\documentclass[11pt,a4paper,german]{book}
\usepackage[T1]{fontenc}
\usepackage[latin9]{inputenc}
\usepackage{amsmath}
\usepackage{amssymb}
\makeatletter

\pdfpageheight\paperheight
\pdfpagewidth\paperwidth

\providecommand{\tabularnewline}{\\}

\usepackage{amsmath} 
\usepackage{amsthm,amscd,amssymb} 
\newtheoremstyle{break}
  {9pt}
  {9pt}
  {\itshape}
  {}
  {\bfseries}
  {}
  {\newline}
  {}
\newtheoremstyle{beispiel}
  {9pt}
  {9pt}
  {\upshape}
  {}
  {\bfseries}
  {}
  {\newline}
  {}

\theoremstyle{break}

\newtheorem{defi}{Definition}[chapter]
\newtheorem{satz}[defi]{Satz}
\newtheorem{hilf}[defi]{Hilfssatz}
\newtheorem{lemma}[defi]{Lemma}
\newtheorem{korollar}[defi]{Korollar}

\theoremstyle{beispiel}
\newtheorem{beisp}[defi]{Beispiel}
\newtheorem{anwendung}[defi]{Anwendung}

\newtheorem{bem}[defi]{Bemerkung}

\makeatother

\usepackage{babel}
\begin{document}
\begin{titlepage}

{\renewcommand{\baselinestretch}{1.1}

\centering

\Huge Testen statistischer Funktionale f\"ur Zweistichprobenprobleme\\[10ex]

\LARGE I n a u g u r a l - D i s s e r t a t i o n\\

\normalsize\ \\

zur\\

Erlangung des Doktorgrades \\

der Mathematisch-Naturwissenschaftlichen Fakult\"at\\

der Heinrich-Heine-Universit\"at D\"usseldorf \\[30ex]

vorgelegt von\\

\LARGE Vladimir Ostrovski \\[1.25ex]

\normalsize

aus Moskau \\[8ex]

D\"usseldorf 2006 \\

}

\newpage\pagestyle{empty}

Aus dem Mathematischen Institut\\

der Heinrich-Heine-Universit\"at D\"usseldorf\\

\vspace*{11cm}

 Gedruckt mit der Genehmigung der \\

Mathematisch-Naturwissenschaftlichen Fakult\"at der\\

Heinrich-Heine-Universit\"at D\"usseldorf

\vspace*{1cm}

$\begin{array}{ll}

\mbox{\hspace{-0.16cm}Referent:} & \mbox{Prof. Dr. A. Janssen, D\"usseldorf} \\

\mbox{\hspace{-0.16cm}Korreferenten:} & \mbox{Prof. Dr. K. Janßen, D\"usseldorf} \\

               & \mbox{Prof. Dr. F. Liese, Rostock} \\

 &   \\

\mbox{\hspace{-0.16cm}Tag der m\"undlichen Pr\"ufung:} & \mbox{26.01.2006} 

\end{array}$ \end{titlepage}\pagenumbering{roman}

\tableofcontents{}

\chapter*{Einführung\pagenumbering{arabic}\addcontentsline{toc}{chapter}{Einführung}}

Viele Testprobleme der parametrischen und nichtparametrischen Statistik
lassen sich mit Hilfe der reellwertigen statistischen Funktionale
formulieren. In dieser Arbeit werden die Probleme des Testens statistischer
Funktionale im Zweistichprobenkontext betrachtet. Es wird stets vorausgesetzt,
dass die beiden Stichproben stochastisch unabhängig sind und die Zufallsvariablen
in jeder Stichprobe unabhängig identisch verteilt sind. Seien $\mathcal{P}\subset\mathcal{M}_{1}\left(\Omega_{1},\mathcal{A}_{1}\right)$
und $\mathcal{Q}\subset\mathcal{M}_{1}\left(\Omega_{2},\mathcal{A}_{2}\right)$
zwei nichtleere Mengen von Wahrscheinlichkeitsmaßen. Wir betrachten
die Familie $\mathcal{P}\otimes\mathcal{Q}:=\left\{ P\otimes Q:P\in\mathcal{P},Q\in\mathcal{Q}\right\} $
von Produktmaßen, die das kanonische Modell für die Zweistichprobenprobleme
darstellt. Jede Abbildung $k:\mathcal{P}\otimes\mathcal{Q}\rightarrow\mathbb{R}$
wird dann als reellwertiges statistisches Funktional bezeichnet. Viele
einseitige und zweiseitige Testprobleme aus der Praxis lassen sich
dann in der Form 
\begin{equation}
\left\{ P\otimes Q\in\mathcal{P}\otimes\mathcal{Q}:k\left(P\otimes Q\right)\leq0\right\} \;\mbox{gegen}\;\left\{ P\otimes Q\in\mathcal{P}\otimes\mathcal{Q}:k\left(P\otimes Q\right)>0\right\} \label{e1}
\end{equation}
 beziehungsweise 
\begin{equation}
\left\{ P\otimes Q\in\mathcal{P}\otimes\mathcal{Q}:k\left(P\otimes Q\right)=0\right\} \;\mbox{gegen}\;\left\{ P\otimes Q\in\mathcal{P}\otimes\mathcal{Q}:k\left(P\otimes Q\right)\neq0\right\} \label{e2}
\end{equation}
 darstellen. An einer beliebigen aber fest gewählten Stelle $P_{0}\otimes Q_{0}\in\mathcal{P}\otimes\mathcal{Q}$
mit $k\left(P_{0}\otimes Q_{0}\right)=0$ werden diese Testprobleme
mit Hilfe der $L_{2}\left(P_{0}\otimes Q_{0}\right)$-differenzierbaren
Kurven von Wahrscheinlichkeitsmaßen parametrisiert. Weiterhin betrachtet
man ein an der Stelle $P_{0}\otimes Q_{0}$ differenzierbares statistisches
Funktional $k:\mathcal{P}\otimes\mathcal{Q}\rightarrow\mathbb{R}$.
Die an der Stelle $P_{0}\otimes Q_{0}$ differenzierbaren statistischen
Funktionale lassen sich lokal entlang der $L_{2}\left(P_{0}\otimes Q_{0}\right)$-differenzierbaren
Kurven linearisieren. Dementsprechend werden die einseitigen und zweiseitigen
Testprobleme (\ref{e1}) und (\ref{e2}) lokal an der Stelle $P_{0}\otimes Q_{0}$
als lineare Testprobleme in einem Hilbertraum $H\subset L_{2}\left(P_{0}\otimes Q_{0}\right)$
aufgefasst. Mit dem Übergang zur Asymptotik im Rahmen der modernen
Theorie von Le Cam erhält man dann die linearen Testprobleme für die
Limesexperimente. Die geeignete Lokalisierung führt dabei zu den unendlich-dimensionalen
Gauß-Shift Experimenten als Limesexperimente. Die linearen Testprobleme
lassen sich für die Gauß-Shift Experimente optimal lösen. Man findet
also für die Limesexperimente die optimalen Tests für die gegebenen
linearen Testprobleme. Mit Hilfe der Theorie von Le Cam und insbesondere
des Hauptsatzes der Testtheorie führt das zu oberen Schranken für
die asymptotischen Gütefunktionen der Testfolgen für die Testprobleme
(\ref{e1}) und (\ref{e2}). Diese oberen Schranken gelten dann für
die asymptotischen Gütefunktionen aller Testfolgen, die einige für
das jeweilige Testproblem sinnvolle Nebenbedingungen erfüllen. Die
Testfolgen, deren asymptotische Gütefunktionen die obere Schranke
erreichen, werden als asymptotisch optimal bezeichnet.

In Kapitel 1 wird an die $L_{2}$-differenzierbaren Kurven von Wahrscheinlichkeitsmaßen
erinnert. Danach werden die differenzierbaren statistischen Funktionale
und ihre kanonischen Gradienten vorgestellt. Der Tangentialraum und
der Tangentenkegel eines Wahrscheinlichkeitsmaßes $P_{0}\in\mathcal{P}$
bezüglich einer Familie $\mathcal{P}$ von Wahrscheinlichkeitsmaßen
werden eingeführt. Insbesondere wird auf die Struktur der Tangentialräume
und Tangentenkegel der Familien von Produktmaßen eingegangen. Die
Kenntnis dieser Struktur ist der Schlüssel zum Testen der differenzierbaren
statistischen Funktionale im Zweistichprobenkontext. Als Letztes werden
in diesem Kapitel die Hilfsmittel vorgestellt, die zum asymptotischen
Vergleich der verschiedenen Testfolgen benötigt werden.

In Kapitel 2 wird zuerst ein Darstellungssatz für den kanonischen
Gradienten eines an der Stelle $P_{0}\otimes Q_{0}$ differenzierbaren
statistischen Funktionals $k:\mathcal{P}\otimes\mathcal{Q}\rightarrow\mathbb{R}$
bewiesen. Danach werden die kanonischen Gradienten für einige wichtige
Funktionale im Zweistichprobenkontext ausgerechnet. Unter anderem
werden die kanonischen Gradienten für das Wilcoxon Funktional und
für die von Mises Funktionale berechnet. Anschließend wird die Differenzierbarkeit
zusammengesetzter statistischer Funktionale untersucht und mit einigen
wichtigen Beispielen begleitet.

Mit $n_{i}$ wird der Umfang der $i$-ten Stichprobe für $i=1,2$
bezeichnet. Der gesamte Stichprobenumfang ist dann $n=n_{1}+n_{2}$.
Als Modell für die unabhängigen Stichproben und die unabhängig identisch
verteilten Zufallsvariablen in jeder Stichprobe betrachtet man das
statistische Experiment 
\[
E_{n}=\left(\Omega_{1}^{n_{1}}\otimes\Omega_{2}^{n_{2}},\mathcal{A}_{1}^{n_{1}}\otimes\mathcal{A}_{2}^{n_{2}},\left\{ P^{n_{1}}\otimes Q^{n_{2}}:P\in\mathcal{P},\,Q\in\mathcal{Q}\right\} \right).
\]
 Seien nun $\left(n_{1}\right)_{n\in\mathbb{N}}$ und $\left(n_{2}\right)_{n\in\mathbb{N}}$
zwei Folgen von natürlichen Zahlen mit $n_{1}+n_{2}=n$ für alle $n\in\mathbb{N}$
und $\lim\limits_{n\rightarrow\infty}\frac{n_{2}}{n}=d$ für ein $d\in\left(0,1\right)$.
In Kapitel 3 wird die Teststatistik $T_{n}$, die auf dem kanonischen
Gradienten eines an der Stelle $P_{0}\otimes Q_{0}$ differenzierbaren
statistischen Funktionals $k:\mathcal{P}\otimes\mathcal{Q}\rightarrow\mathbb{R}$
beruht, motiviert und begründet. Die Zweistichproben-U-Statistik $U_{n_{1},n_{2}}$
mit dem kanonischen Gradienten als Kern besitzt im Vergleich zu den
allgemeinen U-Statistiken eine besonders einfache Gestalt. Die Teststatistik
$T_{n}$ erhält man dann aus $U_{n_{1},n_{2}}$ durch $T_{n}:=\sqrt{n}U_{n_{1},n_{2}}.$
Danach wird an die lokale asymptotische Normalität (LAN) erinnert.
Die lokale asymptotische Normalität der Familien von Produktmaßen
wird untersucht und die zentrale Folge $X_{n}$ wird vorgestellt.
Außerdem wird die gemeinsame asymptotische Verteilung der Teststatistik
$T_{n}$ und der zentralen Folge $X_{n}$ unter dem Wahrscheinlichkeitsmaß
$P_{0}^{n_{1}}\otimes Q_{0}^{n_{2}}$ berechnet. In den letzten zwei
Abschnitten von Kapitel 3 werden die einseitigen und zweiseitigen
Tests auf Basis der Teststatistik $T_{n}$ entwickelt. Die asymptotische
Gütefunktion der Testfolgen der einseitigen und zweiseitigen Tests
wird berechnet. Es wird außerdem gezeigt, dass die Folge der einseitigen
Tests eine asymptotische Maximin-$\alpha$-Testfolge in der Menge
aller asymptotisch Niveau $\alpha$ ähnlichen Testfolgen ist. Der
asymptotisch optimale relative Stichprobenumfang $d=\lim\limits_{n\rightarrow\infty}\frac{n_{2}}{n}$
wird für die Folge der einseitigen Tests bestimmt.

In Kapitel 4 werden die unendlich-dimensionalen Gauß-Shift Experimente
mit einem Hilbertraum als Parameterraum vorgestellt. Danach wird die
schwache Konvergenz von Experimenten eingeführt. Die lokale asymptotische
Normalität wird dann als schwache Konvergenz gegen ein Gauß-Shift
Experiment verallgemeinert. Mit diesen Mitteln werden einige Ergebnisse
aus \cite{Janssen:1999a}, \cite{Janssen:1999b} und \cite{Janssen:2000}
für die Folgen der einseitigen Tests im Einstichprobenkontext erneut
bewiesen. Zusätzlich wird die asymptotische Optimalität der Folgen
der zweiseitigen Tests unter den analogen Voraussetzungen nachgewiesen.
Mit Hilfe der Approximation in dem Limesexperiment wird die asymptotische
Optimalität der Testfolgen auf Basis des kanonischen Gradienten für
die einseitigen und zweiseitigen Testprobleme unter schwachen Voraussetzungen
an den Tangentenkegel gezeigt. Anschließend wird die verallgemeinerte
Theorie der lokalen asymptotischen Normalität auf die Folgen der Experimente,
die als Modelle für die zwei unabhängigen Stichproben dienen, übertragen.
Das Problem des Testens statistischer Funktionale im Zweistichprobenkontext
lässt sich elegant behandeln, indem man das Skalarprodukt in dem Tangentialraum
wechselt und dann auf die Lösung der entsprechenden Einstichprobenprobleme
zurückgreift. Die am Anfang des Kapitels gezeigte Ergebnisse für das
Testen der differenzierbaren statistischen Funktionale im Einstichprobenkontext
lassen sich auf die Situation der Zweistichproben übertragen.

Kapitel 5 widmet sich der Anwendung der Theorie des Testens statistischer
Funktionale für die Zweistichprobenprobleme. Auf die Stärken und Schwierigkeiten
der Anwendung wird hingewiesen. Der Begriff der asymptotischen Äquivalenz
der Testfolgen wird eingeführt. Anhand der Beispiele werden die Techniken
und Methoden vorgestellt, die bei dem Nachweis der asymptotischen
Äquivalenz der Testfolgen besonders hilfreich und nützlich sind. Für
das Wilcoxon Funktional wird gezeigt, dass der wohl bekannte Wilcoxon-Rangsummentest
bereits der asymptotisch optimale Test ist, falls man die Nullhypothese
auf die Produktmaße $P^{2}$ mit dem stetigen Wahrscheinlichkeitsmaß
$P\in\mathcal{M}_{1}\left(\mathbb{R},\mathcal{B}\right)$ einschränkt.
Das Testen der Erwartungswerte zweier unabhängigen Stichproben führt
zur Betrachtung der Summen zweier von Mises Funktionale. Das Problem
des Testens des Erwartungswertes einer randomisierten Summe kann als
Motivation zur Betrachtung des Produktes zweier von Mises Funktionale
dienen. Deswegen werden die Summe und das Produkt zweier von Mises
Funktionale ausführlich behandelt. Die letzte Anwendung widmet sich
wieder dem Testen des Wilcoxon Funktionals. Es gibt sehr umfangreiche
Literatur zu diesem für die Praxis wichtigen Testproblem, vgl. insbesondere
\cite{Hajek:1999}, \cite{Lehmann:1986}, \cite{Janssen:1997}, \cite{Janssen:1999a},
\cite{Janssen:2000} und \cite{Xie:2002}. Hier wird die allgemeine
Fragestellung ohne üblichen Einschränkungen der Nullhypothese betrachtet.
Es werden berechenbare asymptotisch optimale Testfolgen für das einseitige
und zweiseitige Testproblem hergeleitet. Die asymptotische Optimalität
wird mit Hilfe der in dieser Arbeit entwickelten Theorie und der funktionellen
Delta Methode bewiesen. Das besondere Augenmerk gilt der Bestimmung
des kritischen Wertes. Die entwickelten Tests beruhen im Unterschied
zu dem Wilcoxon-Rangsummentest auf den Vergleichen der Realisierungen
der reellwertigen Zufallsvariablen.\\
\\
\\
\emph{Mein herzlicher Dank gilt meinem Lehrer Prof. Dr. Arnold Janssen
sowohl für viele anregende Gespräche als auch für das interessante
und herausfordernde Thema und die ausgezeichnete Betreuung. Ein herzlicher
Dank gebührt Herrn Prof. Dr. K. Janßen und Herrn Prof. Dr. F. Liese
für die Übernahme und Erstellung der weiteren Gutachten. Ich bedanke
mich bei den anderen Mitarbeitern des Mathematischen Instituts der
HHU Düsseldorf für Ihre Unterstützung. Mein tiefer Dank gilt meiner
Frau, meinen Eltern und meinen Großeltern für die Liebe, Unterstützung
und Zuneigung, die sie mir in meinem Leben zuteil werden ließen.}

\chapter{Grundlagen}

In diesem Kapitel wird zunächst die $L_{2}$-Differenzierbarkeit der
Kurven von Wahrscheinlichkeitsmaßen eingeführt. Der Tangentenkegel
und der Tangentialraum eines Wahrscheinlichkeitsmaßes $P\in\mathcal{P}$
bzgl. der Familie $\mathcal{P}$ werden vorgestellt. Danach wird der
Differentialkalkül für statistische Funktionale entwickelt. Dabei
steht der kanonische Gradient eines statistischen Funktionals im Mittelpunkt.
Anschließend werden die Tangentenkegel und die Tangentialräume der
Familien der Produktmaße untersucht. Im letzten Abschnitt werden die
Hilfsmittel zum asymptotischen Vergleich der Testfolgen bereitgestellt.

\section{$L_{2}$-Differenzierbarkeit und Tangentialräume}

Es seien $(\Omega,\mathcal{A})$ ein Messraum und $\mathcal{M}_{1}\left(\Omega,\mathcal{A}\right)$
die Menge aller Wahrscheinlichkeitsmaße auf $(\Omega,\mathcal{A}).$
Jede nichtleere Teilmenge $\mathcal{P}\subset\mathcal{M}_{1}\left(\Omega,\mathcal{A}\right)$
heißt nichtparametrische Familie von Wahrscheinlichkeitsmaßen.\begin{defi}[Dichtequotient]\label{Dichtequotient}Seien
$P,Q\in\mathcal{M}_{1}\left(\Omega,\mathcal{A}\right)$ Wahrscheinlichkeitsmaße.
Jede messbare Funktion $L:\left(\Omega,\mathcal{A}\right)\rightarrow\left(\overline{\mathbb{R}}_{\geq0},\overline{\mathcal{B}}\mid_{\mathbb{R}_{\geq0}}\right)$
mit
\begin{eqnarray*}
P(A) & = & \int_{A}L\,dQ+P(A\cap\{L=\infty\})
\end{eqnarray*}
für alle $A\in\mathcal{A}$ und $Q(L<\infty)=1$ heißt ein Dichtequotient
von $P$ bzgl. $Q$. \end{defi}\begin{bem}

Sei $\mu$ ein $\sigma$-endliches Maß auf $\left(\Omega,\mathcal{A}\right)$
mit $P,Q\ll\mu.$ Setze $f_{P}=\frac{dP}{d\mu}$ und $f_{Q}=\frac{dQ}{d\mu}.$
Dann ist
\begin{equation}
L:=\frac{f_{P}}{f_{Q}}\mathbf{1}_{\left\{ f_{Q}>0\right\} }+\infty\mathbf{1}_{\left\{ f_{Q}=0,\:f_{P}>0\right\} }\label{DQ}
\end{equation}
eine Version des Dichtequotienten von $P$ bzgl. $Q$.\end{bem}\begin{bem}Der
Dichtequotient von $P$ bzgl. $Q$ ist $(P+Q)$-f.ü. eindeutig bestimmt
(vgl. \cite{Witting:1985}, Seite 112, Satz 1.110). Deswegen wird
die Notation $\frac{dP}{dQ}$ für den Dichtequotient von $P$ bzgl.
$Q$ verwendet. \end{bem} \begin{defi}[$L_{2}$-Differenzierbarkeit]\label{L_2_diff}Es
seien $\varepsilon>0$ und $\mathcal{P}\subset\mathcal{M}_{1}\left(\Omega,\mathcal{A}\right)$
eine nichtleere Teilmenge. Eine Abbildung $\left(-\varepsilon,\varepsilon\right)\rightarrow\mathcal{P},t\mapsto P_{t}$
heißt eine $L_{2}\left(P_{0}\right)$-differenzierbare Kurve in $\mathcal{P}$
mit Tangente $g\in L_{2}\left(P_{0}\right)$, falls für $t\rightarrow0$
gilt
\begin{equation}
\left\Vert 2\left(\left(\frac{dP_{t}}{dP_{0}}\right)^{\frac{1}{2}}-1\right)-tg\right\Vert _{L_{2}\left(P_{0}\right)}=o\left(\left|t\right|\right)\label{L2-diffbarkeit1}
\end{equation}
und
\begin{equation}
P_{t}\left(\left\{ \frac{dP_{t}}{dP_{0}}=\infty\right\} \right)=o\left(t^{2}\right).\label{L2-diffbarkeit2}
\end{equation}
\end{defi}Zu (\ref{L2-diffbarkeit1}) und (\ref{L2-diffbarkeit2})
sind äquivalent
\begin{equation}
\left\Vert 2\left(\left(\frac{dP_{t}}{d\mu}\right)^{\frac{1}{2}}-\left(\frac{dP_{0}}{d\mu}\right)^{\frac{1}{2}}\right)-tg\left(\frac{dP_{0}}{d\mu}\right)^{\frac{1}{2}}\right\Vert _{L_{2}(\mu)}=o\left(\left|t\right|\right)\label{'L2-diffbarkeit1}
\end{equation}
und
\begin{equation}
\int\mathbf{1}_{\left\{ \frac{dP_{t}}{dP_{0}}=\infty\right\} }\frac{dP_{t}}{d\mu}d\mu=o\left(t^{2}\right),\label{'L2-diffbarkeit2}
\end{equation}
falls die Familie $\left\{ P_{t}:\left|t\right|<\varepsilon\right\} $
durch ein $\sigma$-endliches Maß $\mu$ dominiert ist. \begin{bem}\label{LBem}\renewcommand{\labelenumi}{\alph{enumi})}
\begin{enumerate}
\item Die Eigenschaft der $L_{2}\left(P_{0}\right)$-Differenzierbarkeit
ist unabhängig von der speziellen Version des Dichtequotienten von
$P_{t}$ bzgl. $P_{0}$.
\item Die Tangente $g\in L_{2}\left(P_{0}\right)$ ist $P_{0}$-f.ü. eindeutig
bestimmt. Jede Version von $g$ bezeichnet man deshalb als Tangente
an die $L_{2}\left(P_{0}\right)$-differenzier\-ba\-re Kurve $t\mapsto P_{t}$
in $\mathcal{P}$.
\item Aus der $L_{2}\left(P_{0}\right)$-Differenzierbarkeit folgen die
$L_{1}\left(P_{0}\right)$-Stetigkeit und die $L_{2}\left(P_{0}\right)$-Stetigkeit,
d.h. die Gültigkeit von $\left\Vert \frac{dP_{t}}{dP_{0}}-1\right\Vert _{L_{1}\left(P_{0}\right)}=o(1)$
und $\left\Vert \left(\frac{dP_{t}}{dP_{0}}\right)^{\frac{1}{2}}-1\right\Vert _{L_{2}\left(P_{0}\right)}=o(1)$
für $t\rightarrow0$.
\item Zur Abkürzung wird es häufig $L_{2}\left(0\right)$ anstelle von $L_{2}\left(P_{0}\right)$
geschrieben, falls aus dem Zusammenhang klar ist, welches Wahrscheinlichkeitsmaß
gemeint ist.
\end{enumerate}
\end{bem}\begin{bem}In der Literatur findet man gelegentlich die
Definition der einseitig $L_{2}(P_{0})$-differenzierbaren Kurven
$\left[0,\varepsilon\right)\rightarrow\mathcal{P},\,t\mapsto P_{t}$
für ein $\varepsilon>0$, vgl. \cite{Vaart:1988}. Es wird hier auf
die Betrachtung der einseitig differenzierbaren Kurven verzichtet,
weil es für die Modellbildung in dieser Arbeit nicht erforderlich
ist. Eine einseitig $L_{2}(P_{0})$-differenzierbare Kurve $\left[0,\varepsilon\right)\rightarrow\mathcal{P},\,t\mapsto P_{t}$
lässt sich außerdem in $\mathcal{M}_{1}\left(\Omega,\mathcal{A}\right)$
so fortsetzen, dass eine zweiseitig $L_{2}(P_{0})$-differenzierbare
Kurve $\left(-\varepsilon,\varepsilon\right)\rightarrow\mathcal{M}_{1}\left(\Omega,\mathcal{A}\right),\,t\mapsto P_{t}$
entsteht. \end{bem}Umfangreiche und ausführliche Informationen zur
$L_{2}$-Differenzierbarkeit (auch allgemeiner zur $L_{r}$-Differenzierbarkeit
für $r\geq1$) findet man in \cite{Witting:1985} und \cite{Strasser:1998}.
\begin{bem}[Tangentenraum]Sei $P\in\mathcal{M}_{1}\left(\Omega,\mathcal{A}\right)$
ein Wahrscheinlichkeitsmaß. Sei $t\mapsto P_{t}$ eine $L_{2}\left(P_{0}\right)$-differenzierbare
Kurve mit Tangente $g\in L_{2}\left(P_{0}\right)$ und $P_{0}=P$.
Es gilt dann $\int g\,dP_{0}=0$ (zum Beweis vgl. \cite{Witting:1985},
Satz 1.190 und Hilfssatz 1.178). Der Vektorraum 
\[
L_{2}^{(0)}\left(P\right):=\left\{ h\in L_{2}\left(P\right):\int h\,dP=0\right\} 
\]
 wird deswegen als Tangentenraum bezeichnet. Der Vektorraum $L_{2}^{(0)}(P)$
ist ein abgeschlossener Unterraum des Hilbertraums $L_{2}(P)$. \end{bem}Sei
$\mathcal{P}\subset\mathcal{M}_{1}\left(\Omega,\mathcal{A}\right)$
eine nichtparametrische Familie von Wahrscheinlichkeitsmaßen.\begin{satz}[lineare Transformation]\label{lin_trafo}
Sei $t\mapsto P_{t}$ eine $L_{2}(P_{0})$-differenzierbare Kurve
in $\mathcal{P}$ mit Tangente $g$. Für jedes $a\in\mathbb{R}\setminus\{0\}$
ist dann die Kurve $t\mapsto P_{at}$ ebenfalls $L_{2}(P_{0})$-differenzierbar
und zwar mit Tangente $ag$. \end{satz}\begin{proof}Wegen $\int ag\,dP_{0}=a\int g\,dP_{0}=0$
gilt $ag\in L_{2}^{(0)}\left(P_{0}\right).$ Die Bedingung (\ref{L2-diffbarkeit1})
ist erfüllt, denn es gilt $P_{at}\left(\left\{ \frac{dP_{at}}{dP_{0}}=\infty\right\} \right)=o\left(\left(at\right)^{2}\right)=o\left(t^{2}\right).$
Es reicht die Bedingung (\ref{L2-diffbarkeit2}) für die Kurve $t\mapsto P_{at}$
mit der Tangente $ag$ nachzuweisen. Mit der Substitution $s:=at$
erhält man
\begin{eqnarray*}
 &  & \lim_{t\rightarrow0}\left\Vert \frac{2}{t}\left(\left(\frac{dP_{at}}{dP_{0}}\right)^{\frac{1}{2}}-1\right)-ag\right\Vert _{L_{2}\left(P_{0}\right)}\\
\textrm{} & = & \lim_{t\rightarrow0}\left|a\right|\left\Vert \frac{2}{at}\left(\left(\frac{dP_{at}}{dP_{0}}\right)^{\frac{1}{2}}-1\right)-g\right\Vert _{L_{2}(P_{0})}\\
 & = & \left|a\right|\lim_{s\rightarrow0}\left\Vert \frac{2}{s}\left(\left(\frac{dP_{s}}{dP_{0}}\right)^{\frac{1}{2}}-1\right)-g\right\Vert _{L_{2}(P_{0})}\\
 & = & 0.
\end{eqnarray*}
\end{proof}\begin{bem}\label{Tangente0} Sei $P\in\mathcal{P}$ beliebig
aber fest gewählt. Die Kurve $\mathbb{R}\rightarrow\mathcal{P},t\mapsto P$
ist offensichtlich $L_{2}(0)$-differenzierbar mit der Tangente $0\in L_{2}^{(0)}\left(P\right).$\end{bem}\begin{defi}[Tangentenkegel]\label{Tangentenkegel}Die
Menge der Tangenten eines Wahrscheinlichkeitsmaßes $P\in\mathcal{P}$
bzgl. der Familie $\mathcal{P}$ wird definiert als $K(P,\mathcal{P}):=\{g\in L_{2}^{(0)}\left(P\right):$
es existiert eine $L_{2}\left(P_{0}\right)$-differenzier\-ba\-re
Kurve $t\mapsto P_{t}$ in $\mathcal{P}$ mit Tangente $g$ und $P_{0}=P\}$.\end{defi}Mit
Hilfe von Satz \ref{lin_trafo} und Bemerkung \ref{Tangente0} sieht
man, dass $K(P,\mathcal{P})$ ein Kegel ist. Die Menge $K(P,\mathcal{P})$
wird deshalb als Tangentenkegel des Wahrscheinlichkeitsmaßes $P$
bzgl. der Familie $\mathcal{P}$ bezeichnet. \begin{defi}[Tangentialraum]\label{Tangentialraum}Der
$L_{2}(P)$-Abschluss der linearen Hülle eines Tangentenkegels $K(P,\mathcal{P})$
wird mit $T(P,\mathcal{P}):=\overline{\mbox{lin}\,K(P,\mathcal{P})}$
notiert. Man bezeichnet $T(P,\mathcal{P})$ als Tangentialraum des
Wahrscheinlichkeitsmaßes $P$ bzgl. der Familie $\mathcal{P}$. \end{defi}Der
Tangentialraum spiegelt die lokale Struktur eines statistischen Experiments
wider. Für jedes Wahrscheinlichkeitsmaß $P\in\mathcal{P}$ ist der
Tangentialraum $T(P,\mathcal{P})$ ein abgeschlossener Unterraum des
Hilbertraums $L_{2}(P)$.

\section{Differenzierbare statistische Funktionale}

Sei $\mathcal{P}\subset\mathcal{M}_{1}\left(\Omega,\mathcal{A}\right)$
eine nichtparametrische Familie von Wahrscheinlichkeitsmaßen. Jede
Abbildung $k:\mathcal{P}\rightarrow\mathbb{R},P\mapsto k\left(P\right)$
wird als statistisches Funktional bezeichnet. Man interessiert sich
für die statistischen Funktionale, die sich entlang der $L_{2}$-differenzierbaren
Kurven differenzieren lassen. Viele nichtparametrischen Testprobleme
lassen sich zum Beispiel mit Hilfe solcher Funktionale mathematisch
exakt formulieren. \begin{defi}[differenzierbares statistisches Funktional]\label{diff-stat-funk}Ein
statistisches Funktional $k:\mathcal{P}\rightarrow\mathbb{R}$ heißt
differenzierbar an einer Stelle $P\in\mathcal{P}$ mit einem Gradienten
$\dot{k}=\dot{k}(P)\in L_{2}^{(0)}(P)$, falls für jede $L_{2}\left(P_{0}\right)$-differenzierbare
Kurve $t\mapsto P_{t}$ in $\mathcal{P}$ mit Tangente $g$ und $P_{0}=P$
gilt
\begin{equation}
\lim_{t\downarrow0}\frac{1}{t}\left(k\left(P_{t}\right)-k\left(P_{0}\right)\right)=\int\dot{k}g\,dP_{0}.\label{gradient}
\end{equation}
\end{defi}\begin{bem}Der einseitige Grenzwert in (\ref{gradient})
stellt keine Einschränkung dar. Es seien $k$ ein statistisches an
einer Stelle $P\in\mathcal{P}$ differenzierbares Funktional, $t\mapsto P_{t}$
eine $L_{2}\left(P_{0}\right)$-differenzierbare Kurve in $\mathcal{P}$
mit Tangente $g$ und $P_{0}=P$.

Die Kurve $t\mapsto P_{-t}$ ist nach Satz \ref{lin_trafo} ebenfalls
$L_{2}\left(P_{0}\right)$-differenzierbar mit Tangente $-g$ und
$P_{0}=P$. Für die Kurve $t\mapsto P_{-t}$ gilt also 
\begin{eqnarray}
\lim_{t\downarrow0}\frac{1}{t}\left(k\left(P_{-t}\right)-k\left(P_{0}\right)\right) & = & -\int\dot{k}g\,dP\label{beidseitigeGW}
\end{eqnarray}
nach Definition \ref{diff-stat-funk}. Aus (\ref{beidseitigeGW})
ergibt sich
\begin{eqnarray}
\lim_{t\uparrow0}\frac{1}{t}\left(k\left(P_{t}\right)-k\left(P_{0}\right)\right) & = & \lim_{t\downarrow0}\frac{1}{\left(-t\right)}\left(k\left(P_{-t}\right)-k\left(P_{0}\right)\right)\nonumber \\
 & = & -\lim_{t\downarrow0}\frac{1}{t}\left(k\left(P_{-t}\right)-k\left(P_{0}\right)\right)\nonumber \\
 & = & \int\dot{k}g\,dP.\label{GW_von_unten}
\end{eqnarray}
 \end{bem}\begin{bem}Ist ein statistisches Funktional $k:\mathcal{P}\rightarrow\mathbb{R}$
differenzierbar an einer Stelle $P\in\mathcal{P}$, so impliziert
(\ref{gradient}) eine Linearisierung von $k$ an der Stelle $P$
entlang der $L_{2}\left(P_{0}\right)$-differenzierbaren Kurven in
$\mathcal{P}$ mit $P_{0}=P$.\end{bem}\begin{bem}Ist ein statistisches
Funktional $k:\mathcal{P}\rightarrow\mathbb{R}$ differenzierbar an
einer Stelle $P\in\mathcal{P}$, so ist der Gradient von $k$ im Allgemeinen
nicht eindeutig bestimmt. \end{bem}\begin{beisp}\label{nicht_eind_gradient}Sei
$P\in\mathcal{P}$ ein Wahrscheinlichkeitsmaß mit $\{0\}\neq T(P,\mathcal{P})\neq L_{2}^{(0)}(P)$.
Sei $k:\mathcal{P}\rightarrow\mathbb{R}$ ein an der Stelle $P$ differenzierbares
statistisches Funktional mit einem Gradienten $\dot{k}\in L_{2}^{(0)}(P)$.

Mit $T(P,\mathcal{P})^{\perp}$ wird der Orthogonalraum von $T(P,\mathcal{P})$
in $L_{2}^{(0)}(P)$ bezeichnet. Aus $L_{2}^{(0)}(P)=T(P,\mathcal{P})\oplus T(P,\mathcal{P})^{\perp}$
folgt $T(P,\mathcal{P})^{\perp}\neq\{0\}$ wegen Voraussetzung $\{0\}\neq T(P,\mathcal{P})\neq L_{2}^{(0)}(P)$.
Für jedes $h\in T(P,\mathcal{P})^{\perp}$ und für jede $L_{2}\left(P_{0}\right)$-differenzierbare
Kurve $t\mapsto P_{t}$ in $\mathcal{P}$ mit Tangente $g$ und $P_{0}=P$
gilt dann
\[
\lim_{t\downarrow0}\frac{1}{t}\left(k\left(P_{t}\right)-k\left(P_{0}\right)\right)=\int\dot{k}g\,dP=\int\dot{k}g\,dP+\int hg\,dP=\int\left(\dot{k}+h\right)g\,dP.
\]
 Hieraus folgt, dass die Funktion $\dot{k}+h$ für jedes $h\in T(P,\mathcal{P})^{\perp}$
ein Gradient von $k$ an der Stelle $P$ ist.\end{beisp}\begin{satz}[kanonischer Gradient]\label{kanon_Gradient}Ein
statistisches Funktional $k:\mathcal{P}\rightarrow\mathbb{R}$ sei
differenzierbar an einer Stelle $P\in\mathcal{P}.$ Es existiert ein
eindeutiger Gradient $\widetilde{k}\in T(P,\mathcal{P})$, der die
kleinste Norm innerhalb der Menge aller Gradienten von $k$ an der
Stelle $P$ besitzt. Man bezeichnet $\widetilde{k}$ als kanonischen
Gradienten von $k$ an der Stelle $P$.\end{satz}\begin{proof}vgl.
\cite{Janssen:1998}, Seite 185 oder \cite{Ostrovski:2004}, Seite
51, Satz 4.13.\end{proof}\begin{bem}\label{kan_grad_orth_pro}Ist
$\dot{k}$ ein Gradient von $k:\mathcal{P}\rightarrow\mathbb{R}$
an der Stelle $P\in\mathcal{P}$, so lässt sich der kanonische Gradient
$\widetilde{k}$ an der Stelle $P$ als die orthogonale Projektion
von $\dot{k}$ in den Tangentialraum $T(P,\mathcal{P})$ berechnen,
vgl. \cite{Janssen:1998}, Seite 185 oder \cite{Pfanzagl:1985}, Seite
119.\end{bem}

\section{$L_{2}$-Differenzierbarkeit und Tangentialräume der Produktmaße}

Die Produktmaße spielen bei Zweistichprobenproblemen der Statistik
eine besondere Rolle, weil die beiden Stichproben meistens als stochastisch
unabhängig angesehen werden. Die Tangentialräume der Familien der
Produktmaße besitzen eine besondere Struktur, die sich als Schlüssel
zum Testen der differenzierbaren statistischen Funktionale für Zweistichprobenprobleme
erweist. Eine ausführliche Information über die Tangentialräume und
$L_{2}$-Differenzierbarkeit der Familien der Produktmaße findet man
in \cite{Pfanzagl:1985}, \cite{Witting:1985}, \cite{Strasser:1989}
und \cite{Ostrovski:2004}. Allgemeinere Ergebnisse über die Struktur
der Tangenten bei gegebenen Randverteilungen und gemeinsamer Verteilung
zweier Zufallsvariablen findet man in \cite{Janssen:2002} und \cite{Rahnenfuehrer:1999}. 

\begin{satz}[$L_2$-Differenzierbarkeit der Produktmaße]\label{L2Produkt}Sei
$n\in\mathbb{N}$. Es seien $\left(\Omega_{i},\mathcal{A}_{i}\right)$
Messräume und $\mathcal{P}_{i}\subset\mathcal{M}_{1}\left(\Omega_{i},\mathcal{A}_{i}\right)$
nichtparametrische Familien von Wahrscheinlichkeitsmaßen für alle
$i\in\left\{ 1,\ldots,n\right\} $. Die Familie der Produktmaße wird
definiert als $\mathcal{P}:=\left\{ \bigotimes\limits_{i=1}^{n}P_{i}:P_{i}\in\mathcal{P}_{i}\right\} $. 

Dann sind folgende Aussagen äquivalent:
\begin{enumerate}
\item Eine Kurve $t\mapsto P_{t}=\bigotimes\limits_{i=1}^{n}P_{i,t}$ in
$\mathcal{P}$ ist $L_{2}(0)$-differenzierbar.
\item Die Kurven $t\mapsto P_{i,t}$ sind $L_{2}(0)$-differenzierbar für
alle $i\in\left\{ 1,\ldots,n\right\} $.
\end{enumerate}
Für jedes $i\in\left\{ 1,\ldots,n\right\} $ sei $g_{i}$ die Tangente
zu der Kurve $t\mapsto P_{i,t}$ an der Stelle $0$. Die Abbildung
\[
\times_{i=1}^{n}\Omega_{i}\rightarrow\mathbb{R},\,(\omega_{1},\ldots,\omega_{n})\mapsto\sum\limits_{i=1}^{n}g_{i}(\omega_{i})
\]
 ist dann die Tangente zu der Kurve $t\mapsto P_{t}=\bigotimes\limits_{i=1}^{n}P_{i,t}$
an der Stelle $0$.\end{satz}\begin{proof}

$\mathbf{1.\Rightarrow2.}$ Sei $i\in\left\{ 1,\ldots,n\right\} $.
Die $i$-te kanonische Projektion 
\[
\pi_{i}:\times_{i=1}^{n}\Omega_{i}\rightarrow\Omega_{i},(\omega_{1},\ldots,\omega_{n})\mapsto\omega_{i}
\]
 ist eine messbare Abbildung. Die Kurve $t\mapsto P_{t}^{\pi_{i}}=P_{i,t}$
ist $L_{2}(0)$-differenzierbar nach \cite{Witting:1985}, Seite 178,
Satz 1.193.

$\mathbf{2.\Rightarrow1.}$ Diese Implikation entspricht \cite{Witting:1985},
Seite 176, Satz 1.191.\end{proof} \begin{bem}\label{prod_bem}Sei
$t\mapsto P_{t}$ eine $L_{2}\left(P_{0}\right)$-differenzierbare
Kurve in $\mathcal{P}$ mit Tangente $g\in L_{2}^{(0)}\left(P_{0}\right)$.
Sei $S:\left(\Omega,\mathcal{A}\right)\rightarrow\left(\mathbb{R},\mathcal{B}\right)$
eine Zufallsvariable. Die Kurve $t\mapsto P_{t}^{S}$ ist dann $L_{2}\left(P_{0}^{S}\right)$-differenzierbar
mit Tangente 
\[
h(s)=E_{P_{0}}\left(g\left|S=s\right.\right).
\]
Zum Beweis vergleiche \cite{Witting:1985}, Seite 176, Satz 1.191.
\end{bem}\begin{satz}\label{produkt_tangentialraum} Sei $n\in\mathbb{N}$.
Für jedes $i\in\left\{ 1,\ldots,n\right\} $ sei $\mathcal{P}_{i}\subset\mathcal{M}_{1}\left(\Omega_{i},\mathcal{A}_{i}\right)$
eine nichtparametrische Familie von Wahrscheinlichkeitsmaßen. Die
Familie der Produktmaße $\mathcal{P}$ sei definiert wie in Satz \ref{L2Produkt}.
Für jedes $P=\bigotimes\limits_{i=1}^{n}P_{i}\in\mathcal{P}$ gilt
dann:
\begin{enumerate}
\item Der Tangentenkegel $K(P,\mathcal{P})$ ist die Menge der Abbildungen
\\
$(\omega_{1},\ldots,\omega_{n})\mapsto\sum\limits_{i=1}^{n}g_{i}(\omega_{i})$
mit $g_{i}\in K(P_{i},\mathcal{P}_{i})$ für jedes $i\in\left\{ 1,\ldots,n\right\} $.
\item Der Tangentialraum $T(P,\mathcal{P})$ ist die Menge der Abbildungen
\\
$(\omega_{1},\ldots,\omega_{n})\mapsto\sum\limits_{i=1}^{n}g_{i}(\omega_{i})$
mit $g_{i}\in T(P_{i},\mathcal{P}_{i})$ für jedes $i\in\left\{ 1,\ldots,n\right\} $.
\end{enumerate}
\end{satz}\begin{proof} Es ist leicht zu zeigen, dass die Menge
der Abbildungen 
\[
(\omega_{1},\ldots,\omega_{n})\mapsto\sum\limits_{i=1}^{n}g_{i}(\omega_{i})
\]
 mit $g_{i}\in T(P_{i},\mathcal{P}_{i})$ für jedes $i\in\left\{ 1,\ldots,n\right\} $
ein bzgl. der $L_{2}\left(\otimes_{i=1}^{n}P_{i}\right)$-Norm abgeschlossener
Vektorraum ist. Die erste Aussage ergibt sich dann unmittelbar aus
Satz \ref{L2Produkt}. Mit Hilfe von Definition \ref{Tangentialraum}
folgt die zweite Aussage. \end{proof}\begin{bem}\label{eind_tangente}Sei
$g\in T\left(\otimes_{i=1}^{n}P_{i},\mathcal{P}\right)$ eine Tangente.
Dann existieren eindeutige Abbildungen $g_{i}\in T\left(P_{i},\mathcal{P}_{i}\right)$
für alle $i\in\left\{ 1,\ldots,n\right\} $ mit $g=\sum\limits_{i=1}^{n}g_{i}\circ\pi_{i}$.
Die Existenz folgt unmittelbar aus Satz \ref{produkt_tangentialraum}.
Die Eindeutigkeit ist zu zeigen. Es gelte 
\[
g=\sum\limits_{i=1}^{n}g_{i}\circ\pi_{i}=\sum\limits_{i=1}^{n}\widetilde{g}_{i}\circ\pi_{i}
\]
mit $g_{i},\widetilde{g}_{i}\in T\left(P_{i},\mathcal{P}_{i}\right)$
für alle $i\in\left\{ 1,\ldots,n\right\} $. Man erhält zunächst
\[
0=\int\left(\sum\limits_{i=1}^{n}g_{i}\circ\pi_{i}-\sum\limits_{i=1}^{n}\widetilde{g}_{i}\circ\pi_{i}\right)^{2}d\otimes_{i=1}^{n}P_{i}=\sum\limits_{i=1}^{n}\int\left(g_{i}-\widetilde{g}_{i}\right)^{2}dP_{i},
\]
weil $T\left(P_{i},\mathcal{P}_{i}\right)\subset L_{2}^{(0)}\left(P_{i}\right)$
für alle $i\in\left\{ 1,\ldots,n\right\} $ gilt. Hieraus folgt $g_{i}=\widetilde{g}_{i}$
$P_{i}$-f.s. für alle $i\in\left\{ 1,\ldots,n\right\} $. \end{bem}

\section{Asymptotisches Vergleichen der Testfolgen}

In der Statistik stellt sich häufig die Frage, wie man zwei Testfolgen
für ein Testproblem vergleichen kann. Es wird zum Beispiel oft gezeigt,
dass zwei Testfolgen die gleiche asymptotische Güte besitzen. Solche
Testfolgen nennt man asymptotisch äquivalent. In diesem Abschnitt
werden verschiedene Hilfsmittel zur Lösung dieser Probleme vorgestellt.\begin{defi}[Norm der totalen Variation]Es
seien $P$ und $Q$ zwei Wahrscheinlichkeitsmaße auf einem Messraum
$\left(\Omega,\mathcal{A}\right)$. Außerdem sei $\mu$ ein $\sigma$-endliches
Maß auf $\left(\Omega,\mathcal{A}\right)$ mit $P,Q\ll\mu$. Dann
heißt 
\[
\left\Vert P-Q\right\Vert :=\frac{1}{2}\int\left|\frac{dP}{d\mu}-\frac{dQ}{d\mu}\right|d\mu
\]
die Norm der totalen Variation.\end{defi}\begin{bem}Die Norm der
totalen Variation $\left\Vert P-Q\right\Vert $ ist unabhängig von
dem speziellen dominierenden Maß $\mu$. Insbesondere kann das Maß
$P+Q$ als dominierendes Maß gewählt werden. \end{bem}\begin{bem}Für
alle Wahrscheinlichkeitsmaße $P,Q$ auf einem Messraum $\left(\Omega,\mathcal{A}\right)$
gilt
\[
\left\Vert P-Q\right\Vert =\sup_{B\in\mathcal{A}}\left|P\left(B\right)-Q\left(B\right)\right|.
\]
 Zum Beweis vergleiche \cite{Witting:1985}, Seite 136.\end{bem}

\begin{hilf}\label{hsatz1}Seien $P_{n},Q_{n}\in\mathcal{M}_{1}\left(\Omega_{n},\mathcal{A}_{n}\right)$
zwei Folgen von Wahrscheinlichkeitsmaßen mit
\[
\lim_{n\rightarrow\infty}\left\Vert P_{n}-Q_{n}\right\Vert =0.
\]
Sei $\varphi_{n}:\left(\Omega_{n},\mathcal{A}_{n}\right)\rightarrow\left(\left[0,1\right],\left.\mathcal{B}\right|_{\left[0,1\right]}\right)$
eine Testfolge, so dass der Grenzwert $\lim\limits_{n\rightarrow\infty}E_{P_{n}}\left(\varphi_{n}\right)$
existiert. Dann gilt $\lim\limits_{n\rightarrow\infty}E_{P_{n}}\left(\varphi_{n}\right)=\lim\limits_{n\rightarrow\infty}E_{Q_{n}}\left(\varphi_{n}\right).$
\end{hilf}\begin{proof}Sei $\mu_{n}$ ein $\sigma$-endliches Maß
mit $P_{n},Q_{n}\ll\mu_{n}$. Dann gilt
\begin{eqnarray*}
\left|E_{P_{n}}\left(\varphi_{n}\right)-E_{Q_{n}}\left(\varphi_{n}\right)\right| & \leq & \int\left|\varphi_{n}\frac{dP_{n}}{d\mu_{n}}-\varphi_{n}\frac{dQ_{n}}{d\mu_{n}}\right|d\mu_{n}\\
 & = & \int\varphi_{n}\left|\frac{dP_{n}}{d\mu_{n}}-\frac{dQ_{n}}{d\mu_{n}}\right|d\mu_{n}
\end{eqnarray*}
\begin{eqnarray*}
 & \leq & \int\left|\frac{dP_{n}}{d\mu_{n}}-\frac{dQ_{n}}{d\mu_{n}}\right|d\mu_{n}\\
 & = & 2\left\Vert P_{n}-Q_{n}\right\Vert \rightarrow0\;\mbox{für}\;n\rightarrow\infty.
\end{eqnarray*}
\end{proof}\begin{satz}[asymptotisch äquivalente Testfolgen]\label{asympt_aquivalent_testfolge}
Sei $\left(P_{n}\right)_{n\in\mathbb{N}}$ eine Folge der Wahrscheinlichkeitsmaße.
Es seien $\left(T_{1n}\right)_{n\in\mathbb{N}}$ und $\left(T_{2n}\right)_{n\in\mathbb{N}}$
zwei Folgen der Teststatistiken, so dass 
\[
\lim\limits_{n\rightarrow\infty}P_{n}\left(\left\{ \left|T_{1n}-T_{2n}\right|>\varepsilon\right\} \right)=0
\]
 für alle $\varepsilon>0$ gilt. Außerdem seien $\left(\varphi_{in}\right)_{n\in\mathbb{N}}$
Testfolgen für $i=1,2$ mit 
\[
\varphi_{in}=\left\{ \begin{array}{cccc}
1 &  & >\\
\gamma_{in} & T_{in} & = & c_{in}\\
0 &  & <
\end{array}\right..
\]
Es gelte $\lim\limits_{n\rightarrow\infty}c_{in}=c$ für $i=1,2$
und $c\in\mathbb{R}$. Die Verteilung $P_{n}^{T_{1n}}$ konvergiere
schwach gegen ein Wahrscheinlichkeitsmaß $\mu$. Ist $\mu\left(\left\{ c\right\} \right)=0$,
so folgt 
\[
\lim_{n\rightarrow\infty}E_{P_{n}}\left(\left|\varphi_{1n}-\varphi_{2n}\right|\right)=0.
\]
 \end{satz}\begin{proof}vgl. \cite{Janssen:1998}, Seite 96 oder
\cite{Witting-Noelle:1970}, Seite 58. \end{proof} Beim Nachweis
der asymptotischen Äquivalenz der Testfolgen für die Familien von
Wahrscheinlichkeitsmaßen ist der Begriff der Benachbartheit unverzichtbar.\begin{defi}Es
seien $P_{n},Q_{n}\in\mathcal{M}_{1}\left(\Omega_{n},\mathcal{A}_{n}\right)$
zwei Folgen von Wahrscheinlichkeitsmaßen. Die Folge $\left(Q_{n}\right)_{n\in\mathbb{N}}$
heißt benachbart zu $\left(P_{n}\right)_{n\in\mathbb{N}}$, wenn für
alle Folgen $A_{n}\in\mathcal{A}_{n}$ von Mengen aus $\lim\limits_{n\rightarrow\infty}P_{n}\left(A_{n}\right)=0$
bereits $\lim\limits_{n\rightarrow\infty}Q_{n}\left(A_{n}\right)=0$
folgt. Die Benachbartheit von $\left(Q_{n}\right)_{n\in\mathbb{N}}$
zu $\left(P_{n}\right)_{n\in\mathbb{N}}$ wird mit $Q_{n}\triangleleft P_{n}$
notiert.\end{defi}Ausführliche Informationen zur Benachbartheit findet
man in \cite{Strasser:1985b}, \cite{Hajek:1999}, \cite{Vaart:1998}
und \cite{LeCam:2000}. Die asymptotische Äquivalenz der Testfolgen
wird häufig zuerst nur für eine Folge $\left(P_{n}\right)_{n\in\mathbb{N}}$
von Wahrscheinlichkeitsmaßen gezeigt, vgl. Satz \ref{asympt_aquivalent_testfolge}.
Diese Aussage lässt sich dann unter schwachen Voraussetzungen auf
jede zu $\left(P_{n}\right)_{n\in\mathbb{N}}$ benachbarte Folge von
Wahrscheinlichkeitsmaßen übertragen. \begin{lemma}\label{benach_imp_asymp_eq}Es
seien $P_{n},Q_{n}\in\mathcal{M}_{1}\left(\Omega_{n},\mathcal{A}_{n}\right)$
zwei Folgen von Wahrscheinlichkeitsmaßen mit $Q_{n}\triangleleft P_{n}$.
Außerdem seien $\left(\varphi_{in}\right)_{n\in\mathbb{N}}$ zwei
Testfolgen für $i=1,2$ mit $\lim\limits_{n\rightarrow\infty}E_{P_{n}}\left(\left|\varphi_{1n}-\varphi_{2n}\right|\right)=0.$
Dann gilt
\[
\lim_{n\rightarrow\infty}E_{Q_{n}}\left(\left|\varphi_{1n}-\varphi_{2n}\right|\right)=0.
\]
\end{lemma}\begin{proof}Sei $\varepsilon>0$. Aus der Konvergenz
$\lim\limits_{n\rightarrow\infty}E_{P_{n}}\left(\left|\varphi_{1n}-\varphi_{2n}\right|\right)=0$
folgt dann $\lim\limits_{n\rightarrow\infty}P_{n}\left(\left\{ \left|\varphi_{1n}-\varphi_{2n}\right|>\varepsilon\right\} \right)=0.$
Die Benachbartheit $Q_{n}\triangleleft P_{n}$ impliziert
\[
\lim\limits_{n\rightarrow\infty}Q_{n}\left(\left\{ \left|\varphi_{1n}-\varphi_{2n}\right|>\varepsilon\right\} \right)=0.
\]
 Man erhält außerdem die Abschätzung
\begin{eqnarray*}
\textrm{} &  & E_{Q_{n}}\left(\left|\varphi_{1n}-\varphi_{2n}\right|\right)\\
 & = & \int\left|\varphi_{1n}-\varphi_{2n}\right|\mathbf{1}_{\left\{ \left|\varphi_{1n}-\varphi_{2n}\right|>\varepsilon\right\} }dQ_{n}+\int\left|\varphi_{1n}-\varphi_{2n}\right|\mathbf{1}_{\left\{ \left|\varphi_{1n}-\varphi_{2n}\right|\leq\varepsilon\right\} }dQ_{n}\\
 & \leq & Q_{n}\left(\left\{ \left|\varphi_{1n}-\varphi_{2n}\right|>\varepsilon\right\} \right)+\varepsilon Q_{n}\left(\left\{ \left|\varphi_{1n}-\varphi_{2n}\right|\leq\varepsilon\right\} \right)\\
 & \leq & Q_{n}\left(\left\{ \left|\varphi_{1n}-\varphi_{2n}\right|>\varepsilon\right\} \right)+\varepsilon.
\end{eqnarray*}
Somit ergibt sich
\[
\limsup\limits_{n\rightarrow\infty}E_{Q_{n}}\left(\left|\varphi_{1n}-\varphi_{2n}\right|\right)\leq\lim\limits_{n\rightarrow\infty}Q_{n}\left(\left\{ \left|\varphi_{1n}-\varphi_{2n}\right|>\varepsilon\right\} \right)+\varepsilon=\varepsilon
\]
 für alle $\varepsilon>0$. Hieraus folgt $\lim\limits_{n\rightarrow\infty}E_{Q_{n}}\left(\left|\varphi_{1n}-\varphi_{2n}\right|\right)=0.$
\end{proof}

\chapter{Differentiation statistischer Funktionale für das Zweistichprobenproblem}

In diesem Abschnitt werden die Differenzierbarkeitseigenschaften statistischer
Funktionale für das Zweistichprobenproblem untersucht. Die besondere
Aufmerksamkeit wird auf die Struktur der Tangentialräume und des kanonischen
Gradienten gerichtet. Es werden leistungsfähige Werkzeuge entwickelt,
mit deren Hilfe die Differenzierbarkeit eines statistischen Funktionals
und die Gestalt des kanonischen Gradienten bestimmt werden können.
Außerdem werden die kanonischen Gradienten für viele wichtige statistische
Funktionale ausgerechnet. 

Es seien $\left(\Omega_{1},\mathcal{A}_{1}\right)$ und $\left(\Omega_{2},\mathcal{A}_{2}\right)$
Messräume, $\mathcal{P}\subset\mathcal{M}_{1}\left(\Omega_{1},\mathcal{A}_{1}\right)$
und $\mathcal{Q}\subset\mathcal{M}_{1}\left(\Omega_{2},\mathcal{A}_{2}\right)$
nichtparametrische Familien von Wahrscheinlichkeitsmaßen.\begin{defi}Die
Familie der Produktmaße $\mathcal{P}\otimes\mathcal{Q}\subset\mathcal{M}_{1}\left(\Omega_{1}\otimes\Omega_{2},\mathcal{A}_{1}\otimes\mathcal{A}_{2}\right)$
wird definiert als $\mathcal{P}\otimes\mathcal{Q}:=\left\{ P\otimes Q:\,P\in\mathcal{P},\,Q\in\mathcal{Q}\right\} .$
Die Abbildungen $\pi_{1}:\Omega_{1}\otimes\Omega_{2}\rightarrow\Omega_{1},\left(\omega_{1},\omega_{2}\right)\mapsto\omega_{1}$
und $\pi_{2}:\Omega_{1}\otimes\Omega_{2}\rightarrow\Omega_{2},\left(\omega_{1},\omega_{2}\right)\mapsto\omega_{2}$
bezeichnen die kanonischen Projektionen.\end{defi} \begin{satz}[Struktur eines kanonischen Gradienten]\label{umrechnung}Sei
$k:\mathcal{P}\otimes\mathcal{Q}\rightarrow\mathbb{R}$ ein statistisches
Funktional, das an einer Stelle $P_{0}\otimes Q_{0}\in\mathcal{P}\otimes\mathcal{Q}$
differenzierbar ist. Es seien $\dot{k}\in L_{2}\left(P_{0}\otimes Q_{0}\right)$
ein Gradient von $k$ an der Stelle $P_{0}\otimes Q_{0}$ und $\widetilde{k}\in L_{2}^{(0)}\left(P_{0}\otimes Q_{0}\right)$
der kanonische Gradient von $k$ an der Stelle $P_{0}\otimes Q_{0}$.
Dann gilt:
\begin{enumerate}
\item Die statistischen Funktionale $k_{1}:\mathcal{P}\rightarrow\mathbb{R},P\mapsto k\left(P\otimes Q_{0}\right)$
und $k_{2}:\mathcal{Q}\rightarrow\mathbb{R},Q\mapsto k\left(P_{0}\otimes Q\right)$
sind differenzierbar an der Stelle $P_{0}$ bzw. $Q_{0}$.
\item Die Abbildung $\dot{k}_{1}:\omega_{1}\mapsto\int\dot{k}\left(\omega_{1},\omega_{2}\right)dQ_{0}(\omega_{2})$
ist ein Gradient von $k_{1}$ an der Stelle $P_{0}$ und die Abbildung
$\dot{k}_{2}:\omega_{2}\mapsto\int\dot{k}\left(\omega_{1},\omega_{2}\right)dP_{0}(\omega_{1})$
ist ein Gradient von $k_{2}$ an der Stelle $Q_{0}$. 
\item Die Abbildung $\widetilde{k}_{1}:\omega_{1}\mapsto\int\widetilde{k}\left(\omega_{1},\omega_{2}\right)dQ_{0}(\omega_{2})$
ist der kanonische Gradient von $k_{1}$ an der Stelle $P_{0}$. Die
Abbildung $\widetilde{k}_{2}:\omega_{2}\mapsto\int\widetilde{k}\left(\omega_{1},\omega_{2}\right)dP_{0}(\omega_{1})$
ist der kanonische Gradient von $k_{2}$ an der Stelle $Q_{0}$. Außerdem
gilt $\widetilde{k}=\widetilde{k}_{1}\circ\pi_{1}+\widetilde{k}_{2}\circ\pi_{2}$.
\end{enumerate}
\end{satz}\begin{proof}

Sei $t\mapsto P_{t}\otimes Q_{t}$ eine $L_{2}(0)$-differenzierbare
Kurve in $\mathcal{P}\otimes\mathcal{Q}$ mit Tangente $h\in L_{2}^{(0)}\left(P_{0}\otimes Q_{0}\right)$.
Nach Satz \ref{produkt_tangentialraum} existieren die Tangenten $h_{1}\in L_{2}^{(0)}\left(P_{0}\right)$
und $h_{2}\in L_{2}^{(0)}\left(Q_{0}\right)$ mit $h=h_{1}\circ\pi_{1}+h_{2}\circ\pi_{2}$
. Die Kurven $t\mapsto P_{t}$ und $t\mapsto Q_{t}$ sind dann $L_{2}(0)$-differenzierbar
mit Tangenten $h_{1}$ und $h_{2}$. Mit erneuter Anwendung von Satz
\ref{produkt_tangentialraum} erhält man, dass die Kurve $t\mapsto P_{t}\otimes Q_{0}$
ebenfalls $L_{2}\left(0\right)$-differenzierbar ist und zwar mit
Tangente $h_{1}\circ\pi_{1}\in L_{2}^{(0)}\left(P_{0}\otimes Q_{0}\right).$
Es gilt somit
\begin{eqnarray*}
\lim_{t\downarrow0}\frac{1}{t}\left(k_{1}\left(P_{t}\right)-k_{1}\left(P_{0}\right)\right) & = & \lim_{t\downarrow0}\frac{1}{t}\left(k\left(P_{t}\otimes Q_{0}\right)-k\left(P_{0}\otimes Q_{0}\right)\right)\\
 & = & \int h_{1}(\omega_{1})\dot{k}(\omega_{1},\omega_{2})\,dP_{0}\otimes Q_{0}\left(\omega_{1},\omega_{2}\right)\\
 & = & \int h_{1}(\omega_{1})\underbrace{\left[\int\dot{k}(\omega_{1},\omega_{2})dQ_{0}(\omega_{2})\right]}_{=:\dot{k}_{1}(\omega_{1})}dP_{0}(\omega_{1}).
\end{eqnarray*}
Es bleibt $\dot{k}_{1}\in L_{2}\left(P_{0}\right)$ zu zeigen. Die
Jensensche Ungleichung liefert
\begin{eqnarray*}
\int\dot{k}_{1}^{2}dP_{0} & = & \int\left[\int\dot{k}(\omega_{1},\omega_{2})dQ_{0}(\omega_{2})\right]^{2}dP_{0}(\omega_{1})\\
 & \leq & \int\int\dot{k}(\omega_{1},\omega_{2})^{2}dQ_{0}(\omega_{2})dP_{0}(\omega_{1})\\
 & = & \int\dot{k}^{2}dP_{0}\otimes Q_{0}<\infty.
\end{eqnarray*}
Somit sind Teile 1 und 2 dieses Satzes bewiesen. 

Der kanonische Gradient $\widetilde{k}\in T\left(P_{0}\otimes Q_{0},\mathcal{P}\otimes\mathcal{Q}\right)$
besitzt nach Satz \ref{produkt_tangentialraum} die Darstellung $\widetilde{k}=g_{1}\circ\pi_{1}+g_{2}\circ\pi_{2}$,
wobei $g_{1}\in T\left(P_{0},\mathcal{P}\right)$ und $g_{2}\in T\left(Q_{0},\mathcal{Q}\right)$
nach Bemerkung \ref{eind_tangente} eindeutig sind. Hieraus folgt
\[
\widetilde{k}_{1}(\omega_{1})=\int\left(g_{1}(\omega_{1})+g_{2}(\omega_{2})\right)dQ_{0}(\omega_{2})=g_{1}(\omega_{1}).
\]
Nach Teil 2 dieses Satzes ist die Funktion $\widetilde{k}_{1}=g_{1}$
ein Gradient von $k_{1}$ an der Stelle $P_{0}$. Nach Satz \ref{kanon_Gradient}
und Bemerkung \ref{kan_grad_orth_pro} folgt, dass die Abbildung $\widetilde{k}_{1}=g_{1}\in T\left(P_{0},\mathcal{P}\right)$
bereits der kanonische Gradient von $k_{1}$ an der Stelle $P_{0}$
ist.\end{proof}\begin{bem}Die Abbildung $\widetilde{k}_{i}$ in
Satz \ref{umrechnung} ergibt sich als bedingte Erwartung 
\[
\widetilde{k}_{i}=E_{P_{0}\otimes Q_{0}}\left(\widetilde{k}\mid\pi_{i}\right)
\]
 für $i=1,2.$ \end{bem}\begin{defi}\label{volle_familie}Eine nichtparametrische
Familie $\mathcal{P}\subset\mathcal{M}_{1}\left(\Omega,\mathcal{A}\right)$
von Wahrscheinlichkeitsmaßen heißt voll, wenn $T\left(P,\mathcal{P}\right)=L_{2}^{(0)}\left(P\right)$
für alle $P\in\mathcal{P}$ gilt. \end{defi}\begin{bem}Sei $\mu$
ein $\sigma$-endliches Maß auf $\left(\Omega,\mathcal{A}\right)$.
Die Familie 
\[
\mathcal{P}_{\mu}:=\left\{ P\in\mathcal{M}_{1}\left(\Omega,\mathcal{A}\right):P\ll\mu\right\} 
\]
 ist voll. Zum Beweis vergleiche Satz \ref{konstruktion}  oder \cite{Ostrovski:2004},
Korollar 4.17.\end{bem}\begin{beisp}[von Mises Funktional]\label{von_mises}Sei
$k:\mathcal{P}\otimes\mathcal{Q}\rightarrow\mathbb{R},P\otimes Q\mapsto\int h\,dP\otimes Q$
ein statistisches an einer Stelle $P_{0}\otimes Q_{0}$ differenzierbares
Funktional und sei $\dot{k}=h-\int h\,dP_{0}\otimes Q_{0}$ ein Gradient
von $k$ in $P_{0}\otimes Q_{0}$. Dies tritt zum Beispiel dann auf,
wenn $h$ eine beschränkte Funktion ist, vgl. \cite{Witting:1985},
Seite 164, Satz 1.179. Allgemeinere Funktionale mit $h\in L_{2}\left(P_{0}\otimes Q_{0}\right)$
sind mit Hilfe von Lemma \ref{LemmaIundH} zu untersuchen. Nach Satz
\ref{umrechnung} gilt dann 
\begin{eqnarray*}
\dot{k}_{1}\left(\omega_{1}\right) & = & \int\left[h\left(\omega_{1},\omega_{2}\right)-\int h\left(\omega_{1},\omega_{2}\right)dP_{0}\otimes Q_{0}\left(\omega_{1},\omega_{2}\right)\right]dQ_{0}\left(\omega_{2}\right)\\
 & = & \int h\left(\omega_{1},\omega_{2}\right)dQ_{0}\left(\omega_{2}\right)-\int h\left(\omega_{1},\omega_{2}\right)dP_{0}\otimes Q_{0}\left(\omega_{1},\omega_{2}\right),
\end{eqnarray*}
\begin{eqnarray*}
\dot{k}_{2}\left(\omega_{2}\right) & = & \int\left[h\left(\omega_{1},\omega_{2}\right)-\int h\left(\omega_{1},\omega_{2}\right)dP_{0}\otimes Q_{0}\left(\omega_{1},\omega_{2}\right)\right]dP_{0}\left(\omega_{1}\right)\\
 & = & \int h\left(\omega_{1},\omega_{2}\right)dP_{0}\left(\omega_{1}\right)-\int h\left(\omega_{1},\omega_{2}\right)dP_{0}\otimes Q_{0}\left(\omega_{1},\omega_{2}\right).
\end{eqnarray*}

Sind die Familien $\mathcal{P}$ und $\mathcal{Q}$ voll, dann folgt
$\widetilde{k}_{1}=\dot{k}_{1}\in T\left(P_{0},\mathcal{P}\right)$
und $\widetilde{k}_{2}=\dot{k}_{2}\in T\left(Q_{0},\mathcal{Q}\right)$
nach Satz \ref{kanon_Gradient} und Bemerkung \ref{kan_grad_orth_pro}.
Der kanonische Gradient des Funktionals $k$ an der Stelle $P_{0}\otimes Q_{0}$
ergibt sich dann als 
\[
\widetilde{k}\left(\omega_{1},\omega_{2}\right)=\int h\left(\omega_{1},\omega_{2}\right)dQ_{0}\left(\omega_{2}\right)+\int h\left(\omega_{1},\omega_{2}\right)dP_{0}\left(\omega_{1}\right)-2k\left(P_{0}\otimes Q_{0}\right).
\]
 \end{beisp}\begin{anwendung}[Wilcoxon Funktional]\label{wilcoxon}Es
seien $\left(\Omega_{1},\mathcal{A}_{1}\right)=\left(\Omega_{2},\mathcal{A}_{2}\right)=\left(\mathbb{R},\mathcal{B}\right)$
und $h:=1_{\left\{ \left(x,y\right)\in\mathbb{R}^{2}:x\geq y\right\} }$.
Die Familien $\mathcal{P}\subset\mathcal{M}\left(\Omega_{1},\mathcal{A}_{1}\right)$
und $\mathcal{Q}\subset\mathcal{M}\left(\Omega_{2},\mathcal{A}_{2}\right)$
seien voll. Das statistische Funktional $k:P\otimes Q\mapsto\int h\,dP\otimes Q$
ist dann differenzierbar an jeder Stelle $P_{0}\otimes Q_{0}\in\mathcal{P}\otimes\mathcal{Q}$
mit dem kanonischen Gradienten
\begin{eqnarray*}
\widetilde{k}(x,y) & = & \int1_{\left\{ \left(x,y\right)\in\mathbb{R}^{2}:x\geq y\right\} }dQ_{0}(y)+\int1_{\left\{ \left(x,y\right)\in\mathbb{R}^{2}:x\geq y\right\} }dP_{0}(x)\\
 &  & -2k\left(P_{0}\otimes Q_{0}\right)\\
 & = & Q_{0}\left(\left(-\infty,x\right]\right)+P_{0}\left(\left[y,+\infty\right)\right)-2k\left(P_{0}\otimes Q_{0}\right)\\
 & = & F_{Q_{0}}(x)+1-F_{P_{0}}^{-}(y)-2k\left(P_{0}\otimes Q_{0}\right),
\end{eqnarray*}
wobei $F_{P}^{-}$ die linksseitige Verteilungsfunktion eines Wahrscheinlichkeitsmaßes
$P$ bezeichnet.\end{anwendung}\begin{bem}Es seien $n\in\mathbb{N}$
und $\left(\Omega_{i},\mathcal{A}_{i}\right)$ Messräume für alle
$i\in\left\{ 1,\ldots,n\right\} $. Die Familie der Produktmaße wird
definiert als $\mathcal{P}=\left\{ \otimes_{i=1}^{n}P_{i}:\,P_{i}\in\mathcal{P}_{i}\right\} ,$
wobei $\mathcal{P}_{i}\subset\mathcal{M}_{1}\left(\Omega_{i},\mathcal{A}_{i}\right)$
für jedes $i\in\left\{ 1,\ldots,n\right\} $ eine nichtparametrische
Familie von Wahrscheinlichkeitsmaßen ist. Der Satz \ref{umrechnung}
lässt sich dann analog auf die statistischen Funktionale $k:\mathcal{P}\rightarrow\mathbb{R}$
verallgemeinern.\end{bem}\begin{beisp}[invariante Funktionale]\label{survival}Seien
$\mathcal{P},\mathcal{Q}\subset\mathcal{M}_{1}(\mathbb{R},\mathcal{B})$.
Die statistischen Funktionale der Form 
\begin{equation}
k:\mathcal{P}\otimes\mathcal{Q}\rightarrow\mathbb{R},P\otimes Q\mapsto\int h\left(F_{Q}(x)\right)dP(x)\label{inv_funk}
\end{equation}
 finden eine breite Anwendung in der Praxis, wobei $h$ eine $\mathcal{B}$-messbare
Funktion ist und $F_{Q}$ die Verteilungsfunktion von $Q$ bezeichnet.
Insbesondere ist das Wilcoxon Funktional von diesem Typ für $h=\mbox{Id}_{\mathbb{R}}$.
Die statistischen Funktionale der Form (\ref{inv_funk}) sind invariant
unter streng isotonen Transformationen im folgenden Sinne. Sei $T:\mathbb{R}\rightarrow\mathbb{R}$
eine streng isotone Abbildung. Für alle $P\otimes Q\in\mathcal{P}\otimes\mathcal{Q}$
gilt dann 
\begin{eqnarray*}
k\left(P^{T}\otimes Q^{T}\right) & = & \int h\left(F_{Q^{T}}\left(x\right)\right)dP^{T}\left(x\right)\\
 & = & \int h\left(\int\mathbf{1}_{\left(-\infty,x\right]}\left(y\right)dQ^{T}\left(y\right)\right)dP^{T}\left(x\right)\\
 & = & \int h\left(\int\mathbf{1}_{\left(-\infty,T\left(x\right)\right]}\left(T\left(y\right)\right)dQ\left(y\right)\right)dP\left(x\right)\\
 & = & \int h\left(\int\mathbf{1}_{\left(-\infty,x\right]}\left(y\right)dQ\left(y\right)\right)dP\left(x\right)\\
 & = & k\left(P\otimes Q\right).
\end{eqnarray*}

Die $\mathcal{B}$-messbare Abbildung $h:[0,1]\rightarrow\mathbb{R}$
sei differenzierbar mit beschränkter Ableitung, d.h. $\left|\dot{h}(t)\right|\leq c$
für alle $t\in\left[0,1\right]$ und ein $c\in\left(0,+\infty\right)$.
Hieraus folgt, dass die Abbildung $h$ ebenfalls beschränkt ist. Es
wird nun gezeigt, dass das statistische Funktional $k$ dann an jeder
Stelle $P_{0}\otimes Q_{0}\in\mathcal{P}\otimes\mathcal{Q}$ differenzierbar
ist und die Abbildung 
\begin{eqnarray*}
\dot{k}\left(x,y\right) & = & h\left(F_{Q_{0}}(x)\right)+\int\dot{h}\left(F_{Q_{0}}(s)\right)1_{\left[y,+\infty\right)}(s)dP_{0}(s)\\
 &  & -k\left(P_{0}\otimes Q_{0}\right)-\int\dot{h}\left(F_{Q_{0}}(s)\right)F_{Q_{0}}\left(s\right)dP_{0}\left(s\right)
\end{eqnarray*}
 ein Gradient von $k$ an der Stelle $P_{0}\otimes Q_{0}$ ist. 

Sei $t\mapsto P_{t}\otimes Q_{t}$ eine $L_{2}\left(0\right)$-differenzierbare
Kurve in $\mathcal{P}\otimes\mathcal{Q}$ mit Tangente $g\in L_{2}^{(0)}\left(P_{0}\otimes Q_{0}\right)$.
Nach Satz \ref{L2Produkt} sind die Kurven $t\mapsto P_{t}$ und $t\mapsto Q_{t}$
ebenfalls $L_{2}\left(0\right)$-differenzierbar mit Tangenten $g_{1}\in L_{2}^{(0)}\left(P_{0}\right)$
und $g_{2}\in L_{2}^{(0)}\left(Q_{0}\right)$. Es gilt außerdem $g=g_{1}\circ\pi_{1}+g_{2}\circ\pi_{2}$.
Zuerst erhält man mit den einfachen Umformungen
\begin{eqnarray*}
\textrm{} &  & \frac{1}{t}\left(k\left(P_{t}\otimes Q_{t}\right)-k\left(P_{0}\otimes Q_{0}\right)\right)\\
 & = & \frac{1}{t}\left(\int h\left(F_{Q_{t}}(x)\right)dP_{t}(x)-\int h\left(F_{Q_{0}}(x)\right)dP_{0}(x)\right)\\
 & = & \frac{1}{t}\left(\int h\left(F_{Q_{t}}(x)\right)dP_{t}(x)-\int h\left(F_{Q_{0}}(x)\right)dP_{t}(x)\right)\\
 &  & +\frac{1}{t}\left(\int h\left(F_{Q_{0}}(x)\right)dP_{t}(x)-\int h\left(F_{Q_{0}}(x)\right)dP_{0}(x)\right).
\end{eqnarray*}
Für den zweiten Summanden ergibt sich 
\begin{eqnarray*}
 &  & \lim_{t\rightarrow0}\frac{1}{t}\left(\int h\left(F_{Q_{0}}(x)\right)dP_{t}(x)-\int h\left(F_{Q_{0}}(x)\right)dP_{0}(x)\right)\\
 & = & \left.\frac{d}{dt}\int h\left(F_{Q_{0}}(x)\right)dP_{t}(x)\right|_{t=0}\\
 & = & \int h\left(F_{Q_{0}}(x)\right)g_{1}\left(x\right)\,dP_{0}(x)
\end{eqnarray*}
nach \cite{Witting:1985}, Seite 164, Satz 1.179, weil $h$ eine beschränkte
Funktion ist. Es bleibt den Grenzwert
\[
\lim_{t\rightarrow0}\frac{1}{t}\left(\int h\left(F_{Q_{t}}(x)\right)dP_{t}(x)-\int h\left(F_{Q_{0}}(x)\right)dP_{t}(x)\right)
\]
 für den ersten Summanden zu bestimmen. Sei $\left(t_{n}\right)_{n\in\mathbb{N}}$
eine beliebige Nullfolge. Es existiert ein Wahrscheinlichkeitsmaß
$\mu$ mit $P_{0}\ll\mu$ und $P_{t_{n}}\ll\mu$ für alle $n\in\mathbb{N}$.
Man erhält dann
\begin{eqnarray*}
\textrm{} &  & \frac{1}{t_{n}}\left(\int h\left(F_{Q_{t_{n}}}(x)\right)dP_{t_{n}}(x)-\int h\left(F_{Q_{0}}(x)\right)dP_{t_{n}}(x)\right)\\
 & = & \int\frac{1}{t_{n}}\left(h\left(F_{Q_{t_{n}}}(x)\right)-h\left(F_{Q_{0}}(x)\right)\right)dP_{t_{n}}(x)
\end{eqnarray*}
\begin{eqnarray*}
\textrm{} & = & \int\frac{1}{t_{n}}\left(h\left(F_{Q_{t_{n}}}(x)\right)-h\left(F_{Q_{0}}(x)\right)\right)\frac{dP_{t_{n}}}{d\mu}(x)d\mu(x).
\end{eqnarray*}
Um den Grenzwertübergang unter dem Integral zu rechtfertigen, wird
der Integrand
\[
f_{n}(x):=\frac{1}{t_{n}}\left(h\left(F_{Q_{t_{n}}}(x)\right)-h\left(F_{Q_{0}}(x)\right)\right)\frac{dP_{t_{n}}}{d\mu}(x)
\]
zunächst abgeschätzt. Nach dem Mittelwertsatz der Differentialrechnung
erhält man 
\[
\frac{1}{t_{n}}\left(h\left(F_{Q_{t_{n}}}(x)\right)-h\left(F_{Q_{0}}(x)\right)\right)=\frac{1}{t_{n}}\left(F_{Q_{t_{n}}}(x)-F_{Q_{0}}(x)\right)\dot{h}\left(\xi_{n}\right)
\]
 für ein $\xi_{n}\in\left[0,1\right]$. Es gilt außerdem
\begin{eqnarray*}
\textrm{} &  & \lim_{n\rightarrow\infty}\left|\frac{1}{t_{n}}\left(F_{Q_{t_{n}}}(x)-F_{Q_{0}}(x)\right)\right|\\
 & = & \left|\left.\frac{d}{dt}\int\mathbf{1}_{\left(-\infty,x\right]}dQ_{t}\right|_{t=0}\right|\\
 & = & \left|\int\mathbf{1}_{\left(-\infty,x\right]}g_{2}\,dQ_{0}\right|\\
 & \leq & \int\mathbf{1}_{\left(-\infty,x\right]}\left|g_{2}\right|\,dQ_{0}\\
 & \leq & \int\left|g_{2}\right|\,dQ_{0}
\end{eqnarray*}
nach \cite{Witting:1985}, Seite 164, Satz 1.179. Für jedes $\varepsilon>0$
existiert also ein $n_{0}\in\mathbb{N}$, so dass
\[
\left|\frac{1}{t_{n}}\left(F_{Q_{t_{n}}}(x)-F_{Q_{0}}(x)\right)\right|\leq\int\left|g_{2}\right|\,dQ_{0}+\varepsilon
\]
für alle $n\geq n_{0}$ gilt. Insgesamt erhält man für alle $n\geq n_{0}$
die Majorante
\begin{eqnarray}
\left|f_{n}(x)\right| & = & \left|\frac{1}{t_{n}}\left(h\left(F_{Q_{t_{n}}}(x)\right)-h\left(F_{Q_{0}}(x)\right)\right)\right|\frac{dP_{t_{n}}}{d\mu}(x)\nonumber \\
 & = & \left|\frac{1}{t_{n}}\left(F_{Q_{t_{n}}}(x)-F_{Q_{0}}(x)\right)\dot{h}\left(\xi_{n}\right)\right|\frac{dP_{t_{n}}}{d\mu}(x)\nonumber \\
 & \leq & \left|\frac{1}{t_{n}}\left(F_{Q_{t_{n}}}(x)-F_{Q_{0}}(x)\right)\right|\frac{dP_{t_{n}}}{d\mu}(x)\max_{\xi\in\left[0,1\right]}\left|\dot{h}\left(\xi\right)\right|\nonumber \\
 & \leq & \left(\int\left|g_{2}\right|\,dQ_{0}+\varepsilon\right)\max_{\xi\in\left[0,1\right]}\left|\dot{h}\left(\xi\right)\right|\frac{dP_{t_{n}}}{d\mu}(x)\nonumber \\
 & =: & g_{n}(x).\label{eq:pr1}
\end{eqnarray}
Es gilt außerdem $\frac{dP_{t_{n}}}{d\mu}\rightarrow\frac{dP_{0}}{d\mu}$
für $n\rightarrow\infty$ in $L_{1}\left(\mu\right)$-Norm nach Bemerkung
\ref{LBem}, weil die $L_{2}\left(P_{0}\right)$-Differenzierbarkeit
der Kurve $t\mapsto P_{t}$ insbesondere die Stetigkeit der Abbildung
$t\mapsto\frac{dP_{t}}{dP_{0}}$ an der Stelle $t=0$ bzgl. der $L_{1}\left(P_{0}\right)$-Norm
impliziert, vgl. \cite{Witting:1985}, Seite 174. Hieraus folgt 
\begin{equation}
g_{n}(x)\rightarrow\left(\int\left|g_{2}\right|\,dQ_{0}+\varepsilon\right)\max_{\xi\in\left[0,1\right]}\left|\dot{h}\left(\xi\right)\right|\frac{dP_{0}}{d\mu}(x)\label{eq:pr2}
\end{equation}
in $L_{1}\left(\mu\right)$-Norm für $n\rightarrow\infty$. Außerdem
gilt 
\begin{equation}
\int g_{n}(x)d\mu(x)=\left(\int\left|g_{2}\right|\,dQ_{0}+\varepsilon\right)\max_{\xi\in\left[0,1\right]}\left|\dot{h}\left(\xi\right)\right|.\label{eq:pr3}
\end{equation}
Mit (\ref{eq:pr1}), (\ref{eq:pr2}) und (\ref{eq:pr3}) erhält man
nach dem Satz von Pratt die Konvergenz
\begin{eqnarray*}
\lim_{n\rightarrow\infty}\int f_{n}(x) & = & \lim_{n\rightarrow\infty}\int\frac{1}{t_{n}}\left(h\left(F_{Q_{t_{n}}}(x)\right)-h\left(F_{Q_{0}}(x)\right)\right)\frac{dP_{t_{n}}}{d\mu}(x)d\mu(x)\\
 & = & \int\dot{h}\left(F_{Q_{0}}(x)\right)\left[\int\mathbf{1}_{\left(-\infty,x\right]}(y)g_{2}(y)dQ_{0}(y)\right]\frac{dP_{0}}{d\mu}(x)d\mu(x)\\
 & = & \int\dot{h}\left(F_{Q_{0}}(x)\right)\left[\int\mathbf{1}_{\left(-\infty,x\right]}(y)g_{2}(y)\,dQ_{0}(y)\right]dP_{0}(x),
\end{eqnarray*}
vgl. \cite{elstrodt:2002}, Seite 258-259. Mit dem Satz von Fubini
ergibt sich 
\begin{eqnarray*}
 &  & \int\dot{h}\left(F_{Q_{0}}(x)\right)\left[\int\mathbf{1}_{\left(-\infty,x\right]}g_{2}\,dQ_{0}\right]dP_{0}(x)\\
 & = & \int\left[\int\dot{h}\left(F_{Q_{0}}(x)\right)\mathbf{1}_{\left(-\infty,x\right]}(y)g_{2}(y)dQ_{0}(y)\right]dP_{0}(x)\\
 & = & \int\left[\int\dot{h}\left(F_{Q_{0}}(x)\right)\mathbf{1}_{\left[y,+\infty\right)}(x)g_{2}(y)dP_{0}(x)\right]dQ_{0}(y)\\
 & = & \int g_{2}(y)\left[\int\dot{h}\left(F_{Q_{0}}(x)\right)\mathbf{1}_{\left[y,+\infty\right)}(x)dP_{0}(x)\right]dQ_{0}(y).
\end{eqnarray*}
 Insgesamt erhält man
\begin{eqnarray*}
\textrm{} &  & \lim_{t\rightarrow0}\frac{1}{t}\left(k\left(P_{t}\otimes Q_{t}\right)-k\left(P_{0}\otimes Q_{0}\right)\right)\\
 & = & \int h\left(F_{Q_{0}}(x)\right)g_{1}\left(x\right)\,dP_{0}(x)\\
 &  & +\int\left[\int\dot{h}\left(F_{Q_{0}}(s)\right)\mathbf{1}_{\left[y,+\infty\right)}(s)dP_{0}(s)\right]g_{2}(y)dQ_{0}(y)\\
 & = & \int\left(g_{1}\left(x\right)+g_{2}\left(y\right)\right)\hat{k}(x,y)dP_{0}\otimes Q_{0}(x,y),
\end{eqnarray*}
wobei $\widehat{k}(x,y):=h\left(F_{Q_{0}}(x)\right)+\int\dot{h}\left(F_{Q_{0}}(s)\right)\mathbf{1}_{\left[y,+\infty\right)}(s)dP_{0}(s)$
ist. Das statistische Funktional $k$ ist also differenzierbar an
der Stelle $P_{0}\otimes Q_{0}$ mit dem Gradienten $\dot{k}:=\widehat{k}-E_{P_{0}\otimes Q_{0}}\left(\widehat{k}\right)$
nach Definition \ref{diff-stat-funk}. Mit dem Satz von Fubini erhält
man dann
\begin{eqnarray*}
\dot{k}\left(x,y\right) & = & h\left(F_{Q_{0}}(x)\right)+\int\dot{h}\left(F_{Q_{0}}(s)\right)1_{\left[y,+\infty\right)}(s)dP_{0}(s)\\
 &  & -\int\left(h\left(F_{Q_{0}}(x)\right)+\int\dot{h}\left(F_{Q_{0}}(s)\right)1_{\left[y,+\infty\right)}(s)dP_{0}(s)\right)dP_{0}\otimes Q_{0}(x,y)\\
 & = & h\left(F_{Q_{0}}(x)\right)+\int\dot{h}\left(F_{Q_{0}}(s)\right)1_{\left[y,+\infty\right)}(s)dP_{0}(s)\\
 &  & -\int h\left(F_{Q_{0}}(x)\right)dP_{0}(x)-\int\int\dot{h}\left(F_{Q_{0}}(s)\right)1_{\left[y,+\infty\right)}(s)dP_{0}(s)dQ_{0}(y)\\
 & = & h\left(F_{Q_{0}}(x)\right)+\int\dot{h}\left(F_{Q_{0}}(s)\right)1_{\left[y,+\infty\right)}(s)dP_{0}(s)\\
 &  & -k\left(P_{0}\otimes Q_{0}\right)-\int\dot{h}\left(F_{Q_{0}}(s)\right)\left[\int1_{\left(-\infty,s\right]}(y)dQ_{0}(y)\right]dP_{0}(s)\\
 & = & h\left(F_{Q_{0}}(x)\right)+\int\dot{h}\left(F_{Q_{0}}(s)\right)1_{\left[y,+\infty\right)}(s)dP_{0}(s)\\
 &  & -k\left(P_{0}\otimes Q_{0}\right)-\int\dot{h}\left(F_{Q_{0}}(s)\right)F_{Q_{0}}\left(s\right)dP_{0}\left(s\right).
\end{eqnarray*}
 Der kanonische Gradient $\widetilde{k}$ an der Stelle $P_{0}\otimes Q_{0}$
lässt sich mit Satz \ref{umrechnung} berechnen, falls die Tangentialräume
$T\left(P_{0},\mathcal{P}\right)$ und $T\left(Q_{0},\mathcal{Q}\right)$
bekannt sind. \end{beisp}Das nachfolgende Lemma liefert eine hinreichende
Bedingung für die Differenzierbarkeit der von Mises Funktionale, die
einen breiten Einzug in die Theorie der asymptotischen Statistik gefunden
hat, vgl. \cite{Bickel:1993}, \cite{Ibragimov:1981}, \cite{Vaart:1988}
und \cite{Witting:1985}. Die Voraussetzungen des Lemmas sind hinreichend
schwach und meistens leicht nachzuprüfen. \begin{lemma}\label{LemmaIundH}Sei
$\mathcal{P}\!\subset\!\mathcal{M}_{1}\left(\Omega,\mathcal{A}\right)$
eine nichtparametrische Familie von Wahrscheinlichkeitsmaßen. Es seien
$t\mapsto P_{t}$ eine $L_{2}(P_{0})$-differenzierbare Kurve in $\mathcal{P}$
mit Tangente $g\in L_{2}^{(0)}\left(P_{0}\right)$ und $T:(\Omega,\mathcal{A})\rightarrow(\mathbb{R},\mathcal{B})$
eine Statistik mit $E_{P_{0}}(T^{2})<\infty$.

Falls $\limsup\limits_{t\downarrow0}E_{P_{t}}(T^{2})<\infty$ gilt,
so ist die Abbildung $t\mapsto E_{P_{t}}(T)$ rechtsseitig differenzierbar
an der Stelle $0$ und es gilt
\begin{equation}
\left.\frac{d}{dt}E_{P_{t}}(T)\right|_{t=0^{+}}=E_{P_{0}}\left(Tg\right)\label{IundH}
\end{equation}
(vgl. \cite{Ibragimov:1981}, Seite 73, Theorem 7.3. oder \cite{Bickel:1993},
Seite 457, Proposition 2).\end{lemma}\begin{proof} vgl. \cite{Vaart:1988},
Seite 167, Lemma 5.21 oder \cite{Ostrovski:2004}, Seite 11, Lemma
2.1.\end{proof} Zur Anwendung des obigen Lemmas in der nichtparametrischen
Situation braucht man eine mit der $L_{2}$-Differenzierbarkeit verträgliche
Metrik auf $\mathcal{M}_{1}\left(\Omega,\mathcal{A}\right)$, d.h.
eine $L_{2}(0)$-differenzierbare Kurve $t\mapsto P_{t}$ in $\mathcal{M}_{1}\left(\Omega,\mathcal{A}\right)$
ist dann bereits stetig an der Stelle $0$ bzgl. dieser Metrik. Die
Hellingerdistanz ist eine solche Metrik, die sich für die nichtparametrischen
Untersuchungen in der Statistik gut eignet.\begin{defi}[Hellingerdistanz]Seien
$P$ und $Q$ Wahrscheinlichkeitsmaße auf $(\Omega,\mathcal{A})$
mit $P,Q\ll\nu$ für ein $\sigma$-endliches Maß $\nu$. Dann heißt
\[
d(P,Q)=\left(\frac{1}{2}\int\left(\left(\frac{dP}{d\nu}\right)^{\frac{1}{2}}-\left(\frac{dQ}{d\nu}\right)^{\frac{1}{2}}\right)^{2}d\nu\right)^{\frac{1}{2}}
\]
 die Hellingerdistanz von $P$ und $Q$.\end{defi}Mit Hilfe der Hellingerdistanz
kann man die Voraussetzungen von Lemma \ref{LemmaIundH} für die nichtparametrischen
Familien von Wahrscheinlichkeitsmaßen formulieren. Im nachfolgenden
Satz werden drei äquivalente Möglichkeiten gegeben, wie man die für
die Differenzierbarkeit hinreichende Voraussetzung aus Lemma \ref{LemmaIundH}
im nichtparametrischen Kontext ausdrücken kann. \begin{satz}[topologische Äquivalenz]\label{topologie} Es
seien $\mathcal{P}\subset\mathcal{M}_{1}\left(\Omega,\mathcal{A}\right)$
eine nichtparametrische Familie von Wahrscheinlichkeitsmaßen und $T:(\Omega,\mathcal{A})\rightarrow(\mathbb{R},\mathcal{B})$
eine Statistik mit $E_{P_{0}}(T^{2})<\infty$ für ein $P_{0}\in\mathcal{P}$.
Es sind folgenden Aussagen äquivalent:\renewcommand{\labelenumi}{(\alph{enumi})}

\begin{enumerate}
\item Für alle Folgen $\left(P_{n}\right)_{n\in\mathbb{N}}$ aus $\mathcal{P}$
mit $\lim\limits_{n\rightarrow\infty}d\left(P_{n},P_{0}\right)=0$
gilt
\[
\limsup_{n\rightarrow\infty}E_{P_{n}}\left(T^{2}\right)<\infty.
\]
\item Es existiert ein $\varepsilon>0$ mit der Eigenschaft
\[
\sup\left\{ E_{P}\left(T^{2}\right):\,P\in\mathcal{P}\;\mbox{mit}\quad d\left(P,P_{0}\right)<\varepsilon\right\} <\infty.
\]
\item Es existieren ein $\varepsilon>0$ und ein $K>0$, so dass 
\[
E_{P}\left(T^{2}\right)<K
\]
für alle $P\in\mathcal{P}$ mit $d\left(P,P_{0}\right)<\varepsilon$
gilt.
\end{enumerate}
\end{satz}\begin{proof}

Die Äquivalenz zwischen $(b)$ und $(c)$ ist leicht zu zeigen.

\textbf{$(a)\Rightarrow(b)$} Angenommen, die Aussage $(b)$ ist falsch.
Dann existiert eine Folge $\left(P_{n}\right)_{n\in\mathbb{N}}\subset\mathcal{P}$
mit $\lim\limits_{n\rightarrow\infty}d\left(P_{n},P_{0}\right)=0$
und $\lim\limits_{n\rightarrow\infty}E_{P_{n}}\left(T^{2}\right)=\infty.$
Hieraus folgt $\limsup_{n\rightarrow\infty}E_{P_{n}}\left(T^{2}\right)=\infty.$
Dies ist ein Widerspruch zu $(a)$.

\textbf{$(b)\Rightarrow(a)$} Sei $\left(P_{n}\right)_{n\in\mathbb{N}}$
eine Folge mit $\lim\limits_{n\rightarrow\infty}d\left(P_{n},P_{0}\right)=0$.
Für jedes $\varepsilon>0$ existiert dann ein $n_{0}\in\mathbb{N},$
so dass $d\left(P_{n},P_{0}\right)<\varepsilon$ für alle $n\geq n_{0}$
gilt. Man erhält somit 
\[
\limsup_{n\rightarrow\infty}E_{P_{n}}\left(T^{2}\right)<\sup\left\{ E_{P}\left(T^{2}\right):\,P\in\mathcal{P}\;\mbox{mit}\quad d\left(P,P_{0}\right)<\varepsilon\right\} <\infty.
\]
\end{proof}Es seien $k_{1}:\mathcal{P}\rightarrow\mathbb{R}$ und
$k_{2}:\mathcal{Q}\rightarrow\mathbb{R}$ statistische Funktionale.
Oft verwendet man die zusammengesetzten statistischen Funktionale
von der Form
\begin{equation}
\mathcal{P}\otimes\mathcal{Q}\rightarrow\mathbb{R},\:P\otimes Q\mapsto f\left(k_{1}\left(P\right),k_{2}\left(Q\right)\right),\label{zusamm_gesetz_stat_funk}
\end{equation}
 wobei $f:\mathbb{R}^{2}\rightarrow\mathbb{R}$ eine Abbildung ist.
Die Differenzierbarkeit der zusammengesetzten statistischen Funktionale
der Form (\ref{zusamm_gesetz_stat_funk}) wird nun untersucht.\begin{satz}\label{kett_reg_stat_funk}Sei
$k_{1}:\mathcal{P}\rightarrow\mathbb{R}$ ein statistisches an einer
Stelle $P_{0}\in\mathcal{P}$ differenzierbares Funktional mit einem
Gradienten $\dot{k}_{1}\in L_{2}^{(0)}\left(P_{0}\right)$ und dem
kanonischen Gradienten $\widetilde{k}_{1}\in T\left(P_{0},\mathcal{P}\right)$.
Sei $k_{2}:\mathcal{Q}\rightarrow\mathbb{R}$ ebenfalls ein statistisches
an einer Stelle $Q_{0}\in\mathcal{Q}$ differenzierbares Funktional
mit einem Gradienten $\dot{k}_{2}\in L_{2}^{(0)}\left(Q_{0}\right)$
und dem kanonischen Gradienten $\widetilde{k}_{2}\in T\left(Q_{0},\mathcal{Q}\right)$.
Sei $U\subset\mathbb{R}^{2}$ eine offene Menge mit $\left\{ k_{1}\left(P\right):\,P\in\mathcal{P}\right\} \times\left\{ k_{2}\left(Q\right):\,Q\in\mathcal{Q}\right\} \subset U$.
Sei $f:U\rightarrow\mathbb{R}$ eine an der Stelle $\left(k_{1}\left(P_{0}\right),k_{2}\left(Q_{0}\right)\right)$
differenzierbare Funktion. Es bezeichne 
\[
c_{1}=\left.\frac{\partial}{\partial x}f\left(x,y\right)\right|_{(x,y)=\left(k_{1}\left(P_{0}\right),k_{2}\left(Q_{0}\right)\right)}
\]
 und 
\[
c_{2}=\left.\frac{\partial}{\partial y}f\left(x,y\right)\right|_{(x,y)=\left(k_{1}\left(P_{0}\right),k_{2}\left(Q_{0}\right)\right)}.
\]
Das statistische Funktional $k:\mathcal{P}\otimes\mathcal{Q}\rightarrow\mathbb{R},\:P\otimes Q\mapsto f\left(k_{1}\left(P\right),k_{2}\left(Q\right)\right)$
ist dann differenzierbar an der Stelle $P_{0}\otimes Q_{0}$ und die
Abbildung 
\[
\dot{k}:\Omega_{1}\times\Omega_{2}\rightarrow\mathbb{R},\:\left(\omega_{1},\omega_{2}\right)\mapsto c_{1}\dot{k}_{1}\left(\omega_{1}\right)+c_{2}\dot{k}_{2}\left(\omega_{2}\right)
\]
 ist ein Gradient von $k$ an der Stelle $P_{0}\otimes Q_{0}$. Die
Abbildung 
\[
\widetilde{k}:\Omega_{1}\times\Omega_{2}\rightarrow\mathbb{R},\:\left(\omega_{1},\omega_{2}\right)\mapsto c_{1}\widetilde{k}_{1}\left(\omega_{1}\right)+c_{2}\widetilde{k}_{2}\left(\omega_{2}\right)
\]
ist der kanonische Gradient von $k$ an der Stelle $P_{0}\otimes Q_{0}$.
\end{satz}\begin{proof}Sei $t\mapsto P_{t}\otimes Q_{t}$ eine $L_{2}\left(0\right)$-differenzierbare
Kurve in $\mathcal{P}\otimes\mathcal{Q}$ mit Tangente $g\in L_{2}^{(0)}\left(P_{0}\otimes Q_{0}\right)$.
Nach Satz \ref{L2Produkt} sind die Kurven $t\mapsto P_{t}$ in $\mathcal{P}$
und $t\mapsto Q_{t}$ in $\mathcal{Q}$ ebenfalls $L_{2}\left(0\right)$-differenzierbar
mit Tangenten $g_{1}\in L_{2}^{(0)}\left(P_{0}\right)$ und $g_{2}\in L_{2}^{(0)}\left(Q_{0}\right)$.
Es gilt außerdem $g=g_{1}\circ\pi_{1}+g_{2}\circ\pi_{2}$. Nach der
Kettenregel der Differentialrechnung (vgl. \cite{heuser:1993}, Seite
267) erhält man
\begin{eqnarray*}
 &  & \left.\frac{d}{dt}k\left(P_{t}\otimes Q_{t}\right)\right|_{t=0}\\
 & = & \left.\frac{d}{dt}f\left(k_{1}\left(P_{t}\right),k_{2}\left(Q_{t}\right)\right)\right|_{t=0}\\
 & = & \left.\frac{\partial}{\partial x}f\left(x,y\right)\right|_{(x,y)=\left(k_{1}\left(P_{0}\right),k_{2}\left(Q_{0}\right)\right)}\left.\frac{d}{dt}k_{1}\left(P_{t}\right)\right|_{t=0}\\
 &  & +\left.\frac{\partial}{\partial y}f\left(x,y\right)\right|_{(x,y)=\left(k_{1}\left(P_{0}\right),k_{2}\left(Q_{0}\right)\right)}\left.\frac{d}{dt}k_{2}\left(Q_{t}\right)\right|_{t=0}\\
 & = & c_{1}\int g_{1}\dot{k}_{1}\,dP_{0}+c_{2}\int g_{2}\dot{k}_{2}\,dQ_{0}\\
 & = & \int\left(g_{1}\circ\pi_{1}+g_{2}\circ\pi_{2}\right)\left(c_{1}\dot{k}_{1}\circ\pi_{1}+c_{2}\dot{k}_{2}\circ\pi_{2}\right)\,dP_{0}\otimes Q_{0}\\
 & = & \int g\dot{k}\,dP_{0}\otimes Q_{0}.
\end{eqnarray*}
Das statistische Funktional $k:\mathcal{P}\otimes\mathcal{Q}\rightarrow\mathbb{R}$
ist somit differenzierbar an der Stelle $P_{0}\otimes Q_{0}$ und
die Abbildung $\dot{k}$ ist ein Gradient von $k$. Insbesondere ist
die Abbildung $\widetilde{k}$ ein Gradient von $k$ an der Stelle
$P_{0}\otimes Q_{0}$. Nach Satz \ref{produkt_tangentialraum} ergibt
sich $\widetilde{k}\in T\left(P_{0}\otimes Q_{0},\mathcal{P}\otimes\mathcal{Q}\right)$,
weil $c_{1}\widetilde{k}_{1}\in T\left(P_{0},\mathcal{P}\right)$
und $c_{2}\widetilde{k}_{2}\in T\left(Q_{0},\mathcal{Q}\right)$ sind.
Nach Satz \ref{kanon_Gradient} und Bemerkung \ref{kan_grad_orth_pro}
erhält man nun, dass die Abbildung $\widetilde{k}$ der kanonische
Gradient von $k$ an der Stelle $P_{0}\otimes Q_{0}$ ist.\end{proof}\begin{beisp}\label{bsp_zusamm_gesetzte_funk}Es
gelten die Voraussetzungen aus Satz \ref{kett_reg_stat_funk}. 
\begin{enumerate}
\item Sei $f:\mathbb{R}^{2}\rightarrow\mathbb{R},\left(x,y\right)\mapsto x+y$.
Das statistische Funktional 
\[
k:\mathcal{P}\otimes\mathcal{Q},\,P\otimes Q\mapsto k_{1}\left(P\right)+k_{2}\left(Q\right)
\]
 ist dann differenzierbar an der Stelle $P_{0}\otimes Q_{0}$ mit
dem kanonischen Gradienten $\widetilde{k}=\widetilde{k}_{1}\circ\pi_{1}+\widetilde{k}_{2}\circ\pi_{2}$.
Die Abbildung $\dot{k}=\dot{k}_{1}\circ\pi_{1}+\dot{k}_{2}\circ\pi_{2}$
ist ein Gradient von $k$ an der Stelle $P_{0}\otimes Q_{0}$.
\item Sei $f:\mathbb{R}^{2}\rightarrow\mathbb{R},\left(x,y\right)\mapsto xy$.
Das statistische Funktional 
\[
k:\mathcal{P}\otimes\mathcal{Q},\,P\otimes Q\mapsto k_{1}\left(P\right)k_{2}\left(Q\right)
\]
 ist differenzierbar an der Stelle $P_{0}\otimes Q_{0}$ mit dem kanonischen
Gradienten 
\[
\widetilde{k}=k_{2}\left(Q_{0}\right)\widetilde{k}_{1}\circ\pi_{1}+k_{1}\left(P_{0}\right)\widetilde{k}_{2}\circ\pi_{2}.
\]
Die Abbildung $\dot{k}=k_{2}\left(Q_{0}\right)\dot{k}_{1}\circ\pi_{1}+k_{1}\left(P_{0}\right)\dot{k}_{2}\circ\pi_{2}$
ist ein Gradient von $k$ an der Stelle $P_{0}\otimes Q_{0}$.
\item Sei $f:\mathbb{R}^{2}\rightarrow\mathbb{R},\left(x,y\right)\mapsto\frac{x}{y}$.
Das statistische Funktional 
\[
k:\mathcal{P}\otimes\mathcal{Q},\,P\otimes Q\mapsto\frac{k_{1}\left(P\right)}{k_{2}\left(Q\right)}
\]
 ist differenzierbar an der Stelle $P_{0}\otimes Q_{0}$ mit dem kanonischen
Gradienten 
\[
\widetilde{k}=\frac{1}{k_{2}\left(Q_{0}\right)}\widetilde{k}_{1}\circ\pi_{1}-\frac{k_{1}\left(P_{0}\right)}{k_{2}\left(Q_{0}\right)^{2}}\widetilde{k}_{2}\circ\pi_{2}.
\]
Die Abbildung $\dot{k}=\frac{1}{k_{2}\left(Q_{0}\right)}\dot{k}_{1}\circ\pi_{1}-\frac{k_{1}\left(P_{0}\right)}{k_{2}\left(Q_{0}\right)^{2}}\dot{k}_{2}\circ\pi_{2}$
ist ein Gradient von $k$ an der Stelle $P_{0}\otimes Q_{0}$.
\end{enumerate}
\end{beisp}

\chapter{Testen impliziter Alternativen im Zweistichprobenkontext\label{cha:Testen-implizit-definierter}}

\section{Zweistichprobenprobleme beim Testen statistischer Funktionale und
Teststatistiken\label{sec:Zweistichprobenprobleme-beim-Testen}}

Es seien $\mathcal{P}\subset\mathcal{M}_{1}\left(\Omega_{1},\mathcal{A}_{1}\right)$,
$\mathcal{Q}\subset\mathcal{M}_{1}\left(\Omega_{2},\mathcal{A}_{2}\right)$
nichtparametrische Familien von Wahrscheinlichkeitsmaßen und $\mathcal{P}\otimes\mathcal{Q}:=\left\{ P\otimes Q:P\in\mathcal{P},Q\in\mathcal{Q}\right\} $
eine Familie von Produktmaßen. Für viele Anwendungen kann die statistische
Aufgabe als ein Testproblem bzgl. eines statistischen Funktionals
$k:\mathcal{P}\otimes\mathcal{Q}\rightarrow\mathbb{R}$ formuliert
werden. Man interessiert sich unter anderem für die einseitigen Testprobleme
\[
H=\left\{ P\otimes Q\in\mathcal{P}\otimes\mathcal{Q}:k\left(P\otimes Q\right)=a\right\} 
\]
\[
\mbox{gegen}\quad K_{1}=\left\{ P\otimes Q\in\mathcal{P}\otimes\mathcal{Q}:k\left(P\otimes Q\right)>a\right\} 
\]
für eine Zahl $a\in\mathbb{R}$. Häufig betrachtet man die erweiterte
Hypothese 
\[
H_{1}=\left\{ P\otimes Q\in\mathcal{P}\otimes\mathcal{Q}:k\left(P\otimes Q\right)\leq a\right\} 
\]
für das einseitige Testproblem. Die zweiseitigen Testprobleme
\[
H=\left\{ P\otimes Q\in\mathcal{P}\otimes\mathcal{Q}:k\left(P\otimes Q\right)=a\right\} 
\]
\[
\mbox{gegen}\quad K_{2}=\left\{ P\otimes Q\in\mathcal{P}\otimes\mathcal{Q}:k\left(P\otimes Q\right)\neq a\right\} 
\]
 sind ebenfalls von großer Bedeutung. Für die analogen Einstichprobenprobleme
ist es bekannt (vgl. \cite{Janssen:1999a}, \cite{Janssen:1999b}
und \cite{Janssen:2000}), dass die Teststatistiken, die auf dem kanonischen
Gradienten beruhen, zu den asymptotisch optimalen Testfolgen führen.
Für die Zweistichprobenprobleme bietet es sich an, die Zweistichproben-U-Statistiken
mit dem kanonischen Gradienten als Kern zu betrachten. Umfangreiche
Informationen zu den allgemeinen U-Statistiken und insbesondere auch
zu den Zweistichproben-U-Statistiken findet man in \cite{Lee:1990},
\cite{Korolyuk:1994}, \cite{Witting:1985} und \cite{Witting:1995}. 

Ein statistisches Funktional $k:\mathcal{P}\otimes\mathcal{Q}\rightarrow\mathbb{R}$
sei differenzierbar an einer Stelle $P_{0}\otimes Q_{0}\in H$ mit
dem kanonischen Gradienten $\widetilde{k}\in L_{2}^{(0)}\left(P_{0}\otimes Q_{0}\right).$
Mit $n_{i}$ wird der Umfang der $i$-ten Stichprobe für $i=1,2$
bezeichnet, wobei $n$ der gesamte Stichprobenumfang ist. Es seien
$\left(n_{1}\right)_{n\in\mathbb{N}}\subset\mathbb{N}$ und $\left(n_{2}\right)_{n\in\mathbb{N}}\subset\mathbb{N}$
zwei Folgen mit $n_{1}+n_{2}=n$ für alle $n\in\mathbb{N}$ und
\begin{equation}
\lim_{n\rightarrow\infty}\frac{n_{2}}{n}=d\label{d-wert}
\end{equation}
für ein $d\in\left(0,1\right).$ Für jedes $n\in\mathbb{N}$ sei 
\[
E_{n}=\left(\Omega_{1}^{n_{1}}\times\Omega_{2}^{n_{2}},\mathcal{A}_{1}^{n_{1}}\otimes\mathcal{A}_{2}^{n_{2}},\left\{ P^{n_{1}}\otimes Q^{n_{2}}:P\in\mathcal{P},Q\in\mathcal{Q}\right\} \right)
\]
 ein statistisches Experiment. Man erhält dann
\begin{eqnarray}
U_{n_{1},n_{2}}:\Omega_{1}^{n_{1}}\times\Omega_{2}^{n_{2}} & \rightarrow & \mathbb{R},\nonumber \\
\left(\omega_{1,1},\ldots,\omega_{1,n_{1}},\omega_{2,1},\ldots,\omega_{2,n_{2}}\right) & \mapsto & \frac{1}{n_{1}n_{2}}\sum_{i=1}^{n_{1}}\sum_{j=1}^{n_{2}}\widetilde{k}\left(\omega_{1,i},\omega_{2,j}\right)\label{u_stat}
\end{eqnarray}
als die Zweistichproben-U-Statistik für das Experiment $E_{n}$ mit
dem kanonischen Gradienten $\widetilde{k}$ als Kern. Diese Statistik
lässt sich wesentlich vereinfachen, weil der kanonische Gradient $\widetilde{k}\in L_{2}^{(0)}\left(P_{0}\otimes Q_{0}\right)$
eine besonders einfache Struktur besitzt. Nach Satz \ref{umrechnung}
existieren $\widetilde{k}_{1}\in L_{2}^{(0)}\left(P_{0}\right)$ und
$\widetilde{k}_{2}\in L_{2}^{(0)}\left(Q_{0}\right)$ mit 
\[
\widetilde{k}:\Omega_{1}\times\Omega_{2}\rightarrow\mathbb{R},\left(\omega_{1},\omega_{2}\right)\mapsto\widetilde{k}_{1}\left(\omega_{1}\right)+\widetilde{k}_{2}\left(\omega_{2}\right).
\]
Für die Zweistichproben-U-Statistik $U_{n_{1},n_{2}}$ ergibt sich
somit
\begin{eqnarray}
 &  & U_{n_{1},n_{2}}\left(\omega_{1,1},\ldots,\omega_{1,n_{1}},\omega_{2,1},\ldots,\omega_{2,n_{2}}\right)\nonumber \\
 & = & \frac{1}{n_{1}n_{2}}\sum_{i=1}^{n_{1}}\sum_{j=1}^{n_{2}}\widetilde{k}\left(\omega_{1,i},\omega_{2,j}\right)\nonumber \\
 & = & \frac{1}{n_{1}n_{2}}\sum_{i=1}^{n_{1}}\sum_{j=1}^{n_{2}}\left(\widetilde{k}_{1}\left(\omega_{1,i}\right)+\widetilde{k}_{2}\left(\omega_{2,j}\right)\right)\nonumber \\
 & = & \frac{1}{n_{1}}\sum_{i=1}^{n_{1}}\widetilde{k}_{1}\left(\omega_{1,i}\right)+\frac{1}{n_{2}}\sum_{j=1}^{n_{2}}\widetilde{k}_{2}\left(\omega_{2,j}\right).\label{u_stat_kanon}
\end{eqnarray}
\begin{beisp}[Wilcoxon Funktional]\label{Ts_WF}Das Wilcoxon-Funktional
$k:P\otimes Q\mapsto\int1_{\left\{ \left(x,y\right)\in\mathbb{R}^{2}:x\geq y\right\} }dP\otimes Q(x,y)$
aus Anwendung \ref{wilcoxon} ist differenzierbar an jeder Stelle
$P_{0}\otimes Q_{0}\in\mathcal{P}\otimes Q.$ Die Familien $\mathcal{P}$
und $\mathcal{Q}$ seien voll, vgl. Definition \ref{volle_familie}.
Nach Anwendung \ref{wilcoxon} erhält man dann den kanonischen Gradienten
$\widetilde{k}$ an der Stelle $P_{0}\otimes Q_{0}$ als 
\[
\widetilde{k}\left(x,y\right)=Q_{0}\left(\left(-\infty,x\right]\right)+P_{0}\left(\left[y,+\infty\right)\right)-2k\left(P_{0}\otimes Q_{0}\right).
\]
 Nach Satz \ref{umrechnung} erhält man außerdem $\widetilde{k}_{1}\left(x\right)=Q_{0}\left(\left(-\infty,x\right]\right)-k\left(P_{0}\otimes Q_{0}\right)$
und $\widetilde{k}_{2}\left(y\right)=P_{0}\left(\left[y,+\infty\right)\right)-k\left(P_{0}\otimes Q_{0}\right).$
Die Statistik $U_{n_{1},n_{2}}$ zu dem kanonischen Gradienten $\widetilde{k}$
an der Stelle $P_{0}\otimes Q_{0}$ ergibt sich als
\begin{eqnarray*}
U_{n_{1},n_{2}}\left(x_{1},\ldots,x_{n_{1}},y_{1},\ldots,y_{n_{2}}\right) & = & \frac{1}{n_{1}}\sum_{i=1}^{n_{1}}Q_{0}\left(\left(-\infty,x_{i}\right]\right)\\
 &  & +\frac{1}{n_{2}}\sum_{j=1}^{n_{2}}P_{0}\left(\left[y_{j},+\infty\right)\right)-2k\left(P_{0}\otimes Q_{0}\right).
\end{eqnarray*}
\end{beisp}\begin{beisp}[von Mises Funktional]Das von Mises Funktional
$k:P\otimes Q\mapsto\int h\left(\omega_{1},\omega_{2}\right)dP\otimes Q\left(\omega_{1},\omega_{2}\right)$
aus Beispiel \ref{von_mises} sei differenzierbar an einer Stelle
$P_{0}\otimes Q_{0}$ mit dem kanonischen Gradienten 
\begin{eqnarray*}
\widetilde{k}\left(\omega_{1},\omega_{2}\right) & = & \int h\left(\omega_{1},\omega_{2}\right)dQ_{0}\left(\omega_{2}\right)+\int h\left(\omega_{1},\omega_{2}\right)dP_{0}\left(\omega_{1}\right)-2k\left(P_{0}\otimes Q_{0}\right).
\end{eqnarray*}
Es gilt außerdem $\widetilde{k}_{1}\left(\omega_{1}\right)=\int h\left(\omega_{1},\omega_{2}\right)dQ_{0}\left(\omega_{2}\right)-k\left(P_{0}\otimes Q_{0}\right)$
und $\widetilde{k}_{2}\left(\omega_{2}\right)=\int h\left(\omega_{1},\omega_{2}\right)dP_{0}\left(\omega_{1}\right)-k\left(P_{0}\otimes Q_{0}\right)$
nach Satz \ref{umrechnung}. Die Statistik $U_{n_{1},n_{2}}$ ergibt
sich dann als 
\begin{eqnarray*}
U_{n_{1},n_{2}}\left(\omega_{1,1},\ldots,\omega_{1,n_{1}},\omega_{2,1},\ldots,\omega_{2,n_{2}}\right) & = & \frac{1}{n_{1}}\sum_{i=1}^{n_{1}}\int h\left(\omega_{1,i},\omega_{2}\right)dQ_{0}\left(\omega_{2}\right)\\
 &  & +\frac{1}{n_{2}}\sum_{j=1}^{n_{2}}\int h\left(\omega_{1},\omega_{2,j}\right)dP_{0}\left(\omega_{1}\right)\\
 &  & -2k\left(P_{0}\otimes Q_{0}\right).
\end{eqnarray*}
 \end{beisp}\begin{beisp}[invariante Funktionale]\label{s_f_f} Sei
$h:[0,1]\rightarrow\mathbb{R}$ eine differenzierbare Abbildung mit
beschränkter Ableitung. Ein invariantes Funktional $k:P\otimes Q\mapsto\int h\left(F_{Q}\left(x\right)\right)dP\left(x\right)$
ist dann differenzierbar an jeder Stelle $P_{0}\otimes Q_{0}\in\mathcal{P}\otimes\mathcal{Q}$
nach Beispiel \ref{survival}. Die Abbildung 
\begin{eqnarray*}
\dot{k}\left(x,y\right) & = & h\left(F_{Q_{0}}(x)\right)+\int\dot{h}\left(F_{Q_{0}}(s)\right)1_{\left[y,+\infty\right)}(s)dP_{0}(s)\\
 &  & -k\left(P_{0}\otimes Q_{0}\right)-\int\dot{h}\left(F_{Q_{0}}(s)\right)F_{Q_{0}}\left(s\right)dP_{0}\left(s\right)
\end{eqnarray*}
 ist ein Gradient des Funktionals $k$ an der Stelle $P_{0}\otimes Q_{0}$.
Es gelte $T\left(P_{0},\mathcal{P}\right)=L_{2}^{(0)}\left(P_{0}\right)$
und $T\left(Q_{0},\mathcal{Q}\right)=L_{2}^{(0)}\left(Q_{0}\right)$.
Nach Satz \ref{umrechnung} erhält man dann 
\begin{eqnarray*}
\widetilde{k}_{1}(x) & = & \int\dot{k}\left(x,y\right)dQ_{0}(y)\\
 & = & h\left(F_{Q_{0}}(x)\right)+\int\int\dot{h}\left(F_{Q_{0}}(s)\right)1_{\left[y,+\infty\right)}(s)dP_{0}(s)dQ_{0}(y)\\
 &  & -k\left(P_{0}\otimes Q_{0}\right)-\int\dot{h}\left(F_{Q_{0}}(s)\right)F_{Q_{0}}\left(s\right)dP_{0}\left(s\right)
\end{eqnarray*}
\begin{eqnarray*}
\textrm{} & = & h\left(F_{Q_{0}}(x)\right)-k\left(P_{0}\otimes Q_{0}\right)
\end{eqnarray*}
und
\begin{eqnarray*}
\widetilde{k}_{2}(y) & = & \int\dot{k}\left(x,y\right)dP_{0}(x)\\
 & = & \int h\left(F_{Q_{0}}(x)\right)dP_{0}(x)+\int\dot{h}\left(F_{Q_{0}}(s)\right)1_{\left[y,+\infty\right)}(s)dP_{0}(s)\\
 &  & -k\left(P_{0}\otimes Q_{0}\right)-\int\dot{h}\left(F_{Q_{0}}(s)\right)F_{Q_{0}}\left(s\right)dP_{0}\left(s\right)\\
 & = & \int\dot{h}\left(F_{Q_{0}}(s)\right)1_{\left[y,+\infty\right)}(s)dP_{0}(s)-\int\dot{h}\left(F_{Q_{0}}(s)\right)F_{Q_{0}}\left(s\right)dP_{0}\left(s\right).
\end{eqnarray*}
Die von dem kanonischen Gradienten erzeugte Statistik $U_{n_{1},n_{2}}$
ergibt sich als
\begin{eqnarray*}
U_{n_{1},n_{2}}\left(x_{1},\ldots,x_{n_{1}},y_{1},\ldots,y_{n_{2}}\right) & = & \frac{1}{n_{1}}\sum_{i=1}^{n_{1}}h\left(F_{Q_{0}}(x_{i})\right)\\
 &  & +\frac{1}{n_{2}}\sum_{j=1}^{n_{2}}\int\dot{h}\left(F_{Q_{0}}(s)\right)1_{\left[y_{j},+\infty\right)}(s)dP_{0}(s)\\
 &  & -k\left(P_{0}\otimes Q_{0}\right)-\int\dot{h}\left(F_{Q_{0}}(s)\right)F_{Q_{0}}\left(s\right)dP_{0}\left(s\right).
\end{eqnarray*}
\end{beisp}

\section{Lokale asymptotische Normalität der Produktexperimente}

Die lokale asymptotische Normalität (kurz: LAN) kann man als zentralen
Grenzwertsatz für Experimente betrachten, der eine Grundlage für die
Untersuchungen der lokalen asymptotischen Optimalität der Test- und
Schätzverfahren bildet. Umfangreiche Informationen zu der lokalen
asymptotischen Normalität findet man in \cite{Strasser:1985b}, \cite{LeCam:2000}
und \cite{Witting:1995}. An dieser Stelle wird kurz an die Definition
der lokalen asymptotischen Normalität erinnert. Für die nichtparametrischen
Untersuchungen der asymptotischen Optimalität in diesem Abschnitt
reicht die eindimensionale Version der LAN aus.\begin{defi}[LAN]\label{lan}Sei
$E_{n}=\left(\Omega_{n},\mathcal{A}_{n},\left\{ P_{n,\vartheta}:\vartheta\in\Theta_{n}\right\} \right)$
eine Folge von Experimenten mit $0\in\Theta_{n}\subset\mathbb{R}$
und $\Theta_{n}\uparrow\Theta$ für $n\rightarrow\infty.$ Die Folge
$E_{n}$ heißt lokal asymptotisch normal (LAN), wenn es Folgen von
Zufallsvariablen $X_{n}:\Omega_{n}\rightarrow\mathbb{R}$ und $R_{n,\vartheta}:\Omega_{n}\rightarrow\left[-\infty,\infty\right]$
gibt, so dass für ein $\sigma>0$ gilt:\renewcommand{\labelenumi}{(\arabic{enumi})}
\begin{enumerate}
\item $\log\frac{dP_{n,\vartheta}}{dP_{n,0}}=\vartheta X_{n}-\frac{1}{2}\vartheta^{2}\sigma^{2}+R_{n,\vartheta},$
\item $\mathcal{L}\left(X_{n}\mid P_{n,0}\right)\rightarrow N\left(0,\sigma^{2}\right)$
für $n\rightarrow\infty$,
\item $\lim\limits_{n\rightarrow\infty}P_{n,0}\left(\left\{ \left|R_{n,\vartheta}\right|>\varepsilon\right\} \right)=0$
für alle $\varepsilon>0$ und für alle $\vartheta\in\Theta$.
\end{enumerate}
Die Zufallsvariablen $\left(X_{n}\right)_{n\in\mathbb{N}}$ werden
als zentrale Folge bezeichnet.

\end{defi}\begin{bem}Häufig sind $P_{n,\vartheta}:=P_{\vartheta_{0}+\frac{\vartheta}{\sqrt{n}}}^{n}$
Produktmaße, wobei $\vartheta_{0}$ als globaler Parameter und $\vartheta$
als lokaler Parameter bezeichnet werden. Der globale Parameter $\vartheta_{0}\in\Theta$
wird festgehalten. Der lokale Parameter $\vartheta$ kann beliebige
Werte aus $\left\{ \vartheta\in\mathbb{R}:\vartheta_{0}+\frac{\vartheta}{\sqrt{n}}\in\Theta_{n}\right\} $
annehmen, so dass für jedes $n\in\mathbb{N}$ eine parametrische Familie
$\left\{ P_{\vartheta_{0}+\frac{\vartheta}{\sqrt{n}}}^{n}:\vartheta_{0}+\frac{\vartheta}{\sqrt{n}}\in\Theta_{n}\right\} $
von Wahrscheinlichkeitsmaßen entsteht.\end{bem}\begin{bem}Eine Folge
$E_{n}=\left(\Omega_{n},\mathcal{A}_{n},\left\{ P_{n,\vartheta}:\vartheta\in\Theta_{n}\right\} \right)$
von Experimenten sei LAN. Dann gilt
\[
P_{n,\vartheta}\triangleleft P_{n,0}\quad\mbox{und}\quad P_{n,0}\triangleleft P_{n,\vartheta}.
\]
\end{bem}\begin{proof}Die Aussage folgt unmittelbar aus dem ersten
Lemma von Le Cam, vgl. \cite{Hajek:1999}, Seite 251, \cite{Witting:1995},
Seite 331 oder \cite{Janssen:1998}, Seite 113.\end{proof}\begin{bem}\label{asympt_banachb_equiv_tests}Eine
Folge $E_{n}=\left(\Omega_{n},\mathcal{A}_{n},\left\{ P_{n,\vartheta}:\vartheta\in\Theta_{n}\right\} \right)$
von Experimenten sei LAN. Es seien $\left(\varphi_{in}\right)_{n\in\mathbb{N}}$
zwei Testfolgen für $i=1,2$ mit $\lim\limits_{n\rightarrow\infty}\int\left|\varphi_{1n}-\varphi_{2n}\right|dP_{n,0}=0$.
 Nach Lemma \ref{benach_imp_asymp_eq} ergibt sich dann
\[
\lim_{n\rightarrow\infty}\int\left|\varphi_{1n}-\varphi_{2n}\right|dP_{n,\vartheta}=0
\]
 für alle $\vartheta\in\Theta$. \end{bem}Beim Testen statistischer
Funktionale ist es nützlich, wenn eine verschärfte Version der LAN-Bedingung
zur Verfügung steht. Dann kann man die lokale Trenngeschwindigkeit
$n^{-\frac{1}{2}}$ durch eine Folge $t_{n}$ mit $\lim\limits_{n\rightarrow\infty}n^{\frac{1}{2}}t_{n}=1$
ersetzen. Man betrachtet dann das asymptotische Verhalten der Experimente
\[
E_{n}=\left(\Omega_{n},\mathcal{A}_{n},\left\{ P_{\vartheta_{0}+t_{n}\vartheta}^{n}:\vartheta_{0}+t_{n}\vartheta\in\Theta_{n}\right\} \right)
\]
 für einen fest gewählten globalen Parameter $\vartheta_{0}\in\Theta$.
Aus diesem Grund wird eine gleichmäßige LAN-Bedingung angegeben, die
in der Literatur ULAN-Bedingung heißt. \begin{defi}[ULAN]\label{ulan}Sei
$E_{n}=\left(\Omega_{n},\mathcal{A}_{n},\left\{ P_{n,\vartheta}:\vartheta\in\Theta_{n}\right\} \right)$
eine LAN-Folge von Experimenten. Die Folge $E_{n}$ heißt ULAN, falls
für alle kompakten Mengen $K\subset\mathbb{R}$ und für alle $\varepsilon>0$
folgt
\begin{equation}
\lim_{n\rightarrow\infty}\sup_{\vartheta\in K\cap\Theta_{n}}P_{n,0}\left(\left\{ \left|R_{n,\vartheta}\right|>\varepsilon\right\} \right)=0.\label{ulan}
\end{equation}
 Die äquivalente Bedingung an die Restglieder $R_{n,\vartheta}$ ist
die Konvergenz
\begin{equation}
\lim_{n\rightarrow\infty}P_{n,0}\left(\left\{ \left|R_{n,\vartheta_{n}}\right|>\varepsilon\right\} \right)=0\label{ulan_einfach}
\end{equation}
 für alle konvergenten Folgen $\vartheta_{n}\rightarrow\vartheta$
und für alle $\varepsilon>0$. \end{defi}Für die Zweistichprobenprobleme
braucht man zusätzliche Hilfsmittel, die die Arbeit mit der LAN für
die Familien der Produktmaße erleichtern. Sind zwei Folgen von Experimenten
gegeben, so interessiert man sich für die lokalen asymptotischen Eigenschaften
der Folge der Experimenten, die auf Familien der Produktmaße beruhen.
Eine umfassende Antwort auf diese Fragestellung wird in Satz \ref{lan_produkt}
gegeben.\begin{defi}[Produktexperiment]\label{prod_exp_defi}Es seien
$\left(n_{1}\right)_{n\in\mathbb{N}}\subset\mathbb{N}$ und $\left(n_{2}\right)_{n\in\mathbb{N}}\subset\mathbb{N}$
zwei Folgen mit $n_{1}+n_{2}=n$ für alle $n\in\mathbb{N}$ und $\lim\limits_{n\rightarrow\infty}\frac{n_{2}}{n}=d\in(0,1)$.
Es seien $E_{1,n}=\left(\Omega_{1,n},\mathcal{A}_{1,n},\left\{ P_{n,\vartheta}:\vartheta\in\Theta_{n}\right\} \right)$
und $E_{2,n}=(\Omega_{2,n},\mathcal{A}_{2,n},\left\{ Q_{n,\vartheta}:\vartheta\in\Theta_{n}\right\} )$
zwei Folgen von Experimenten. Das Produktexperiment wird definiert
als 
\begin{equation}
E_{n}=\left(\Omega_{1,n_{1}}\times\Omega_{2,n_{2}},\,\mathcal{A}_{1,n_{1}}\otimes\mathcal{A}_{2,n_{2}},\left\{ P_{n_{1},\vartheta}\otimes Q_{n_{2},\vartheta}:\vartheta\in\Theta_{n}\right\} \right).\label{produktexperiment}
\end{equation}
\end{defi}\begin{satz}[LAN von Produktexperimenten]\label{lan_produkt}Die
Voraussetzungen aus Definition \ref{prod_exp_defi} seien erfüllt.
Es gelte außerdem $0\in\Theta_{n}\subset\mathbb{R}$ für alle $n\in\mathbb{N}$
und $\Theta_{n}\uparrow\Theta$ für $n\rightarrow\infty$. Die Folge
$\left(E_{i,n}\right)_{n\in\mathbb{N}}$ von Experimenten sei LAN
mit $X_{i,n}$ als zentrale Folge für $i=1,2$. Dann ist die Folge
$\left(E_{n}\right)_{n\in\mathbb{N}}$ der Produktexperimente LAN
mit
\begin{eqnarray*}
X_{n}:\Omega_{1,n_{1}}\times\Omega_{2,n_{2}} & \rightarrow & \mathbb{R},\\
\left(\omega_{1,n_{1}},\omega_{2,n_{2}}\right) & \mapsto & X_{1,n_{1}}\left(\omega_{1,n_{1}}\right)+X_{2,n_{2}}\left(\omega_{2,n_{2}}\right)
\end{eqnarray*}
als zentrale Folge. \end{satz}\begin{proof}Nach Definition \ref{lan}
existieren Folgen $\left(R_{n,\vartheta}\right)_{n\in\mathbb{N}}$
und $\left(\widetilde{R}_{n,\vartheta}\right)_{n\in\mathbb{N}}$ von
Zufallsvariablen für jedes $\vartheta\in\Theta$, so dass gilt:\renewcommand{\labelenumi}{(\arabic{enumi})}
\begin{enumerate}
\item $\log\frac{dP_{n,\vartheta}}{dP_{n,0}}=\vartheta X_{1,n}-\frac{1}{2}\vartheta^{2}\sigma_{1}^{2}+R_{n,\vartheta}$
und $\log\frac{dQ_{n,\vartheta}}{dQ_{n,0}}=\vartheta X_{2,n}-\frac{1}{2}\vartheta^{2}\sigma_{2}^{2}+\widetilde{R}_{n,\vartheta}$,
\item $\mathcal{L}\left(X_{1,n}\mid P_{n,0}\right)\rightarrow N\left(0,\sigma_{1}^{2}\right)$
und $\mathcal{L}\left(X_{2,n}\mid Q_{n,0}\right)\rightarrow N\left(0,\sigma_{2}^{2}\right)$
in Verteilung für $n\rightarrow\infty$,
\item $\lim\limits_{n\rightarrow\infty}P_{n,0}\left(\left\{ \left|R_{n,\vartheta}\right|>\varepsilon\right\} \right)=0$
und $\lim\limits_{n\rightarrow\infty}Q_{n,0}\left(\left\{ \left|\widetilde{R}_{n,\vartheta}\right|>\varepsilon\right\} \right)=0$
für alle $\vartheta\in\Theta$ und alle $\varepsilon>0$.
\end{enumerate}
Man erhält zunächst 
\begin{eqnarray*}
\textrm{} &  & \log\frac{dP_{n_{1},\vartheta}\otimes Q_{n_{2},\vartheta}}{dP_{n_{1},0}\otimes Q_{n_{2},0}}\left(\omega_{1,n_{1}},\omega_{2,n_{2}}\right)\\
 & = & \log\left(\frac{dP_{n_{1},\vartheta}}{dP_{n_{1},0}}\left(\omega_{1,n_{1}}\right)\frac{dQ_{n_{2},\vartheta}}{dQ_{n_{2},0}}\left(\omega_{2,n_{2}}\right)\right)
\end{eqnarray*}
\begin{eqnarray*}
 & = & \log\left(\frac{dP_{n_{1},\vartheta}}{dP_{n_{1},0}}\left(\omega_{1,n_{1}}\right)\right)+\log\left(\frac{dQ_{n_{2},\vartheta}}{dQ_{n_{2},0}}\left(\omega_{2,n_{2}}\right)\right)\\
 & = & \vartheta X_{1,n_{1}}\left(\omega_{1,n_{1}}\right)-\frac{1}{2}\vartheta^{2}\sigma_{1}^{2}+R_{n_{1},\vartheta}\left(\omega_{1,n_{1}}\right)\\
 &  & +\vartheta X_{2,n_{2}}\left(\omega_{2,n_{2}}\right)-\frac{1}{2}\vartheta^{2}\sigma_{2}^{2}+\widetilde{R}_{n_{2},\vartheta}\left(\omega_{2,n_{2}}\right)\\
 & = & \vartheta\left(X_{1,n_{1}}\left(\omega_{1,n_{1}}\right)+X_{2,n_{2}}\left(\omega_{2,n_{2}}\right)\right)-\frac{1}{2}\vartheta^{2}\left(\sigma_{1}^{2}+\sigma_{2}^{2}\right)\\
 &  & +R_{n_{1},\vartheta}\left(\omega_{1,n_{1}}\right)+\widetilde{R}_{n_{2},\vartheta}\left(\omega_{2,n_{2}}\right).
\end{eqnarray*}
Für die Folge $\left(X_{n}\right)_{n\in\mathbb{N}}$ von Zufallsvariablen
erhält man unmittelbar 
\begin{eqnarray*}
\mathcal{L}\left(X_{n}\mid P_{n_{1},0}\otimes Q_{n_{2},0}\right) & = & \mathcal{L}\left(X_{1,n_{1}}\mid P_{n_{1},0}\right)*\mathcal{L}\left(X_{2,n_{2}}\mid Q_{n_{2},0}\right)\\
 & \rightarrow & N\left(0,\sigma_{1}^{2}\right)*N\left(0,\sigma_{2}^{2}\right)\\
 & = & N\left(0,\sigma_{1}^{2}+\sigma_{2}^{2}\right)
\end{eqnarray*}
 für $n\rightarrow\infty$, weil $\mathcal{L}\left(X_{1,n_{1}}\mid P_{n_{1},0}\right)$
und $\mathcal{L}\left(X_{2,n_{2}}\mid Q_{n_{2},0}\right)$ stochastisch
unabhängig sind. Für jedes $\varepsilon>0$ gilt außerdem
\begin{eqnarray*}
 &  & P_{n_{1},0}\otimes Q_{n_{2},0}\left(\left\{ \left|R_{n_{1},\vartheta}+\widetilde{R}_{n_{2},\vartheta}\right|>\varepsilon\right\} \right)\\
 & \leq & P_{n_{1},0}\otimes Q_{n_{2},0}\left(\left\{ \left|R_{n_{1},\vartheta}\right|>\frac{\varepsilon}{2}\right\} \right)+P_{n_{1},0}\otimes Q_{n_{2},0}\left(\left\{ \left|\widetilde{R}_{n_{2},\vartheta}\right|>\frac{\varepsilon}{2}\right\} \right)\\
 & = & \underbrace{P_{n_{1},0}\left(\left\{ \left|R_{n_{1},\vartheta}\right|>\frac{\varepsilon}{2}\right\} \right)}_{\rightarrow0}+\underbrace{Q_{n_{2},0}\left(\left\{ \left|\widetilde{R}_{n_{2},\vartheta}\right|>\frac{\varepsilon}{2}\right\} \right)}_{\rightarrow0}
\end{eqnarray*}
für $n\rightarrow\infty$. Somit ist alles bewiesen. \end{proof}\begin{satz}[ULAN von Produktexperimenten]\label{ulan_produkt}Es
gelten die Voraussetzungen aus Satz \ref{lan_produkt}. Die Folgen
$E_{1,n}$ und $E_{2,n}$ von Experimenten seien ULAN. Die Folge $\left(E_{n}\right)_{n\in\mathbb{N}}$
der Produktexperimente ist dann ebenfalls ULAN.\end{satz}\begin{proof}Nach
Satz \ref{lan_produkt} reicht es die Bedingung (\ref{ulan_einfach})
nachzuweisen. Sei $\left(\vartheta_{n}\right)_{n\in\mathbb{N}}$ eine
Folge mit $\vartheta_{n}\in\Theta_{n}$ für alle $n\in\mathbb{N}$
und $\lim\limits_{n\rightarrow\mathbb{N}}\vartheta_{n}=\vartheta$
für ein $\vartheta\in\Theta.$ Man erhält dann für jedes $\varepsilon>0$
die Konvergenz
\begin{eqnarray*}
 &  & P_{n_{1},0}\otimes Q_{n_{2},0}\left(\left\{ \left|R_{n_{1},\vartheta_{n}}+\widetilde{R}_{n_{2},\vartheta_{n}}\right|>\varepsilon\right\} \right)\\
\textrm{} & \leq & \underbrace{P_{n_{1},0}\left(\left\{ \left|R_{n_{1},\vartheta_{n}}\right|>\frac{\varepsilon}{2}\right\} \right)}_{\rightarrow0}+\underbrace{Q_{n_{2},0}\left(\left\{ \left|\widetilde{R}_{n_{2},\vartheta_{n}}\right|>\frac{\varepsilon}{2}\right\} \right)}_{\rightarrow0}
\end{eqnarray*}
für $n\rightarrow\infty$. Die Bedingung (\ref{ulan_einfach}) ist
somit erfüllt. \end{proof} \begin{defi}Sei $\Theta\subset\mathbb{R}$
mit $\left(-\varepsilon,\varepsilon\right)\subset\Theta$ für ein
$\varepsilon>0$. Ein Experiment $E\left(\Omega,\mathcal{A},\left\{ P_{\vartheta}:\vartheta\in\Theta\right\} \right)$
heißt $L_{2}\left(P_{0}\right)$-differenzierbar mit Tangente $g\in L_{2}^{(0)}\left(P_{0}\right)$,
falls die Kurve $\vartheta\mapsto P_{\vartheta}$ $L_{2}\left(P_{0}\right)$-differenzierbar
mit Tangente $g\in L_{2}^{(0)}\left(P_{0}\right)$ ist.\end{defi}Die
$L_{2}$-Differenzierbarkeit eines Experimentes ist eine Glattheitseigenschaft.
Der folgende Satz von Le Cam verifiziert LAN für die geeigneten Folgen
von $L_{2}(0)$-differenzierbaren Experimenten und zeigt dabei die
zentrale Rolle der Tangente an der Stelle $0$.\begin{satz}[von Le Cam]\label{lecam}Sei
$E=\left(\Omega,\mathcal{A},\left\{ P_{\vartheta}:\vartheta\in\Theta\right\} \right)$
ein $L_{2}(0)$-differenzierbares Experiment mit Tangente $g\in L_{2}^{(0)}\left(P_{0}\right)$
und $g\neq0.$ Sei $\left(\left(c_{n,i}\right)_{1\leq i\leq n}\right)_{n\in\mathbb{N}}\subset\mathbb{R}$
eine Dreiecksschema von Regressionskoeffizienten mit\renewcommand{\labelenumi}{(\roman{enumi})}
\begin{enumerate}
\item $\max\limits_{1\leq i\leq n}\left|c_{n,i}\right|\rightarrow0$ für
$n\rightarrow\infty,$
\item $\sum\limits_{i=1}^{n}c_{n,i}^{2}\rightarrow c^{2}>0$ für $n\rightarrow\infty$.
\end{enumerate}
Sei $E_{n}=\left(\Omega^{n},\mathcal{A}^{n},\left\{ \bigotimes\limits_{i=1}^{n}P_{c_{n,i}\vartheta}:\vartheta\in\Theta_{n}\right\} \right)$
eine Folge von Experimenten mit $\Theta_{n}\subset\left\{ \vartheta:c_{n,i}\vartheta\in\Theta\;\mbox{für alle}\;i\in\left\{ 1,\ldots,n\right\} \right\} $.
Dann erfüllt $\left(E_{n}\right)_{n\in\mathbb{N}}$ die LAN-Bedingung
und die Folge
\[
X_{n}\left(\omega_{1},\ldots,\omega_{n}\right):=\sum_{i=1}^{n}c_{n,i}g\left(\omega_{i}\right)
\]
von Zufallsvariablen ist eine zentrale Folge. Außerdem gilt
\[
\lim_{n\rightarrow\infty}\int X_{n}^{2}\,dP_{0}^{n}=c^{2}\int g^{2}dP_{0}.
\]
 \end{satz}\begin{proof}vgl. \cite{Witting:1995}, Seite 317, Satz
6.130 oder \cite{Janssen:1998}, Seite 126, Satz 14.8.\end{proof}\begin{satz}\label{ulan_le_cam}Es
gelten die Voraussetzungen aus Satz \ref{lecam}. Die Folge $\left(E_{n}\right)_{n\in\mathbb{N}}$
der Experimente aus Satz \ref{lecam} ist dann ULAN.\end{satz}\begin{proof}vgl.
\cite{Janssen:1998}, Seite 138, Beispiel 14.18 oder \cite{Witting:1995},
Seite 317, Satz 6.130.\end{proof}\begin{anwendung}\label{ulan_anwendung}Seien
$E_{1}=\left(\Omega_{1},\mathcal{A}_{1},\left\{ P_{\vartheta}:\vartheta\in\Theta\right\} \right)$
und $E_{2}=\left(\Omega_{2},\mathcal{A}_{2},\left\{ Q_{\vartheta}:\vartheta\in\Theta\right\} \right)$
zwei $L_{2}(0)$-differenzierbare Experimente mit Tangenten $g_{1}\in L_{2}^{(0)}\left(P_{0}\right)$
und $g_{2}\in L_{2}^{(0)}\left(Q_{0}\right)$. Außerdem gelte $g_{1}\neq0$
und $g_{2}\neq0$. Seien $\left(\left(c_{n,i}\right)_{1\leq i\leq n}\right)_{n\in\mathbb{N}}\subset\mathbb{R}$
und $\left(\left(\widetilde{c}_{n,i}\right)_{1\leq i\leq n}\right)_{n\in\mathbb{N}}\subset\mathbb{R}$
zwei Dreiecksschemata von Regressionskoeffi\-zienten, die den Bedingungen
\emph{$(i)$} und $(ii)$ aus Satz \ref{lecam} entsprechen. Dann
gilt 
\[
\lim_{n\rightarrow\infty}\sum_{i=1}^{n}c_{n,i}^{2}=c_{1}^{2}>0\;\mbox{und}\;\lim_{n\rightarrow\infty}\sum_{i=1}^{n}\widetilde{c}_{n,i}^{2}=c_{2}^{2}>0.
\]
Die Folgen $E_{1,n}=\left(\Omega_{1}^{n},\mathcal{A}_{1}^{n},\left\{ \bigotimes\limits_{i=1}^{n}P_{c_{n,i}\vartheta}:\vartheta\in\Theta_{n}\right\} \right)$
und \\
$E_{2,n}=\left(\Omega_{2}^{n},\mathcal{A}_{2}^{n},\left\{ \bigotimes\limits_{i=1}^{n}Q_{\widetilde{c}_{n,i}\vartheta}:\vartheta\in\Theta_{n}\right\} \right)$
von Experimenten erfüllen die ULAN-Bedingung nach Satz \ref{lecam}.
Seien $\left(n_{1}\right)_{n\in\mathbb{N}}\subset\mathbb{N}$ und
$\left(n_{2}\right)_{n\in\mathbb{N}}\subset\mathbb{N}$ zwei Folgen
mit $\lim\limits_{n\rightarrow\infty}\frac{n_{2}}{n}=d\in(0,1)$ und
$n_{1}+n_{2}=n$ für alle $n\in\mathbb{N}$. Die Folge $E_{n}$ der
Produktexperimente erfüllt dann nach Satz \ref{lan_produkt} die ULAN-Bedingung
mit
\[
X_{n}\left(\omega_{1,1},\ldots,\omega_{1,n_{1}},\omega_{2,1},\ldots,\omega_{2,n_{2}}\right)=\sum_{i=1}^{n_{1}}c_{n_{1},i}g_{1}\left(\omega_{1,i}\right)+\sum_{i=1}^{n_{2}}\widetilde{c}_{n_{2},i}g_{2}\left(\omega_{2,i}\right)
\]
als zentrale Folge. Es gilt außerdem 
\[
\lim_{n\rightarrow\infty}\mbox{Var}_{P_{0}^{n_{1}}\otimes Q_{0}^{n_{2}}}\left(X_{n}\right)=c_{1}^{2}\int g_{1}^{2}dP_{0}+c_{2}^{2}\int g_{2}^{2}dQ_{0}.
\]
 Die Regressionskoeffizienten können zum Beispiel als 
\begin{equation}
c_{n_{1},i}:=\frac{1}{\sqrt{n}}\quad\mbox{und}\quad\widetilde{c}_{n_{2},i}:=\frac{1}{\sqrt{n}}\label{regr_koeff}
\end{equation}
 gewählt werden. Man erhält dann $c_{1}^{2}=1-d$ und $c_{2}^{2}=d$.
Hieraus folgt
\begin{eqnarray*}
\lim_{n\rightarrow\infty}\mbox{Var}_{P_{0}^{n_{1}}\otimes Q_{0}^{n_{2}}}\left(X_{n}\right) & = & \left(1-d\right)\int g_{1}^{2}dP_{0}+d\int g_{2}^{2}dQ_{0}.
\end{eqnarray*}
Die Abbildung $g:\Omega_{1}\times\Omega_{2}\rightarrow\mathbb{R},\left(\omega_{1},\omega_{2}\right)\mapsto g_{1}\left(\omega_{1}\right)+g_{2}\left(\omega_{2}\right)$
ist die Tangente zu der $L_{2}(0)$-differenzierbaren Kurve $\vartheta\mapsto P_{\vartheta}\otimes Q_{\vartheta}$
an der Stelle $0$, vgl. Satz \ref{L2Produkt}. Die Tangente $g$
und die Regressionskoeffi\-zienten bestimmen also die asymptotischen
Eigenschaften der Produktexperimente $\left(E_{n}\right)_{n\in\mathbb{N}}$.
\end{anwendung}\newpage 

\section{Asymptotisches Verhalten der Teststatistik }

Im weiteren Verlauf dieser Arbeit werden die lokalen asymptotischen
Eigenschaften der Teststatistik
\[
T_{n}:=\sqrt{n}U_{n_{1},n_{2}}
\]
 analysiert, wobei $U_{n_{1},n_{2}}$ die von dem kanonischen Gradient
$\widetilde{k}$ erzeugte Zwei\-stich\-pro\-ben-U-Sta\-tis\-tik
ist, vgl. (\ref{u_stat}). Mit (\ref{u_stat_kanon}) ergibt sich zunächst
\begin{equation}
T_{n}\left(\omega_{1,1},\ldots,\omega_{1,n_{1}},\omega_{2,1},\ldots,\omega_{2,n_{2}}\right)=\frac{\sqrt{n}}{n_{1}}\sum_{i=1}^{n_{1}}\widetilde{k}_{1}\left(\omega_{1,i}\right)+\frac{\sqrt{n}}{n_{2}}\sum_{j=1}^{n_{2}}\widetilde{k}_{2}\left(\omega_{2,j}\right),\label{teststatistik T}
\end{equation}
wobei $\widetilde{k}:\Omega_{1}\times\Omega_{2}\rightarrow\mathbb{R},\left(\omega_{1},\omega_{2}\right)\mapsto\widetilde{k}_{1}\left(\omega_{1}\right)+\widetilde{k}_{2}\left(\omega_{2}\right)$
der kanonische Gradient eines statistischen Funktionals $k:\mathcal{P}\otimes\mathcal{Q}\rightarrow\mathbb{R}$
an der Stelle $P_{0}\otimes Q_{0}$ ist, vgl. Satz \ref{umrechnung}.
\begin{satz}[Asymptotisches Verhalten von $T_n$ unter $P_0\otimes Q_0$]\label{asymp_verhalt_der_T_stat} Seien
$\left(n_{1}\right)_{n\in\mathbb{N}}\subset\mathbb{N}$ und $\left(n_{2}\right)_{n\in\mathbb{N}}\subset\mathbb{N}$
zwei Folgen mit $n_{1}+n_{2}=n$ für alle $n\in\mathbb{N}$ und $\lim\limits_{n\rightarrow\infty}\frac{n_{2}}{n}=d\in(0,1)$.
Dann gilt
\[
\mathcal{L}\left(T_{n}\left|P_{0}^{n_{1}}\otimes Q_{0}^{n_{2}}\right.\right)\rightarrow N\left(0,\frac{1}{1-d}\left\Vert \widetilde{k}_{1}\right\Vert _{L_{2}\left(P_{0}\right)}^{2}+\frac{1}{d}\left\Vert \widetilde{k}_{2}\right\Vert _{L_{2}\left(Q_{0}\right)}^{2}\right)
\]
für $n\rightarrow\infty$, wobei $\left\Vert \widetilde{k}_{1}\right\Vert _{L_{2}\left(P_{0}\right)}^{2}=\int\widetilde{k}_{1}^{2}\,dP_{0}$
und $\left\Vert \widetilde{k}_{2}\right\Vert _{L_{2}\left(Q_{0}\right)}^{2}=\int\widetilde{k}_{2}^{2}\,dQ_{0}$
sind. \end{satz}\begin{proof}Die $i$-te kanonische Projektion wird
mit $\pi_{i}$ bezeichnet. Nach dem zentralen Grenzwertsatz von Lindeberg-Feller
(vgl. \cite{Bauer:1991}, Satz 28.3) erhält man unmittelbar die Konvergenz
\[
\mathcal{L}\left(\left.\frac{\sqrt{n}}{n_{1}}\sum_{i=1}^{n_{1}}\widetilde{k}_{1}\circ\pi_{i}\right|P_{0}^{n_{1}}\right)\rightarrow N\left(0,\frac{1}{1-d}\int\widetilde{k}_{1}^{2}\,dP_{0}\right)
\]
und
\[
\mathcal{L}\left(\left.\frac{\sqrt{n}}{n_{2}}\sum_{j=1}^{n_{2}}\widetilde{k}_{2}\circ\pi_{j}\right|Q_{0}^{n_{2}}\right)\rightarrow N\left(0,\frac{1}{d}\int\widetilde{k}_{2}^{2}\,dQ_{0}\right)
\]
 für $n\rightarrow\infty$, weil $\widetilde{k}_{1}\in L_{2}^{(0)}\left(P_{0}\right)$
und $\widetilde{k}_{2}\in L_{2}^{(0)}\left(Q_{0}\right)$ sind. Insgesamt
folgt
\begin{eqnarray*}
\mathcal{L}\left(T_{n}\left|P_{0}^{n_{1}}\otimes Q_{0}^{n_{2}}\right.\right) & = & \mathcal{L}\left(\left.\frac{\sqrt{n}}{n_{1}}\sum_{i=1}^{n_{1}}\widetilde{k}_{1}\circ\pi_{i}+\frac{\sqrt{n}}{n_{2}}\sum_{j=1}^{n_{2}}\widetilde{k}_{2}\circ\pi_{n_{1}+j}\right|P_{0}^{n_{1}}\otimes Q_{0}^{n_{2}}\right)\\
 & \rightarrow & N\left(0,\frac{1}{1-d}\int\widetilde{k}_{1}^{2}\,dP_{0}+\frac{1}{d}\int\widetilde{k}_{2}^{2}\,dQ_{0}\right)
\end{eqnarray*}
für $n\rightarrow\infty$, weil $\lim\limits_{n\rightarrow\infty}\frac{n_{2}}{n}=d\in\left(0,1\right)$
nach Voraussetzung gilt und die Verteilungen $\mathcal{L}\left(\left.\frac{\sqrt{n}}{n_{1}}\sum_{i=1}^{n_{1}}\widetilde{k}_{1}\circ\pi_{i}\right|P_{0}^{n_{1}}\right)$
und $\mathcal{L}\left(\left.\frac{\sqrt{n}}{n_{2}}\sum_{j=1}^{n_{2}}\widetilde{k}_{2}\circ\pi_{j}\right|Q_{0}^{n_{2}}\right)$
stochastisch unabhängig sind.\end{proof}Zur Anwendung der asymptotischen
Testtheorie bleibt noch die gemeinsame asymptotische Verteilung der
zentralen Folge 
\begin{equation}
X_{n}=\sum_{i=1}^{n_{1}}c_{n_{1},i}\,g_{1}\circ\pi_{i}+\sum_{i=n_{1}+1}^{n}\widetilde{c}_{n_{2},i}\,g_{2}\circ\pi_{i}\label{zfolge}
\end{equation}
 und der Statistik $T_{n}$ zu untersuchen. \begin{satz}[gemeinsame asymptotische Verteilung]\label{gem-asymp-vert}Seien
$\left(n_{1}\right)_{n\in\mathbb{N}}\subset\mathbb{N}$ und $\left(n_{2}\right)_{n\in\mathbb{N}}\subset\mathbb{N}$
zwei Folgen mit $n_{1}+n_{2}=n$ und $\lim\limits_{n\rightarrow\infty}\frac{n_{2}}{n}=d\in(0,1)$.
Die Regressionskoeffizienten $\left(\left(c_{n,i}\right)_{1\leq i\leq n}\right)_{n\in\mathbb{N}}\subset\mathbb{R}$
und $\left(\left(\widetilde{c}_{n,i}\right)_{1\leq i\leq n}\right)_{n\in\mathbb{N}}\subset\mathbb{R}$
der zentralen Folge $X_{n}$ seien allgemein gewählt und erfüllen
die Bedingungen aus Satz \ref{lecam}. Außerdem existieren die Grenzwerte
\[
\lim_{n\rightarrow\infty}n^{-\frac{1}{2}}\sum_{i=1}^{n}c_{n,i}=a_{1}\quad\mbox{und}\quad\lim_{n\rightarrow\infty}n^{-\frac{1}{2}}\sum_{i=1}^{n}\widetilde{c}_{n,i}=a_{2}.
\]
 Dann gilt
\[
\mathcal{L}\left(\left(T_{n},X_{n}\right)^{t}\left|P_{0}^{n_{1}}\otimes Q_{0}^{n_{2}}\right.\right)\rightarrow N\left(\left(\begin{array}{c}
0\\
0
\end{array}\right),\left(\begin{array}{cc}
\sigma_{1}^{2} & \sigma_{12}\\
\sigma_{12} & \sigma_{2}^{2}
\end{array}\right)\right)
\]
 für $n\rightarrow\infty$. Die Einträge der Kovarianzmatrix ergeben
sich als 
\[
\sigma_{1}^{2}=\frac{1}{1-d}\int\widetilde{k}_{1}^{2}\,dP_{0}+\frac{1}{d}\int\widetilde{k}_{2}^{2}\,dQ_{0},
\]
\[
\sigma_{2}^{2}=c_{1}^{2}\int g_{1}^{2}\,dP+c_{2}^{2}\int g_{2}^{2}\,dQ,
\]
\[
\sigma_{12}=a_{1}\frac{1}{\sqrt{1-d}}\int\widetilde{k}_{1}g_{1}\,dP_{0}+a_{2}\frac{1}{\sqrt{d}}\int\widetilde{k}_{2}g_{2}\,dQ_{0}.
\]
\end{satz}\begin{proof}Nach dem Cram\'er-Wold-Device (vgl. \cite{gaenssler:1977},
Seite 357) reicht es für alle $\left(\lambda_{1},\lambda_{2}\right)\in\mathbb{R}^{2}$
zu zeigen, dass 
\[
\mathcal{L}\left(\lambda_{1}T_{n}+\lambda_{2}X_{n}\left|P_{0}^{n_{1}}\otimes Q_{0}^{n_{2}}\right.\right)\rightarrow N\left(0,\lambda_{1}^{2}\sigma_{1}^{2}+2\lambda_{1}\lambda_{2}\sigma_{12}+\lambda_{2}^{2}\sigma_{2}^{2}\right)
\]
für $n\rightarrow\infty$ konvergiert. Man erhält zuerst 
\begin{eqnarray*}
\lambda_{1}T_{n}+\lambda_{2}X_{n} & = & \lambda_{1}\left(\frac{\sqrt{n}}{n_{1}}\sum_{i=1}^{n_{1}}\widetilde{k}_{1}\circ\pi_{i}+\frac{\sqrt{n}}{n_{2}}\sum_{j=1}^{n_{2}}\widetilde{k}_{2}\circ\pi_{n_{1}+j}\right)\\
 &  & +\lambda_{2}\left(\sum_{i=1}^{n_{1}}c_{n_{1},i}\,g_{1}\circ\pi_{i}+\sum_{j=1}^{n_{2}}\widetilde{c}_{n_{2},j}\,g_{2}\circ\pi_{n_{1}+j}\right)
\end{eqnarray*}
\begin{eqnarray*}
 & = & \sum_{i=1}^{n_{1}}\left(\lambda_{1}\frac{\sqrt{n}}{n_{1}}\widetilde{k}_{1}\circ\pi_{i}+\lambda_{2}c_{n_{1},i}\,g_{1}\circ\pi_{i}\right)\\
 &  & +\sum_{j=1}^{n_{2}}\left(\lambda_{1}\frac{\sqrt{n}}{n_{2}}\widetilde{k}_{2}\circ\pi_{n_{1}+j}+\lambda_{2}\widetilde{c}_{n_{2},j}\,g_{2}\circ\pi_{n_{1}+j}\right).
\end{eqnarray*}
Nach dem zentralen Grenzwertsatz von Lindeberg-Feller (vgl. \cite{Bauer:1991},
Satz 28.3) ergibt sich die Konvergenz
\begin{eqnarray*}
\mathcal{L}\left(\left.\sum_{i=1}^{n_{1}}\left(\lambda_{1}\frac{\sqrt{n}}{n_{1}}\widetilde{k}_{1}\circ\pi_{i}+\lambda_{2}c_{n_{1},i}g_{1}\circ\pi_{i}\right)\right|P_{0}^{n_{1}}\right) & \rightarrow & N\left(0,v_{1}^{2}\right)
\end{eqnarray*}
und
\[
\mathcal{L}\left(\left.\sum_{j=1}^{n_{2}}\left(\lambda_{1}\frac{\sqrt{n}}{n_{2}}\widetilde{k}_{2}\circ\pi_{j}+\lambda_{2}\widetilde{c}_{n_{2},j}g_{2}\circ\pi_{j}\right)\right|Q_{0}^{n_{2}}\right)\rightarrow N\left(0,v_{2}^{2}\right)
\]
für $n\rightarrow\infty$, weil $\widetilde{k}_{1},g_{1}\in L_{2}^{(0)}\left(P_{0}\right)$
und $\widetilde{k}_{2},g_{2}\in L_{2}^{(0)}\left(Q_{0}\right)$ sind.
Die Varianzen $v_{1}^{2}$ und $v_{2}^{2}$ erhält man als folgende
Grenzwerte: 
\begin{eqnarray*}
v_{1}^{2} & = & \lim_{n\rightarrow\infty}\mbox{Var}_{P_{0}^{n_{1}}}\left(\sum_{i=1}^{n_{1}}\left(\lambda_{1}\frac{\sqrt{n}}{n_{1}}\widetilde{k}_{1}\circ\pi_{i}+\lambda_{2}c_{n_{1},i}g_{1}\circ\pi_{i}\right)\right)\\
 & = & \lim_{n\rightarrow\infty}\left(\sum_{i=1}^{n_{1}}\mbox{Var}_{P_{0}^{n_{1}}}\left(\lambda_{1}\frac{\sqrt{n}}{n_{1}}\widetilde{k}_{1}\circ\pi_{i}+\lambda_{2}c_{n_{1},i}g_{1}\circ\pi_{i}\right)\right)\\
 & = & \lim_{n\rightarrow\infty}\left(\sum_{i=1}^{n_{1}}\mbox{Var}_{P_{0}^{n_{1}}}\left(\lambda_{1}\frac{\sqrt{n}}{n_{1}}\widetilde{k}_{1}\circ\pi_{i}\right)+\sum_{i=1}^{n_{1}}\mbox{Var}_{P_{0}^{n_{1}}}\left(\lambda_{2}c_{n_{1},i}g_{1}\circ\pi_{i}\right)\right.\\
 &  & +\left.2\sum_{i=1}^{n_{1}}\mbox{Cov}_{P_{0}^{n_{1}}}\left(\lambda_{1}\frac{\sqrt{n}}{n_{1}}\widetilde{k}_{1}\circ\pi_{i},\,\lambda_{2}c_{n_{1},i}g_{1}\circ\pi_{i}\right)\right)\\
 & = & \lambda_{1}^{2}\mbox{Var}_{P_{0}}\left(\widetilde{k}_{1}\right)\lim_{n\rightarrow\infty}\left(\frac{n}{n_{1}}\right)+\lambda_{2}^{2}\mbox{Var}_{P_{0}}\left(g_{1}\right)\lim_{n\rightarrow\infty}\left(\sum_{i=1}^{n_{1}}c_{n_{1},i}^{2}\right)\\
 &  & +2\lambda_{1}\lambda_{2}\mbox{Cov}_{P_{0}}\left(\widetilde{k}_{1},\,g_{1}\right)\lim_{n\rightarrow\infty}\left(\frac{\sqrt{n}}{\sqrt{n_{1}}}\frac{1}{\sqrt{n_{1}}}\sum_{i=1}^{n_{1}}c_{n_{1},i}\right)\\
 & = & \lambda_{1}^{2}\frac{1}{1-d}\int\widetilde{k}_{1}^{2}\,dP_{0}+\lambda_{2}^{2}c_{1}^{2}\int g_{1}^{2}dP_{0}+2\lambda_{1}\lambda_{2}\frac{1}{\sqrt{1-d}}a_{1}\int\widetilde{k}_{1}g_{1}\,dP_{0}
\end{eqnarray*}
und
\begin{eqnarray*}
v_{2}^{2} & = & \lim_{n\rightarrow\infty}\mbox{Var}_{Q_{0}^{n_{2}}}\left(\sum_{j=1}^{n_{2}}\left(\lambda_{1}\frac{\sqrt{n}}{n_{2}}\widetilde{k}_{2}\circ\pi_{j}+\lambda_{2}\widetilde{c}_{n_{2},j}g_{2}\circ\pi_{j}\right)\right)
\end{eqnarray*}
\begin{eqnarray*}
 & = & \lambda_{1}^{2}\mbox{Var}_{Q_{0}}\left(\widetilde{k}_{2}\right)\lim_{n\rightarrow\infty}\left(\frac{n}{n_{2}}\right)+\lambda_{2}^{2}\mbox{Var}_{Q_{0}}\left(g_{2}\right)\lim_{n\rightarrow\infty}\left(\sum_{j=1}^{n_{2}}\widetilde{c}_{n_{2},j}^{2}\right)\\
 &  & +2\lambda_{1}\lambda_{2}\mbox{Cov}_{Q_{0}}\left(\widetilde{k}_{2},\,g_{2}\right)\lim_{n\rightarrow\infty}\left(\frac{\sqrt{n}}{\sqrt{n_{2}}}\frac{1}{\sqrt{n_{2}}}\sum_{j=1}^{n_{2}}\widetilde{c}_{n_{2},j}\right)\\
 & = & \lambda_{1}^{2}\frac{1}{d}\int\widetilde{k}_{2}^{2}\,dQ_{0}+\lambda_{2}^{2}c_{2}^{2}\int g_{2}^{2}dQ_{0}+2\lambda_{1}\lambda_{2}\frac{1}{\sqrt{d}}a_{2}\int\widetilde{k}_{2}g_{2}\,dQ_{0}.
\end{eqnarray*}
Insgesamt ergibt sich 
\begin{eqnarray*}
\textrm{} &  & \mathcal{L}\left(\lambda_{1}T_{n}+\lambda_{2}X_{n}\left|P_{0}^{n_{1}}\otimes Q_{0}^{n_{2}}\right.\right)\\
 & = & \mathcal{L}\left(\sum_{i=1}^{n_{1}}\left(\lambda_{1}\frac{\sqrt{n}}{n_{1}}\widetilde{k}_{1}\circ\pi_{i}+\lambda_{2}c_{n_{1},i}g_{1}\circ\pi_{i}\right)\left|P_{0}^{n_{1}}\right.\right)\\
 &  & *\mathcal{L}\left(\sum_{j=1}^{n_{2}}\left(\lambda_{1}\frac{\sqrt{n}}{n_{2}}\widetilde{k}_{2}\circ\pi_{j}+\lambda_{2}\widetilde{c}_{n_{2},j}g_{2}\circ\pi_{j}\right)\left|Q_{0}^{n_{2}}\right.\right)\\
 & \rightarrow & N\left(0,v_{1}^{2}\right)*N\left(0,v_{2}^{2}\right)\\
 & = & N\left(0,v_{1}^{2}+v_{2}^{2}\right),
\end{eqnarray*}
für $n\rightarrow\infty$, weil die Verteilungen $\mathcal{L}\left(\sum_{i=1}^{n_{1}}\left(\lambda_{1}\frac{\sqrt{n}}{n_{1}}\widetilde{k}_{1}\circ\pi_{i}+\lambda_{2}c_{n_{1},i}\,g_{1}\circ\pi_{i}\right)\left|P_{0}^{n_{1}}\right.\right)$
und $\mathcal{L}\left(\sum_{j=1}^{n_{2}}\left(\lambda_{1}\frac{\sqrt{n}}{n_{2}}\widetilde{k}_{2}\circ\pi_{j}+\lambda_{2}\widetilde{c}_{n_{2},j}g_{2}\circ\pi_{j}\right)\left|Q_{0}^{n_{2}}\right.\right)$
stochastisch unabhängig sind. Die Varianz $v_{1}^{2}+v_{2}^{2}$ erhält
man als
\begin{eqnarray*}
v_{1}^{2}+v_{2}^{2} & = & \lambda_{1}^{2}\frac{1}{1-d}\int\widetilde{k}_{1}^{2}\,dP_{0}+\lambda_{2}^{2}c_{1}^{2}\int g_{1}^{2}dP_{0}+2\lambda_{1}\lambda_{2}\frac{1}{\sqrt{1-d}}a_{1}\int\widetilde{k}_{1}g_{1}\,dP_{0}\\
 &  & +\lambda_{1}^{2}\frac{1}{d}\int\widetilde{k}_{2}^{2}\,dQ_{0}+\lambda_{2}^{2}c_{2}^{2}\int g_{2}^{2}dQ_{0}+2\lambda_{1}\lambda_{2}\frac{1}{\sqrt{d}}a_{2}\int\widetilde{k}_{2}g_{2}\,dQ_{0}\\
 & = & \lambda_{1}^{2}\left(\frac{1}{1-d}\int\widetilde{k}_{1}^{2}\,dP_{0}+\frac{1}{d}\int\widetilde{k}_{2}^{2}\,dQ_{0}\right)\\
 &  & +\lambda_{2}^{2}\left(c_{1}^{2}\int g_{1}^{2}dP_{0}+c_{2}^{2}\int g_{2}^{2}dQ_{0}\right)\\
 &  & +2\lambda_{1}\lambda_{2}\left(a_{1}\frac{1}{\sqrt{1-d}}\int\widetilde{k}_{1}g_{1}\,dP_{0}+a_{2}\frac{1}{\sqrt{d}}\int\widetilde{k}_{2}g_{2}\,dQ_{0}\right).
\end{eqnarray*}
Somit ist alles bewiesen.\end{proof} 

\section{Asymptotische Eigenschaften der einseitigen Tests\label{sec:Asymptotische-Eigenschaften-der} }

Die Ergebnisse aus Satz \ref{gem-asymp-vert} werden nun zum asymptotisch
optimalen Testen der statistischen Funktionale angewandt. Dies erfordert
allerdings noch einige Vorbereitungen. 

Eine Folge $E_{n}=\left(\Omega_{n},\mathcal{A}_{n},\left\{ P_{n,\vartheta}:\vartheta\in\Theta_{n}\right\} \right)$
von Experimenten sei LAN mit $\left(X_{n}\right)_{n\in\mathbb{N}}$
als zentrale Folge. Man betrachtet die Folge der einseitigen Testprobleme
\[
H_{1,n}=\left\{ \vartheta\in\Theta_{n}:\vartheta\leq0\right\} \quad\mbox{gegen}\quad K_{1,n}=\left\{ \vartheta\in\Theta_{n}:\vartheta>0\right\} .
\]
Die gebräuchlichen Tests sind von der Gestalt
\[
\varphi_{n}=\left\{ \begin{array}{cccc}
1 &  & >\\
\lambda_{n} & S_{n} & = & c_{n}\\
0 &  & <
\end{array}\right.,
\]
 die durch die Nebenbedingung 
\begin{equation}
\lim_{n\rightarrow\infty}E_{P_{n,0}}\left(\varphi_{n}\right)=\alpha\label{asympt_unverf}
\end{equation}
festgelegt sind.\begin{defi}\label{asymp_0_aplha_anlich_testfolge}
Eine Testfolge $\left(\varphi_{n}\right)_{n\in\mathbb{N}}$ heißt
asymptotisch $\left\{ 0\right\} $-$\alpha$-ähnlich, falls die Bedingung
(\ref{asympt_unverf}) erfüllt ist.\end{defi} Derartige Testfolgen
lassen sich mit Hilfe der Theorie von Le Cam vergleichen. Die Vorgehensweise
von Le Cam führt dann zu den wichtigen Ergebnissen der asymptotischen
Testtheorie. Insbesondere kann die asymptotische Gütefunktion einer
asymptotisch $\left\{ 0\right\} $-$\alpha$-ähnlichen Testfolge berechnet
werden. \begin{satz}[asymptotische G\"utefunktion]\label{asymp-guete}Sei
$\left(\varphi_{n}\right)_{n\in\mathbb{N}}$ eine asymptotisch $\left\{ 0\right\} $-$\alpha$-ähnliche
Testfolge. Außerdem gelte
\begin{eqnarray*}
\mathcal{L}\left(\left(S_{n},X_{n}\right)^{t}\left|P_{n,0}\right.\right) & \rightarrow & N\left(\left(\begin{array}{c}
0\\
0
\end{array}\right),\left(\begin{array}{cc}
\sigma_{1}^{2} & \sigma_{12}\\
\sigma_{12} & \sigma_{2}^{2}
\end{array}\right)\right)
\end{eqnarray*}
für $n\rightarrow\infty.$ Für jedes $\vartheta\in\Theta$ gilt dann
\[
\lim_{n\rightarrow\infty}E_{P_{n,\vartheta}}\left(\varphi_{n}\right)=\Phi\left(\vartheta\frac{\sigma_{12}}{\sigma_{1}}-u_{1-\alpha}\right),
\]
 wobei $\Phi$ die Verteilungsfunktion der Standardnormalverteilung
bezeichnet und $u_{1-\alpha}:=\Phi^{-1}\left(1-\alpha\right)$ das
$\left(1-\alpha\right)$-Quantil von $\Phi$ ist. \end{satz}\begin{proof}vgl.
\cite{Janssen:1998}, Seite 147, Satz 15.2. oder \cite{Hajek:1999},
Chapter 7.\end{proof}Für die einseitigen Testprobleme
\[
H_{1}=\left\{ P\otimes Q\in\mathcal{P}\otimes\mathcal{Q}:k\left(P\otimes Q\right)\leq a\right\} 
\]
\begin{equation}
\mbox{gegen}\quad K_{1}=\left\{ P\otimes Q\in\mathcal{P}\otimes\mathcal{Q}:k\left(P\otimes Q\right)>a\right\} \label{einseit_testproblem}
\end{equation}
 wird eine lokale Parametrisierung des Testproblems in Abhängigkeit
von einem statistischen Funktional $k:\mathcal{P}\otimes\mathcal{Q}\rightarrow\mathbb{R}$
gesucht. 

Ein statistisches Funktional $k:\mathcal{P}\otimes\mathcal{Q}\rightarrow\mathbb{R}$
sei differenzierbar an der Stelle $P_{0}\otimes Q_{0}$ mit dem kanonischen
Gradienten $\widetilde{k}\in L_{2}^{(0)}\left(P_{0}\otimes Q_{0}\right).$
Außerdem gelte 
\[
\left\Vert \widetilde{k}\right\Vert _{L_{2}\left(P_{0}\otimes Q_{0}\right)}^{2}\neq0.
\]
 Sei $t\mapsto P_{t}\otimes Q_{t}$ eine $L_{2}\left(0\right)$-differenzierbare
Kurve in $\mathcal{P}\otimes\mathcal{Q}$ mit Tangente $g\in L_{2}^{(0)}\left(P_{0}\otimes Q_{0}\right)$.
Ferner gilt $g=g_{1}\circ\pi_{1}+g_{2}\circ\pi_{2}$ für die Tangenten
$g_{1}\in L_{2}^{(0)}\left(P_{0}\right)$ und $g_{2}\in L_{2}^{(0)}\left(Q_{0}\right)$,
vgl. Satz \ref{L2Produkt}. Eine lokale Parametrisierung des Testproblems
(\ref{einseit_testproblem}) erhält man zum Beispiel durch
\begin{equation}
H_{n}^{p}=\left\{ P_{0}^{n_{1}}\otimes Q_{0}^{n_{2}}\right\} \quad\mbox{gegen}\quad K_{n}^{p}=\left\{ P_{\frac{t}{\sqrt{n}}}^{n_{1}}\otimes Q_{\frac{t}{\sqrt{n}}}^{n_{2}}\right\} \label{asympt_testproblem}
\end{equation}
für ein $t>0$, wobei die Bedingung
\begin{equation}
\lim_{n\rightarrow\infty}\frac{k\left(P_{\frac{t}{\sqrt{n}}}\otimes Q_{\frac{t}{\sqrt{n}}}\right)-k\left(P_{0}\otimes Q_{0}\right)}{n^{-\frac{1}{2}}}>0\label{rechtseitege_bedingung}
\end{equation}
für alle $t>0$ vorausgesetzt wird, vgl. \cite{Janssen:1999a} und
\cite{Janssen:1999b}. Die Bedingung (\ref{rechtseitege_bedingung})
impliziert insbesondere 
\[
\int\widetilde{k}g\,dP_{0}\otimes Q_{0}>0
\]
wegen 
\[
\lim_{n\rightarrow\infty}\frac{k\left(P_{\frac{t}{\sqrt{n}}}\otimes Q_{\frac{t}{\sqrt{n}}}\right)-k\left(P_{0}\otimes Q_{0}\right)}{n^{-\frac{1}{2}}}=t\int\widetilde{k}g\,dP_{0}\otimes Q_{0}.
\]
 Die Nullhypothese $H_{n}^{p}$ lässt sich ebenfalls erweitern. Sei
$s\mapsto P_{s}\otimes Q_{s}$ eine weitere $L_{2}(0)$-differenzierbare
Kurve in $\mathcal{P}\otimes\mathcal{Q}$ mit Tangente $h\in L_{2}\left(P_{0}\otimes Q_{0}\right)$.
Außerdem gelte 
\begin{equation}
\lim_{n\rightarrow\infty}\frac{k\left(P_{\frac{s}{\sqrt{n}}}\otimes Q_{\frac{s}{\sqrt{n}}}\right)-k\left(P_{0}\otimes Q_{0}\right)}{n^{-\frac{1}{2}}}\leq0\label{linsseitige_bedingung}
\end{equation}
 für alle $s>0$. Die erweiterte Nullhypothese $H_{n}^{p,e}$ ergibt
sich dann als 
\begin{equation}
H_{n}^{p,e}=\left\{ P_{\frac{s}{\sqrt{n}}}^{n_{1}}\otimes Q_{\frac{s}{\sqrt{n}}}^{n_{2}}\right\} \label{Hnpe}
\end{equation}
 für ein $s>0$. Sowohl für das Testproblem (\ref{asympt_testproblem})
als auch für das erweiterte Testproblem 
\[
H_{n}^{p,e}\quad\mbox{gegen}\quad K_{n}^{p}
\]
eignet sich die Testfolge
\begin{eqnarray}
\psi_{n} & = & \left\{ \begin{array}{cccc}
1 &  & >\\
 & T_{n} &  & c\\
0 &  & \leq
\end{array}\right.,\label{einseit_asym_opt_testfolge}
\end{eqnarray}
wobei $T_{n}$ die Teststatistik wie in (\ref{teststatistik T}) ist.
Der kritische Wert 
\[
c:=u_{1-\alpha}\left(\frac{1}{1-d}\int\widetilde{k}_{1}^{2}\,dP_{0}+\frac{1}{d}\int\widetilde{k}_{2}^{2}\,dQ_{0}\right)^{\frac{1}{2}}
\]
wird unabhängig von $n$ gewählt. Es ist zu beachten, dass 
\[
\left\Vert \widetilde{k}\right\Vert _{L_{2}\left(P_{0}\otimes Q_{0}\right)}^{2}=\int\widetilde{k}_{1}^{2}\,dP_{0}+\int\widetilde{k}_{2}^{2}\,dQ_{0}\neq0
\]
 nach Voraussetzung gilt. 

Die impliziten Alternativen aus (\ref{asympt_testproblem}) und die
impliziten Hypothesen (\ref{Hnpe}) lassen sich noch etwas verallgemeinern,
indem man die lokalen Parameter $\frac{t}{\sqrt{n}}$ durch die geeigneten
Nullfolgen $\left(t_{n}\right)_{n\in\mathbb{N}}$ ersetzt, vgl. \cite{Janssen:1999a}
und \cite{Janssen:1999b}. Die Parametrisierung ist dann durch die
lokalen Werte des Funktionals gegeben. Für ein $\vartheta\in\mathbb{R}\setminus\left\{ 0\right\} $
sei $\left(t_{n}\left(\vartheta\right)\right)_{n\in\mathbb{N}}$ eine
Nullfolge, welche die Bedingung
\begin{equation}
k\left(P_{t_{n}\left(\vartheta\right)}\otimes Q_{t_{n}\left(\vartheta\right)}\right)=k\left(P_{0}\otimes Q_{0}\right)+n^{-\frac{1}{2}}\vartheta+o\left(n^{-\frac{1}{2}}\right)\label{implizite Alternative}
\end{equation}
für $n\rightarrow\infty$ erfüllt. Die Folge $\left(P_{t_{n}\left(\vartheta\right)}\otimes Q_{t_{n}\left(\vartheta\right)}\right)_{n\in\mathbb{N}}$
der Wahrscheinlichkeitsmaße wird für $\vartheta>0$ als implizite
Alternative und für $\vartheta<0$ als implizite Hypothese bezeichnet.
Es sei zusätzlich vorausgesetzt, dass $\int\widetilde{k}g\,dP_{0}\otimes Q_{0}\neq0$
ist. Dann gilt
\begin{eqnarray}
\vartheta & = & \lim_{n\rightarrow\infty}\frac{k\left(P_{t_{n}\left(\vartheta\right)}\otimes Q_{t_{n}\left(\vartheta\right)}\right)-k\left(P_{0}\otimes Q_{0}\right)}{n^{-\frac{1}{2}}}\nonumber \\
 & = & \lim_{n\rightarrow\infty}\frac{k\left(P_{t_{n}\left(\vartheta\right)}\otimes Q_{t_{n}\left(\vartheta\right)}\right)-k\left(P_{0}\otimes Q_{0}\right)}{t_{n}\left(\vartheta\right)}\frac{t_{n}\left(\vartheta\right)}{n^{-\frac{1}{2}}}\nonumber \\
 & = & \left(\int\widetilde{k}g\,dP_{0}\otimes Q_{0}\right)\lim_{n\rightarrow\infty}\left(n^{\frac{1}{2}}t_{n}\left(\vartheta\right)\right).\label{diff_t_n_grad}
\end{eqnarray}
Die Folge
\[
\widetilde{t}_{n}(\vartheta)=n^{-\frac{1}{2}}\vartheta\left(\int\widetilde{k}g\,dP_{0}\otimes Q_{0}\right)^{-1}
\]
erfüllt somit die Bedingung (\ref{implizite Alternative}) nach Konstruktion.
Für alle $\vartheta\in\!\mathbb{R}\!\setminus\left\{ 0\right\} $
existiert also eine Nullfolge $\left(t_{n}\left(\vartheta\right)\right)_{n\in\mathbb{N}},$
die der Bedingung (\ref{implizite Alternative}) genügt. Seien $\left(n_{1}\right)_{n\in\mathbb{N}}\subset\mathbb{N}$
und $\left(n_{2}\right)_{n\in\mathbb{N}}\subset\mathbb{N}$ zwei Folgen
mit $n_{1}+n_{2}=n$ für alle $n\in\mathbb{N}$ und $\lim\limits_{n\rightarrow\infty}\frac{n_{2}}{n}=d\in(0,1)$.
Die Folge 
\[
E_{n}=\left(\Omega_{1}^{n_{1}}\times\Omega_{2}^{n_{2}},\,\mathcal{A}_{1}^{n_{1}}\otimes\mathcal{A}_{2}^{n_{2}},\left\{ P_{\widetilde{t}_{n}(\vartheta)}^{n_{1}}\otimes Q_{\widetilde{t}_{n}(\vartheta)}^{n_{2}}:\vartheta\in\Theta_{n}\right\} \right)
\]
der Produktexperimente ist dann ULAN nach Anwendung \ref{ulan_anwendung}.
\\
Wegen $\lim\limits_{n\rightarrow\infty}\frac{\widetilde{t}_{n}(\vartheta)}{t_{n}(\vartheta)}=1$
folgt
\[
\lim_{n\rightarrow\infty}\left\Vert P_{t_{n}(\vartheta)}^{n_{1}}\otimes Q_{t_{n}(\vartheta)}^{n_{2}}-P_{\widetilde{t}_{n}(\vartheta)}^{n_{1}}\otimes Q_{\widetilde{t}_{n}(\vartheta)}^{n_{2}}\right\Vert =0
\]
 nach \cite{Janssen:1998}, Satz $14.17$, weil die Folge $\left(E_{n}\right)_{n\in\mathbb{N}}$
von Experimenten ULAN ist. Man erhält somit 
\begin{eqnarray}
\lim_{n\rightarrow\infty}\int\varphi_{n}\,dP_{t_{n}(\vartheta)}^{n_{1}}\otimes Q_{t_{n}(\vartheta)}^{n_{2}} & = & \lim_{n\rightarrow\infty}\int\varphi_{n}\,dP_{\widetilde{t}_{n}(\vartheta)}^{n_{1}}\otimes Q_{\widetilde{t}_{n}(\vartheta)}^{n_{2}}\label{t-strich-vereinfachung}
\end{eqnarray}
 für jede Testfolge $\left(\varphi_{n}\right)_{n\in\mathbb{N}}$ nach
Hilfssatz \ref{hsatz1}, falls einer der Grenzwerte existiert. Man
sagt, dass alle Testfolgen unter $P_{t_{n}(\vartheta)}^{n_{1}}\otimes Q_{t_{n}(\vartheta)}^{n_{2}}$
und $P_{\widetilde{t}_{n}(\vartheta)}^{n_{1}}\otimes Q_{\widetilde{t}_{n}(\vartheta)}^{n_{2}}$
die gleiche asymptotische Güte besitzen. Die asymptotischen Eigenschaften
der Testfolge $\left(\psi_{n}\right)_{n\in\mathbb{N}}$ werden nun
mit Hilfe der Theorie von Le Cam untersucht. \begin{satz}[asymptotische G\"utefunktion]\label{asymp_gute}Seien
$\left(n_{1}\right)_{n\in\mathbb{N}}\subset\mathbb{N}$ und $\left(n_{2}\right)_{n\in\mathbb{N}}\subset\mathbb{N}$
zwei Folgen mit $n_{1}+n_{2}=n$ für alle $n\in\mathbb{N}$ und $\lim\limits_{n\rightarrow\infty}\frac{n_{2}}{n}=d\in(0,1)$.
Es gelte 
\begin{equation}
\int\widetilde{k}^{2}\,dP_{0}\otimes Q_{0}\neq0\label{vor-kanon-grad}
\end{equation}
und
\[
\int\widetilde{k}g\,dP_{0}\otimes Q_{0}\neq0.
\]
 Außerdem sei $\left(P_{t_{n}}\otimes Q_{t_{n}}\right)_{n\in\mathbb{N}}$
eine Folge mit 
\[
k\left(P_{t_{n}}\otimes Q_{t_{n}}\right)=k\left(P_{0}\otimes Q_{0}\right)+n^{-\frac{1}{2}}\vartheta+o\left(n^{-\frac{1}{2}}\right)
\]
 für ein $\vartheta\in\mathbb{R}\setminus\{0\}.$ Die Testfolge $\left(\psi_{n}\right)_{n\in\mathbb{N}}$
ist dann asymptotisch $\{0\}$-$\alpha$-ähnlich und es gilt 
\begin{equation}
\lim_{n\rightarrow\infty}\int\psi_{n}\,dP_{t_{n}}^{n_{1}}\otimes Q_{t_{n}}^{n_{2}}=\Phi\left(\vartheta\left(\frac{1}{1-d}\int\widetilde{k}_{1}^{2}\,dP_{0}+\frac{1}{d}\int\widetilde{k}_{2}^{2}\,dQ_{0}\right)^{-\frac{1}{2}}-u_{1-\alpha}\right).\label{asyp_gutefunktion}
\end{equation}
Insbesondere hält die Testfolge $\left(\psi_{n}\right)_{n\in\mathbb{N}}$
asymptotisch das Niveau $\alpha$ für $\vartheta<0$ ein.\end{satz}\begin{proof}
Wegen (\ref{t-strich-vereinfachung}) reicht es, den Grenzwert 
\[
\lim_{n\rightarrow\infty}\int\psi_{n}\,dP_{\widetilde{t}_{n}}^{n_{1}}\otimes Q_{\widetilde{t}_{n}}^{n_{2}}
\]
 für 
\[
\widetilde{t}_{n}(\vartheta)=n^{-\frac{1}{2}}\vartheta\left(\int\widetilde{k}g\,dP_{0}\otimes Q_{0}\right)^{-1}
\]
zu berechnen. Die Testfolge $\left(\psi_{n}\right)_{n\in\mathbb{N}}$
ist asymptotisch $\{0\}$-$\alpha$-ähnlich wegen
\begin{eqnarray*}
\lim_{n\rightarrow\infty}\int\psi_{n}\,dP_{0}^{n_{1}}\otimes Q_{0}^{n_{2}} & = & \lim_{n\rightarrow\infty}P_{0}^{n_{1}}\otimes Q_{0}^{n_{2}}\left(\left\{ T_{n}>c\right\} \right)\\
 & = & N\left(0,\left(\frac{1}{1-d}\int\widetilde{k}_{1}^{2}\,dP_{0}+\frac{1}{d}\int\widetilde{k}_{2}^{2}\,dQ_{0}\right)\right)\left(\left(c,+\infty\right)\right)\\
 & = & 1-\Phi\left(\frac{u_{1-\alpha}\left(\frac{1}{1-d}\int\widetilde{k}_{1}^{2}\,dP_{0}+\frac{1}{d}\int\widetilde{k}_{2}^{2}\,dQ_{0}\right)}{\frac{1}{1-d}\int\widetilde{k}_{1}^{2}\,dP_{0}+\frac{1}{d}\int\widetilde{k}_{2}^{2}\,dQ_{0}}\right)\\
 & = & \alpha.
\end{eqnarray*}
 Nach Satz \ref{ulan_produkt} und Anwendung \ref{ulan_anwendung}
erhält man, dass die Folge der Experimente 
\[
E_{n}=\left(\Omega_{1}^{n_{1}}\times\Omega_{2}^{n_{2}},\mathcal{A}_{1}^{n_{1}}\otimes\mathcal{A}_{2}^{n_{2}},\left\{ P_{\widetilde{t}_{n}(\vartheta)}^{n_{1}}\otimes Q_{\widetilde{t}_{n}(\vartheta)}^{n_{2}}:\vartheta\in\Theta_{n}\right\} \right)
\]
 ULAN ist. Eine zentrale Folge ist dann
\begin{eqnarray*}
X_{n} & = & \sum_{i=1}^{n_{1}}c_{n_{1},i}\,g_{1}\circ\pi_{i}+\sum_{j=1}^{n_{2}}\widetilde{c}_{n_{2},j}\,g_{2}\circ\pi_{j+n_{1}},
\end{eqnarray*}
wobei die Koeffizienten $c_{n_{1},i}$ und $\widetilde{c}_{n_{2},j}$
für alle $i\in\left\{ 1,\ldots,n_{1}\right\} $ und $j\in\left\{ 1,\ldots,n_{2}\right\} $
durch
\[
c_{n_{1},i}:=n^{-\frac{1}{2}}\left(\int\widetilde{k}g\,dP_{0}\otimes Q_{0}\right)^{-1},
\]
\[
\widetilde{c}_{n_{2},j}:=n^{-\frac{1}{2}}\left(\int\widetilde{k}g\,dP_{0}\otimes Q_{0}\right)^{-1}
\]
 gegeben sind. Nach Satz \ref{gem-asymp-vert} ergibt sich die Konvergenz
\[
\mathcal{L}\left(\left(T_{n},X_{n}\right)^{t}\left|P_{0}^{n_{1}}\otimes Q_{0}^{n_{2}}\right.\right)\rightarrow N\left(\left(\begin{array}{c}
0\\
0
\end{array}\right),\left(\begin{array}{cc}
\sigma_{1}^{2} & \sigma_{12}\\
\sigma_{12} & \sigma_{2}^{2}
\end{array}\right)\right)
\]
für $n\rightarrow\infty$. Als Einträge der Kovarianzmatrix erhält
man 
\[
\sigma_{1}^{2}=\frac{1}{1-d}\int\widetilde{k}_{1}^{2}\,dP_{0}+\frac{1}{d}\int\widetilde{k}_{2}^{2}\,dQ_{0},
\]
\[
\sigma_{2}^{2}=c_{1}^{2}\int g_{1}^{2}\,dP+c_{2}^{2}\int g_{2}^{2}\,dQ,
\]
\[
\sigma_{12}=a_{1}\frac{1}{\sqrt{1-d}}\int\widetilde{k}_{1}g_{1}\,dP_{0}+a_{2}\frac{1}{\sqrt{d}}\int\widetilde{k}_{2}g_{2}\,dQ_{0}.
\]
Die Werte $a_{1}$ und $a_{2}$ sind nun zu bestimmen. Es gilt zunächst
\begin{eqnarray*}
a_{1} & = & \lim_{n_{1}\rightarrow\infty}n_{1}^{-\frac{1}{2}}\sum_{i=1}^{n_{1}}c_{n_{1},i}\\
 & = & \lim_{n_{1}\rightarrow\infty}n_{1}^{\frac{1}{2}}n^{-\frac{1}{2}}\left(\int\widetilde{k}g\,dP_{0}\otimes Q_{0}\right)\\
 & = & \sqrt{\left(1-d\right)}\left(\int\widetilde{k}g\,dP_{0}\otimes Q_{0}\right)^{-1}
\end{eqnarray*}
und
\begin{eqnarray*}
a_{2} & = & \lim_{n_{2}\rightarrow\infty}n_{2}^{-\frac{1}{2}}\sum_{i=1}^{n_{2}}\widetilde{c}_{n_{2},i}\\
 & = & \lim_{n_{1}\rightarrow\infty}n_{2}^{\frac{1}{2}}n^{-\frac{1}{2}}\left(\int\widetilde{k}g\,dP_{0}\otimes Q_{0}\right)^{-1}\\
 & = & \sqrt{d}\left(\int\widetilde{k}g\,dP_{0}\otimes Q_{0}\right)^{-1}.
\end{eqnarray*}
Für die Kovarinz $\sigma_{12}$ ergibt sich dann
\begin{eqnarray*}
\sigma_{12} & = & a_{1}\frac{1}{\sqrt{1-d}}\int\widetilde{k}_{1}g_{1}\,dP_{0}+a_{2}\frac{1}{\sqrt{d}}\int\widetilde{k}_{2}g_{2}\,dQ_{0}\\
 & = & \int\widetilde{k}_{1}g_{1}\,dP_{0}\left(\int\widetilde{k}g\,dP_{0}\otimes Q_{0}\right)^{-1}+\int\widetilde{k}_{2}g_{2}\,dQ_{0}\left(\int\widetilde{k}g\,dP_{0}\otimes Q_{0}\right)^{-1}\\
 & = & \left(\int\widetilde{k}g\,dP_{0}\otimes Q_{0}\right)^{-1}\left(\int\widetilde{k}_{1}g_{1}\,dP_{0}+\int\widetilde{k}_{2}g_{2}\,dQ_{0}\right)\\
 & = & 1.
\end{eqnarray*}
Nach Satz \ref{asymp-guete} erhält man die asymptotische Güte als
\begin{eqnarray*}
 &  & \lim_{n\rightarrow\infty}\int\psi_{n}\,dP_{\widetilde{t}_{n}(\vartheta)}^{n_{1}}\otimes Q_{\widetilde{t}_{n}(\vartheta)}^{n_{2}}\\
 & = & \Phi\left(\vartheta\frac{\sigma_{12}}{\sigma_{1}}-u_{1-\alpha}\right)\\
 & = & \Phi\left(\vartheta\left(\frac{1}{1-d}\int\widetilde{k}_{1}^{2}\,dP_{0}+\frac{1}{d}\int\widetilde{k}_{2}^{2}\,dQ_{0}\right)^{-\frac{1}{2}}-u_{1-\alpha}\right).
\end{eqnarray*}
 \end{proof}\begin{bem}Die asymptotische Gütefunktion wird durch
$d=\lim\limits_{n\rightarrow\infty}\frac{n_{2}}{n}$, die Normen $\left\Vert \widetilde{k}_{1}\right\Vert _{L_{2}\left(P_{0}\right)}$
und $\left\Vert \widetilde{k}_{2}\right\Vert _{L_{2}\left(Q_{0}\right)}$
der Komponenten des kanonischen Gradienten $\widetilde{k}$ und die
lokalen Werte $\vartheta$ vollständig beschrieben. Es ist bemerkenswert,
dass die asymptotische Gütefunktion von der Tangente $g\in L_{2}\left(P_{0}\otimes Q_{0}\right)$
unabhängig ist. Die $L_{2}(0)$-differenzierbaren Kurven sind daher
als mathematische Konstruktionen anzusehen, die asymptotisch wieder
eliminiert werden, vgl. \cite{Janssen:1999a} und \cite{Janssen:1998}.
\end{bem}\begin{satz}\label{niveau_alpha_trivial}Es seien $t\mapsto P_{t}\otimes Q_{t}$
eine $L_{2}\left(0\right)$-differenzierbare Kurve in $\mathcal{P}\otimes\mathcal{Q}$
mit Tangente $g=0\in L_{2}\left(P_{0}\otimes Q_{0}\right)$ und $\left(\phi_{n}\right)_{n\in\mathbb{N}}$
eine Testfolge mit 
\[
\lim\limits_{n\rightarrow\infty}\int\phi_{n}\,dP_{0}^{n_{1}}\otimes Q_{0}^{n_{2}}=\alpha.
\]
 Für jede Nullfolge $\left(t_{n}\right)_{n\in\mathbb{N}}$ mit $\lim_{n\rightarrow\infty}t_{n}\sqrt{n}>0$
gilt dann 
\[
\lim\limits_{n\rightarrow\infty}\int\phi_{n}\,dP_{t_{n}}^{n_{1}}\otimes Q_{t_{n}}^{n_{2}}=\alpha.
\]
 \end{satz}\begin{proof}Die Kurven $t\mapsto P_{t}$ in $\mathcal{P}$
und $t\mapsto Q_{t}$ in $\mathcal{Q}$ sind nach Satz \ref{L2Produkt} 
und Bemerkung \ref{prod_bem} beide $L_{2}\left(0\right)$-differenzierbar
mit Tangenten $0\in L_{2}\left(P_{0}\right)$ und $0\in L_{2}\left(Q_{0}\right)$.
Wegen $\lim_{n\rightarrow\infty}t_{n}\sqrt{n}>0$ erhält man dann
\begin{eqnarray*}
\left\Vert P_{t_{n}}^{n_{1}}\otimes Q_{t_{n}}^{n_{2}}-P_{0}^{n_{1}}\otimes Q_{0}^{n_{2}}\right\Vert  & \leq & \left\Vert P_{t_{n}}^{n_{1}}\otimes Q_{t_{n}}^{n_{2}}-P_{t_{n}}^{n_{1}}\otimes Q_{0}^{n_{2}}\right\Vert \\
 &  & +\left\Vert P_{t_{n}}^{n_{1}}\otimes Q_{0}^{n_{2}}-P_{0}^{n_{1}}\otimes Q_{0}^{n_{2}}\right\Vert \\
 & = & \left\Vert Q_{t_{n}}^{n_{2}}-Q_{0}^{n_{2}}\right\Vert +\left\Vert P_{t_{n}}^{n_{1}}-P_{0}^{n_{1}}\right\Vert \\
 & \rightarrow & 0
\end{eqnarray*}
 für $n\rightarrow\infty$ nach \cite{Strasser:1985b}, Seite 386,
Theorem 75.8. Nach Hilfssatz \ref{hsatz1} gilt dann

\[
\lim\limits_{n\rightarrow\infty}\int\phi_{n}\,dP_{t_{n}}^{n_{1}}\otimes Q_{t_{n}}^{n_{2}}=\lim\limits_{n\rightarrow\infty}\int\phi_{n}\,dP_{0}^{n_{1}}\otimes Q_{0}^{n_{2}}=\alpha.
\]
 \end{proof}\begin{satz}[asymptotisch unverf\"alschte Testfolge]\label{asymp_unverf_testfolge}

Es gelte $\int\widetilde{k}^{2}\,dP_{0}\otimes Q_{0}\neq0.$ Sei $P_{0}\otimes Q_{0}=\widetilde{P}_{0}\otimes\widetilde{Q}_{0}\in\mathcal{P}\otimes\mathcal{Q}$.
Es seien $t\mapsto P_{t}\otimes Q_{t}$ und $s\mapsto\widetilde{P}_{s}\otimes\widetilde{Q}_{s}$
zwei $L_{2}(0)$-differenzierbare Kurven mit Tangenten $g\in L_{2}^{(0)}\left(P_{0}\otimes Q_{0}\right)$
und $h\in L_{2}^{(0)}\left(\widetilde{P}_{0}\otimes\widetilde{Q}_{0}\right)$.
Es gelten außerdem die Bedingungen 
\begin{equation}
\lim_{n\rightarrow\infty}\frac{k\left(P_{t_{n}}\otimes Q_{t_{n}}\right)-k\left(P_{0}\otimes Q_{0}\right)}{n^{-\frac{1}{2}}}>0\label{lb1}
\end{equation}
und
\begin{equation}
\lim_{n\rightarrow\infty}\frac{k\left(\widetilde{P}_{s_{n}}\otimes\widetilde{Q}_{s_{n}}\right)-k\left(\widetilde{P}_{0}\otimes\widetilde{Q}_{0}\right)}{n^{-\frac{1}{2}}}\leq0,\label{rb1}
\end{equation}
für alle Nullfolgen $\left(t_{n}\right)_{n\in\mathbb{N}}$ mit $\lim_{n\rightarrow\infty}t_{n}\sqrt{n}>0$
und alle Nullfolgen $\left(s_{n}\right)_{n\in\mathbb{N}}$ mit $\lim_{n\rightarrow\infty}s_{n}\sqrt{n}>0$.
Die Testfolge $\left(\psi_{n}\right)_{n\in\mathbb{N}}$ für die Testprobleme
\begin{equation}
H_{n}^{p,e}=\left\{ \widetilde{P}_{s_{n}}^{n_{1}}\otimes\widetilde{Q}_{s_{n}}^{n_{2}}\right\} \quad\mbox{gegen}\quad K_{n}^{p}=\left\{ P_{t_{n}}^{n_{1}}\otimes Q_{t_{n}}^{n_{2}}\right\} \label{he-gegen-k}
\end{equation}
 ist dann asymptotisch unverfälscht zum Niveau $\alpha$.\end{satz}\begin{proof}
Man setze zuerst
\begin{eqnarray}
\vartheta & = & \lim_{n\rightarrow\infty}\frac{k\left(P_{t_{n}}\otimes Q_{t_{n}}\right)-k\left(P_{0}\otimes Q_{0}\right)}{n^{-\frac{1}{2}}}\nonumber \\
 & = & \lim_{n\rightarrow\infty}\frac{t_{n}}{n^{-\frac{1}{2}}}\frac{k\left(P_{t_{n}}\otimes Q_{t_{n}}\right)-k\left(P_{0}\otimes Q_{0}\right)}{t_{n}}\nonumber \\
 & = & t\int\widetilde{k}g\,dP_{0}\otimes Q_{0},\label{theta-funk-umrechnung}
\end{eqnarray}
wobei $t:=\lim\limits_{n\rightarrow\infty}\frac{t_{n}}{n^{-\frac{1}{2}}}$
ist. Unmittelbar aus Voraussetzung (\ref{lb1}) folgt dann $\vartheta>0$
und $\int\widetilde{k}g\,dP_{0}\neq0.$ Mit Hilfe von Satz \ref{asymp_gute}
erhält man 
\begin{eqnarray*}
 &  & \lim_{n\rightarrow\infty}\int\psi_{n}dP_{t_{n}}^{n_{1}}\otimes Q_{t_{n}}^{n_{2}}\\
 & = & \Phi\left(\underbrace{\vartheta\left(\frac{1}{1-d}\int\widetilde{k}_{1}^{2}\,dP_{0}+\frac{1}{d}\int\widetilde{k}_{2}^{2}\,dQ_{0}\right)^{-\frac{1}{2}}}_{>0}-u_{1-\alpha}\right)\\
 & > & \Phi\left(-u_{1-\alpha}\right)\\
 & = & \alpha.
\end{eqnarray*}
 Es gelte nun
\[
\lim_{n\rightarrow\infty}\frac{k\left(\widetilde{P}_{s_{n}}\otimes\widetilde{Q}_{s_{n}}\right)-k\left(\widetilde{P}_{0}\otimes\widetilde{Q}_{0}\right)}{n^{-\frac{1}{2}}}<0.
\]
Der Parameter $\widetilde{\vartheta}$ wird durch 
\[
\widetilde{\vartheta}=\lim_{n\rightarrow\infty}\frac{k\left(\widetilde{P}_{s_{n}}\otimes\widetilde{Q}_{s_{n}}\right)-k\left(\widetilde{P}_{0}\otimes\widetilde{Q}_{0}\right)}{n^{-\frac{1}{2}}}=s\int\widetilde{k}h\,dP_{0}\otimes Q_{0}
\]
bestimmt, wobei $s:=\lim\limits_{n\rightarrow\infty}s_{n}n^{\frac{1}{2}}$
ist. Hieraus folgt $\widetilde{\vartheta}<0$ und $\int\!\widetilde{k}h\,dP_{0}\otimes Q_{0}\!\neq\!0$
nach Voraussetzung. Nach Satz \ref{asymp_gute} ergibt sich
\begin{eqnarray*}
\textrm{} &  & \lim_{n\rightarrow\infty}\int\psi_{n}dP_{s_{n}}^{n_{1}}\otimes Q_{s_{n}}^{n_{2}}\\
 & = & \Phi\left(\underbrace{\widetilde{\vartheta}\left(\frac{1}{1-d}\int\widetilde{k}_{1}^{2}\,dP_{0}+\frac{1}{d}\int\widetilde{k}_{2}^{2}\,dQ_{0}\right)^{-\frac{1}{2}}}_{<0}-u_{1-\alpha}\right)
\end{eqnarray*}
\begin{eqnarray*}
 & < & \Phi\left(-u_{1-\alpha}\right)\\
 & = & \alpha.
\end{eqnarray*}
 Sei nun 
\[
\lim_{n\rightarrow\infty}\frac{k\left(\widetilde{P}_{s_{n}}\otimes\widetilde{Q}_{s_{n}}\right)-k\left(\widetilde{P}_{0}\otimes\widetilde{Q}_{0}\right)}{n^{-\frac{1}{2}}}=0.
\]
Das impliziert $\int\widetilde{k}h\,dP_{0}\otimes Q_{0}=0$. Ohne
Einschränkung sei $\int h^{2}\,dP_{0}\otimes Q_{0}\neq0$, denn sonst
gilt 
\[
\lim_{n\rightarrow\infty}\int\psi_{n}dP_{s_{n}}^{n_{1}}\otimes Q_{s_{n}}^{n_{2}}=\alpha
\]
 nach Satz \ref{niveau_alpha_trivial}. Sei $m\in\mathbb{N}$, so
dass $\widetilde{P}_{s_{n}}\otimes\widetilde{Q}_{s_{n}}\in\mathcal{P}\otimes\mathcal{Q}$
für alle $n>m$ gilt. Es sei wieder $s:=\lim\limits_{n\rightarrow\infty}s_{n}n^{\frac{1}{2}}.$
Die Folge $\left(\left\{ \widetilde{P}_{s_{n}}^{n_{1}}\otimes\widetilde{Q}_{s_{n}}^{n_{2}}\right\} \right)_{n\in\mathbb{N}\setminus\left\{ 1,\ldots,m\right\} }$
ist nach Anwendung \ref{ulan_anwendung} ULAN mit der zentralen Folge
\begin{eqnarray*}
X_{n} & = & \sum_{i=1}^{n_{1}}s_{n}h_{1}\circ\pi_{i}+\sum_{j=1}^{n_{2}}s_{n}h_{2}\circ\pi_{j+n_{1}}.
\end{eqnarray*}
 Nach Satz \ref{gem-asymp-vert} erhält man
\[
\mathcal{L}\left(\left(T_{n},X_{n}\right)^{t}\left|\widetilde{P}_{0}^{n_{1}}\otimes\widetilde{Q}_{0}^{n_{2}}\right.\right)\rightarrow N\left(\left(\begin{array}{c}
0\\
0
\end{array}\right),\left(\begin{array}{cc}
\sigma_{1}^{2} & \sigma_{12}\\
\sigma_{12} & \sigma_{2}^{2}
\end{array}\right)\right)
\]
für $n\rightarrow\infty$. Es gilt außerdem 
\[
\sigma_{1}^{2}=\lim_{n\rightarrow\infty}\int T_{n}\,dP_{0}^{n_{1}}\otimes Q_{0}^{n_{2}}=\frac{1}{1-d}\int\widetilde{k}_{1}^{2}\,dP_{0}+\frac{1}{d}\int\widetilde{k}_{2}^{2}\,dQ_{0}>0.
\]
 Für die Koeffizienten $a_{1}$ und $a_{2}$ aus Satz \ref{gem-asymp-vert}
ergibt sich
\[
a_{1}=\lim_{n\rightarrow\infty}n_{1}^{-\frac{1}{2}}\sum_{i=1}^{n_{1}}s_{n}=\lim_{n\rightarrow\infty}\frac{n_{1}^{\frac{1}{2}}}{n^{\frac{1}{2}}}s_{n}n^{\frac{1}{2}}=\sqrt{\left(1-d\right)}s
\]
und
\[
a_{2}=\lim_{n\rightarrow\infty}n_{2}^{-\frac{1}{2}}\sum_{i=1}^{n_{2}}s_{n}=\sqrt{d}s.
\]

\noindent Hieraus folgt
\begin{eqnarray*}
\sigma_{12} & = & \frac{1}{\sqrt{1-d}}a_{1}\int\widetilde{k}_{1}h_{1}\,dP_{0}+\frac{1}{\sqrt{d}}a_{2}\int\widetilde{k}_{2}h_{2}\,dQ_{0}\\
 & = & s\int\widetilde{k}_{1}h_{1}\,dP_{0}+s\int\widetilde{k}_{2}h_{2}\,dQ_{0}\\
 & = & s\int\widetilde{k}h\,dP_{0}\otimes Q_{0}\\
 & = & 0.
\end{eqnarray*}
Nach Satz \ref{asymp-guete} erhält man nun
\begin{eqnarray*}
\lim_{n\rightarrow\infty}\int\psi_{n}\,d\widetilde{P}_{s_{n}}^{n_{1}}\otimes\widetilde{Q}_{s_{n}}^{n_{2}}= & \Phi\left(\underbrace{\frac{\sigma_{12}}{\sigma_{1}}}_{=0}-u_{1-\alpha}\right)=\Phi\left(-u_{1-\alpha}\right)=\alpha.
\end{eqnarray*}
 \end{proof}\begin{defi}Die Menge $\mathcal{F}_{2}$ enthalte alle
Folgen $\left(P_{t_{n}}\otimes Q_{t_{n}}\right)_{n\in\mathbb{N}}$,
die die folgenden Bedingungen erfüllen:
\begin{enumerate}
\item Die Nullfolge $\left(t_{n}\right)_{n\in\mathbb{N}}$ erfüllt $\lim_{n\rightarrow\infty}t_{n}\sqrt{n}>0.$
\item Es existieren ein $\varepsilon>0$ und eine $L_{2}\left(P_{0}\otimes Q_{0}\right)$-differenzierbare
Kurve $f:\left(-\varepsilon,\varepsilon\right)\rightarrow\mathcal{P}\otimes\mathcal{Q},\,t\mapsto P_{t}\otimes Q_{t},$
so dass $f\left(t_{n}\right)=P_{t_{n}}\otimes Q_{t_{n}}$ für alle
$n\in\mathbb{N}$ gilt.
\end{enumerate}
\noindent\end{defi}\begin{defi}\label{impl_alternativen} Die Menge
$\mathcal{K}_{1}$ aller impliziten Alternativen für das Testproblem
$H_{1}$ gegen $K_{1}$ ist definiert durch 
\[
\mathcal{K}_{1}:=\left\{ \left(P_{t_{n}}\otimes Q_{t_{n}}\right)_{n\in\mathbb{N}}\in\mathcal{F}_{2}:\lim_{n\rightarrow\infty}\sqrt{n}\left(k\left(P_{t_{n}}\otimes Q_{t_{n}}\right)-k\left(P_{0}\otimes Q_{0}\right)\right)>0\right\} .
\]

\noindent Die Menge $\mathcal{H}_{1}$ aller impliziten Hypothesen
für das Testproblem $H_{1}$ gegen $K_{1}$ ist gegeben durch 
\[
\mathcal{H}_{1}:=\left\{ \left(P_{t_{n}}\otimes Q_{t_{n}}\right)_{n\in\mathbb{N}}\in\mathcal{F}_{2}:\lim_{n\rightarrow\infty}\sqrt{n}\left(k\left(P_{t_{n}}\otimes Q_{t_{n}}\right)-k\left(P_{0}\otimes Q_{0}\right)\right)\leq0\right\} .
\]
\end{defi}\begin{bem}Die Statistik $T_{n}$ setzt sich aus zwei
Teilen $\frac{\sqrt{n}}{n_{1}}\sum\limits_{i=1}^{n_{1}}\widetilde{k}_{1}\circ\pi_{i}$
und $\frac{\sqrt{n}}{n_{2}}\sum\limits_{i=n_{1}+1}^{n}\widetilde{k}_{2}\circ\pi_{i}$
zusammen, die jeweils nur von einer der beiden Stichproben abhängen.
Es stellt sich die Frage, ob eine gewichtete Summe dieser beiden Anteile
eine geeignete Teststatistik für das Testproblem $\mathcal{H}_{1}$
gegen $\mathcal{K}_{1}$ darstellt. Es wird hier jedoch gezeigt, dass
die Testfolgen auf Basis der gewichteten Summe der beiden Anteile
von $T_{n}$ im Allgemeinen asymptotisch verfälscht sind und das Niveau
$\alpha$ asymptotisch nicht einhalten. Seien $\left(a_{n}\right)_{n\in\mathbb{N}}\subset\mathbb{R}$
und $\left(b_{n}\right)_{n\in\mathbb{N}}\subset\mathbb{R}$ zwei Folgen
mit $\lim\limits_{n\rightarrow\infty}a_{n}=a$, $\lim\limits_{n\rightarrow\infty}b_{n}=b$
und $a\neq b.$ Außerdem seien $\int\widetilde{k}_{1}^{2}\,dP_{0}\neq0$
und $\int\widetilde{k}_{2}^{2}\,dQ_{0}\neq0$. Die Teststatistik $S_{n}$
wird durch 
\[
S_{n}=a_{n}\frac{\sqrt{n}}{n_{1}}\sum\limits_{i=1}^{n_{1}}\widetilde{k}_{1}\circ\pi_{i}+b_{n}\frac{\sqrt{n}}{n_{2}}\sum\limits_{i=n_{1}+1}^{n}\widetilde{k}_{2}\circ\pi_{i}
\]
 für alle $n\in\mathbb{N}$ definiert. Als gemeinsame asymptotische
Verteilung von $S_{n}$ und $X_{n}$ erhält man 
\[
\mathcal{L}\left(\left(S_{n},X_{n}\right)^{t}\left|P_{0}^{n_{1}}\otimes Q_{0}^{n_{2}}\right.\right)\rightarrow N\left(\left(\begin{array}{c}
0\\
0
\end{array}\right),\left(\begin{array}{cc}
\sigma_{1}^{2} & \sigma_{12}\\
\sigma_{12} & \sigma_{2}^{2}
\end{array}\right)\right)
\]
für $n\rightarrow\infty$. Die Einträge der Kovarianzmatrix ergeben
sich als 
\[
\sigma_{1}^{2}=a^{2}\frac{1}{1-d}\int\widetilde{k}_{1}^{2}\,dP_{0}+b^{2}\frac{1}{d}\int\widetilde{k}_{2}^{2}\,dQ_{0},
\]
\[
\sigma_{2}^{2}=c_{1}^{2}\int g_{1}^{2}\,dP+c_{2}^{2}\int g_{2}^{2}\,dQ,
\]
\[
\sigma_{12}=\left(a\int\widetilde{k}_{1}g_{1}\,dP_{0}+b\int\widetilde{k}_{2}g_{2}\,dQ_{0}\right)\left(\int\widetilde{k}g\,dP_{0}\otimes Q_{0}\right)^{-1}.
\]
Die Testfolge 
\begin{eqnarray*}
\phi_{n} & = & \left\{ \begin{array}{cccc}
1 &  & >\\
 & S_{n} &  & u_{1-\alpha}\sigma_{1}^{2}\\
0 &  & \leq
\end{array}\right.
\end{eqnarray*}
 ist asymptotisch $\{0\}$-$\alpha$-ähnlich, d.h. es gilt
\[
\lim_{n\rightarrow\infty}\int\phi_{n}\,dP_{0}^{n_{1}}\otimes Q_{0}^{n_{2}}=\alpha.
\]
Die Voraussetzungen aus Satz \ref{asymp_gute} seien erfüllt. Analog
wie in Satz \ref{asymp_gute} erhält man dann die asymptotische Gütefunktion
\begin{eqnarray*}
 &  & \lim_{n\rightarrow\infty}\int\phi_{n}\,dP_{t_{n}}^{n_{1}}\otimes Q_{t_{n}}^{n_{2}}\\
\textrm{} & = & \Phi\left(\vartheta\frac{a\int\widetilde{k}_{1}g_{1}\,dP_{0}+b\int\widetilde{k}_{2}g_{2}\,dQ_{0}}{\int\widetilde{k}g\,dP_{0}\otimes Q_{0}\left(a^{2}\frac{1}{1-d}\int\widetilde{k}_{1}^{2}\,dP_{0}+b^{2}\frac{1}{d}\int\widetilde{k}_{2}^{2}\,dQ_{0}\right)^{\frac{1}{2}}}-u_{1-\alpha}\right),
\end{eqnarray*}
 wobei 
\[
\vartheta=\lim_{n\rightarrow\infty}\frac{k\left(P_{t_{n}}\otimes Q_{t_{n}}\right)-k\left(P_{0}\otimes Q_{0}\right)}{n^{-\frac{1}{2}}}=\int\widetilde{k}g\,dP_{0}\otimes Q_{0}\lim_{n\rightarrow\infty}n^{\frac{1}{2}}t_{n}
\]
 ist. Falls der Tangentenkegel $K\left(P_{0}\otimes Q_{0},\mathcal{P}\otimes\mathcal{Q}\right)$
reichhaltig genug ist, so existiert eine Tangente $g\in K\left(P_{0}\otimes Q_{0},\mathcal{P}\otimes\mathcal{Q}\right)$
mit
\[
\int\widetilde{k}_{1}g_{1}\,dP_{0}+\int\widetilde{k}_{2}g_{2}\,dQ_{0}=\int\widetilde{k}g\,dP_{0}\otimes Q_{0}>0
\]
und 
\[
a\int\widetilde{k}_{1}g_{1}\,dP_{0}+b\int\widetilde{k}_{2}g_{2}\,dQ_{0}<0,
\]
 wobei $g=g_{1}\circ\pi_{1}+g_{2}\circ\pi_{2}$ mit $g_{1}\in L_{2}^{(0)}\left(P_{0}\right)$
und $g_{2}\in L_{2}^{(0)}\left(Q_{0}\right)$ nach Satz \ref{L2Produkt}
gilt. Es seien $t\mapsto P_{t}\otimes Q_{t}$ eine $L_{2}\left(0\right)$-differenzierbare
Kurve mit Tangente $g$ und $\left(t_{n}\right)_{n\in\mathbb{N}}$
eine Nullfolge mit $\lim_{n\rightarrow\infty}t_{n}\sqrt{n}>0.$ Die
Folge $\left(\phi_{n}\right)_{n\in\mathbb{N}}$ der Tests ist asymptotisch
verfälscht zum Niveau $\alpha$ für die implizite Alternative $\left(P_{t_{n}}^{n_{1}}\otimes Q_{t_{n}}^{n_{2}}\right)_{n\in\mathbb{N}}.$
Es gilt nämlich 
\begin{eqnarray*}
 &  & \lim_{n\rightarrow\infty}\int\phi_{n}\,dP_{t_{n}}^{n_{1}}\otimes Q_{t_{n}}^{n_{2}}\\
 & = & \Phi\left(\vartheta\underbrace{\frac{a\int\widetilde{k}_{1}g_{1}\,dP_{0}+b\int\widetilde{k}_{2}g_{2}\,dQ_{0}}{\int\widetilde{k}g\,dP_{0}\otimes Q_{0}\left(a^{2}\frac{1}{1-d}\int\widetilde{k}_{1}^{2}\,dP_{0}+b^{2}\frac{1}{d}\int\widetilde{k}_{2}^{2}\,dQ_{0}\right)^{\frac{1}{2}}}}_{<0}-u_{1-\alpha}\right)\\
 & < & \Phi\left(-u_{1-\alpha}\right)\\
 & = & \alpha,
\end{eqnarray*}
 wobei
\[
\vartheta=\int\widetilde{k}g\,dP_{0}\otimes Q_{0}\lim_{n\rightarrow\infty}n^{\frac{1}{2}}t_{n}>0
\]
 ist. Analog erhält man implizite Hypothesen, so dass die Testfolge
$\left(\phi_{n}\right)_{n\in\mathbb{N}}$ das Niveau $\alpha$ asymptotisch
nicht einhält. \end{bem} Als Nächstes werden asymptotische Optimalitätseigenschaften
der Testfolge $\left(\psi_{n}\right)_{n\in\mathbb{N}}$ aus (\ref{einseit_asym_opt_testfolge})
untersucht. Die Neyman-Pearson Tests spielen dabei die Rolle der Vergleichsobjekte,
weil sie die optimale Gütefunktion besitzen. Deswegen wird im nächsten
Lemma die asymptotische Gütefunktion für eine Folge der Neyman-Pearson
Tests berechnet. \begin{lemma}\label{asymp_neymann_pearson_einfach}
Es seien $\left(P_{t_{n}}^{n_{1}}\otimes Q_{t_{n}}^{n_{2}}\right)_{n\in\mathbb{N}}\in\mathcal{K}_{1}$
eine implizite Alternative und $t\mapsto P_{t}\otimes Q_{t}$ die
zugehörige $L_{2}\left(P_{0}\otimes Q_{0}\right)$-differenzierbare
Kurve in $\mathcal{P}\otimes\mathcal{Q}$ mit Tangente $g\in L_{2}^{(0)}\left(P_{0}\otimes Q_{0}\right)$.
Die Tangente $g$ besitzt dann die eindeutige Darstellung $g=g_{1}\circ\pi_{1}+g_{2}\circ\pi_{2}$
mit $g_{1}\in L_{2}^{(0)}\left(P_{0}\right)$ und $g_{2}\in L_{2}^{(0)}\left(Q_{0}\right).$
Sei $\left(\alpha_{n}\right)_{n\in\mathbb{N}}\subset\left(0,1\right)$
eine Folge mit $\lim\limits_{n\rightarrow\infty}\alpha_{n}=\alpha$.
Es bezeichne $\varphi_{n}$ den Neyman-Pearson Test für $P_{0}^{n_{1}}\otimes Q_{0}^{n_{2}}$
gegen $P_{t_{n}}^{n_{1}}\otimes Q_{t_{n}}^{n_{2}}$ zum Niveau $\alpha_{n}.$
Die asymptotische Gütefunktion der Testfolge $\left(\varphi_{n}\right)_{n\in\mathbb{N}}$
ergibt sich dann als 
\[
\lim_{n\rightarrow\infty}\int\varphi_{n}dP_{t_{n}}^{n_{1}}\otimes Q_{t_{n}}^{n_{2}}=\Phi\left(t\sigma-u_{1-\alpha}\right),
\]
wobei $t=\lim\limits_{n\rightarrow\infty}n^{\frac{1}{2}}t_{n}$ und
$\sigma=\left(\left(1-d\right)\int g_{1}^{2}\,dP_{0}+d\int g_{2}^{2}\,dQ_{0}\right)^{\frac{1}{2}}$
sind. \end{lemma}\begin{proof} Der Parameter $\vartheta$ wird durch
\[
\vartheta=\lim\limits_{n\rightarrow\infty}\frac{k\left(P_{t_{n}}\otimes Q_{t_{n}}\right)-k\left(P_{0}\otimes Q_{0}\right)}{n^{-\frac{1}{2}}}
\]
definiert. Nach Voraussetzung gilt dann $\vartheta>0.$ Wie in Satz
\ref{asymp_unverf_testfolge} erhält man 
\begin{equation}
\vartheta=t\int\widetilde{k}g\,dP_{0}\otimes Q_{0}.\label{vartheta_wert}
\end{equation}
 Die LAN-Eigenschaften der Produktexperimente implizieren
\[
\log\left(\frac{dP_{t_{n}}^{n_{1}}\otimes Q_{t_{n}}^{n_{2}}}{dP_{0}^{n_{1}}\otimes Q_{0}^{n_{2}}}\right)=\vartheta X_{n}-\frac{1}{2}\vartheta\sigma_{2}^{2}+R_{n,\vartheta},
\]
wobei 
\[
X_{n}=\sum_{i=1}^{n_{1}}\frac{t_{n}}{\vartheta}g_{1}\circ\pi_{i}+\sum_{i=1}^{n_{2}}\frac{t_{n}}{\vartheta}g_{2}\circ\pi_{i+n_{1}}
\]
und $\sigma_{2}^{2}=\lim\limits_{n\rightarrow\infty}\int X_{n}^{2}dP_{0}^{n_{1}}\otimes Q_{0}^{n_{2}}$
sind. Für $\sigma_{2}^{2}$ ergibt sich
\begin{eqnarray}
\sigma_{2}^{2} & = & \lim_{n\rightarrow\infty}\left(\frac{t_{n}^{2}}{\vartheta^{2}}n_{1}\int g_{1}^{2}\,dP_{0}+\frac{t_{n}^{2}}{\vartheta^{2}}n_{2}\int g_{2}^{2}\,dQ_{0}\right)\nonumber \\
 & = & \frac{1}{\vartheta^{2}}\int g_{1}^{2}\,dP_{0}\lim_{n\rightarrow\infty}nt_{n}^{2}\frac{n_{1}}{n}+\frac{1}{\vartheta^{2}}\int g_{2}^{2}\,dQ_{0}\lim_{n\rightarrow\infty}nt_{n}^{2}\frac{n_{2}}{n}\nonumber \\
 & = & \frac{1}{\vartheta^{2}}t^{2}\left(1-d\right)\int g_{1}^{2}\,dP_{0}+\frac{1}{\vartheta^{2}}t^{2}d\int g_{2}^{2}\,dQ_{0}\nonumber \\
 & = & \frac{1}{\vartheta^{2}}t^{2}\sigma{{}^2}\nonumber \\
 & = & \sigma{{}^2}\left(\int\widetilde{k}g\,dP_{0}\otimes Q_{0}\right)^{-2}.\label{sigma_2_wert}
\end{eqnarray}
 Sei $S_{n}:=\log\left(\frac{dP_{t_{n}}^{n_{1}}\otimes Q_{t_{n}}^{n_{2}}}{dP_{0}^{n_{1}}\otimes Q_{0}^{n_{2}}}\right)+\frac{1}{2}\vartheta\sigma_{2}^{2}$
die Teststatistik des Neyman-Pearson Tests $\varphi_{n}.$ Für jedes
$\varepsilon>0$ gilt dann 
\begin{eqnarray*}
\lim_{n\rightarrow\infty}P_{0}^{n_{1}}\otimes Q_{0}^{n_{2}}\left(\left|\vartheta X_{n}-S_{n}\right|>\varepsilon\right) & = & \lim_{n\rightarrow\infty}P_{0}^{n_{1}}\otimes Q_{0}^{n_{2}}\left(\left|R_{n,\vartheta}\right|>\varepsilon\right)\\
 & = & 0.
\end{eqnarray*}
 Den Test $\widetilde{\varphi}_{n}$ erhält man aus dem Neyman-Pearson
Test $\varphi_{n}$, indem die Teststatistik $S_{n}$ durch $\vartheta X_{n}$
ersetzt wird. Nach Satz \ref{asympt_aquivalent_testfolge} gilt dann
\[
\lim_{n\rightarrow\infty}\int\left|\varphi_{n}-\widetilde{\varphi}_{n}\right|dP_{0}^{n_{1}}\otimes Q_{0}^{n_{2}}=0.
\]
 Nach Bemerkung \ref{asympt_banachb_equiv_tests} ergibt sich
\[
\lim_{n\rightarrow\infty}\int\left|\varphi_{n}-\widetilde{\varphi}_{n}\right|dP_{t_{n}}^{n_{1}}\otimes Q_{t_{n}}^{n_{2}}=0.
\]
 Hieraus folgt
\[
\lim_{n\rightarrow\infty}\int\widetilde{\varphi}_{n}\,dP_{0}^{n_{1}}\otimes Q_{0}^{n_{2}}=\lim_{n\rightarrow\infty}\int\varphi_{n}dP_{0}^{n_{1}}\otimes Q_{0}^{n_{2}}=\alpha
\]
und
\[
\lim_{n\rightarrow\infty}\int\widetilde{\varphi}_{n}\,dP_{t_{n}}^{n_{1}}\otimes Q_{t_{n}}^{n_{2}}=\lim_{n\rightarrow\infty}\int\varphi_{n}dP_{t_{n}}^{n_{1}}\otimes Q_{t_{n}}^{n_{2}}.
\]
Somit reicht es, den Grenzwert $\lim_{n\rightarrow\infty}\int\widetilde{\varphi}_{n}\,dP_{t_{n}}^{n_{1}}\otimes Q_{t_{n}}^{n_{2}}$
zu bestimmen. Die gemeinsame asymptotische Verteilung von $\vartheta X_{n}$
und $X_{n}$ ergibt sich unmittelbar als
\begin{eqnarray*}
\mathcal{L}\left(\left(\vartheta X_{n},X_{n}\right)^{t}\left|P_{n,0}\right.\right) & \rightarrow & N\left(\left(\begin{array}{c}
0\\
0
\end{array}\right),\left(\begin{array}{cc}
\vartheta^{2}\sigma_{2}^{2} & \vartheta\sigma_{2}^{2}\\
\vartheta\sigma_{2}^{2} & \sigma_{2}^{2}
\end{array}\right)\right)
\end{eqnarray*}
für $n\rightarrow\infty.$ Mit Satz \ref{asymp-guete} erhält man
nun
\begin{eqnarray*}
\lim_{n\rightarrow\infty}\int\widetilde{\varphi}_{n}\,dP_{t_{n}}^{n_{1}}\otimes Q_{t_{n}}^{n_{2}} & = & \Phi\left(\vartheta\frac{\vartheta\sigma_{2}^{2}}{\vartheta\sigma_{2}}-u_{1-\alpha}\right)\\
 & = & \Phi\left(\vartheta\sigma_{2}-u_{1-\alpha}\right)\\
 & = & \Phi\left(t\sigma-u_{1-\alpha}\right)
\end{eqnarray*}
mit Hilfe von (\ref{vartheta_wert}) und (\ref{sigma_2_wert}).\end{proof}
\begin{defi}Die Menge $K^{+}\left(P_{0}\otimes Q_{0},\mathcal{P}\otimes\mathcal{Q}\right)\subset K\left(P_{0}\otimes Q_{0},\mathcal{P}\otimes\mathcal{Q}\right)$
wird definiert durch 
\[
K^{+}\left(P_{0}\otimes Q_{0},\mathcal{P}\otimes\mathcal{Q}\right):=\left\{ g\in K\left(P_{0}\otimes Q_{0},\mathcal{P}\otimes\mathcal{Q}\right):\int g\widetilde{k}\,dP_{0}\otimes Q_{0}>0\right\} .
\]
 \end{defi}\begin{satz}[Maximin-$\alpha$-Testfolge]\label{maximin-testfolge}Der
Kegel $K\left(P_{0}\otimes Q_{0},\mathcal{P}\otimes\mathcal{Q}\right)$
sei konvex. Es gelten die Voraussetzungen aus Satz \ref{asymp_unverf_testfolge}.
Dann ist die Testfolge $\left(\psi_{n}\right)_{n\in\mathbb{N}}$ aus
(\ref{einseit_asym_opt_testfolge}) eine asymptotische Maximin-$\alpha$-Testfolge
für die Folge $\left(P_{0}^{n_{1}}\otimes Q_{0}^{n_{2}}\right)_{n\in\mathbb{N}}$
der Hypothesen gegen die impliziten Alternativen $\left(P_{t_{n}}^{n_{1}}\otimes Q_{t_{n}}^{n_{2}}\right)_{n\in\mathbb{N}}\in\mathcal{K}_{1}$
mit einem fest gewählten impliziten Parameter $\vartheta>0$, d.h.
es gilt
\[
\inf_{\mathcal{K}_{1}}\,\lim_{n\rightarrow\infty}\int\psi_{n}\,dP_{t_{n}}^{n_{1}}\otimes Q_{t_{n}}^{n_{2}}=\max_{\phi_{n}}\,\inf_{\mathcal{K}_{1}}\,\limsup_{n\rightarrow\infty}\int\phi_{n}\,dP_{t_{n}}^{n_{1}}\otimes Q_{t_{n}}^{n_{2}}.
\]
Das Infimum wird über alle impliziten Alternativen $\left(P_{t_{n}}^{n_{1}}\otimes Q_{t_{n}}^{n_{2}}\right)_{n\in\mathbb{N}}\in\mathcal{K}_{1}$
mit 
\begin{equation}
\lim\limits_{n\rightarrow\infty}\frac{k\left(P_{t_{n}}\otimes Q_{t_{n}}\right)-k\left(P_{0}\otimes Q_{0}\right)}{n^{-\frac{1}{2}}}=\vartheta\label{stern_gl}
\end{equation}
 gebildet. Das Maximum wird über alle asymptotisch $\{0\}$-$\alpha$-ähnlichen
Testfolgen gebildet, vgl. Definition \ref{asymp_0_aplha_anlich_testfolge}.
\end{satz}\begin{proof}Sei $\left(P_{t_{n}}^{n_{1}}\otimes Q_{t_{n}}^{n_{2}}\right)_{n\in\mathbb{N}}\in\mathcal{K}_{1}$
eine implizite Alternative. Dann existiert eine zugehörige $L_{2}(0)$-differenzierbare
Kurve $t\mapsto P_{t}\otimes Q_{t}$ in $\mathcal{P}\otimes\mathcal{Q}$
mit Tangente $g\in L_{2}^{(0)}\left(P_{0}\otimes Q_{0}\right).$ Die
Tangente $g$ besitzt nach Satz \ref{L2Produkt} die Darstellung $g=g_{1}\circ\pi_{1}+g_{2}\circ\pi_{2}$
mit $g_{1}\in L_{2}^{(0)}\left(P_{0}\right)$ und $g_{2}\in L_{2}^{(0)}\left(Q_{0}\right)$.
Wie in Satz \ref{asymp_unverf_testfolge} ergibt sich 
\[
\vartheta=\lim\limits_{n\rightarrow\infty}\frac{k\left(P_{t_{n}}\otimes Q_{t_{n}}\right)-k\left(P_{0}\otimes Q_{0}\right)}{n^{-\frac{1}{2}}}=t\int\widetilde{k}g\,dP_{0}\otimes Q_{0},
\]
wobei $t=\lim\limits_{n\rightarrow\infty}\frac{t_{n}}{n^{-\frac{1}{2}}}$
ist. Sei $\left(\phi_{n}\right)_{n\in\mathbb{N}}$ eine asymptotisch
$\{0\}$-$\alpha$-ähnliche Testfolge. Es bezeichne $\varphi_{n}$
den Neyman-Pearson Test für $P_{0}^{n}\otimes Q_{0}^{n}$ gegen $P_{t_{n}}^{n_{1}}\otimes Q_{t_{n}}^{n_{2}}$
zum Niveau $\max\left(\alpha,\int\phi_{n}\,dP_{0}^{n}\otimes Q_{0}^{n}\right).$
Wegen des Neyman-Pearson Lemma (vgl. \cite{Witting:1985}, Seite 196,
Satz 2.7) folgt dann
\[
\int\phi_{n}\,dP_{t_{n}}^{n_{1}}\otimes Q_{t_{n}}^{n_{2}}\leq\int\varphi_{n}\,dP_{t_{n}}^{n_{1}}\otimes Q_{t_{n}}^{n_{2}}.
\]
Nach Lemma \ref{asymp_neymann_pearson_einfach} ergibt sich 
\begin{eqnarray*}
\lim_{n\rightarrow\infty}\int\varphi_{n}\,dP_{t_{n}}^{n_{1}}\otimes Q_{t_{n}}^{n_{2}} & = & \Phi\left(t\left(\left(1-d\right)\int g_{1}^{2}\,dP_{0}+d\int g_{2}^{2}\,dQ_{0}\right)^{\frac{1}{2}}-u_{1-\alpha}\right)\\
 & = & \Phi\left(\vartheta\frac{\left(\left(1-d\right)\int g_{1}^{2}\,dP_{0}+d\int g_{2}^{2}\,dQ_{0}\right)^{\frac{1}{2}}}{\int\widetilde{k}g\,dP_{0}\otimes Q_{0}}-u_{1-\alpha}\right).
\end{eqnarray*}
 Hieraus folgt
\begin{eqnarray*}
 &  & \limsup_{n\rightarrow\infty}\int\phi_{n}\,dP_{t_{n}}^{n_{1}}\otimes Q_{t_{n}}^{n_{2}}\\
 & \leq & \lim_{n\rightarrow\infty}\int\varphi_{n}\,dP_{t_{n}}^{n_{1}}\otimes Q_{t_{n}}^{n_{2}}\\
 & = & \Phi\left(\vartheta\frac{\left(\left(1-d\right)\int g_{1}^{2}\,dP_{0}+d\int g_{2}^{2}\,dQ_{0}\right)^{\frac{1}{2}}}{\int\widetilde{k}g\,dP_{0}\otimes Q_{0}}-u_{1-\alpha}\right).
\end{eqnarray*}
Nun wird das Infimum über alle impliziten Alternativen $\left(P_{t_{n}}^{n_{1}}\otimes Q_{t_{n}}^{n_{2}}\right)_{n\in\mathbb{N}}$
aus $\mathcal{K}_{1}$, die die Bedingung (\ref{stern_gl}) erfüllen,
gebildet. Man erhält dann 
\begin{eqnarray*}
\textrm{} &  & \inf_{\mathcal{K}}\,\limsup_{n\rightarrow\infty}\int\phi_{n}\,dP_{t_{n}}^{n_{1}}\otimes Q_{t_{n}}^{n_{2}}\\
 & \leq & \inf_{g\in K^{+}\left(P_{0}\otimes Q_{0},\mathcal{P}\otimes\mathcal{Q}\right)}\Phi\left(\vartheta\frac{\left(\left(1-d\right)\int g_{1}^{2}\,dP_{0}+d\int g_{2}^{2}\,dQ_{0}\right)^{\frac{1}{2}}}{\int\widetilde{k}g\,dP_{0}\otimes Q_{0}}-u_{1-\alpha}\right).
\end{eqnarray*}
 Die asymptotische Güte 
\begin{equation}
\Phi\left(\vartheta\frac{\left(\left(1-d\right)\int g_{1}^{2}\,dP_{0}+d\int g_{2}^{2}\,dQ_{0}\right)^{\frac{1}{2}}}{\int\widetilde{k}g\,dP_{0}\otimes Q_{0}}-u_{1-\alpha}\right)\label{obere_schranke}
\end{equation}
der Folge der Neyman-Pearson Tests wird mit der asymptotischen Güte
\begin{eqnarray*}
 &  & \lim_{n\rightarrow\infty}\int\psi_{n}\,dP_{t_{n}}^{n_{1}}\otimes Q_{t_{n}}^{n_{2}}\\
 & = & \Phi\left(\vartheta\left(\frac{1}{1-d}\int\widetilde{k}_{1}^{2}\,dP_{0}+\frac{1}{d}\int\widetilde{k}_{2}^{2}\,dQ_{0}\right)^{-\frac{1}{2}}-u_{1-\alpha}\right)
\end{eqnarray*}
der Testfolge $\left(\psi_{n}\right)_{n\in\mathbb{N}}$ verglichen.
Mit der Cauchy-Schwarzschen Ungleichung erhält man zunächst
\begin{eqnarray*}
 &  & \int\widetilde{k}g\,dP_{0}\otimes Q_{0}\\
 & = & \int\left(\frac{1}{\sqrt{1-d}}\widetilde{k}_{1}+\frac{1}{\sqrt{d}}\widetilde{k}_{2}\right)\left(\sqrt{\left(1-d\right)}g_{1}+\sqrt{d}g_{2}\right)\,dP_{0}\otimes Q_{0}\\
 & \leq & \left(\int\left(\frac{1}{\sqrt{1-d}}\widetilde{k}_{1}+\frac{1}{\sqrt{d}}\widetilde{k}_{2}\right)^{2}dP_{0}\otimes Q_{0}\int\left(\sqrt{\left(1-d\right)}g_{1}+\sqrt{d}g_{2}\right)^{2}dP_{0}\otimes Q_{0}\right)^{\frac{1}{2}}\\
 & = & \left(\left(1-d\right)\int g_{1}^{2}\,dP_{o}+d\int g_{2}^{2}\,dQ_{0}\right)^{\frac{1}{2}}\left(\frac{1}{1-d}\int\widetilde{k}_{1}^{2}\,dP_{0}+\frac{1}{d}\int\widetilde{k}_{2}^{2}\,dQ_{0}\right)^{\frac{1}{2}}.
\end{eqnarray*}
Hieraus folgt 
\begin{eqnarray}
\textrm{} &  & \vartheta\left(\frac{1}{1-d}\int\widetilde{k}_{1}^{2}\,dP_{0}+\frac{1}{d}\int\widetilde{k}_{2}^{2}\,dQ_{0}\right)^{-\frac{1}{2}}\nonumber \\
 & \leq & \vartheta\frac{\left(\left(1-d\right)\int g_{1}^{2}\,dP_{o}+d\int g_{2}^{2}\,dQ_{0}\right)^{\frac{1}{2}}}{\int\widetilde{k}g\,dP_{0}\otimes Q_{0}}.\label{min_max_ungl}
\end{eqnarray}
In dieser Ungleichung erhält man das Gleichheitszeichen, falls auf
beiden Seiten das Infimum über $g\in K^{+}\left(P_{0}\otimes Q_{0},\mathcal{P}\otimes\mathcal{Q}\right)$
gebildet wird. Der Tangentenkegel $K\left(P_{0}\otimes Q_{0},\mathcal{P}\otimes\mathcal{Q}\right)$
ist bereits ein Vektorraum, weil er nach Voraussetzung konvex ist.
Der Vektorraum $K\left(P_{0}\otimes Q_{0},\mathcal{P}\otimes\mathcal{Q}\right)$
liegt deshalb dicht in $T\left(P_{0}\otimes Q_{0},\mathcal{P}\otimes\mathcal{Q}\right)$
nach Definition \ref{Tangentialraum}. Wegen $\widetilde{k}\in T\left(P_{0}\otimes Q_{0},\mathcal{P}\otimes\mathcal{Q}\right)$
existiert also eine Folge $\left(g_{n}\right)_{n\in\mathbb{N}}\subset K\left(P_{0}\otimes Q_{0},\mathcal{P}\otimes\mathcal{Q}\right)$
mit 
\[
\lim_{n\rightarrow\infty}\int\left(\widetilde{k}-g_{n}\right)^{2}dP_{0}\otimes Q_{0}=0.
\]
Die Tangente $g_{n}$ besitzt nach Satz \ref{produkt_tangentialraum}
die Darstellung $g_{n}=g_{1n}\circ\pi_{1}+g_{2n}\circ\pi_{2}$ mit
$g_{1n}\in K\left(P_{0},\mathcal{P}\right)$ und $g_{2n}\in K\left(Q_{0},\mathcal{Q}\right)$.
Man erhält somit
\begin{eqnarray*}
0 & = & \lim_{n\rightarrow\infty}\int\left(\widetilde{k}-g_{n}\right)^{2}dP_{0}\otimes Q_{0}\\
 & = & \lim_{n\rightarrow\infty}\int\left(\widetilde{k}_{1}-g_{1n}\right)^{2}dP_{0}+\lim_{n\rightarrow\infty}\int\left(\widetilde{k}_{2}-g_{2n}\right)^{2}dQ_{0}.
\end{eqnarray*}
Die Tangente $h_{n}$ ist durch
\[
h_{n}=\frac{1}{1-d}g_{1n}\circ\pi_{1}+\frac{1}{d}g_{2n}\circ\pi_{2}
\]
für alle $n\in\mathbb{N}$ definiert. Es gilt $h_{n}\in K\left(P_{0}\otimes Q_{0},\mathcal{P}\otimes\mathcal{Q}\right)$
für alle $n\in\mathbb{N}$ nach Satz \ref{produkt_tangentialraum}.
Wegen $\int\widetilde{k}_{1}^{2}\,dP_{0}+\int\widetilde{k}_{2}^{2}\,dQ_{0}>0$
ergibt sich
\begin{eqnarray*}
\lim_{n\rightarrow\infty}\int h_{n}\widetilde{k}\,dP_{0}\otimes Q_{0} & = & \frac{1}{1-d}\lim_{n\rightarrow\infty}\int\widetilde{k}_{1}g_{1n}dP_{0}+\frac{1}{d}\lim_{n\rightarrow\infty}\int\widetilde{k}_{2}g_{2n}dQ_{0}\\
 & = & \frac{1}{1-d}\int\widetilde{k}_{1}^{2}\,dP_{0}+\frac{1}{d}\int\widetilde{k}_{2}^{2}\,dQ_{0}\\
 & > & 0.
\end{eqnarray*}
Hieraus folgt $\int h_{n}\widetilde{k}\,dP_{0}\otimes Q_{0}>0$ und
somit $h_{n}\in K^{+}\left(P_{0}\otimes Q_{0},\mathcal{P}\otimes\mathcal{Q}\right)$
für hinreichend große $n\in\mathbb{N}$. Die obere Schranke
\[
\inf_{g\in K^{+}\left(P_{0}\otimes Q_{0},\mathcal{P}\otimes\mathcal{Q}\right)}\Phi\left(\vartheta\frac{\left(\left(1-d\right)\int g_{1}^{2}\,dP_{0}+d\int g_{2}^{2}\,dQ_{0}\right)^{\frac{1}{2}}}{\int\widetilde{k}g\,dP_{0}\otimes Q_{0}}-u_{1-\alpha}\right)
\]
kann nun exakt berechnet werden. Setze $h_{1n}=\frac{1}{1-d}g_{1n}$
und $h_{2n}=\frac{1}{d}g_{2n}.$ Man erhält einerseits
\begin{eqnarray*}
\textrm{} &  & \inf_{g\in K^{+}\left(P_{0}\otimes Q_{0},\mathcal{P}\otimes\mathcal{Q}\right)}\Phi\left(\vartheta\frac{\left(\left(1-d\right)\int g_{1}^{2}\,dP_{0}+d\int g_{2}^{2}\,dQ_{0}\right)^{\frac{1}{2}}}{\int\widetilde{k}g\,dP_{0}\otimes Q_{0}}-u_{1-\alpha}\right)\\
 & \leq & \lim_{n\rightarrow\infty}\Phi\left(\vartheta\frac{\left(\left(1-d\right)\int h_{1n}^{2}\,dP_{0}+d\int h_{2n}^{2}\,dQ_{0}\right)^{\frac{1}{2}}}{\int\widetilde{k}h_{n}dP_{0}\otimes Q_{0}}-u_{1-\alpha}\right)\\
 & = & \lim_{n\rightarrow\infty}\Phi\left(\frac{\left(\frac{1}{1-d}\int g_{1n}^{2}\,dP_{0}+\frac{1}{d}\int g_{2n}^{2}\,dQ_{0}\right)^{\frac{1}{2}}}{\frac{1}{1-d}\int g_{1n}\widetilde{k}_{1}\,dP_{0}+\frac{1}{d}\int g_{2n}\widetilde{k}_{2}\,dQ_{0}}-u_{1-\alpha}\right)\\
 & = & \Phi\left(\vartheta\frac{\left(\frac{1}{1-d}\int\widetilde{k}_{1}^{2}\,dP_{0}+\frac{1}{d}\int\widetilde{k}_{2}^{2}\,dQ_{0}\right)^{\frac{1}{2}}}{\frac{1}{1-d}\int\widetilde{k}_{1}^{2}\,dP_{0}+\frac{1}{d}\int\widetilde{k}_{2}^{2}\,dQ_{0}}-u_{1-\alpha}\right)\\
 & = & \Phi\left(\vartheta\left(\frac{1}{1-d}\int\widetilde{k}_{1}^{2}\,dP_{0}+\frac{1}{d}\int\widetilde{k}_{2}^{2}\,dQ_{0}\right)^{-\frac{1}{2}}-u_{1-\alpha}\right),
\end{eqnarray*}
andererseits gilt
\begin{eqnarray*}
\textrm{} &  & \inf_{g\in K^{+}\left(P_{0}\otimes Q_{0},\mathcal{P}\otimes\mathcal{Q}\right)}\Phi\left(\vartheta\frac{\left(\left(1-d\right)\int g_{1}^{2}\,dP_{o}+d\int g_{2}^{2}\,dQ_{0}\right)^{\frac{1}{2}}}{\int\widetilde{k}g\,dP_{0}\otimes Q_{0}}-u_{1-\alpha}\right)\\
 & \geq & \Phi\left(\vartheta\left(\frac{1}{1-d}\int\widetilde{k}_{1}^{2}\,dP_{0}+\frac{1}{d}\int\widetilde{k}_{2}^{2}\,dQ_{0}\right)^{-\frac{1}{2}}-u_{1-\alpha}\right)
\end{eqnarray*}
wegen der Ungleichung (\ref{min_max_ungl}). Somit ist alles bewiesen.
\end{proof}\begin{bem}Die zum Satz \ref{maximin-testfolge} analoge
Aussage für die Einstichprobenprobleme findet man in \cite{Janssen:1999a}
und \cite{Janssen:1998}. Der wesentliche Unterschied zu \cite{Janssen:1999a}
besteht darin, dass sowohl die Testfolgen als auch die impliziten
Alternativen bei den Zweistichprobenproblemen von denen bei Einstichprobenproblemen
abweichen. \end{bem}\begin{bem}[Optimierung des Parameters $d$]Der
Parameter $d$ wird meistens durch die praktischen Gegebenheiten bestimmt.
Man kann sich allerdings die Frage stellen, wie dieser Parameter am
günstigsten gewählt werden kann. Das führt zum Optimierungsproblem
für die Funktion
\begin{eqnarray*}
g:\left(0,1\right) & \rightarrow & \mathbb{R},\\
s & \mapsto & \Phi\left(\vartheta\left(\frac{1}{1-s}\int\widetilde{k}_{1}^{2}\,dP_{0}+\frac{1}{s}\int\widetilde{k}_{2}^{2}\,dQ_{0}\right)^{-\frac{1}{2}}-u_{1-\alpha}\right),
\end{eqnarray*}
 wobei $\vartheta\in\mathbb{R}\setminus\left\{ 0\right\} $ ist. Da
$\Phi$ streng isoton ist, reicht es, das Argument
\begin{eqnarray*}
f:\left(0,1\right) & \rightarrow & \mathbb{R},\\
s & \mapsto & \vartheta\left(\frac{1}{1-s}\int\widetilde{k}_{1}^{2}\,dP_{0}+\frac{1}{s}\int\widetilde{k}_{2}^{2}\,dQ_{0}\right)^{-\frac{1}{2}}-u_{1-\alpha}
\end{eqnarray*}
 zu maximieren. Es seien $\int\widetilde{k}_{1}^{2}\,dP_{0}\neq0$
und $\int\widetilde{k}_{2}^{2}\,dQ_{0}\neq0$. Man erhält zunächst
\begin{eqnarray*}
\frac{d}{ds}f\left(s\right) & = & -\frac{1}{2}\vartheta h(s)\left(\frac{1}{\left(1-s\right)^{2}}\int\widetilde{k}_{1}^{2}\,dP_{0}-\frac{1}{s^{2}}\int\widetilde{k}_{2}^{2}\,dQ_{0}\right),
\end{eqnarray*}
wobei 
\[
h(s):=\left(\frac{1}{1-s}\int\widetilde{k}_{1}^{2}\,dP_{0}+\frac{1}{s}\int\widetilde{k}_{2}^{2}\,dQ_{0}\right)^{-\frac{3}{2}}
\]
ist. Nun ist die Gleichung $\frac{d}{ds}f(s)=0$ zu lösen. Es ist
leicht zu sehen, dass $h(s)>0$ für alle $s\in\left(0,1\right)$ gilt.
Man erhält somit für alle $s\in\left(0,1\right)$ folgende Äquivalenzen:
\begin{eqnarray*}
\frac{d}{ds}f(s)=0 & \Leftrightarrow & -\frac{1}{2}\vartheta h(s)\left(\frac{1}{\left(1-s\right)^{2}}\int\widetilde{k}_{1}^{2}\,dP_{0}-\frac{1}{s^{2}}\int\widetilde{k}_{2}^{2}\,dQ_{0}\right)=0\\
 & \Leftrightarrow & \frac{1}{\left(1-s\right)^{2}}\int\widetilde{k}_{1}^{2}\,dP_{0}-\frac{1}{s^{2}}\int\widetilde{k}_{2}^{2}\,dQ_{0}=0\\
 & \Leftrightarrow & \frac{s^{2}}{\left(1-s\right)^{2}}=\frac{\int\widetilde{k}_{2}^{2}\,dQ_{0}}{\int\widetilde{k}_{1}^{2}\,dP_{0}}\\
 & \Leftrightarrow & \frac{s}{1-s}=\pm\sqrt{\frac{\int\widetilde{k}_{2}^{2}\,dQ_{0}}{\int\widetilde{k}_{1}^{2}\,dP_{0}}}.
\end{eqnarray*}
Sei 
\[
p:=\sqrt{\frac{\int\widetilde{k}_{2}^{2}\,dQ_{0}}{\int\widetilde{k}_{1}^{2}\,dP_{0}}}=\frac{\left\Vert \widetilde{k}_{2}\right\Vert _{L_{2}\left(Q_{0}\right)}}{\left\Vert \widetilde{k}_{1}\right\Vert _{L_{2}\left(P_{0}\right)}}.
\]
Es gilt $p>0$ nach Voraussetzung. Die Lösung 
\[
\frac{s}{1-s}=-p
\]
wird wegen $\frac{s}{1-s}>0$ für $s\in\left(0,1\right)$ verworfen.
Aus $\frac{s}{1-s}=p$ folgt $s=\frac{p}{1+p}$. Daher existiert eine
eindeutige Lösung $\frac{p}{1+p}\in\left(0,1\right)$ der Gleichung
$\frac{d}{ds}f(s)=0$ in dem Intervall $\left(0,1\right)$. Für eine
Optimalitätsaussage muss noch das Vorzeichenverhalten der Ableitung
$\frac{d}{ds}f(s)$ in dem Intervall $\left(0,1\right)$ untersucht
werden. Für $s\in\left(0,\frac{p}{1+p}\right)$ ergibt sich 
\begin{eqnarray*}
\frac{d}{ds}f(s) & = & -\frac{1}{2}\vartheta h(s)\left(\frac{1}{\left(1-s\right)^{2}}\int\widetilde{k}_{1}^{2}\,dP_{0}-\frac{1}{s^{2}}\int\widetilde{k}_{2}^{2}\,dQ_{0}\right)\\
 & = & -\frac{1}{2}\vartheta h(s)\frac{1}{\left(1-s\right)^{2}}\frac{1}{s^{2}}\int\widetilde{k}_{1}^{2}\,dP_{0}\left(s^{2}-\left(1-s\right)^{2}\frac{\int\widetilde{k}_{2}^{2}\,dQ_{0}}{\int\widetilde{k}_{1}^{2}\,dP_{0}}\right)\\
 & = & -\frac{1}{2}\vartheta h(s)\frac{1}{\left(1-s\right)^{2}}\frac{1}{s^{2}}\int\widetilde{k}_{1}^{2}\,dP_{0}\left(s^{2}-\left(1-s\right)^{2}p^{2}\right)\\
 &  & -\frac{1}{2}\vartheta h(s)\frac{1}{\left(1-s\right)^{2}}\frac{1}{s^{2}}\int\widetilde{k}_{1}^{2}\,dP_{0}\left(s-p+sp\right)\left(s+\left(1-s\right)p\right)\\
 & = & -\frac{1}{2}\vartheta\underbrace{h(s)\frac{1}{\left(1-s\right)^{2}}\frac{1}{s^{2}}\int\widetilde{k}_{1}^{2}\,dP_{0}}_{>0}\underbrace{\left(\underbrace{s\left(1+p\right)}_{<p}-p\right)}_{<0}\underbrace{\left(s+\left(1-s\right)p\right)}_{>0}.
\end{eqnarray*}
Für $s\in\left(\frac{p}{1+p},1\right)$ erhält man analog 
\begin{eqnarray*}
\frac{d}{ds}f(s) & =-\frac{1}{2}\vartheta & \underbrace{h(s)\frac{1}{\left(1-s\right)^{2}}\frac{1}{s^{2}}\int\widetilde{k}_{1}^{2}\,dP_{0}}_{>0}\underbrace{\left(\underbrace{s\left(1+p\right)}_{>p}-p\right)}_{>0}\underbrace{\left(s+\left(1-s\right)p\right)}_{>0}.
\end{eqnarray*}
Das Vorzeichenverhalten der Ableitung $\dot{f}(s):=\frac{d}{ds}f(s)$
und die Bedeutung der Stelle $\frac{p}{1+p}$ werden nun in der nachfolgenden
Tabelle dargestellt:\\

\begin{tabular}{|l|c|c|c|}
\hline 
 & $s\in\left(0,\frac{p}{1+p}\right)$ & $s\in\left(\frac{p}{1+p},1\right)$ & $\frac{p}{1+p}$\tabularnewline
\hline 
$\vartheta>0$ & $\dot{f}(s)>0$ & $\dot{f}(s)<0$ & globale Maximumstelle\tabularnewline
\hline 
$\vartheta<0$ & $\dot{f}(s)<0$ & $\dot{f}(s)>0$ & globale Minimumstelle\tabularnewline
\hline 
\end{tabular}\\
\\
\\
Der Wert 
\begin{equation}
d_{\mbox{opt}}:=\frac{p}{1+p}=\frac{\left\Vert \widetilde{k}_{2}\right\Vert _{L_{2}\left(Q_{0}\right)}}{\left\Vert \widetilde{k}_{1}\right\Vert _{L_{2}\left(P_{0}\right)}+\left\Vert \widetilde{k}_{2}\right\Vert _{L_{2}\left(Q_{0}\right)}}\label{d_opt}
\end{equation}
 des Parameters $d$ optimiert die asymptotische Gütefunktion (\ref{asyp_gutefunktion})
für alle $\vartheta\in\mathbb{R}\setminus\left\{ 0\right\} $ gleichmäßig.
Es ist bemerkenswert, dass $d_{\mbox{opt}}$ an der Stelle $P_{0}\otimes Q_{0}$
nur von den Werten $\left\Vert \widetilde{k}_{1}\right\Vert _{L_{2}\left(P_{0}\right)}$
und $\left\Vert \widetilde{k}_{2}\right\Vert _{L_{2}\left(Q_{0}\right)}$
abhängt. \end{bem}

\section{Asymptotische Eigenschaften der zweiseitigen Tests\label{sec:zw:twsts}}

Es werden zunächst die Hilfsmittel vorgestellt, die zur asymptotischen
Behandlung der zweiseitigen Tests benötigt werden. Es seien $\Theta_{n}\subset\mathbb{R}$
eine Folge von Mengen mit $0\in\Theta_{n}$ für alle $n\in\mathbb{N}$
und $\Theta_{n}\uparrow\Theta$ für $n\rightarrow\infty$. Sei $E_{n}=\left(\Omega_{n},\mathcal{A}_{n},\left\{ P_{n,\vartheta}:\vartheta\in\Theta_{n}\right\} \right)$
eine Folge von Experimenten, die LAN-Bedingung mit einer zentralen
Folge $X_{n}$ erfüllt. Es wird das zweiseitige Testproblem $H=\left\{ \vartheta\in\Theta:\vartheta=0\right\} $
gegen $K=\left\{ \vartheta\in\Theta:\vartheta\neq0\right\} $ betrachtet.
Die gebräuchlichen Tests, die auf einer Teststatistik $S_{n}$ beruhen,
sind dann von der Gestalt
\begin{equation}
\varphi_{n}=\left\{ \begin{array}{cccc}
1 &  & >\\
\lambda_{n} & \left|S_{n}\right| & = & c_{n}\\
0 &  & <
\end{array}\right.\label{zweiseit_testfolge}
\end{equation}
mit $\lambda_{n}\in\left[0,1\right]$ und $c_{n}\geq0,$ die durch
die Nebenbedingungen 
\begin{equation}
\lim_{n\rightarrow\infty}E_{P_{n,0}}\left(\varphi_{n}\right)=\alpha\label{asymp_niveau_alpha}
\end{equation}
und 
\begin{equation}
\lim_{n\rightarrow\infty}E_{P_{n,\vartheta}}\left(\varphi_{n}\right)\geq\alpha\label{zweiseitig_asymp_unverf}
\end{equation}
festgelegt sind.\begin{satz}[asymptotische G\"utefunktion]\label{asymp-guete_zw}Eine
Testfolge $\left(\varphi_{n}\right)_{n\in\mathbb{N}}$ sei von der
Form (\ref{zweiseit_testfolge}) und erfülle die Bedingung (\ref{asymp_niveau_alpha}).
Außerdem sei
\begin{eqnarray*}
\mathcal{L}\left(\left(S_{n},X_{n}\right)^{t}\left|P_{n,0}\right.\right) & \rightarrow & N\left(\left(\begin{array}{c}
0\\
0
\end{array}\right),\left(\begin{array}{cc}
\sigma_{1}^{2} & \sigma_{12}\\
\sigma_{12} & \sigma_{2}^{2}
\end{array}\right)\right)
\end{eqnarray*}
für $n\rightarrow\infty.$ Für jedes $\vartheta\in\Theta$ gilt dann
\[
\lim_{n\rightarrow\infty}E_{P_{n,\vartheta}}\left(\varphi_{n}\right)=\Phi\left(\frac{\vartheta\sigma_{12}}{\sigma_{1}}-u_{1-\frac{\alpha}{2}}\right)+\Phi\left(-\frac{\vartheta\sigma_{12}}{\sigma_{1}}-u_{1-\frac{\alpha}{2}}\right).
\]
 \end{satz}\begin{proof}Setze $\Lambda_{n}=\log\frac{dP_{n,\vartheta}}{dP_{n,0}}.$
Die LAN-Bedingung impliziert 
\[
\Lambda_{n}=\vartheta X_{n}-\frac{1}{2}\vartheta^{2}\sigma_{2}^{2}+R_{n,\vartheta}.
\]
Nach der Voraussetzung ergibt sich nun 
\begin{eqnarray*}
\mathcal{L}\left(\left(S_{n},\Lambda_{n}\right)^{t}\left|P_{n,0}\right.\right) & \rightarrow & N\left(\left(\begin{array}{c}
0\\
-\frac{1}{2}\vartheta^{2}\sigma_{2}^{2}
\end{array}\right),\left(\begin{array}{cc}
\sigma_{1}^{2} & \vartheta\sigma_{12}\\
\vartheta\sigma_{12} & \vartheta^{2}\sigma_{2}^{2}
\end{array}\right)\right)
\end{eqnarray*}
für $n\rightarrow\infty.$ Man erhält somit 
\begin{eqnarray*}
\alpha & = & \lim_{n\rightarrow\infty}E_{P_{n,0}}\left(\varphi_{n}\right)\\
 & = & \lim_{n\rightarrow\infty}\left(\lambda_{n}P_{n,0}\left(\left|S_{n}\right|=c_{n}\right)+P_{n,0}\left(S_{n}>c_{n}\right)+P_{n,0}\left(S_{n}<-c_{n}\right)\right)\\
 & = & 1-\Phi\left(\frac{\lim\limits_{n\rightarrow\infty}c_{n}}{\sigma_{1}}\right)+\Phi\left(-\frac{\lim\limits_{n\rightarrow\infty}c_{n}}{\sigma_{1}}\right)\\
 & = & 2-2\Phi\left(\frac{\lim\limits_{n\rightarrow\infty}c_{n}}{\sigma_{1}}\right).
\end{eqnarray*}
Hieraus folgt 
\begin{equation}
\lim\limits_{n\rightarrow\infty}c_{n}=\sigma_{1}u_{1-\frac{\alpha}{2}}.\label{krit_wert_kogz}
\end{equation}
 Das dritte Lemma von Le Cam (vgl. \cite{Hajek:1999}, Seite 257)
liefert 
\[
\mathcal{L}\left(\left.S_{n}\right|P_{n,\vartheta}\right)\rightarrow N\left(\vartheta\sigma_{12},\sigma_{1}^{2}\right)
\]
 für $n\rightarrow\infty.$ Mit (\ref{krit_wert_kogz}) erhält man
die Konvergenz
\begin{eqnarray*}
E_{P_{n,\vartheta}}\left(\varphi_{n}\right) & = & \lambda_{n}P_{n,\vartheta}\left(\left|S_{n}\right|=c_{n}\right)+P_{n,\vartheta}\left(S_{n}>c_{n}\right)+P_{n,\vartheta}\left(S_{n}<-c_{n}\right)\\
 & = & \lambda_{n}P_{n,\vartheta}\left(\left|S_{n}\right|=c_{n}\right)+P_{n,\vartheta}\left(\frac{S_{n}-\vartheta\sigma_{12}}{\sigma_{1}}<\frac{-c_{n}-\vartheta\sigma_{12}}{\sigma_{1}}\right)\\
 &  & +P_{n,\vartheta}\left(\frac{S_{n}-\vartheta\sigma_{12}}{\sigma_{1}}>\frac{c_{n}-\vartheta\sigma_{12}}{\sigma_{1}}\right)\\
 & \rightarrow & \Phi\left(\frac{\vartheta\sigma_{12}}{\sigma_{1}}-u_{1-\frac{\alpha}{2}}\right)+\Phi\left(-\frac{\vartheta\sigma_{12}}{\sigma_{1}}-u_{1-\frac{\alpha}{2}}\right)
\end{eqnarray*}
für $n\rightarrow\infty.$ \end{proof}Wir kehren nun zu dem zweiseitigen
Zweistichprobenproblem 
\[
H_{2}=\left\{ P\otimes Q\in\mathcal{P}\otimes\mathcal{Q}:k\left(P\otimes Q\right)=a\right\} 
\]
gegen
\[
K_{2}=\left\{ P\otimes Q\in\mathcal{P}\otimes\mathcal{Q}:k\left(P\otimes Q\right)\neq a\right\} ,
\]
 das beim Testen statistischer Funktionale häufig auftritt, zurück.
Zur Lösung dieses Testproblems wird eine geeignete lokale Parametrisierung
vorgestellt, die analog wie in Abschnitt \ref{sec:Asymptotische-Eigenschaften-der}
begründet werden kann. Ein statistisches Funktional $k:\mathcal{P}\otimes\mathcal{Q}\rightarrow\mathbb{R}$
sei differenzierbar an einer Stelle $P_{0}\otimes Q_{0}\in H_{2}$
mit dem kanonischen Gradienten $\widetilde{k}\in L_{2}^{(0)}\left(P_{0}\otimes Q_{0}\right).$
Außerdem gelte 
\begin{equation}
\int\widetilde{k}^{2}dP_{0}\otimes Q_{0}\neq0.\label{nnf}
\end{equation}
\begin{defi}\label{H_2_impl} Die Menge $\mathcal{H}_{2}$ der impliziten
Hypothesen ist definiert durch 
\[
\mathcal{H}_{2}:=\left\{ \left(P_{t_{n}}\otimes Q_{t_{n}}\right)_{n\in\mathbb{N}}\in\mathcal{F}_{2}:\lim\limits_{n\rightarrow\infty}\sqrt{n}\left(k\left(P_{t_{n}}\otimes Q_{t_{n}}\right)-k\left(P_{0}\otimes Q_{0}\right)\right)=0\right\} .
\]
 Die Menge $\mathcal{K}_{2}$ der impliziten Alternativen ist gegeben
durch 
\[
\mathcal{K}_{2}:=\left\{ \left(P_{t_{n}}\otimes Q_{t_{n}}\right)_{n\in\mathbb{N}}\in\mathcal{F}_{2}:\lim\limits_{n\rightarrow\infty}\sqrt{n}\left(k\left(P_{t_{n}}\otimes Q_{t_{n}}\right)-k\left(P_{0}\otimes Q_{0}\right)\right)\neq0\right\} .
\]
 \end{defi}Die Testfolge
\begin{eqnarray}
\psi_{n} & = & \left\{ \begin{array}{cccc}
1 &  & >\\
 & \left|T_{n}\right| &  & c\\
0 &  & \leq
\end{array}\right.\label{zweiset_asymp_opt_testfolge}
\end{eqnarray}
 eignet sich für das zweiseitige Testproblem $H_{2}$ gegen $K_{2},$
wobei 
\[
T_{n}\left(\omega_{1,1},\ldots,\omega_{1,n_{1}},\omega_{2,1},\ldots,\omega_{2,n_{2}}\right)=\frac{\sqrt{n}}{n_{1}}\sum_{i=1}^{n_{1}}\widetilde{k}_{1}\left(\omega_{1,i}\right)+\frac{\sqrt{n}}{n_{2}}\sum_{j=1}^{n_{2}}\widetilde{k}_{2}\left(\omega_{2,j}\right)
\]
 die Teststatistik wie in (\ref{teststatistik T}) ist. Der kritische
Wert 
\[
c:=u_{1-\frac{\alpha}{2}}\left(\frac{1}{1-d}\int\widetilde{k}_{1}^{2}\,dP_{0}+\frac{1}{d}\int\widetilde{k}_{2}^{2}\,dQ_{0}\right)^{\frac{1}{2}}
\]
ist unabhängig von $n$. Es ist zu beachten, dass 
\[
\left\Vert \widetilde{k}\right\Vert _{L_{2}\left(P_{0}\otimes Q_{0}\right)}^{2}=\int\widetilde{k}_{1}^{2}\,dP_{0}+\int\widetilde{k}_{2}^{2}\,dQ_{0}\neq0
\]
 vorausgesetzt wird.\begin{satz}\label{sgl_zw_test_asymp_gute}Sei
$\left(P_{t_{n}}\otimes Q_{t_{n}}\right)_{n\in\mathbb{N}}\in\mathcal{F}_{2}$.
Der Parameter $\vartheta$ sei durch 
\[
\vartheta:=\lim\limits_{n\rightarrow\infty}\sqrt{n}\left(k\left(P_{t_{n}}\otimes Q_{t_{n}}\right)-k\left(P_{0}\otimes Q_{0}\right)\right)
\]
 gegeben. Außerdem seien $\left(n_{1}\right)_{n\in\mathbb{N}}\subset\mathbb{N}$
und $\left(n_{2}\right)_{n\in\mathbb{N}}\subset\mathbb{N}$ zwei Folgen
mit $n_{1}+n_{2}=n$ für alle $n\in\mathbb{N}$ und $\lim\limits_{n\rightarrow\infty}\frac{n_{2}}{n}=d\in(0,1)$.
Dann gilt
\begin{equation}
\lim_{n\rightarrow\infty}\int\psi_{n}\,dP_{t_{n}}^{n_{1}}\otimes Q_{t_{n}}^{n_{2}}=\Phi\left(\frac{\vartheta}{\sigma_{1}}-u_{1-\frac{\alpha}{2}}\right)+\Phi\left(-\frac{\vartheta}{\sigma_{1}}-u_{1-\frac{\alpha}{2}}\right),\label{sqrt}
\end{equation}
wobei 
\[
\sigma_{1}:=\left(\frac{1}{1-d}\int\widetilde{k}_{1}^{2}\,dP_{0}+\frac{1}{d}\int\widetilde{k}_{2}^{2}\,dQ_{0}\right)^{\frac{1}{2}}.
\]
 \end{satz}\begin{proof}Es wird zuerst gezeigt, dass die Folge $\left(\psi_{n}\right)_{n\in\mathbb{N}}$
die Bedingung (\ref{asymp_niveau_alpha}) erfüllt. Nach Satz \ref{asymp_verhalt_der_T_stat}
ergibt sich
\begin{eqnarray*}
\lim_{n\rightarrow\infty}\int\psi_{n}\,dP_{0}^{n_{1}}\otimes Q_{0}^{n_{2}} & = & \lim_{n\rightarrow\infty}\left(P_{0}^{n_{1}}\otimes Q_{0}^{n_{2}}\left(\left\{ T_{n}>c\right\} \right)+P_{0}^{n_{1}}\otimes Q_{0}^{n_{2}}\left(\left\{ T_{n}<-c\right\} \right)\right)\\
 & = & N\left(0,\sigma_{1}^{2}\right)\left(\left(c,+\infty\right)\right)+N\left(0,\sigma_{1}^{2}\right)\left(\left(-\infty,-c\right)\right)\\
 & = & N\left(0,\sigma_{1}^{2}\right)\left(\left(u_{1-\frac{\alpha}{2}}\sigma_{1},+\infty\right)\right)\\
 &  & +N\left(0,\sigma_{1}^{2}\right)\left(\left(-\infty,-u_{1-\frac{\alpha}{2}}\sigma_{1}\right)\right)\\
 & = & 1-\Phi\left(u_{1-\frac{\alpha}{2}}\right)+\Phi\left(-u_{1-\frac{\alpha}{2}}\right)\\
 & = & \alpha.
\end{eqnarray*}
Für den Parameter $\vartheta$ erhält man 
\begin{eqnarray*}
\vartheta & = & \lim\limits_{n\rightarrow\infty}\sqrt{n}\left(k\left(P_{t_{n}}\otimes Q_{t_{n}}\right)-k\left(P_{0}\otimes Q_{0}\right)\right)\\
 & = & \lim\limits_{n\rightarrow\infty}t_{n}\sqrt{n}\frac{\left(k\left(P_{t_{n}}\otimes Q_{t_{n}}\right)-k\left(P_{0}\otimes Q_{0}\right)\right)}{t_{n}}\\
 & = & t\int\widetilde{k}g\,dP_{0}\otimes Q_{0},
\end{eqnarray*}
wobei $t:=\lim\limits_{n\rightarrow\infty}t_{n}n^{\frac{1}{2}}$ ist.
Sei $\left(P_{t_{n}}\otimes Q_{t_{n}}\right)_{n\in\mathbb{N}}\in\mathcal{K}_{2}\cup\mathcal{H}_{2}$
mit der zugehörigen Tangente $g\in L_{2}\left(P_{0}^{n_{1}}\otimes Q_{0}^{n_{2}}\right).$
Angenommen, es gilt $\int g^{2}\,dP_{0}\otimes Q_{0}=0.$ Nach Satz
\ref{niveau_alpha_trivial} folgt
\[
\lim_{n\rightarrow\infty}\int\psi_{n}\,dP_{t_{n}}^{n_{1}}\otimes Q_{t_{n}}^{n_{2}}=\lim_{n\rightarrow\infty}\int\psi_{n}\,dP_{0}^{n_{1}}\otimes Q_{0}^{n_{2}}=\alpha
\]
 und die Behauptung (\ref{sqrt}) gilt wegen 
\[
\vartheta=t\int\widetilde{k}g\,dP_{0}\otimes Q_{0}=0.
\]
Es sei nun $\int g^{2}\,dP_{0}\otimes Q_{0}\neq0.$ Nach Satz \ref{L2Produkt}
existieren $g_{1}\in L_{2}^{(0)}\left(P_{0}\right)$ und $g_{2}\in L_{2}^{(0)}\left(Q_{0}\right)$
mit $g=g_{1}\circ\pi_{1}+g_{2}\circ\pi_{2}.$ Die Folge $\left(P_{t_{n}}^{n_{1}}\otimes Q_{t_{n}}^{n_{2}}\right)_{n\in\mathbb{N}}$
der Wahrscheinlichkeitsmaße ist LAN nach Satz \ref{lan_produkt} mit
der zentralen Folge
\[
X_{n}=\sum_{i=1}^{n_{1}}t_{n}g_{1}\circ\pi_{i}+\sum_{i=n_{1}+1}^{n}t_{n}g_{2}\circ\pi_{i}
\]
 und 
\begin{eqnarray*}
\sigma_{2}^{2} & = & \lim_{n\rightarrow\infty}\mbox{Var}_{P_{0}^{n_{1}}\otimes Q_{0}^{n_{2}}}\left(\sum_{i=1}^{n_{1}}t_{n}g_{1}\circ\pi_{i}+\sum_{i=n_{1}+1}^{n}t_{n}g_{2}\circ\pi_{i}\right)\\
 & = & \lim_{n\rightarrow\infty}t_{n}^{2}n\left(\frac{n_{1}}{n}\int g_{1}^{2}\,dP_{0}+\frac{n_{2}}{n}\int g_{2}^{2}\,dQ_{0}\right)\\
 & = & t^{2}\left(\left(1-d\right)\int g_{1}^{2}\,dP_{0}+d\int g_{2}^{2}\,dQ_{0}\right).
\end{eqnarray*}
Nach Satz \ref{gem-asymp-vert} erhält man die Konvergenz
\[
\mathcal{L}\left(\left(T_{n},X_{n}\right)^{t}\left|P_{0}^{n_{1}}\otimes Q_{0}^{n_{2}}\right.\right)\rightarrow N\left(\left(\begin{array}{c}
0\\
0
\end{array}\right),\left(\begin{array}{cc}
\sigma_{1}^{2} & \sigma_{12}\\
\sigma_{12} & \sigma_{2}^{2}
\end{array}\right)\right)
\]
für $n\rightarrow\infty$ mit 
\[
\sigma_{12}=a_{1}\frac{1}{\sqrt{1-d}}\int\widetilde{k}_{1}g_{1}\,dP_{0}+a_{2}\frac{1}{\sqrt{d}}\int\widetilde{k}_{2}g_{2}\,dQ_{0}.
\]
Für die Koeffizienten $a_{1}$ und $a_{2}$ ergibt sich
\[
a_{1}=\lim_{n\rightarrow\infty}n_{1}^{-\frac{1}{2}}\sum_{i=1}^{n_{1}}t_{n}=\lim_{n\rightarrow\infty}\left(\frac{n_{1}}{n}\right)^{\frac{1}{2}}n^{\frac{1}{2}}t_{n}=t\sqrt{1-d}
\]
und
\[
a_{2}=\lim_{n\rightarrow\infty}n_{2}^{-\frac{1}{2}}\sum_{i=1}^{n_{2}}t_{n}=\lim_{n\rightarrow\infty}\left(\frac{n_{2}}{n}\right)^{\frac{1}{2}}n^{\frac{1}{2}}t_{n}=t\sqrt{d}.
\]
 Hieraus folgt 
\[
\sigma_{12}=t\int\widetilde{k}_{1}g_{1}\,dP_{0}+t\int\widetilde{k}_{2}g_{2}\,dQ_{0}=t\int\widetilde{k}g\,dP_{0}\otimes Q_{0}.
\]
Mit Hilfe von Satz \ref{asymp-guete_zw} erhält man nun
\begin{eqnarray*}
\lim_{n\rightarrow\infty}\int\psi_{n}dP_{t_{n}}^{n_{1}}\otimes Q_{t_{n}}^{n_{2}} & = & \Phi\left(\frac{\sigma_{12}}{\sigma_{1}}-u_{1-\frac{\alpha}{2}}\right)+\Phi\left(-\frac{\sigma_{12}}{\sigma_{1}}-u_{1-\frac{\alpha}{2}}\right)\\
 & = & \Phi\left(\frac{\vartheta}{\sigma_{1}}-u_{1-\frac{\alpha}{2}}\right)+\Phi\left(-\frac{\vartheta}{\sigma_{1}}-u_{1-\frac{\alpha}{2}}\right).
\end{eqnarray*}
\end{proof}\begin{korollar}Die Testfolge $\left(\psi_{n}\right)_{n\in\mathbb{N}}$
ist asymptotisch unverfälscht für $\mathcal{K}_{2}$ gegen $\mathcal{H}_{2}$.\end{korollar}\begin{proof}Wegen
$\vartheta=\lim\limits_{n\rightarrow\infty}\sqrt{n}\left(k\left(P_{t_{n}}\otimes Q_{t_{n}}\right)-k\left(P_{0}\otimes Q_{0}\right)\right)$
reicht es zu zeigen, dass die asymptotische Gütefunktion 
\[
g:\mathbb{R}\rightarrow\left(0,1\right),\,\,\,\vartheta\mapsto\Phi\left(\frac{\vartheta}{\sigma_{1}}-u_{1-\frac{\alpha}{2}}\right)+\Phi\left(-\frac{\vartheta}{\sigma_{1}}-u_{1-\frac{\alpha}{2}}\right)
\]
ihr globales Minimum an der Stelle $\vartheta=0$ annimmt. Man erhält
zunächst
\begin{eqnarray*}
\frac{d}{d\vartheta}g\left(\vartheta\right) & = & \frac{1}{\sigma_{1}}f\left(\frac{\vartheta}{\sigma_{1}}-u_{1-\frac{\alpha}{2}}\right)-\frac{1}{\sigma_{1}}f\left(-\frac{\vartheta}{\sigma_{1}}-u_{1-\frac{\alpha}{2}}\right),
\end{eqnarray*}
wobei $f(x)=\frac{1}{\sqrt{2\pi}}\exp\left(-\frac{x^{2}}{2}\right)$
die Lebesgue-Dichte der Standardnormalverteilung ist. Es gelten nun
folgende Äquivalenzen:
\begin{eqnarray*}
 &  & \frac{d}{d\vartheta}g\left(\vartheta\right)=0\\
 & \Leftrightarrow & f\left(\frac{\vartheta}{\sigma_{1}}-u_{1-\frac{\alpha}{2}}\right)-f\left(-\frac{\vartheta}{\sigma_{1}}-u_{1-\frac{\alpha}{2}}\right)=0\\
 & \Leftrightarrow & \exp\left(-\frac{1}{2}\left(\frac{\vartheta}{\sigma_{1}}-u_{1-\frac{\alpha}{2}}\right)^{2}+\frac{1}{2}\left(-\frac{\vartheta}{\sigma_{1}}-u_{1-\frac{\alpha}{2}}\right)^{2}\right)=1\\
 & \Leftrightarrow & \exp\left(2u_{1-\frac{\alpha}{2}}\frac{\vartheta}{\sigma_{1}}\right)=1\\
 & \Leftrightarrow & 2u_{1-\frac{\alpha}{2}}\frac{\vartheta}{\sigma_{1}}=0\\
 & \Leftrightarrow & \vartheta=0.
\end{eqnarray*}
 Außerdem gilt $\frac{d}{d\vartheta}g\left(\vartheta\right)<0$ für
$\vartheta<0$ und $\frac{d}{d\vartheta}g\left(\vartheta\right)>0$
für $\vartheta>0.$ Man erhält somit, dass $\vartheta=0$ die globale
Minimumstelle ist. Somit ist alles bewiesen.\end{proof}\begin{bem}In
dem nächsten Kapitel wird unter anderem gezeigt, dass $\left(\psi_{n}\right)_{n\in\mathbb{N}}$
unter schwachen Voraussetzungen an den Tangentenkegel $K\left(P_{0}\otimes Q_{0},\mathcal{P}\otimes\mathcal{Q}\right)$
eine asymptotisch optimale unverfälschte Testfolge für das Testproblem
$\mathcal{K}_{2}$ gegen $\mathcal{H}_{2}$ ist.\end{bem}

\chapter{Asymptotisch optimales Testen statistischer Funktionale\label{cha:Nichtparametrisches-asymptotisch-optimale}}

\section{Konvergenz gegen unendlich-dimensionale Gauß-Shift Experimente}

In diesem Abschnitt werden die für weitere Untersuchungen benötigten
Kenntnisse über unendlich-dimensionale Gauß-Shift Experimente bereitgestellt.
Ausführliche Informationen und viele Beispiele zu diesem Thema findet
man in \cite{Strasser:1985b}, \cite{Strasser:1985s}, \cite{LeCam:2000}
und \cite{LeCam:1986}.\begin{defi}[Standard-Gaußprozess]Es seien
$\left(H,\left\langle \cdot,\cdot\right\rangle \right)$ ein Hilbertraum
und $\left(\Omega,\mathcal{A},P_{0}\right)$ ein Wahrscheinlichkeitsraum.
Eine lineare Abbildung $L:H\rightarrow L_{2}\left(P_{0}\right)$ heißt
Standard-Gaußprozess, wenn 
\[
\mathcal{L}\left(\left.L\left(h\right)\right|P_{0}\right)=N\left(0,\left\Vert h\right\Vert ^{2}\right)
\]
 für alle $h\in H$ gilt. \end{defi}\begin{defi}[Gauß-Shift] Sei
$\left(H,\left\langle \cdot,\cdot\right\rangle \right)$ ein Hilbertraum.
Ein Experiment $E=\left(\Omega,\mathcal{A},\left\{ P_{h}:h\in H\right\} \right)$
heißt Gauß-Shift auf $H$, wenn
\[
\frac{dP_{h}}{dP_{0}}=\exp\left(L\left(h\right)-\frac{1}{2}\left\Vert h\right\Vert ^{2}\right)
\]
für alle $h\in H$ gilt, wobei $L:H\rightarrow L_{2}\left(P_{0}\right)$
ein Standard-Gaußprozess ist.\end{defi}

Der nachfolgende Satz zeigt die Verbindung zwischen dem Hilbertraum
$H$ und der Topologie der Familie $\left\{ P_{h}:h\in H\right\} $
von Wahrscheinlichkeitsmaßen bzgl. der Hellingerdistanz. Dieser Zusammenhang
überträgt sich dann ebenfalls auf die Norm der totalen Variation.
\begin{satz}\label{topologie_gaus_shift}Sei $E=\left(\Omega,\mathcal{A},\left\{ P_{h}:h\in H\right\} \right)$
ein Gauß-Shift. Die Hellingerdistanz zwischen zwei Wahrscheinlichkeitsmaßen
des Gauß-Shift Experimentes $E$ ergibt sich als 
\[
d^{2}\left(P_{h_{1}},P_{h_{2}}\right)=1-\exp\left(-\frac{1}{8}\left\Vert h_{1}-h_{2}\right\Vert ^{2}\right)
\]
 für alle $h_{1},h_{2}\in H$.\end{satz}\begin{proof}vgl. \cite{Strasser:1985b},
Seite 346, Remarks 69.8. \end{proof} \begin{korollar}\label{topologie_korollar}Sei
$\left(h_{n}\right)_{n\in\mathbb{N}}\subset H$ eine konvergente Folge
mit $\lim\limits_{n\rightarrow\infty}\left\Vert h_{n}-h\right\Vert =0$
für ein $h\in H.$ Dann gilt $\lim_{n\rightarrow\infty}d\left(P_{h},P_{h_{n}}\right)=0$
und $\lim_{n\rightarrow\infty}\left\Vert P_{h_{n}}-P_{h}\right\Vert =0.$\end{korollar}
\begin{proof}Mit Hilfe von Satz \ref{topologie_gaus_shift} ergibt
sich
\[
\lim_{n\rightarrow\infty}d^{2}\left(P_{h},P_{h_{n}}\right)=\lim_{n\rightarrow\infty}\left(1-\exp\left(-\frac{1}{8}\left\Vert h_{1}-h_{2}\right\Vert ^{2}\right)\right)=0.
\]
Für alle Wahrscheinlichkeitsmaße $P,Q\in\mathcal{M}_{1}\left(\Omega,\mathcal{A}\right)$
gilt 
\[
d^{2}\left(P,Q\right)\leq\left\Vert P-Q\right\Vert \leq\sqrt{2}d\left(P,Q\right)
\]
 nach \cite{Witting:1985}, Seite 136, Hilfssatz 1.142. Man erhält
somit
\[
0\leq\lim_{n\rightarrow\infty}\left\Vert P_{h_{n}}-P_{h}\right\Vert \leq\lim_{n\rightarrow\infty}\sqrt{2}d\left(P_{h},P_{h_{n}}\right)=0.
\]
 \end{proof}Sei nun $\mathcal{P}\subset\mathcal{M}_{1}\left(\Omega,\mathcal{A}\right)$
eine nichtparametrische Familie von Wahrscheinlichkeitsmaßen. Sei
$P_{0}\in\mathcal{P}$ ein Wahrscheinlichkeitsmaß. Die lokalen Eigenschaften
der Familie $\mathcal{P}$ an der Stelle $P_{0}$ werden mit Hilfe
des Tangentenkegels $K\left(P_{0},\mathcal{P}\right)$ und des Tangentialraums
$T\left(P_{0},\mathcal{P}\right)$ beschrieben. Um die Theorie der
Gauß-Shift Experimente anzuwenden fehlt jedoch die Zuordnung zwischen
den Wahrscheinlichkeitsmaßen aus $\mathcal{P}$ und Tangenten aus
$T\left(P_{0},\mathcal{P}\right),$ weil im Allgemeinen $K\left(P_{0},\mathcal{P}\right)\neq T\left(P_{0},\mathcal{P}\right)$
gilt. Deswegen wird eine semiparametrische Familie $\mathcal{\widetilde{P}}\left(P_{0}\right)$
von Wahrscheinlichkeitsmaßen konstruiert, die lokal an der Stelle
$P_{0}$ den gleichen Tangentialraum $T\left(P_{0},\mathcal{P}\right)$
besitzt. Die Konstruktion der $L_{2}(0)$-differenzierbaren Kurven
entspricht dem Master-Modell aus \cite{Janssen:2004b}. \begin{satz}[Konstruktion  $L_2(0)$-differenzierbarer Kurven] \label{konstruktion} Es
seien $P\in\mathcal{M}_{1}(\Omega,\mathcal{A})$ ein Wahrscheinlichkeitsmaß
und $g\in L_{2}^{(0)}\left(P\right)$ eine Tangente. Dann existiert
eine $L_{2}(0)$-differenzierbare Kurve 
\[
\mathbb{R}\rightarrow\mathcal{M}_{1}(\Omega,\mathcal{A}),\,t\mapsto P_{tg}
\]
 mit der Tangente $g$ und $P_{0}=P$. Außerdem gilt $P_{tg}\ll P_{0}$
für alle $t\in\mathbb{R}.$ \end{satz}\begin{proof}Die Funktion
$f_{tg}:=\left(1+\frac{1}{2}tg\right)^{2}c(tg)^{-1}$ mit $c(tg):=\int\left(1+\frac{1}{2}tg\right)^{2}dP$
ist für jedes $t\in\mathbb{R}$ wohldefiniert, weil
\begin{eqnarray*}
c(tg) & = & \int\left(1+\frac{1}{2}tg\right)^{2}dP=1+\frac{t^{2}}{4}\int g^{2}\,dP\geq1
\end{eqnarray*}
 für alle $t\in\mathbb{R}$ gilt. Man erhält außerdem 
\[
\left.\frac{d}{dt}c\left(tg\right)\right|_{t=0}=\left.\frac{t}{2}\int g^{2}\,dP\right|_{t=0}=0.
\]
 Durch $P_{tg}:=f_{tg}P$ wird ein Wahrscheinlichkeitsmaß definiert,
da die Bedingungen $f_{tg}\geq0$ und $\int f_{tg}dP=1$ für alle
$t\in\mathbb{R}$ und $g\in L_{2}^{(0)}\left(P\right)$ erfüllt sind.
Insbesondere erhält man $P_{0}=f_{0}P=P$. 

Es ist zu zeigen, dass die Kurve $t\mapsto P_{t}$ $L_{2}\left(P_{0}\right)$-differenzierbar
mit der Tangente $g$ ist. Die Bedingung (\ref{L2-diffbarkeit2})
ist wegen $P_{tg}\ll P_{0}$ für alle $t\in\mathbb{R}$ erfüllt. Es
bleibt also die Bedingung (\ref{L2-diffbarkeit1}) nachzuweisen. Wegen
$g\in L_{2}\left(P_{0}\right)$ gilt $P\left(\left\{ g=\infty\right\} \right)=0$.
Für alle $\omega\in\Omega\setminus\left\{ g=\infty\right\} $ erhält
man zunächst
\[
\lim_{t\rightarrow0}\frac{2}{t}\left(\left|1+\frac{1}{2}tg\left(\omega\right)\right|-1\right)=\lim_{t\rightarrow0}\frac{2}{t}\left(1+\frac{1}{2}tg\left(\omega\right)-1\right)=g\left(\omega\right).
\]
Hieraus folgt die Konvergenz $\frac{2}{t}\left(\left|1+\frac{1}{2}tg\right|-1\right)\rightarrow g$
$P$-f.s. für $t\rightarrow0$. Mit der Dreiecksungleichung $\left|\left|a\right|-\left|b\right|\right|\leq\left|a-b\right|$
erhält man die Majorante
\[
\left|\frac{2}{t}\left(\left|1+\frac{1}{2}tg\right|-1\right)\right|\leq\frac{2}{\left|t\right|}\left|1+\frac{1}{2}tg-1\right|=\left|g\right|.
\]
 Nach dem Satz von der dominierten Konvergenz ergibt sich nun
\begin{equation}
\lim_{t\rightarrow0}\left\Vert \frac{2}{t}\left(\left|1+\frac{1}{2}tg\right|-1\right)-g\right\Vert _{L_{2}\left(P_{0}\right)}=0,\label{dom_kogz_major}
\end{equation}
vgl. \cite{elstrodt:2002}, Seite 259, Satz 5.3. Es gilt außerdem
\begin{eqnarray}
\lim_{t\rightarrow0}\frac{c(tg)^{-\frac{1}{2}}-1}{t} & = & \left.\frac{d}{dt}c\left(tg\right)^{-\frac{1}{2}}\right|_{t=0}\nonumber \\
 & = & -\frac{1}{2}c\left(tg\right)^{-\frac{3}{2}}\left.\frac{d}{dt}c\left(tg\right)\right|_{t=0}\nonumber \\
 & = & 0.\label{kogz_normierung}
\end{eqnarray}
 Mit (\ref{dom_kogz_major}) und (\ref{kogz_normierung}) erhält man
insgesamt
\begin{eqnarray*}
 &  & \left\Vert \frac{2}{t}\left(\left(\frac{dP_{tg}}{dP_{0}}\right)^{\frac{1}{2}}-1\right)-g\right\Vert _{L_{2}\left(P_{0}\right)}\\
 & = & \left\Vert \frac{2}{t}\left(\left|1+\frac{1}{2}tg\right|c(tg)^{-\frac{1}{2}}-1\right)-g\right\Vert _{L_{2}\left(P_{0}\right)}
\end{eqnarray*}
\begin{eqnarray*}
 & = & \left\Vert c(tg)^{-\frac{1}{2}}\left(\frac{2}{t}\left(\left|1+\frac{t}{2}g\right|-1\right)-g\right)+\frac{2}{t}\left(c(tg)^{-\frac{1}{2}}-1\right)\right.\\
 &  & \left.+g\left(c(tg)^{-\frac{1}{2}}-1\right)\right\Vert _{L_{2}\left(P_{0}\right)}\\
\textrm{} & \leq & c(tg)^{-\frac{1}{2}}\left\Vert \frac{2}{t}\left(\left|1+\frac{1}{2}tg\right|-1\right)-g\right\Vert _{L_{2}\left(P_{0}\right)}+\left|\frac{2}{t}\left(c(tg)^{-\frac{1}{2}}-1\right)\right|\\
 &  & +\left\Vert g\right\Vert _{L_{2}\left(P_{0}\right)}\left|c(tg)^{-\frac{1}{2}}-1\right|\\
 & \rightarrow & 0
\end{eqnarray*}
für $t\rightarrow0$. Somit ist alles bewiesen.\end{proof}\begin{defi}[semiparametrische Familie]Die
semiparametrische Familie $\mathcal{\widetilde{P}}\left(P_{0}\right)$
wird definiert durch
\[
\mathcal{\widetilde{P}}\left(P_{0}\right):=\left\{ P_{tg}\in\mathcal{M}_{1}\left(\Omega,\mathcal{A}\right):g\in T\left(P_{0},\mathcal{P}\right),\,t\in\mathbb{R}\right\} ,
\]
 wobei $t\mapsto P_{tg}$ eine $L_{2}\left(0\right)$-differenzierbare
Kurve mit Tangente $g$ ist. Die Wahrscheinlichkeitsmaße $P_{tg}$
sind durch den Dichtequotient 
\[
\frac{dP_{tg}}{dP_{0}}=\left(1+\frac{1}{2}tg\right)^{2}c(tg)^{-1}
\]
 mit $c(tg)=\int\left(1+\frac{1}{2}tg\right)^{2}dP_{0}$ für alle
$g\in T\left(P_{0},\mathcal{P}\right)$ und $t\in\mathbb{R}$ gegeben,
vgl. Satz \ref{konstruktion}. \end{defi} \begin{bem}\label{lineare_vereinfachung}Es
gilt 
\[
\mathcal{\widetilde{P}}\left(P_{0}\right)=\left\{ P_{g}\in\mathcal{M}_{1}\left(\Omega,\mathcal{A}\right):g\in T\left(P_{0},\mathcal{P}\right)\right\} ,
\]
weil $T\left(P_{0},\mathcal{P}\right)$ ein Vektorraum ist.\end{bem}Im
Folgenden wird die schwache Konvergenz von Experimenten kurz vorgestellt.
Dies führt dann zu der Verallgemeinerung der lokalen asymptotischen
Normalität, die sich als schwache Konvergenz gegen ein Gauß-Shift
Experiment neu definieren lässt (vgl. \cite{Strasser:1985b}, \cite{LeCam:2000}
und \cite{LeCam:1986}).\begin{defi}[schwache Konvergenz von Experimenten]
Es seien $\Theta$ eine nichtleere Menge und $\left(\Theta_{n}\right)_{n\in\mathbb{N}}$
eine Folge der Teilmengen von $\Theta$ mit $\Theta_{n}\uparrow\Theta$
für $n\rightarrow\infty$. Eine Folge von Experimenten $E_{n}=\left(\Omega_{n},\mathcal{A}_{n},\left\{ P_{n,h}:h\in\Theta_{n}\right\} \right),$
$n\in\mathbb{N}$ konvergiert schwach gegen ein Experiment $E=\left(\Omega,\mathcal{A},\left\{ P_{h}:h\in\Theta\right\} \right)$,
wenn 
\[
\mathcal{L}\left(\left.\left(\frac{dP_{n,h}}{dP_{n,s}}\right)_{h\in Z}\right|P_{n,s}\right)\rightarrow\mathcal{L}\left(\left.\left(\frac{dP_{h}}{dP_{s}}\right)_{h\in Z}\right|P_{s}\right)
\]
 für alle endliche Teilmengen $Z\subset\Theta$ und für jedes $s\in\Theta$
gilt. \end{defi}\begin{defi}[LAN]\label{gaus_shift_lan}Es seien
$H$ ein Hilbertraum und $\left(H_{n}\right)_{n\in\mathbb{N}}$ eine
Folge der Teilmengen von $H$ mit $H_{n}\uparrow H$ für $n\rightarrow\infty$.
Eine Folge $E_{n}=\left(\Omega_{n},\mathcal{A}_{n},\left\{ P_{n,h}:h\in H_{n}\right\} \right),n\in\mathbb{N}$
von Experimenten heißt lokal asymptotisch normal, falls sie schwach
gegen ein Gauß-Shift Experiment auf $H$ konvergiert. \end{defi}\begin{bem}Die
Definition \ref{gaus_shift_lan} ist verträglich mit der eindimensionalen
Version von LAN, vgl. Definition \ref{lan}. \end{bem}\begin{satz}[Hinreichendes und notwendiges Kriterium für LAN]\label{hin_not_lan}Eine
Folge $\left(E_{n}\right)_{n\in\mathbb{N}}$ von Experimenten ist
genau dann lokal asymptotisch normal, wenn der stochastische Prozess
\[
L_{n}\left(h\right):=\log\left(\frac{dP_{n,h}}{dP_{n,0}}\right)+\frac{1}{2}\left\Vert h\right\Vert ^{2}
\]
die nachfolgenden Bedingungen erfüllt:
\begin{enumerate}
\item Es gilt $\mathcal{L}\left(\left.L_{n}\left(h\right)\right|P_{n,0}\right)\rightarrow N\left(0,\left\Vert h\right\Vert ^{2}\right)$
für $n\rightarrow\infty.$
\item Für jedes $\varepsilon>0,$ für alle $a,b\in\mathbb{R}$ und für alle
$h_{1},h_{2}\in H$ gilt 
\[
\lim_{n\rightarrow\infty}P_{n,0}\left(\left|aL_{n}\left(h_{1}\right)+bL_{n}\left(h_{2}\right)-L_{n}\left(ah_{1}+bh_{2}\right)\right|>\varepsilon\right)=0.
\]
 
\end{enumerate}
\end{satz}\begin{proof}vgl. \cite{Strasser:1985b}, Seite 409, Theorem
80.2.\end{proof}\begin{satz}\label{tang_raum_lan}Für alle $n\in\mathbb{N}$
und $g\in T\left(P_{0},\mathcal{P}\right)$ sei $P_{n,g}:=P_{\frac{1}{\sqrt{n}}g}^{n}$
mit $P_{\frac{1}{\sqrt{n}}g}\in\mathcal{\widetilde{P}}\left(P_{0}\right).$
Außerdem sei $E_{n}:=\left(\Omega^{n},\mathcal{A}^{n},\left\{ P_{n,g}:g\in T\left(P_{0},\mathcal{P}\right)\right\} \right).$
Die Folge $\left(E_{n}\right)_{n\in\mathbb{N}}$ der Experimente ist
dann lokal asymptotisch normal.\end{satz}\begin{proof}Zum Beweis
wird Satz \ref{hin_not_lan} angewandt. Nach Satz \ref{lecam} gilt
\[
\log\left(\frac{dP_{n,h}}{dP_{n,0}}\right)=\frac{1}{\sqrt{n}}\sum_{i=1}^{n}h\circ\pi_{i}-\frac{1}{2}\int h^{2}dP_{0}+R_{n,h},
\]
 wobei $\lim\limits_{n\rightarrow\infty}P_{n,0}\left(\left|R_{n,h}\right|>\varepsilon\right)=0$
für jedes $\varepsilon>0$ gilt. Hieraus folgt 
\[
L_{n}\left(h\right)=\log\left(\frac{dP_{n,h}}{dP_{n,0}}\right)+\frac{1}{2}\int h^{2}dP_{0}=\frac{1}{\sqrt{n}}\sum_{i=1}^{n}h\circ\pi_{i}+R_{n,h}.
\]
 Man erhält die schwache Konvergenz 
\begin{eqnarray*}
\mathcal{L}\left(\left.L_{n}\left(h\right)\right|P_{n,0}\right) & = & \mathcal{L}\left(\left.\frac{1}{\sqrt{n}}\sum_{i=1}^{n}h\circ\pi_{i}+R_{n,h}\right|P_{n,0}\right)\\
 & \rightarrow & N\left(0,\int h^{2}dP_{0}\right)
\end{eqnarray*}
für $n\rightarrow\infty$ nach dem Zentralen Grenzwertsatz und nach
dem Lemma von Slutsky (vgl. \cite{Witting:1995}, Seite 76). Die Bedingung
$1$ aus Satz \ref{hin_not_lan} ist somit erfüllt. Es bleibt die
Bedingung $2$ nachzuweisen. Für jedes $\varepsilon>0,$ für alle
$a,b\in\mathbb{R}$ und für alle $h_{1},h_{2}\in T\left(P_{0},\mathcal{P}\right)$
ergibt sich 
\begin{eqnarray*}
\textrm{} &  & P_{n,0}\left(\left|aL_{n}\left(h_{1}\right)+bL_{n}\left(h_{2}\right)-L_{n}\left(ah_{1}+bh_{2}\right)\right|>\varepsilon\right)\\
 & = & P_{n,0}\left(\left|aR_{n,h_{1}}+bR_{n,h_{2}}-R_{n,ah_{1}+bh_{2}}\right|>\varepsilon\right)\\
 & \leq & P_{n,0}\left(\left|aR_{n,h_{1}}+bR_{n,h_{2}}\right|>\frac{\varepsilon}{2}\right)+P_{n,0}\left(\left|R_{n,ah_{1}+bh_{2}}\right|>\frac{\varepsilon}{2}\right)\\
 & \leq & P_{n,0}\left(\left|aR_{n,h_{1}}\right|>\frac{\varepsilon}{4}\right)+P_{n,0}\left(\left|bR_{n,h_{2}}\right|>\frac{\varepsilon}{4}\right)+P_{n,0}\left(\left|R_{n,ah_{1}+bh_{2}}\right|>\frac{\varepsilon}{2}\right)\\
 & \rightarrow & 0
\end{eqnarray*}
für $n\rightarrow0.$ Somit ist alles bewiesen.\end{proof}\begin{bem}Es
gelten die Voraussetzungen aus Satz \ref{tang_raum_lan}. Die Folge
$\left(E_{n}\right)_{n\in\mathbb{N}}$ der Experimente konvergiert
dann schwach gegen ein Gauß-Shift Experiment $E$. Das Limesexperiment
$E=\left(\widetilde{\Omega},\widetilde{\mathcal{A}},\left\{ Q_{h}:h\in T\left(P_{0},\mathcal{P}\right)\right\} \right)$
wird nur von dem Tangentialraum $T\left(P_{0},\mathcal{P}\right)$
bestimmt. Die spezielle Konstruktion der $L_{2}\left(P_{0}\right)$-differenzierbaren
Kurven $t\mapsto P_{th}$ in der Familie $\mathcal{\widetilde{P}}\left(P_{0}\right)$
spielt also keine Rolle für das Limesexperiment $E$. Insbesondere
gibt es keinen Unterschied zwischen den $L_{2}\left(P_{0}\right)$-differenzierbaren
Kurven, die in $\mathcal{P}$ liegen und deren Tangenten dann entsprechend
im Tangentenkegel $K\left(P_{0},\mathcal{P}\right)$ enthalten sind,
und den konstruierten $L_{2}\left(P_{0}\right)$-differenzierbaren
Kurven mit den Tangenten aus $T\left(P_{0},\mathcal{P}\right)$. Man
kann die $L_{2}\left(P_{0}\right)$-differenzierbaren Kurven in $\mathcal{\widetilde{P}}\left(P_{0}\right)$
als Hilfskonstruktionen betrachten, die asymptotisch wieder eliminiert
werden.\end{bem}

\section{Hauptsatz der asymptotischen Testtheorie und obere Schranken für
die asymptotischen Gütefunktionen }

Der folgende Satz stellt ein mächtiges Mittel zur Herleitung der oberen
Schranken für die asymptotische Gütefunktion dar, vgl. \cite{Strasser:1985s}
und \cite{Strasser:1985b}.

\begin{satz}[Hauptsatz der asymptotischen Testtheorie]\label{hauptstz_testtheorie}Es
seien $\Theta$ eine nichtleere Menge und $\left(\Theta_{n}\right)_{n\in\mathbb{N}}$
eine Folge der Teilmengen von $\Theta$ mit $\Theta_{n}\uparrow\Theta$
für $n\rightarrow\infty$. Eine Folge von Experimenten $E_{n}=\left(\Omega_{n},\mathcal{A}_{n},\left\{ P_{n,h}:h\in\Theta_{n}\right\} \right),$
$n\in\mathbb{N}$ konvergiere schwach gegen ein Limesexperiment $E=\left(\Omega,\mathcal{A},\left\{ P_{h}:h\in\Theta\right\} \right).$
Außerdem seien $H,K\subset\Theta$ nichtleere Teilmengen mit $H\cup K=\Theta$
und $H\cap K=\emptyset.$ Sei $\left(\varphi_{n}\right)_{n\in\mathbb{N}}$
eine Testfolge für das Testproblem $H$ gegen $K.$ Zu jeder Teilfolge
$\left(\varphi_{n_{j}}\right)_{j\in\mathbb{N}}$ existiert ein Test
$\varphi$ für $E,$ so dass gilt: 
\[
\liminf_{j\rightarrow\infty}\int\varphi_{n_{j}}\,dP_{n_{j},h}\leq\int\varphi\,dP_{h}\quad\mbox{für alle}\:h\in K
\]
 und
\[
\limsup_{j\rightarrow\infty}\int\varphi_{n_{j}}\,dP_{n_{j},h}\geq\int\varphi\,dP_{h}\quad\mbox{für alle}\:h\in H.
\]
 \end{satz}\begin{proof}vgl. \cite{Strasser:1985s}, Seite 87, Satz
11.8. oder \cite{Strasser:1985b}, Seite 310, Theorem 62.5.\end{proof}\begin{korollar}Es
gelten die Voraussetzungen aus Satz \ref{hauptstz_testtheorie}. Es
seien 
\[
\Phi_{\alpha}\left(E_{n}\right):=\left\{ \varphi:\Omega_{n}\rightarrow\left[0,1\right],\,\int\varphi\,dP_{n,h}\leq\alpha\quad\mbox{für alle}\:h\in H\cap\Theta_{n}\right\} 
\]
die Menge der Niveau $\alpha$-Tests für $E_{n}$ und 
\[
\Phi_{\alpha,u}\left(E_{n}\right):=\left\{ \varphi\in\Phi_{\alpha}\left(E_{n}\right):\int\varphi\,dP_{n,h}\geq\alpha\quad\mbox{für alle}\:h\in K\cap\Theta_{n}\right\} 
\]
die Menge der unverfälschten Niveau $\alpha$-Tests für $E_{n}.$
Sei $\left(\alpha_{n}\right)_{n\in\mathbb{N}}\subset\left(0,1\right)$
eine Folge mit $\lim\limits_{n\rightarrow\infty}\alpha_{n}=\alpha\in\left(0,1\right).$
Dann gilt
\begin{equation}
\limsup_{n\rightarrow\infty}\sup_{\varphi_{n}\in\Phi_{\alpha_{n}}\left(E_{n}\right)}\int\varphi_{n}\,dP_{n,h}\leq\sup_{\psi\in\Phi_{\alpha}\left(E\right)}\int\psi\,dP_{h}\label{kor1}
\end{equation}
und 
\begin{equation}
\limsup_{n\rightarrow\infty}\sup_{\varphi_{n}\in\Phi_{\alpha_{n},u}\left(E_{n}\right)}\int\varphi_{n}\,dP_{n,h}\leq\sup_{\psi\in\Phi_{\alpha,u}\left(E\right)}\int\psi\,dP_{h}\label{kor2}
\end{equation}
für alle $h\in K.$\end{korollar}\begin{proof}Zuerst wird die Ungleichung
(\ref{kor1}) bewiesen. Sei $g\in K$ fest gewählt. Sei $\left(\phi_{n_{j}}\right)_{j\in\mathbb{N}}$
eine Folge der Tests mit $\phi_{n_{j}}\in\Phi_{\alpha_{n_{j}}}\left(E_{n_{j}}\right)$
und 
\[
\lim_{j\rightarrow\infty}\int\phi_{n_{j}}\,dP_{n_{j},g}=\limsup_{n\rightarrow\infty}\sup_{\varphi_{n}\in\Phi_{\alpha_{n}}\left(E_{n}\right)}\int\varphi_{n}\,dP_{n,g}.
\]
 Nach Satz \ref{hauptstz_testtheorie} existiert ein Test $\phi\in\Phi_{\alpha}\left(E\right)$
mit $\liminf_{j\rightarrow\infty}\int\phi_{n_{j}}\,dP_{n_{j},h}\leq\int\phi\,dP_{h}$
für alle $h\in K$ und $\limsup_{j\rightarrow\infty}\int\phi_{n_{j}}\,dP_{n_{j},h}\geq\int\phi\,dP_{h}$
für alle $h\in H$. Hieraus folgt $\phi\in\Phi_{\alpha}\left(E\right)$.
Man erhält somit 
\begin{eqnarray*}
\limsup_{n\rightarrow\infty}\sup_{\varphi_{n}\in\Phi_{\alpha_{n}}\left(E_{n}\right)}\int\varphi_{n}\,dP_{n,g} & = & \lim_{j\rightarrow\infty}\int\phi_{n_{j}}\,dP_{n_{j},g}\\
 & = & \liminf_{j\rightarrow\infty}\int\phi_{n_{j}}\,dP_{n_{j},g}\\
 & \leq & \int\phi\,dP_{g}\\
 & \leq & \sup_{\psi\in\Phi_{\alpha}\left(E\right)}\int\psi\,dP_{g}.
\end{eqnarray*}
Die Ungleichung (\ref{kor2}) beweist man analog. \end{proof}Liegt
ein Gauß-Shift Experiment als Limesexperiment vor, so kann man für
einige wichtige Testprobleme die oberen Schranken für die asymptotische
Gütefunktion exakt ausrechnen. Es seien $H$ ein Hilbertraum mit der
zugehörigen Norm $\left\Vert \cdot\right\Vert :H\rightarrow\mathbb{R}_{\geq0}$
und $f:H\rightarrow\mathbb{R}$ eine lineare Funktion. Sei $E=\left(\Omega,\mathcal{A},\left\{ P_{h}:h\in H\right\} \right)$
ein Gauß-Shift Experiment. Wir interessieren uns für das einseitige
Testproblem
\begin{equation}
H_{1}^{l}=\left\{ h\in H:f\left(h\right)\leq0\right\} \quad\mbox{gegen}\quad K_{1}^{l}=\left\{ h\in H:f\left(h\right)>0\right\} \label{einseit_lin_test_prob}
\end{equation}
und für das zweiseitige Testproblem
\begin{equation}
H_{2}^{l}=\left\{ h\in H:f\left(h\right)=0\right\} \quad\mbox{gegen}\quad K_{2}^{l}=\left\{ h\in H:f\left(h\right)\neq0\right\} .\label{zweiseit_lin_test_prob}
\end{equation}
\begin{defi}Ein Test $\varphi$ zum Niveau $\alpha$ für das einseitige
Testproblem (\ref{einseit_lin_test_prob}) heißt Niveau $\alpha$-ähnlich,
falls $\int\varphi\,dP_{h}=\alpha$ für alle $h\in H$ mit $f\left(h\right)=0$
gilt.\end{defi}\begin{satz}Sei $\varphi$ ein Niveau $\alpha$-ähnlicher
Test für $H_{1}^{l}$ gegen $K_{1}^{l}.$ Dann gilt
\[
\int\varphi\,dP_{h}\leq\Phi\left(\frac{f\left(h\right)}{\left\Vert f\right\Vert }-u_{1-\alpha}\right)\,\,\mbox{für alle}\,\,h\in K_{1}^{l}
\]
und 
\[
\int\varphi\,dP_{h}\geq\Phi\left(\frac{f\left(h\right)}{\left\Vert f\right\Vert }-u_{1-\alpha}\right)\,\,\mbox{für alle}\,\,h\in H_{1}^{l}.
\]
\end{satz}\begin{proof}vgl. \cite{Strasser:1985b}, Seite 357, Lemma
71.1.\end{proof}\begin{satz}Sei $\varphi$ ein unverfälschter Niveau
$\alpha$-Test für $H_{2}^{l}$ gegen $K_{2}^{l}.$ Für alle $h\in H$
gilt dann
\[
\int\varphi\,dP_{h}\leq\Phi\left(\frac{f\left(h\right)}{\left\Vert f\right\Vert }-u_{1-\frac{\alpha}{2}}\right)+\Phi\left(-\frac{f\left(h\right)}{\left\Vert f\right\Vert }-u_{1-\frac{\alpha}{2}}\right).
\]
\end{satz}\begin{proof}vgl. \cite{Strasser:1985b}, Seite 358, Lemma
71.5.\end{proof}\begin{defi}Eine Folge $\left(\varphi_{n}\right)_{n\in\mathbb{N}}$
der Tests für das Testproblem $H_{1}^{l}$ gegen $K_{1}^{l}$ heißt
asymptotisch Niveau $\alpha$-ähnlich, falls 
\[
\lim_{n\rightarrow\infty}\int\varphi_{n}\,dP_{n,h}=\alpha
\]
für alle $h\in H$ mit $f\left(h\right)=0$ gilt. \end{defi}\begin{defi}Es
seien $\widehat{H}$ und $\widehat{K}$ zwei nichtleere Teilmengen
von $H$ mit $H=\widehat{H}+\widehat{K}$. Eine Folge $\left(\varphi_{n}\right)_{n\in\mathbb{N}}$
der Tests für das Testproblem $\widehat{H}$ gegen $\widehat{K}$
heißt asymptotisch unverfälscht zum Niveau $\alpha$, falls 
\[
\limsup_{n\rightarrow\infty}\int\varphi_{n}\,dP_{n,h}\leq\alpha\,\,\mbox{für alle}\,\,h\in\widehat{H}
\]
und 
\[
\liminf_{n\rightarrow\infty}\int\varphi_{n}\,dP_{n,h}\geq\alpha\,\,\mbox{für alle}\,\,h\in\widehat{K}.
\]
\end{defi}\begin{bem}Jede asymptotisch unverfälschte Niveau $\alpha$-Testfolge
für $H_{1}^{l}$ gegen $K_{1}^{l}$ ist ebenfalls asymptotisch Niveau
$\alpha$-ähnlich (vgl. \cite{Strasser:1985b}, Seite 429, Definition
82.5).\end{bem}\begin{satz}\label{asymp_einseitig} Sei $\left(\varphi_{n}\right)_{n\in\mathbb{N}}$
eine asymptotisch Niveau $\alpha$-ähnliche Testfolge für das Testproblem
$H_{1}^{l}$ gegen $K_{1}^{l}.$ Dann gilt
\[
\limsup_{n\rightarrow\infty}\int\varphi_{n}\,dP_{n,h}\leq\Phi\left(\frac{f\left(h\right)}{\left\Vert f\right\Vert }-u_{1-\alpha}\right)\,\,\mbox{für alle}\,\,h\in K_{1}^{l}
\]
und 
\[
\liminf_{n\rightarrow\infty}\int\varphi\,dP_{n,h}\geq\Phi\left(\frac{f\left(h\right)}{\left\Vert f\right\Vert }-u_{1-\alpha}\right)\,\,\mbox{für alle}\,\,h\in H_{1}^{l}.
\]
\end{satz}\begin{proof}vgl. \cite{Strasser:1985b}, Seite 429, Lemma
82.6.\end{proof}\begin{satz}\label{asymp_zweiseitig} Eine Folge
$\left(\varphi_{n}\right)_{n\in\mathbb{N}}$ der Tests für das Testproblem
$H_{2}^{l}$ gegen $K_{2}^{l}$ sei asymptotisch unverfälscht zum
Niveau $\alpha.$ Für alle $h\in H$ gilt dann
\[
\limsup_{n\rightarrow\infty}\int\varphi_{n}\,dP_{n,h}\leq\Phi\left(\frac{f\left(h\right)}{\left\Vert f\right\Vert }-u_{1-\frac{\alpha}{2}}\right)+\Phi\left(-\frac{f\left(h\right)}{\left\Vert f\right\Vert }-u_{1-\frac{\alpha}{2}}\right).
\]
\end{satz}\begin{proof}vgl. \cite{Strasser:1985b}, Seite 431, Lemma
82.12.\end{proof}

\section{Asymptotisch optimales Testen statistischer Funktionale bei Einstichprobenproblemen\label{sec:asymp_opt_ein}}

Es seien $\left(\Omega,\mathcal{A}\right)$ ein Messraum, $\mathcal{P}\subset\mathcal{M}_{1}\left(\Omega,\mathcal{A}\right)$
eine nichtparametrische Familie von Wahrscheinlichkeitsmaßen und $k:\mathcal{P}\rightarrow\mathbb{R}$
ein statistisches Funktional. Sei $a\in\mathbb{R}$ eine Zahl. Die
einseitigen Testprobleme 
\begin{equation}
H_{3}=\left\{ P\in\mathcal{P}:k\left(P\right)\leq a\right\} \quad\mbox{gegen}\quad K_{3}=\left\{ P\in\mathcal{P}:k\left(P\right)>a\right\} \label{testproblem3}
\end{equation}
 und die zweiseitigen Testprobleme
\begin{equation}
H_{4}=\left\{ P\in\mathcal{P}:k\left(P\right)=a\right\} \quad\mbox{gegen}\quad K_{4}=\left\{ P\in\mathcal{P}:k\left(P\right)\neq a\right\} \label{testproblem4}
\end{equation}
 sind von großer Bedeutung, weil viele Testprobleme aus der Praxis
sich auf diese Gestalt bringen lassen. Das statistische Funktional
$k:\mathcal{P}\rightarrow\mathbb{R}$ sei differenzierbar an einer
Stelle $P_{0}\in H_{4}$ mit $k\left(P_{0}\right)=a.$ Der kanonische
Gradient $\widetilde{k}\in L_{2}^{(0)}\left(P_{0}\right)$ des Funktionals
$k$ an der Stelle $P_{0}$ erfülle 
\begin{equation}
\int\widetilde{k}^{2}dP_{0}\neq0.\label{nicht_0_vor}
\end{equation}
Die Parametrisierung der Testprobleme (\ref{testproblem3}) und (\ref{testproblem4})
erfolgt dann wie in Abschnitten \ref{sec:Asymptotische-Eigenschaften-der}
und \ref{sec:zw:twsts}. Die ausführliche Darstellung der Parametrisierung
findet man außerdem in \cite{Janssen:1998}, \cite{Janssen:1999a}
und \cite{Janssen:1999b}.

\subsection{Asymptotisch optimale einseitige Tests}

\begin{defi}Die Menge $\mathcal{F}_{1}$ enthalte alle Folgen $\left(P_{t_{n}}\right)_{n\in\mathbb{N}}$,
die den folgenden Bedingungen genügen:
\begin{enumerate}
\item Die Nullfolge $\left(t_{n}\right)_{n\in\mathbb{N}}$ erfüllt $\lim_{n\rightarrow\infty}t_{n}\sqrt{n}>0.$ 
\item Es existieren ein $\varepsilon>0$ und eine $L_{2}\left(P_{0}\right)$-differenzierbare
Kurve $f:\left(-\varepsilon,\varepsilon\right)\rightarrow\mathcal{P},\,t\mapsto P_{t}$,
so dass $f\left(t_{n}\right)=P_{t_{n}}$ für alle $n\in\mathbb{N}$
gilt.
\end{enumerate}
\end{defi}\begin{defi}\label{impl_alternativen_einstichprobe} Die
Menge $\mathcal{K}_{3}$ der impliziten Alternativen für das Testproblem
(\ref{testproblem3}) ist definiert durch 
\[
\mathcal{K}_{3}:=\left\{ \left(P_{t_{n}}^{n}\right)_{n\in\mathbb{N}}\in\mathcal{F}_{1}:\lim_{n\rightarrow\infty}\sqrt{n}\left(k\left(P_{t_{n}}\right)-k\left(P_{0}\right)\right)>0\right\} .
\]

\noindent Die Menge $\mathcal{H}_{3}$ der impliziten Hypothesen für
das Testproblem (\ref{testproblem3}) ist gegeben durch
\[
\mathcal{H}_{3}:=\left\{ \left(P_{t_{n}}^{n}\right)_{n\in\mathbb{N}}\in\mathcal{F}_{1}:\lim_{n\rightarrow\infty}\sqrt{n}\left(k\left(P_{t_{n}}\right)-k\left(P_{0}\right)\right)\leq0\right\} .
\]
 \end{defi}\begin{bem}\label{bem_verallg_def_impl_alternat} In Definition
\ref{impl_alternativen_einstichprobe} reicht es wohl, wenn eine Folge
$\left(P_{t_{n}}^{n}\right)_{n\in\mathbb{N}}$ in $\mathcal{P}$ liegt.
Die zugehörige $L_{2}\left(P_{0}\right)$-differenzierbare Kurve $t\mapsto P_{t}$
braucht dann nicht in $\mathcal{P}$ zu liegen. Hier wird jedoch die
bewährte elegante Formulierung mit Hilfe des Tangentenkegels $K\left(P_{0},\mathcal{P}\right)$
gewählt. Der Tangentenkegel $K\left(P_{0},\mathcal{P}\right)$ beschreibt
in diesem Fall die lokalen Eigenschaften der Mengen aller impliziten
Alternativen $\mathcal{K}_{3}$ und aller impliziten Hypothesen $\mathcal{H}_{3}$.
Auf die Verallgemeinerung von Definition \ref{impl_alternativen_einstichprobe}
wird deswegen verzichtet. Außerdem ist der Unterschied unwesentlich,
weil die Familie $\mathcal{P}$ in der nichtparametrischen Statistik
meistens nicht genau genug festgelegt ist. Die Modellbildung kann
viel mehr mit Hilfe des Tangentenkegels $K\left(P_{0},\mathcal{P}\right)$
geschehen. \end{bem} Sei $\alpha\in\left(0,1\right).$ Für das einseitige
Testproblem $H_{3}$ gegen $K_{3}$ eignet sich die asymptotisch $\left\{ 0\right\} $-$\alpha$-ähnliche
Testfolge
\begin{equation}
\varphi_{n}=\left\{ \begin{array}{cccc}
1 &  & >\\
 & \frac{1}{\sqrt{n}}\sum\limits_{i=1}^{n}\widetilde{k}\circ\pi_{i} &  & u_{1-\alpha}\\
0 &  & \leq
\end{array}\left\Vert \widetilde{k}\right\Vert _{L_{2}\left(P_{0}\right)}\right.\label{eins_einstichprob_opt_test}
\end{equation}
mit $\left\Vert \widetilde{k}\right\Vert _{L_{2}\left(P_{0}\right)}=\left(\int\widetilde{k}^{2}\,dP_{0}\right)^{\frac{1}{2}}.$
Ausführliche Informationen über diese Testfolge findet man in \cite{Janssen:1999a},
\cite{Janssen:1998} und \cite{Strasser:1985b}. 

\noindent\begin{satz}\label{ep_es_asymp_gute} Die Voraussetzung
(\ref{nicht_0_vor}) sei erfüllt. Dann sind folgende Aussagen gültig:
\begin{enumerate}
\item Für jede Folge $\left(P_{t_{n}}^{n}\right)_{n\in\mathbb{N}}\in\mathcal{F}_{1}$
gilt dann
\[
\lim_{n\rightarrow\infty}\int\varphi_{n}\,dP_{t_{n}}^{n}=\Phi\left(\frac{\vartheta}{\left\Vert \widetilde{k}\right\Vert _{L_{2}\left(P_{0}\right)}}-u_{1-\alpha}\right),
\]
wobei $\vartheta=\lim\limits_{n\rightarrow\infty}\sqrt{n}\left(k\left(P_{t_{n}}\right)-k\left(P_{0}\right)\right)$
ist.
\item Die Testfolge $\left(\varphi_{n}\right)_{n\in\mathbb{N}}$ ist asymptotisch
unverfälscht zum Niveau $\alpha$ für $\mathcal{H}_{3}$ gegen $\mathcal{K}_{3}.$
\end{enumerate}
\noindent\end{satz}\begin{proof}vgl. Satz \ref{asymp_gute} und
Satz \ref{asymp_unverf_testfolge} oder \cite{Janssen:1999a}, Theorem
3.1.\end{proof} \begin{defi} Eine Testfolge $\left(\psi_{n}\right)_{n\in\mathbb{N}}$
für $\mathcal{H}_{3}$ gegen $\mathcal{K}_{3}$ heißt asymptotisch
Niveau $\alpha$-ähnlich, falls für alle implizite Hypothesen $\left(P_{t_{n}}^{n}\right)_{n\in\mathbb{N}}\in\mathcal{H}_{3}$
mit
\[
\lim_{n\rightarrow\infty}\sqrt{n}\left(k\left(P_{t_{n}}\right)-k\left(P_{0}\right)\right)=0
\]
bereits 
\[
\lim_{n\rightarrow\infty}\int\psi_{n}\,dP_{t_{n}}^{n}=\alpha
\]
 gilt. \end{defi}\begin{satz}\label{asymp_opt_einseit_eistichprob_tang}
Es gelte 
\begin{equation}
K\left(P_{0},\mathcal{P}\right)=T\left(P_{0},\mathcal{P}\right).\label{RaumKegel}
\end{equation}
 Die Voraussetzung (\ref{nicht_0_vor}) sei erfüllt. Die Testfolge
$\left(\varphi_{n}\right)_{n\in\mathbb{N}}$ ist dann asymptotisch
optimal in der Menge aller Testfolgen für $\mathcal{H}_{3}$ gegen
$\mathcal{K}_{3}$, die asymptotisch Niveau $\alpha$-ähnlich sind.\end{satz}\begin{proof}Es
seien $\left(P_{t_{n}}^{n}\right)_{n\in\mathbb{N}}\in\mathcal{F}_{1}$
eine Folge von Wahrscheinlichkeitsmaßen und $t\mapsto P_{t}$ die
zugehörige $L_{2}\left(P_{0}\right)$-differenzierbare Kurve in $\mathcal{P}$
mit Tangente $h\in K\left(P_{0},\mathcal{P}\right).$ Wie in (\ref{diff_t_n_grad})
erhält man dann
\begin{eqnarray}
\vartheta=\lim_{n\rightarrow\infty}\sqrt{n}\left(k\left(P_{t_{n}}\right)-k\left(P_{0}\right)\right) & = & t\int h\widetilde{k}\,dP_{0}\label{eq:lim_funk}
\end{eqnarray}
 mit $t=\lim\limits_{n\rightarrow\infty}\left(n^{\frac{1}{2}}t_{n}\right).$
Es gilt außerdem $t>0.$ Die Testfolge $\left(\varphi_{n}\right)_{n\in\mathbb{N}}$
ist asymptotisch Niveau $\alpha$-ähnlich, denn aus Bedingung 
\[
\lim_{n\rightarrow\infty}\sqrt{n}\left(k\left(P_{t_{n}}\right)-k\left(P_{0}\right)\right)=0
\]
 folgt $\lim_{n\rightarrow\infty}\int\varphi_{n}\,dP_{t_{n}}^{n}=\alpha$
nach Satz \ref{ep_es_asymp_gute}. Die Folge $\left(P_{t_{n}}^{n}\right)_{n\in\mathbb{N}}$
von Wahrscheinlichkeitsmaßen ist ULAN nach Satz \ref{ulan_le_cam}.
Wegen 
\[
\lim\limits_{n\rightarrow\infty}\frac{n^{-\frac{1}{2}}t}{t_{n}}=1
\]
 folgt
\[
\lim_{n\rightarrow\infty}\left\Vert P_{t_{n}}^{n}-P_{n^{-\frac{1}{2}}t}^{n}\right\Vert =0
\]
 nach \cite{Janssen:1998}, Satz $14.17$. Man erhält somit 
\begin{eqnarray}
\lim_{n\rightarrow\infty}\int\psi_{n}\,dP_{t_{n}}^{n} & = & \lim_{n\rightarrow\infty}\int\psi_{n}\,dP_{n^{-\frac{1}{2}}t}^{n}\label{ulan-vereinfachung}
\end{eqnarray}
 für jede Testfolge $\left(\psi_{n}\right)_{n\in\mathbb{N}}$ nach
Hilfssatz \ref{hsatz1}, falls einer der Grenzwerte existiert. 

\noindent Die Abbildung $f:T\left(P_{0},\mathcal{P}\right)\rightarrow\mathbb{R},h\mapsto\int h\widetilde{k}\,dP_{0}$
ist linear und stetig. Es gilt 
\[
\left\Vert f\right\Vert =\left(\int\widetilde{k}^{2}\,dP_{0}\right)^{\frac{1}{2}}=\left\Vert \widetilde{k}\right\Vert _{L_{2}\left(P_{0}\right)}.
\]
 Aus (\ref{eq:lim_funk}) ergeben sich folgende zwei Äquivalenzen:
\[
\left(P_{t_{n}}^{n}\right)_{n\in\mathbb{N}}\in\mathcal{K}_{3}\Leftrightarrow\int h\widetilde{k}\,dP_{0}>0,
\]
 
\[
\left(P_{t_{n}}^{n}\right)_{n\in\mathbb{N}}\in\mathcal{H}_{3}\Leftrightarrow\int h\widetilde{k}\,dP_{0}\leq0.
\]
Sei $\left(\psi_{n}\right)_{n\in\mathbb{N}}$ eine asymptotisch Niveau
$\alpha$-ähnliche Testfolge für $\mathcal{H}_{3}$ gegen $\mathcal{K}_{3}.$
Wegen (\ref{ulan-vereinfachung}), (\ref{RaumKegel}) und Bemerkung
\ref{lineare_vereinfachung} erhält man dann, dass $\psi_{n}$ ein
Test für das einseitige Testproblem 
\[
H_{1}^{l}=\left\{ h\in T\left(P_{0},\mathcal{P}\right):f\left(h\right)\leq0\right\} \quad\mbox{gegen}\quad K_{1}^{l}=\left\{ h\in T\left(P_{0},\mathcal{P}\right):f\left(h\right)>0\right\} 
\]
für das Experiment $E_{n}=\left(\Omega^{n},\mathcal{A}^{n},\left\{ P_{n,h}:h\in T\left(P_{0},\mathcal{P}\right)\right\} \right)$
aus Satz \ref{tang_raum_lan} ist. Beachte, dass an dieser Stelle
die Voraussetzung (\ref{RaumKegel}) eine wichtige Rolle spielt. Die
Folge $\left(E_{n}\right)_{n\in N}$ der Experimente konvergiert schwach
gegen ein Gauß-Shift Experiment $E=\left(\widehat{\Omega},\widehat{\mathcal{A}},\left\{ Q_{h}:h\in T\left(P_{0},\mathcal{P}\right)\right\} \right)$
nach Satz \ref{tang_raum_lan}. Die Testfolge $\left(\psi_{n}\right)_{n\in\mathbb{N}}$
ist somit eine asymptotisch Niveau $\alpha$-ähnliche Testfolge für
das Testproblem $H_{1}^{l}$ gegen $K_{1}^{l}.$ Nach Satz \ref{asymp_einseitig}
ergibt sich dann
\[
\limsup_{n\rightarrow\infty}\int\psi_{n}\,dP_{n,h}\leq\Phi\left(\frac{\int h\widetilde{k}\,dP_{0}}{\left\Vert \widetilde{k}\right\Vert _{L_{2}\left(P_{0}\right)}}-u_{1-\alpha}\right)\,\,\mbox{für alle}\,\,h\in K_{1}^{l}
\]
und 
\[
\liminf_{n\rightarrow\infty}\int\psi_{n}\,dP_{n,h}\geq\Phi\left(\frac{\int h\widetilde{k}\,dP_{0}}{\left\Vert \widetilde{k}\right\Vert _{L_{2}\left(P_{0}\right)}}-u_{1-\alpha}\right)\,\,\mbox{für alle}\,\,h\in H_{1}^{l}.
\]
 Falls $\left(P_{t_{n}}^{n}\right)_{n\in\mathbb{N}}\in\mathcal{K}_{3}$
eine implizite Alternative ist, so erhält man 
\begin{eqnarray*}
\limsup_{n\rightarrow\infty}\int\psi_{n}\,dP_{t_{n}}^{n} & = & \limsup_{n\rightarrow\infty}\int\psi_{n}\,dP_{\frac{1}{\sqrt{n}}t}^{n}\\
 & = & \limsup_{n\rightarrow\infty}\int\psi_{n}\,dP_{n,th}\\
 & \leq & \Phi\left(\frac{t\int h\widetilde{k}\,dP_{0}}{\left\Vert \widetilde{k}\right\Vert _{L_{2}\left(P_{0}\right)}}-u_{1-\alpha}\right)\\
 & = & \lim_{n\rightarrow\infty}\int\varphi_{n}\,dP_{t_{n}}^{n}.
\end{eqnarray*}
Ist $\left(P_{t_{n}}^{n}\right)_{n\in\mathbb{N}}\in\mathcal{H}_{3}$
eine implizite Hypothese, so ergibt sich
\begin{eqnarray*}
\liminf_{n\rightarrow\infty}\int\psi_{n}\,dP_{t_{n}}^{n} & = & \liminf_{n\rightarrow\infty}\int\psi_{n}\,dP_{\frac{1}{\sqrt{n}}t}^{n}\\
 & = & \liminf_{n\rightarrow\infty}\int\psi_{n}\,dP_{n,th}\\
 & \geq & \Phi\left(\frac{t\int h\widetilde{k}\,dP_{0}}{\left\Vert \widetilde{k}\right\Vert _{L_{2}\left(P_{0}\right)}}-u_{1-\alpha}\right)\\
 & = & \lim_{n\rightarrow\infty}\int\varphi_{n}\,dP_{t_{n}}^{n}.
\end{eqnarray*}
Somit ist alles bewiesen.\end{proof}Die Voraussetzung $K\left(P_{0},\mathcal{P}\right)=T\left(P_{0},\mathcal{P}\right)$
von Satz \ref{asymp_opt_einseit_eistichprob_tang} kann noch abgeschwächt
werden. Der Vektorraum 
\[
N\left(\widetilde{k},P_{0},\mathcal{P}\right):=\left\{ h\in T\left(P_{0},\mathcal{P}\right):\int\widetilde{k}h\,dP_{0}=0\right\} 
\]
 ist das orthogonale Komplement des kanonischen Gradienten $\widetilde{k}$
in dem Tangentialraum $T\left(P_{0},\mathcal{P}\right).$ Der Vektorraum
$N\left(\widetilde{k},P_{0},\mathcal{P}\right)$ ist somit ein abgeschlossener
Unterraum des Hilbertraums $T\left(P_{0},\mathcal{P}\right).$ \begin{satz}\label{asymp_opt_einseit_eistichprob_null}
Die Voraussetzung (\ref{nicht_0_vor}) sei erfüllt. Gilt $N\left(\widetilde{k},P_{0},\mathcal{P}\right)\subset K\left(P_{0},\mathcal{P}\right),$
so ist $\left(\varphi_{n}\right)_{n\in\mathbb{N}}$ eine asymptotisch
optimale Testfolge in der Menge aller asymptotisch Niveau $\alpha$-ähnlichen
Testfolgen für $\mathcal{H}_{3}$ gegen $\mathcal{K}_{3}$.\end{satz}\begin{proof}Es
reicht zu zeigen, dass jede asymptotisch Niveau $\alpha$-ähnliche
Testfolge $\left(\psi_{n}\right)_{n\in\mathbb{N}}$ für $\mathcal{H}_{3}$
gegen $\mathcal{K}_{3}$ bereits eine asymptotisch Niveau $\alpha$-ähnliche
Testfolge für das Testproblem $H_{1}^{l}=\left\{ h\in T\left(P_{0},\mathcal{P}\right):\int h\widetilde{k}\,dP_{0}\leq0\right\} $
gegen $K_{1}^{l}=\left\{ h\in T\left(P_{0},\mathcal{P}\right):\int h\widetilde{k}\,dP_{0}>0\right\} $
ist. Die Behauptung folgt dann analog wie im Beweis von Satz \ref{asymp_opt_einseit_eistichprob_tang}.
Sei $h\in N\left(\widetilde{k},P_{0},\mathcal{P}\right)$ beliebig.
Wegen Voraussetzung $N\left(\widetilde{k},P_{0},\mathcal{P}\right)\subset K\left(P_{0},\mathcal{P}\right)$
existiert eine $L_{2}\left(P_{0}\right)$-differenzierbare Kurve $t\mapsto P_{t}$
in $\mathcal{P}$ mit der Tangente $h.$ Man erhält nun
\begin{eqnarray*}
\lim_{n\rightarrow\infty}\int\psi_{n}dP_{n,h} & = & \lim_{n\rightarrow\infty}\int\psi_{n}dP_{\frac{1}{\sqrt{n}}}^{n}=\alpha
\end{eqnarray*}
für jede asymptotisch Niveau $\alpha$-ähnliche Testfolge $\left(\psi_{n}\right)_{n\in\mathbb{N}}$
für $\mathcal{H}_{3}$ gegen $\mathcal{K}_{3}$. Somit ist alles bewiesen.\end{proof}Die
Voraussetzungen aus Satz \ref{asymp_opt_einseit_eistichprob_null}
lassen sich weiter abschwächen. Dafür braucht man ein Ergebnis aus
\cite{Vaart:1991}, welches in Satz \ref{existenz_limis_test} vorgestellt
wird.\begin{satz}\label{existenz_limis_test} Es seien $\Theta$ eine
beliebige nichtleere Menge und $E_{n}=\left(\Omega_{n},\mathcal{A}_{n},\left\{ P_{n,\vartheta}:\vartheta\in\Theta\right\} \right)$
eine Folge von Experimenten, die schwach gegen ein dominiertes Limesexperiment
$E=\left(\Omega,\mathcal{A},\left\{ P_{\vartheta}:\vartheta\in\Theta\right\} \right)$
konvergiert. Sei $\left(\phi_{n}\right)_{n\in\mathbb{N}}$ eine Folge
der Tests für die Experimente $\left(E_{n}\right)_{n\in\mathbb{N}}.$
Die Folge der Gütefunktionen 
\[
f_{n}:\Theta\rightarrow\left[0,1\right],\,\,\vartheta\mapsto\int\phi_{n}dP_{n,\vartheta}
\]
 konvergiere punktweise gegen eine Funktion $f:\Theta\rightarrow\left[0,1\right],$
d.h. für alle $\vartheta\in\Theta$ gelte
\[
\lim_{n\rightarrow\infty}\int\phi_{n}\,dP_{n,\vartheta}=f(\vartheta).
\]
Dann existiert ein Test $\phi$ für das Limesexperiment $E$ mit der
Gütefunktion $f,$ d.h. $\int\phi\,dP_{\vartheta}=f(\vartheta)$ gilt
für alle $\vartheta\in\Theta.$ \end{satz}\begin{proof}vgl. \cite{Vaart:1991},
Theorem 7.1.\end{proof} \begin{satz}\label{asymp_opt_einseit_eistichprob_null_app}Die
Voraussetzung (\ref{nicht_0_vor}) sei erfüllt. Es existiere eine
Teilmenge 
\[
D\subset K\left(P_{0},\mathcal{P}\right)\cap N\left(\widetilde{k},P_{0},\mathcal{P}\right),
\]
 die bzgl. der $L_{2}\left(P_{0}\right)$-Norm dicht in $N\left(\widetilde{k},P_{0},\mathcal{P}\right)$
liegt. Die Testfolge $\left(\varphi_{n}\right)_{n\in\mathbb{N}}$
ist dann asymptotisch optimal in der Menge aller asymptotisch Niveau
$\alpha$-ähnlichen Testfolgen für $\mathcal{H}_{3}$ gegen $\mathcal{K}_{3}$.\end{satz}\begin{proof}Es
reicht zu zeigen, dass jede asymptotisch Niveau $\alpha$-ähnliche
Testfolge $\left(\psi_{n}\right)_{n\in\mathbb{N}}$ für $\mathcal{H}_{3}$
gegen $\mathcal{K}_{3}$ bereits eine asymptotisch Niveau $\alpha$-ähnliche
Testfolge für das Testproblem $H_{1}^{l}$ gegen $K_{1}^{l}$ ist.
Es muss also gezeigt werden, dass 
\begin{equation}
\lim_{n\rightarrow\infty}\int\psi_{n}\,dP_{n,h}=\alpha\label{eq:alpha}
\end{equation}
 für alle $h\in N\left(\widetilde{k},P_{0},\mathcal{P}\right)$ gilt.
Die Bedingung (\ref{eq:alpha}) ist für alle $h\in D$ nach Voraussetzung
erfüllt. 

\noindent Sei $g\in N\left(\widetilde{k},P_{0},\mathcal{P}\right)\setminus D$
beliebig aber fest gewählt. Nach Voraussetzung existiert eine konvergente
Folge $\left(g_{m}\right)_{m\in\mathbb{N}}\subset D$ mit 
\[
\lim\limits_{m\rightarrow\infty}\left\Vert g-g_{m}\right\Vert _{L_{2}\left(P_{0}\right)}=0.
\]
 Die Folge der Experimente 
\[
E_{n}=\left(\Omega^{n},\mathcal{A}^{n},\left\{ P_{n,h}:h\in T\left(P_{0},\mathcal{P}\right)\right\} \right),\,\,n\in\mathbb{N}
\]
 konvergiert nach Satz \ref{tang_raum_lan} schwach gegen ein Gauß-Shift
Experiment 
\[
E=\left(\widehat{\Omega},\widehat{\mathcal{A}},\left\{ P_{h}:h\in T\left(P_{0},\mathcal{P}\right)\right\} \right).
\]
Für das Limesexperiment $E$ ergibt sich dann 
\begin{equation}
\lim\limits_{m\rightarrow\infty}\left\Vert P_{g_{m}}-P_{g}\right\Vert =0\label{lim_exp_app}
\end{equation}
nach Korollar \ref{topologie_korollar}. Im Folgenden zieht man sich
auf eine geeignete Folge der Teilexperimente zurück. Sei $\left(\psi_{n_{i}}\right)_{i\in\mathbb{N}}$
eine Teilfolge mit 
\[
\lim_{i\rightarrow\infty}\int\psi_{n_{i}}\,dP_{n_{i},g}=\limsup_{n\rightarrow\infty}\int\psi_{n}\,dP_{n,g}.
\]
 Sei $\Theta^{g}=\left\{ g_{m}:m\in\mathbb{N}\right\} \cup\left\{ g\right\} $.
Für alle $i\in\mathbb{N}$ definiere 
\[
E_{i}^{g,\sup}=\left(\Omega^{n_{i}},\mathcal{A}^{n_{i}},\left\{ P_{n_{i},h}:h\in\Theta^{g}\right\} \right).
\]
 Die Folge $\left(E_{i}^{g,\sup}\right)_{i\in\mathbb{N}}$ der Experimente
konvergiert dann schwach gegen das Limesexperiment
\[
E^{g,\sup}=\left(\widehat{\Omega},\widehat{\mathcal{A}},\left\{ P_{h}:h\in\Theta^{g}\right\} \right).
\]
Die Familie $\left\{ P_{h}:h\in\Theta^{g}\right\} $ der Wahrscheinlichkeitsmaße
ist dominiert. Die Folge 
\[
f_{i}^{g,\sup}:\Theta^{g}\rightarrow\left[0,1\right],\,h\mapsto\int\psi_{n_{i}}\,dP_{n_{i},h}
\]
 der Gütefunktionen konvergiert punktweise gegen eine Funktion $f^{g,\sup}:\Theta^{g}\rightarrow\left[0,1\right]$
wegen $\lim_{i\rightarrow\infty}\int\psi_{n_{i}}\,dP_{n_{i},g_{m}}=\lim_{n\rightarrow\infty}\int\psi_{n}\,dP_{n,g_{m}}=\alpha$
für alle $m\in\mathbb{N}$ und $\lim_{i\rightarrow\infty}\int\psi_{n_{i}}\,dP_{n_{i},g}=\limsup_{n\rightarrow\infty}\int\psi_{n}\,dP_{n,g}.$

\noindent Nach Satz \ref{existenz_limis_test} existiert ein Test
$\psi^{\sup}$ für das Limesexperiment $E^{g,\sup}$ mit der Gütefunktion
\[
\int\psi^{\sup}\,dP_{h}=f^{g,\sup}\left(h\right)
\]
 für alle $h\in\Theta^{g}$. Für jedes $m\in\mathbb{N}$ ergibt sich
somit
\[
\int\psi^{\sup}\,dP_{g_{m}}=f^{g,\sup}\left(g_{m}\right)=\lim_{i\rightarrow\infty}\int\psi_{n_{i}}\,dP_{n_{i},g_{m}}=\alpha.
\]

\noindent Wegen (\ref{lim_exp_app}) erhält man nun 
\[
\int\psi^{\sup}\,dP_{g}=\lim_{m\rightarrow\infty}\int\psi^{\sup}\,dP_{g_{m}}=\alpha
\]
mit Hilfe von Lemma \ref{hsatz1}. Hieraus folgt 
\begin{equation}
\limsup_{n\rightarrow\infty}\int\psi_{n}\,dP_{n,g}=\lim_{i\rightarrow\infty}\int\psi_{n_{i}}\,dP_{n_{i},g}=f^{g,\sup}\left(g\right)=\int\psi^{\sup}\,dP_{g}=\alpha.\label{lim_sup_alpha}
\end{equation}
 Sei nun $\left(\psi_{n_{j}}\right)_{j\in\mathbb{N}}$ eine Teilfolge
mit 
\[
\lim_{j\rightarrow\infty}\int\psi_{n_{j}}\,dP_{n_{j},g}=\liminf_{n\rightarrow\infty}\int\psi_{n}\,dP_{n,g}.
\]
Für jedes $j\in\mathbb{N}$ definiere 
\[
E_{j}^{g,\inf}=\left(\Omega^{n_{j}},\mathcal{A}^{n_{j}},\left\{ P_{n_{j},h}:h\in\Theta^{g}\right\} \right).
\]
Die Folge $\left(E_{j}^{g,\inf}\right)_{j\in\mathbb{N}}$ der Experimente
konvergiert schwach gegen das dominierte Limesexperiment $E^{g,\inf}=\left(\widehat{\Omega},\widehat{\mathcal{A}},\left\{ P_{h}:h\in\Theta^{g}\right\} \right).$
Die Folge 
\[
f_{j}^{g,\inf}:\Theta^{g}\rightarrow\left[0,1\right],\,h\mapsto\int\psi_{n_{j}}\,dP_{n_{j},h}
\]
 der Gütefunktionen konvergiert punktweise gegen eine Funktion $f^{g,\inf}:\Theta^{g}\rightarrow\left[0,1\right].$
Nach Satz \ref{existenz_limis_test} existiert ein Test $\psi^{\inf}$
für das Limesexperiment $E^{g,\inf}$ mit der Gütefunktion $f^{g,\inf}.$
Für alle $m\in\mathbb{N}$ gilt dann 
\[
\int\psi^{\inf}dP_{g_{m}}=\lim_{j\rightarrow\infty}\int\psi_{n_{j}}dP_{n_{j},g_{m}}=\alpha.
\]
 Wegen (\ref{lim_exp_app}) und Lemma \ref{hsatz1} ergibt sich
\[
\int\psi^{\inf}\,dP_{g}=\lim_{m\rightarrow\infty}\int\psi^{\inf}\,dP_{g_{m}}=\alpha.
\]
 Man erhält somit 
\begin{equation}
\liminf_{n\rightarrow\infty}\int\psi_{n}\,dP_{n,g}=\lim_{j\rightarrow\infty}\int\psi_{n_{j}}\,dP_{n_{j},g}=\int\psi^{\inf}\,dP_{g}=\alpha.\label{lim_inf_alpha}
\end{equation}
Aus (\ref{lim_sup_alpha}) und (\ref{lim_inf_alpha}) folgt
\[
\lim_{n\rightarrow\infty}\int\psi_{n}\,dP_{n,g}=\limsup_{n\rightarrow\infty}\int\psi_{n}\,dP_{n,g}=\liminf_{n\rightarrow\infty}\int\psi_{n}\,dP_{n,g}=\alpha.
\]
Somit ist alles bewiesen.\end{proof}\begin{bem}Ist der Tangentenkegel
$K\left(P_{0},\mathcal{P}\right)$ ein Vektorraum, so ist die Voraussetzung
der Existenz einer Menge $D\subset K\left(P_{0},\mathcal{P}\right)\cap N\left(\widetilde{k},P_{0},\mathcal{P}\right),$
die bzgl. der $L_{2}\left(P_{0}\right)$-Norm dicht in $N\left(\widetilde{k},P_{0},\mathcal{P}\right)$
liegt, bereits erfüllt. Zum Beweis vergleiche Bemerkung \ref{dichte_bemerkung}.\end{bem}

\subsection{Asymptotisch optimale zweiseitige Tests\label{opt:zw:test:einstichprobe}}

Zur Behandlung der zweiseitigen Testprobleme

\begin{equation}
H_{4}=\left\{ P\in\mathcal{P}:k\left(P\right)=a\right\} \quad\mbox{gegen}\quad K_{4}=\left\{ P\in\mathcal{P}:k\left(P\right)\neq a\right\} \label{tp_4}
\end{equation}
wird zuerst eine geeignete Parametrisierung vorgestellt. Analog zu
den einseitigen Testproblemen wird die nachfolgende Definition begründet:\begin{defi}\label{K_4_H_4}Die
Menge $\mathcal{K}_{4}$ der impliziten Alternativen für das Testproblem
(\ref{tp_4}) ist definiert durch 
\[
\mathcal{K}_{4}:=\left\{ \left(P_{t_{n}}^{n}\right)_{n\in\mathbb{N}}\in\mathcal{F}_{1}:\lim_{n\rightarrow\infty}\sqrt{n}\left(k\left(P_{t_{n}}\right)-k\left(P_{0}\right)\right)\neq0\right\} .
\]

\noindent Die Menge $\mathcal{H}_{4}$ der impliziten Hypothesen für
das Testproblem (\ref{tp_4}) ist gegeben durch
\[
\mathcal{H}_{4}:=\left\{ \left(P_{t_{n}}^{n}\right)_{n\in\mathbb{N}}\in\mathcal{F}_{1}:\lim_{n\rightarrow\infty}\sqrt{n}\left(k\left(P_{t_{n}}\right)-k\left(P_{0}\right)\right)=0\right\} .
\]

\end{defi}\begin{bem}Definition \ref{K_4_H_4} lässt sich analog
wie Definition \ref{impl_alternativen_einstichprobe} verallgemeinern.
Allerdings treten die gleichen Schwierigkeiten auf und der Vorteil
der Verallgemeinerung ist aus der Sicht der nichtparametrischen Statistik
gering, vgl. Bemerkung \ref{bem_verallg_def_impl_alternat}.\end{bem}Für
das zweiseitige Testproblem (\ref{tp_4}) sind die Tests
\[
\varphi_{n}=\left\{ \begin{array}{cccc}
1 &  & >\\
 & \left|\frac{1}{\sqrt{n}}\sum\limits_{i=1}^{n}\widetilde{k}\circ\pi_{i}\right| &  & u_{1-\frac{\alpha}{2}}\\
0 &  & \leq
\end{array}\left\Vert \widetilde{k}\right\Vert _{L_{2}\left(P_{0}\right)}\right.
\]
geeignet. Zunächst wird die asymptotische Gütefunktion der Testfolge
$\left(\varphi_{n}\right)_{n\in\mathbb{N}}$ berechnet. Dann werden
die Optimalitätseigenschaften dieser Testfolge unter verschiedenen
Voraussetzungen an den Tangentenkegel $K\left(P_{0},\mathcal{P}\right)$
untersucht.\begin{satz}\label{alpha_trivial}Sei $t\mapsto P_{t}$
eine $L_{2}\left(P_{0}\right)$-differenzierbare Kurve in $\mathcal{P}$
mit Tangente $g=0\in L_{2}\left(P_{0}\right)$. Außerdem sei $\left(\phi_{n}\right)_{n\in\mathbb{N}}$
eine Testfolge mit $\lim\limits_{n\rightarrow\infty}\int\phi_{n}\,dP_{0}^{n}=\alpha.$
Für jede Nullfolge $\left(t_{n}\right)_{n\in\mathbb{N}}$ mit $\lim_{n\rightarrow\infty}t_{n}\sqrt{n}>0$
gilt dann $\lim\limits_{n\rightarrow\infty}\int\phi_{n}\,dP_{t_{n}}^{n}=\alpha.$
\end{satz}\begin{proof} Nach \cite{Strasser:1985b}, Seite 386,
Theorem 75.8 erhält man 
\[
\lim_{n\rightarrow\infty}\left\Vert P_{t_{n}}^{n}-P_{0}^{n}\right\Vert =0.
\]
 Nach Hilfssatz \ref{hsatz1} gilt dann

\[
\lim\limits_{n\rightarrow\infty}\int\phi_{n}\,dP_{t_{n}}^{n}=\lim\limits_{n\rightarrow\infty}\int\phi_{n}\,dP_{0}^{n}=\alpha.
\]
 \end{proof}\begin{satz}\label{zw_asymp_gutefunktion}Die Voraussetzung
(\ref{nicht_0_vor}) sei erfüllt. Für jede Folge $\left(P_{t_{n}}^{n}\right)_{n\in\mathbb{N}}\in\mathcal{F}_{1}$
gilt dann
\begin{equation}
\lim_{n\rightarrow\infty}\int\varphi_{n}\,dP_{t_{n}}^{n}=\Phi\left(\vartheta\left\Vert \widetilde{k}\right\Vert _{L_{2}\left(P_{0}\right)}^{-1}-u_{1-\frac{\alpha}{2}}\right)+\Phi\left(-\vartheta\left\Vert \widetilde{k}\right\Vert _{L_{2}\left(P_{0}\right)}^{-1}-u_{1-\frac{\alpha}{2}}\right),\label{gf_zt_es}
\end{equation}
wobei $\vartheta=\lim\limits_{n\rightarrow\infty}\sqrt{n}\left(k\left(P_{t_{n}}\right)-k\left(P_{0}\right)\right)$
ist.\end{satz} \begin{proof}Es wird zunächst gezeigt, dass die Bedingung
(\ref{asymp_niveau_alpha}) erfüllt ist. Der Zentrale Grenzwertsatz
impliziert die schwache Konvergenz 
\[
\mathcal{L}\left(\left.\frac{1}{\sqrt{n}}\sum\limits_{i=1}^{n}\widetilde{k}\circ\pi_{i}\right|P_{0}^{n}\right)\rightarrow N\left(0,\left\Vert \widetilde{k}\right\Vert _{L_{2}\left(P_{0}\right)}^{2}\right)
\]
für $n\rightarrow\infty.$ Man erhält somit 
\begin{eqnarray*}
\lim_{n\rightarrow\infty}\int\varphi_{n}\,dP_{0}^{n} & = & \lim_{n\rightarrow\infty}P_{0}^{n}\left(\frac{1}{\sqrt{n}}\sum\limits_{i=1}^{n}\widetilde{k}\circ\pi_{i}>u_{1-\frac{\alpha}{2}}\left\Vert \widetilde{k}\right\Vert _{L_{2}\left(P_{0}\right)}\right)\\
 &  & +\lim_{n\rightarrow\infty}P_{0}^{n}\left(\frac{1}{\sqrt{n}}\sum\limits_{i=1}^{n}\widetilde{k}\circ\pi_{i}<-u_{1-\frac{\alpha}{2}}\left\Vert \widetilde{k}\right\Vert _{L_{2}\left(P_{0}\right)}\right)\\
 & = & 1-\Phi\left(u_{1-\frac{\alpha}{2}}\right)+\Phi\left(-u_{1-\frac{\alpha}{2}}\right)\\
 & = & \alpha.
\end{eqnarray*}
 Es seien $\left(P_{t_{n}}^{n}\right)_{n\in\mathbb{N}}\in\mathcal{F}_{1}$
eine beliebige aber fest gewählte Folge und $g\in L_{2}^{\left(0\right)}\left(P_{0}\right)$
die zugehörige Tangente. Angenommen, es gilt $\int g^{2}\,dP_{0}=0.$
Nach Satz \ref{alpha_trivial} ergibt sich dann 
\[
\lim_{n\rightarrow\infty}\int\varphi_{n}\,dP_{t_{n}}^{n}=\lim_{n\rightarrow\infty}\int\varphi_{n}\,dP_{0}^{n}=\alpha.
\]
Mit (\ref{eq:lim_funk}) erhält man außerdem 
\[
\vartheta=\lim\limits_{n\rightarrow\infty}\sqrt{n}\left(k\left(P_{t_{n}}\right)-k\left(P_{0}\right)\right)=\int\widetilde{k}g\,dP_{0}\lim_{n\rightarrow\infty}n^{\frac{1}{2}}t_{n}=0.
\]
 Hieraus folgt die Gültigkeit von (\ref{gf_zt_es}), falls $\int g^{2}\,dP_{0}=0$
ist. \\
Sei nun $\int g^{2}\,dP_{0}\neq0.$ Die Folge $\left(P_{t_{n}}^{n}\right)_{n\in\mathbb{N}}$
der Wahrscheinlichkeitsmaße ist dann LAN mit der zentralen Folge 
\[
X_{n}=\sum_{i=1}^{n}t_{n}g\circ\pi_{i}.
\]
Man erhält zunächst 
\[
\sigma_{2}^{2}=\lim_{n\rightarrow\infty}\mbox{Var}_{P_{0}^{n}}\left(X_{n}\right)=\left\Vert g\right\Vert _{L_{2}\left(P_{0}\right)}^{2}\lim_{n\rightarrow\infty}nt_{n}^{2}=t^{2}\left\Vert g\right\Vert _{L_{2}\left(P_{0}\right)}^{2},
\]
 wobei $t:=\lim_{n\rightarrow\infty}n^{\frac{1}{2}}t_{n}$ ist. Mit
Hilfe des Cram\'er-Wold-Device und des Zentralen Grenzwertsatzes
ergibt sich
\[
\mathcal{L}\left(\left.\left(\frac{1}{\sqrt{n}}\sum\limits_{i=1}^{n}\widetilde{k}\circ\pi_{i},X_{n}\right)^{t}\right|P_{0}^{n}\right)\rightarrow N\left(\left(\begin{array}{c}
0\\
0
\end{array}\right),\left(\begin{array}{cc}
\sigma_{1}^{2} & \sigma_{12}\\
\sigma_{12} & \sigma_{2}^{2}
\end{array}\right)\right)
\]
 für $n\rightarrow\infty$ mit $\sigma_{1}=\left\Vert \widetilde{k}\right\Vert _{L_{2}\left(P_{0}\right)}.$
Mit (\ref{eq:lim_funk}) erhält man 
\begin{eqnarray*}
\sigma_{12} & = & \lim_{n\rightarrow\infty}\mbox{Cov}_{P_{0}^{n}}\left(\frac{1}{\sqrt{n}}\sum\limits_{i=1}^{n}\widetilde{k}\circ\pi_{i},\sum_{i=1}^{n}t_{n}g\circ\pi_{i}\right)\\
 & = & \mbox{Cov}_{P_{0}}\left(\widetilde{k},g\right)\lim_{n\rightarrow\infty}n^{\frac{1}{2}}t_{n}\\
 & = & t\int\widetilde{k}g\,dP_{0}\\
 & = & \vartheta.
\end{eqnarray*}
 Nach Satz \ref{asymp-guete} folgt nun
\begin{eqnarray*}
\lim_{n\rightarrow\infty}\int\varphi_{n}\,dP_{t_{n}}^{n} & = & \Phi\left(\frac{\sigma_{12}}{\sigma_{1}}-u_{1-\frac{\alpha}{2}}\right)+\Phi\left(-\frac{\sigma_{12}}{\sigma_{1}}-u_{1-\frac{\alpha}{2}}\right)\\
 & = & \left(\vartheta\left\Vert \widetilde{k}\right\Vert _{L_{2}\left(P_{0}\right)}^{-1}-u_{1-\frac{\alpha}{2}}\right)+\Phi\left(-\vartheta\left\Vert \widetilde{k}\right\Vert _{L_{2}\left(P_{0}\right)}^{-1}-u_{1-\frac{\alpha}{2}}\right).
\end{eqnarray*}
 \end{proof}\begin{korollar}Die Voraussetzung (\ref{nicht_0_vor})
sei erfüllt. Die Testfolge $\left(\varphi_{n}\right)_{n\in\mathbb{N}}$
ist dann asymptotisch unverfälscht zum Niveau $\alpha$ für das Testproblem
$\mathcal{H}_{4}$ gegen $\mathcal{K}_{4}$.\end{korollar}\begin{proof}Die
Aussage folgt aus Satz \ref{zw_asymp_gutefunktion} mit Hilfe von
(\ref{eq:lim_funk}). \end{proof}\begin{satz}\label{zwei_seitige_opt}Die
Voraussetzung (\ref{nicht_0_vor}) sei erfüllt. Es gelte $K\left(P_{0},\mathcal{P}\right)=T\left(P_{0},\mathcal{P}\right)$.
Die Testfolge $\left(\varphi_{n}\right)_{n\in\mathbb{N}}$ ist dann
asymptotisch optimal in der Menge aller Testfolgen für das Testproblem
$\mathcal{H}_{4}$ gegen $\mathcal{K}_{4}$, die zum Niveau $\alpha$
asymptotisch unverfälscht sind. \end{satz}\begin{proof}Betrachte
die lineare Funktion $f:T\left(P_{0},\mathcal{P}\right)\rightarrow\mathbb{R},h\mapsto\int\widetilde{k}h\,dP_{0}$.
Analog zum Beweis von Satz \ref{asymp_opt_einseit_eistichprob_tang}
zeigt man, dass jede asymptotisch unverfälschte Niveau $\alpha$ Testfolge
$\left(\psi_{n}\right)_{n\in\mathbb{N}}$ für das Testproblem $\mathcal{H}_{4}$
gegen $\mathcal{K}_{4}$ bereits eine asymptotisch unverfälschte Niveau
$\alpha$ Testfolge für das Testproblem $H_{2}^{l}=\left\{ h\in T\left(P_{0},\mathcal{P}\right):f(h)=0\right\} $
gegen $K_{2}^{l}=\left\{ h\in T\left(P_{0},\mathcal{P}\right):f(h)\neq0\right\} $
ist. Nach Satz \ref{asymp_zweiseitig} ergibt sich 
\begin{eqnarray*}
 &  & \limsup_{n\rightarrow\infty}\int\psi_{n}\,dP_{n,h}\\
 & \leq & \Phi\left(\frac{f\left(h\right)}{\left\Vert f\right\Vert }-u_{1-\frac{\alpha}{2}}\right)+\Phi\left(-\frac{f\left(h\right)}{\left\Vert f\right\Vert }-u_{1-\frac{\alpha}{2}}\right)\\
 & = & \Phi\left(\int\widetilde{k}h\,dP_{0}\left\Vert \widetilde{k}\right\Vert _{L_{2}\left(P_{0}\right)}^{-1}-u_{1-\frac{\alpha}{2}}\right)+\Phi\left(-\int\widetilde{k}h\,dP_{0}\left\Vert \widetilde{k}\right\Vert _{L_{2}\left(P_{0}\right)}^{-1}-u_{1-\frac{\alpha}{2}}\right)
\end{eqnarray*}
für alle $h\in T\left(P_{0},\mathcal{P}\right)$. Es seien $\left(P_{t_{n}}^{n}\right)_{n\in\mathbb{N}}\in\mathcal{F}_{1}$
eine Folge von Wahrscheinlichkeitsmaßen und $t\mapsto P_{t}$ die
zugehörige $L_{2}\left(P_{0}\right)$-differenzierbare Kurve in $\mathcal{P}$
mit Tangente $h\in L_{2}^{(0)}\left(P_{0}\right).$ Mit (\ref{ulan-vereinfachung})
und (\ref{eq:lim_funk}) erhält man nun
\begin{eqnarray*}
\textrm{} &  & \limsup_{n\rightarrow\infty}\int\psi_{n}\,dP_{t_{n}}^{n}\\
 & = & \limsup_{n\rightarrow\infty}\int\psi_{n}\,dP_{\frac{t}{\sqrt{n}}}^{n}\\
 & = & \limsup_{n\rightarrow\infty}\int\psi_{n}\,dP_{n,th}\\
 & \leq & \Phi\left(t\int\widetilde{k}h\,dP_{0}\left\Vert \widetilde{k}\right\Vert _{L_{2}\left(P_{0}\right)}^{-1}-u_{1-\frac{\alpha}{2}}\right)+\Phi\left(-t\int\widetilde{k}h\,dP_{0}\left\Vert \widetilde{k}\right\Vert _{L_{2}\left(P_{0}\right)}^{-1}-u_{1-\frac{\alpha}{2}}\right)
\end{eqnarray*}
\begin{eqnarray*}
 & = & \Phi\left(\vartheta\left\Vert \widetilde{k}\right\Vert _{L_{2}\left(P_{0}\right)}^{-1}-u_{1-\frac{\alpha}{2}}\right)+\Phi\left(-\vartheta\left\Vert \widetilde{k}\right\Vert _{L_{2}\left(P_{0}\right)}^{-1}-u_{1-\frac{\alpha}{2}}\right)\\
 & = & \lim_{n\rightarrow\infty}\int\varphi_{n}\,dP_{t_{n}}^{n},
\end{eqnarray*}
wobei $t=\lim\limits_{n\rightarrow\infty}n^{\frac{1}{2}}t_{n}$ und
$\vartheta=\lim\limits_{n\rightarrow\infty}\sqrt{n}\left(k\left(P_{t_{n}}\right)-k\left(P_{0}\right)\right)$
sind. Somit ist alles bewiesen.\end{proof}Die Voraussetzung $K\left(P_{0},\mathcal{P}\right)=T\left(P_{0},\mathcal{P}\right)$
aus Satz \ref{zwei_seitige_opt} kann noch abgeschwächt werden.\begin{satz}\label{zw_zw_adv_asymp_opt}Die
Voraussetzung (\ref{nicht_0_vor}) sei erfüllt. Der Tangentenkegel
$K\left(P_{0},\mathcal{P}\right)$ liege bzgl. der $L_{2}\left(P_{0}\right)$-Norm
dicht in $T\left(P_{0},\mathcal{P}\right).$ Außerdem sei vorausgesetzt,
dass die Menge $K\left(P_{0},\mathcal{P}\right)\cap N\left(\widetilde{k},P_{0},\mathcal{P}\right)$
bzgl. der $L_{2}\left(P_{0}\right)$-Norm dicht in $N\left(\widetilde{k},P_{0},\mathcal{P}\right)$
liegt. Die Testfolge $\left(\varphi_{n}\right)_{n\in\mathbb{N}}$
ist dann asymptotisch optimal in der Menge aller Testfolgen für das
Testproblem $\mathcal{H}_{4}$ gegen $\mathcal{K}_{4}$, die asymptotisch
unverfälscht zum Niveau $\alpha$ sind.\end{satz}\begin{proof}Sei
$\left(\psi_{n}\right)_{n\in\mathbb{N}}$ eine beliebige asymptotisch
unverfälschte Testfolge zum Niveau $\alpha$ für das Testproblem $\mathcal{H}_{4}$
gegen $\mathcal{K}_{4}$. Wie im Beweis von Satz \ref{zwei_seitige_opt}
reicht es zu zeigen, dass die Testfolge $\left(\psi_{n}\right)_{n\in\mathbb{N}}$
dann bereits asymptotisch unverfälscht zum Niveau $\alpha$ für das
Testproblem $H_{2}^{l}$ gegen $K_{2}^{l}$ ist. Die Behauptung folgt
dann analog wie in Satz \ref{zwei_seitige_opt}.

An dieser Stelle sei daran erinnert, dass die Folge der Experimente
$E_{n}=\left(\Omega^{n},\mathcal{A}^{n},\left\{ P_{n,h}:h\in T\left(P_{0},\mathcal{P}\right)\right\} \right)$
nach Satz \ref{tang_raum_lan} schwach gegen ein Gauß-Shift Experiment
$E=\left(\widetilde{\Omega},\widetilde{\mathcal{A}},\left\{ P_{h}:h\in T\left(P_{0},\mathcal{P}\right)\right\} \right)$
konvergiert. Sei $t\mapsto P_{t}$ eine $L_{2}\left(P_{0}\right)$-differenzierbare
Kurve in $\mathcal{P}$ mit Tangente $g\in K\left(P_{0},\mathcal{P}\right)$.
Nach Voraussetzung gilt dann
\begin{eqnarray}
\liminf_{n\rightarrow\infty}\int\psi_{n}dP_{n,g} & = & \liminf_{n\rightarrow\infty}\int\psi_{n}dP_{\frac{1}{\sqrt{n}}}^{n}\geq\alpha\label{liminf_geq_alpha}
\end{eqnarray}
 für alle $g\in K\left(P_{0},\mathcal{P}\right)\cap K_{2}^{l}$ und
\begin{eqnarray}
\limsup_{n\rightarrow\infty}\int\psi_{n}dP_{n,g} & = & \limsup_{n\rightarrow\infty}\int\psi_{n}dP_{\frac{1}{\sqrt{n}}}^{n}\leq\alpha\label{limsup_leq_alpha}
\end{eqnarray}
 für alle $g\in K\left(P_{0},\mathcal{P}\right)\cap H_{2}^{l}$, weil
die Testfolge $\left(\psi_{n}\right)_{n\in\mathbb{N}}$ asymptotisch
unverfälscht zum Niveau $\alpha$ für das Testproblem $\mathcal{H}_{4}$
gegen $\mathcal{K}_{4}$ ist.

Sei nun $h\in K_{2}^{l}\setminus K\left(P_{0},\mathcal{P}\right)$
beliebig. Nach Voraussetzung existiert dann eine Folge $\left(h_{m}\right)_{m\in\mathbb{N}}\subset K\left(P_{0},\mathcal{P}\right)$
mit 
\[
\lim_{m\rightarrow\infty}\left\Vert h-h_{m}\right\Vert _{L_{2}\left(P_{0}\right)}=0.
\]
Wegen $\int h\widetilde{k}\,dP_{0}\neq0$ existiert ein $m_{0}\in\mathbb{N},$
so dass $\int h_{m}\widetilde{k}\,dP_{0}\neq0$ für alle $m\geq m_{0}$
gilt. Deswegen wird es ohne Einschränkung angenommen, dass $\int h_{m}\widetilde{k}\,dP_{0}\neq0$
für alle $m\in\mathbb{N}$ gilt. Diese Annahme ist äquivalent zu $\left(h_{m}\right)_{m\in\mathbb{N}}\subset K\left(P_{0},\mathcal{P}\right)\cap K_{2}^{l}.$
Für das Gauß-Shift Experiment $E$ ergibt sich 
\[
\lim_{m\rightarrow\infty}\left\Vert P_{h}-P_{h_{m}}\right\Vert =0
\]
nach Korollar \ref{topologie_korollar}. Sei nun $\varepsilon>0$
beliebig aber fest gewählt. Dann existiert ein $k\in\mathbb{N}$ mit
\[
\left\Vert P_{h}-P_{h_{k}}\right\Vert \leq\frac{\varepsilon}{2}.
\]
Sei $\left(\psi_{n_{j}}\right)_{j\in\mathbb{N}}$ eine Teilfolge,
so dass 
\begin{equation}
\liminf_{n\rightarrow\infty}\int\psi_{n}dP_{n,h}=\lim_{j\rightarrow\infty}\int\psi_{n_{j}}dP_{n_{j},h}\label{itm1}
\end{equation}
gilt und der Grenzwert 
\begin{equation}
\lim_{j\rightarrow\infty}\int\psi_{n_{j}}dP_{n_{j},h_{k}}\label{itm2}
\end{equation}
existiert. Die Existenz einer Teilfolge mit (\ref{itm1}) und (\ref{itm2})
ist folgendermaßen gesichert. Wähle zunächst eine Teilfolge $\left(\psi_{n}\right)_{n\in N},N\subset\mathbb{N}$
von $\left(\psi_{n}\right)_{n\in\mathbb{N}},$ so dass die Bedingung
(\ref{itm1}) erfüllt ist. Danach wähle eine weitere Teilfolge $\left(\psi_{n_{j}}\right)_{j\in\mathbb{N}}$
der Folge $\left(\psi_{n}\right)_{n\in N}$, so dass der Grenzwert
(\ref{itm2}) existiert. Die erste Bedingung (\ref{itm1}) bleibt
dabei erfüllt, weil alle Teilfolgen der Folge $\left(\psi_{n}\right)_{n\in N}$
die Bedingung (\ref{itm1}) ebenfalls erfüllen.

Die Folge $E_{j}^{h,\varepsilon}:=\left(\Omega^{n_{j}},\mathcal{A}^{n_{j}},\left\{ P_{n_{j},h},\,P_{n_{j},h_{k}}\right\} \right)$
der Experimente konvergiert schwach gegen das Limesexperiment $E^{h,\varepsilon}=\left(\widetilde{\Omega},\widetilde{\mathcal{A}},\left\{ P_{h},\,P_{h_{k}}\right\} \right).$
Nach Satz \ref{existenz_limis_test} existiert dann ein Test $\psi$
für das Experiment $E^{h,\varepsilon}$, so dass 
\[
\lim_{j\rightarrow\infty}\int\psi_{n_{j}}dP_{n_{j},h}=\int\psi\,dP_{h}
\]
und
\[
\lim_{j\rightarrow\infty}\int\psi_{n_{j}}dP_{n_{j},h_{k}}=\int\psi\,dP_{h_{k}}
\]
 gelten. Einerseits erhält man 
\begin{equation}
\int\psi\,dP_{h_{k}}=\lim_{j\rightarrow\infty}\int\psi_{n_{j}}dP_{n_{j},h_{k}}\geq\liminf_{n\rightarrow\infty}\int\psi_{n}dP_{n,h_{k}}\geq\alpha,\label{vor_eps_ungl}
\end{equation}
andererseits gilt
\begin{equation}
\left|\int\psi\,dP_{h}-\int\psi\,dP_{h_{k}}\right|\leq2\left\Vert P_{h}-P_{h_{k}}\right\Vert \leq\varepsilon.\label{epsilon_ungl}
\end{equation}
Die Ungleichung (\ref{epsilon_ungl}) wird nun genau analysiert. Es
sei zuerst angenommen, dass $\int\psi\,dP_{h}\geq\int\psi\,dP_{h_{k}}$
gilt. Man erhält dann unmittelbar 
\[
\liminf_{n\rightarrow\infty}\int\psi_{n}dP_{n,h}=\lim_{j\rightarrow\infty}\int\psi_{n_{j}}dP_{n_{j},h}=\int\psi\,dP_{h}\geq\int\psi\,dP_{h_{k}}\geq\alpha.
\]
 Es gelte nun $\int\psi\,dP_{h_{k}}>\int\psi\,dP_{h}.$ Aus Ungleichung
(\ref{epsilon_ungl}) ergibt sich dann
\begin{eqnarray*}
\textrm{} &  & \int\psi\,dP_{h_{k}}-\int\psi\,dP_{h}\leq\varepsilon\\
 & \Leftrightarrow & \int\psi\,dP_{h_{k}}\leq\int\psi\,dP_{h}+\varepsilon.
\end{eqnarray*}
Hieraus folgt
\[
\int\psi\,dP_{h}+\varepsilon\geq\int\psi\,dP_{h_{k}}\geq\alpha,
\]
vgl. (\ref{vor_eps_ungl}). Man erhält somit $\int\psi\,dP_{h}\geq\alpha$,
weil $\varepsilon>0$ beliebig klein gewählt werden kann. Insgesamt
ergibt sich 
\[
\liminf_{n\rightarrow\infty}\int\psi_{n}dP_{n,h}=\lim_{j\rightarrow\infty}\int\psi_{n_{j}}dP_{n_{j},h}=\int\psi\,dP_{h}\geq\alpha
\]
für alle $h\in K_{2}^{l}.$ 

Sei nun $h\in H_{2}^{l}\setminus K\left(P_{0},\mathcal{P}\right)$
beliebig aber fest gewählt. Nach Voraussetzung existiert eine Folge
$\left(h_{m}\right)_{m\in\mathbb{N}}\subset K\left(P_{0},\mathcal{P}\right)\cap N\left(\widetilde{k},P_{0},\mathcal{P}\right)$
mit 
\[
\lim_{m\rightarrow\infty}\left\Vert h-h_{m}\right\Vert _{L_{2}\left(P_{0}\right)}=0.
\]
 Es gilt insbesondere $\left(h_{m}\right)_{m\in\mathbb{N}}\subset K\left(P_{0},\mathcal{P}\right)\cap H_{2}^{l}.$
Sei $\varepsilon>0$ beliebig aber fest gewählt. Wegen $\lim\limits_{m\rightarrow0}\left\Vert P_{h}-P_{h_{m}}\right\Vert =0$
existiert ein $k\in\mathbb{N}$ mit $\left\Vert P_{h}-P_{h_{k}}\!\right\Vert \!\leq\frac{\varepsilon}{2}.$
Sei nun $\left(\psi_{n_{j}}\right)_{j\in\mathbb{N}}$ eine Teilfolge,
so dass 
\[
\limsup_{n\rightarrow\infty}\int\psi_{n}dP_{n,h}=\lim_{j\rightarrow\infty}\int\psi_{n_{j}}dP_{n_{j},h}
\]
gilt und der Grenzwert
\[
\lim_{j\rightarrow\infty}\int\psi_{n_{j}}dP_{n_{j},h_{k}}
\]
 existiert. Nach Satz \ref{existenz_limis_test} existiert ein Test
$\psi$ für das Experiment $E^{h,\varepsilon}=\left(\widetilde{\Omega},\widetilde{\mathcal{A}},\left\{ P_{h},\,\,P_{h_{k}}\right\} \right)$
mit 
\[
\lim_{j\rightarrow\infty}\int\psi_{n_{j}}dP_{n_{j},h}=\int\psi\,dP_{h}
\]
und
\[
\lim_{j\rightarrow\infty}\int\psi_{n_{j}}dP_{n_{j},h_{k}}=\int\psi\,dP_{h_{k}}.
\]
Es gilt außerdem
\[
\int\psi\,dP_{h_{k}}=\lim_{j\rightarrow\infty}\int\psi_{n_{j}}dP_{n_{j},h_{k}}\leq\limsup_{n\rightarrow\infty}\int\psi_{n}dP_{n,h_{k}}\leq\alpha
\]
wegen $h_{k}\in K\left(P_{0},\mathcal{P}\right)\cap H_{2}^{l},$ vgl.
(\ref{limsup_leq_alpha}). Analog wie in (\ref{epsilon_ungl}) ergibt
sich $\left|\int\psi\,dP_{h}-\int\psi\,dP_{h_{k}}\right|\leq\varepsilon$.
Wäre $\int\psi\,dP_{h}\leq\int\psi\,dP_{h_{k}},$ so würde man 
\[
\limsup_{n\rightarrow\infty}\int\psi_{n}dP_{n,h}=\lim_{j\rightarrow\infty}\int\psi_{n_{j}}dP_{n_{j},h}=\int\psi\,dP_{h}\leq\int\psi\,dP_{h_{k}}\leq\alpha.
\]
unmittelbar erhalten. Angenommen, es gilt $\int\psi\,dP_{h}>\int\psi\,dP_{h_{k}}.$
Es ergibt sich dann 
\[
\int\psi\,dP_{h}\leq\varepsilon+\int\psi\,dP_{h_{k}}\leq\varepsilon+\alpha.
\]
Hieraus folgt $\int\psi\,dP_{h}\leq\alpha,$ weil $\varepsilon>0$
beliebig klein gewählt werden kann. Insgesamt erhält man 
\[
\limsup_{n\rightarrow\infty}\int\psi_{n}dP_{n,h}=\lim_{j\rightarrow\infty}\int\psi_{n_{j}}dP_{n_{j},h}=\int\psi\,dP_{h}\leq\alpha
\]
 für alle $h\in H_{2}^{l}.$ Die Testfolge $\left(\psi_{n}\right)_{n\in\mathbb{N}}$
ist also asymptotisch unverfälscht zum Niveau $\alpha$ für das Testproblem
$H_{2}^{l}$ gegen $K_{2}^{l}$. Somit ist alles bewiesen.\end{proof}\begin{bem}\label{dichte_bemerkung}Die
Voraussetzungen aus Satz \ref{zw_zw_adv_asymp_opt} bzgl. des Tangentenkegels
$K\left(P_{0},\mathcal{P}\right)$ sind erfüllt, falls $K\left(P_{0},\mathcal{P}\right)$
ein Vektorraum ist.

Ist der Tangentenkegel $K\left(P_{0},\mathcal{P}\right)$ ein Vektorraum,
so liegt er dicht in dem Tangentialraum $T\left(P_{0},\mathcal{P}\right)$
bzgl. der $L_{2}\left(P_{0}\right)$-Norm nach Definition \ref{Tangentialraum}.
Es bleibt zu zeigen, dass die Menge 
\[
W:=K\left(P_{0},\mathcal{P}\right)\cap N\left(\widetilde{k},P_{0},\mathcal{P}\right)
\]
 bzgl. der $L_{2}\left(P_{0}\right)$-Norm dicht in $N\left(\widetilde{k},P_{0},\mathcal{P}\right)$
liegt. Die Menge $W$ ist ein Vektorraum, weil die Menge $N\left(\widetilde{k},P_{0},\mathcal{P}\right)$
und der Tangentenkegel $K\left(P_{0},\mathcal{P}\right)$ jeweils
Vektorräume sind. Sei $h\in N\left(\widetilde{k},P_{0},\mathcal{P}\right)$
beliebig aber fest gewählt. Dann existiert eine Folge $\left(h_{n}\right)_{n\in\mathbb{N}}$
in $K\left(P_{0},\mathcal{P}\right)$ mit 
\begin{equation}
\lim_{n\rightarrow\infty}\left\Vert h-h_{n}\right\Vert _{L_{2}\left(P_{0}\right)}=0.\label{eq:67}
\end{equation}
Wegen $\int h\widetilde{k}\,dP_{0}=0$ erhält man insbesondere
\begin{equation}
\lim\limits_{n\rightarrow\infty}\int h_{n}\widetilde{k}\,dP_{0}=0.\label{eq:68}
\end{equation}
 Besitzt die Folge $\left(h_{n}\right)_{n\in\mathbb{N}}$ eine weitere
Teilfolge, die in dem Vektorraum $W$ enthalten ist, so ist die Behauptung
bereits bewiesen. Es sei also angenommen, dass $h_{n}\notin W$ für
schließlich alle $n\in\mathbb{N}$ gilt. Ohne Einschränkung sei $h_{1}\notin W,$
das bedeutet $\int h_{1}\widetilde{k}\,dP_{0}\neq0.$ Für jedes $n\in\mathbb{N}$
definiere
\[
\widehat{h}_{n}:=h_{n}-h_{1}\frac{\int h_{n}\widetilde{k}\,dP_{0}}{\int h_{1}\widetilde{k}\,dP_{0}}.
\]
 Man erhält $\widehat{h}_{n}\in W$ für alle $n\in\mathbb{N}$, weil
$K\left(P_{0},\mathcal{P}\right)$ ein Vektorraum ist und 
\[
\int\widehat{h}_{n}\widetilde{k}\,dP_{0}=\int h_{n}\widetilde{k}\,dP_{0}-\int h_{1}\widetilde{k}\,dP_{0}\frac{\int h_{n}\widetilde{k}\,dP_{0}}{\int h_{1}\widetilde{k}\,dP_{0}}=0
\]
für alle $n\in\mathbb{N}$ gilt. Für die Folge $\left(\widehat{h}_{n}\right)_{n\in\mathbb{N}}$
erhält man mit Hilfe von (\ref{eq:67}) und (\ref{eq:68}) die Konvergenz
\begin{eqnarray*}
\left\Vert \widehat{h}_{n}-h\right\Vert _{L_{2}\left(P_{0}\right)} & = & \left\Vert h_{n}-h_{1}\frac{\int h_{n}\widetilde{k}\,dP_{0}}{\int h_{1}\widetilde{k}\,dP_{0}}-h\right\Vert _{L_{2}\left(P_{0}\right)}\\
 & \leq & \left\Vert h_{n}-h\right\Vert _{L_{2}\left(P_{0}\right)}+\left\Vert h_{1}\frac{\int h_{n}\widetilde{k}\,dP_{0}}{\int h_{1}\widetilde{k}\,dP_{0}}\right\Vert _{L_{2}\left(P_{0}\right)}\\
 & = & \left\Vert h_{n}-h\right\Vert _{L_{2}\left(P_{0}\right)}+\left\Vert h_{1}\right\Vert _{L_{2}\left(P_{0}\right)}\left|\frac{\int h_{n}\widetilde{k}\,dP_{0}}{\int h_{1}\widetilde{k}\,dP_{0}}\right|\\
 & \rightarrow & 0
\end{eqnarray*}
für $n\rightarrow\infty$. Somit ist alles bewiesen.\end{bem}\begin{bem}Der
Tangentenkegel $K\left(P_{0},\mathcal{P}\right)$ ist genau dann ein
Vektorraum, wenn er konvex ist. \end{bem}

\section{Asymptotisch optimales Testen statistischer Funktionale bei Zweistichprobenproblemen}

In diesem Abschnitt kehren wir zu den Zweistichprobenproblemen beim
Testen statistischer Funktionale zurück, vgl. Abschnitt \ref{sec:Zweistichprobenprobleme-beim-Testen}.
Im Folgenden seien $\left(n_{1}\right)_{n\in\mathbb{N}}$ und $\left(n_{2}\right)_{n\in\mathbb{N}}$
zwei Folgen mit $\lim\limits_{n\rightarrow\infty}\frac{n_{2}}{n}=d\in\left(0,1\right)$
und $n_{1}+n_{2}=n$ für alle $n\in\mathbb{N}$. Dabei bezeichnet
$n_{i}$ den Umfang der $i$-ten Stichprobe für $i=1,2$. Ein statistisches
Funktional $k:\mathcal{P}\otimes\mathcal{Q}\rightarrow\mathbb{R}$
sei differenzierbar an einer Stelle $P_{0}\otimes Q_{0}$ mit dem
kanonischen Gradienten $\widetilde{k}\in L_{2}^{(0)}\left(P_{0}\otimes Q_{0}\right)$.
Nach Satz \ref{umrechnung} existieren dann eindeutige $\widetilde{k}_{1}\in L_{2}^{(0)}\left(P_{0}\right)$
und $\widetilde{k}_{2}\in L_{2}^{(0)}\left(Q_{0}\right)$ mit $\widetilde{k}=\widetilde{k}_{1}\circ\pi_{1}+\widetilde{k}_{2}\circ\pi_{2}$.
Es wird weiterhin immer vorausgesetzt, dass die Bedingung 
\begin{equation}
\left\Vert \widetilde{k}\right\Vert _{L_{2}\left(P_{0}\otimes Q_{0}\right)}\neq0\label{zw_nicht_null_bed}
\end{equation}
 erfüllt ist. Für das einseitige Testproblem wurde eine lokale Parametrisierung
an der Stelle $P_{0}\otimes Q_{0}$ entwickelt, die sich aus der Mengen
der lokalen impliziten Alternativen $\mathcal{K}_{1}$ und der lokalen
impliziten Hypothesen $\mathcal{H}_{1}$ zusammensetzt, vgl. Definition
\ref{impl_alternativen}. Für das zweiseitige Testproblem findet man
in Abschnitt \ref{sec:zw:twsts} eine lokale Parametrisierung in $P_{0}\otimes Q_{0}$,
die entsprechend aus der lokalen impliziten Alternativen $\mathcal{K}_{2}$
und aus der lokalen impliziten Hypothesen $\mathcal{H}_{2}$ besteht.

Um die Optimalitätseigenschaften der in Abschnitten \ref{sec:Asymptotische-Eigenschaften-der}
und \ref{sec:zw:twsts} vorgestellten Testfolgen nachzuweisen, benötigt
man einige Vorbereitungen. Zuerst wird ein neues Skalarprodukt auf
dem Tangentialraum $T\left(P_{0}\otimes Q_{0},\mathcal{P}\otimes\mathcal{Q}\right)$
eingeführt. 

\begin{satz}Die Abbildung 
\begin{eqnarray*}
\left\langle \cdot,\cdot\right\rangle _{d}:\,\,T\left(P_{0}\otimes Q_{0},\mathcal{P}\otimes\mathcal{Q}\right)^{2} & \rightarrow & \mathbb{R}\\
\left(g_{1}\circ\pi_{1}+g_{2}\circ\pi_{2},h_{1}\circ\pi_{1}+h_{2}\circ\pi_{2}\right) & \mapsto & \left(1-d\right)\int g_{1}h_{1}\,dP_{0}+d\int g_{2}h_{2}\,dQ_{0}
\end{eqnarray*}
ist ein Skalarprodukt zum Parameter $d\in\left(0,1\right)$, wobei
$g_{1},h_{1}\in T\left(P_{0},\mathcal{P}\right)$ und $g_{2},h_{2}\in T\left(Q_{0},\mathcal{Q}\right)$
sind. \end{satz}

\begin{proof} Die Abbildung $\left\langle \cdot,\cdot\right\rangle _{d}$
ist wohl definiert, weil die Darstellung $g=g_{1}\circ\pi_{1}+g_{2}\circ\pi_{2}$
mit $g_{1}\in T\left(P_{0},\mathcal{P}\right)$ und $g_{2}\in T\left(Q_{0},\mathcal{Q}\right)$
für jedes $g\in T\left(P_{0}\otimes Q_{0},\mathcal{P}\otimes\mathcal{Q}\right)$
nach Bemerkung \ref{eind_tangente} eindeutig ist. \\
Die Symmetrie $\left\langle h,g\right\rangle _{d}=\left\langle g,h\right\rangle _{d}$
ist für alle $h,g\in T\left(P_{0}\otimes Q_{0},\mathcal{P}\otimes\mathcal{Q}\right)$
leicht zu sehen. Es seien nun $h,f\in T\left(P_{0}\otimes Q_{0},\mathcal{P}\otimes\mathcal{Q}\right)$
und $r,s\in\mathbb{R}$ beliebig. Dann existieren eindeutige $h_{1},f_{1}\in T\left(P_{0},\mathcal{P}\right)$
und $h_{2},f_{2}\in T\left(Q_{0},\mathcal{Q}\right)$ mit $h=h_{1}\circ\pi_{1}+h_{2}\circ\pi_{2}$
und $f=f_{1}\circ\pi_{1}+f_{2}\circ\pi_{2}$. Man erhält zunächst
\begin{eqnarray*}
\left\langle rh+sf,g\right\rangle _{d} & = & \left(1-d\right)\int\left(rh_{1}+sf_{1}\right)g_{1}\,dP_{0}+d\int\left(rh_{2}+sf_{2}\right)g_{2}\,dQ_{0}
\end{eqnarray*}
\begin{eqnarray*}
 & = & r\left(\left(1-d\right)\int h_{1}g_{1}\,dP_{0}+d\int h_{2}g_{2}\,dQ_{0}\right)\\
 &  & +s\left(\left(1-d\right)\int f_{1}g_{1}\,dP_{0}+d\int f_{2}g_{2}\,dQ_{0}\right)\\
 & = & r\left\langle h,g\right\rangle _{d}+s\left\langle f,g\right\rangle _{d}.
\end{eqnarray*}
Es ist somit gezeigt, dass $\left\langle \cdot,\cdot\right\rangle _{d}$
eine symmetrische bilineare Abbildung ist. Es gilt außerdem
\begin{eqnarray*}
\textrm{} &  & \left\langle g,g\right\rangle _{d}=0\\
 & \Leftrightarrow & \left(1-d\right)\int g_{1}^{2}\,dP_{0}+d\int g_{2}^{2}\,dQ_{0}=0\\
 & \Leftrightarrow & \int g_{1}^{2}\,dP_{0}=0\quad\mbox{und}\quad\int g_{2}^{2}\,dQ_{0}=0\\
 & \Leftrightarrow & g_{1}=0\quad P_{0}\mbox{-f.s.}\quad\mbox{und}\quad g_{2}=0\quad Q_{0}\mbox{-f.s.}\\
 & \Leftrightarrow & g=0\quad P_{0}\otimes Q_{0}\mbox{-f.s.}
\end{eqnarray*}

\end{proof}Die von dem Skalarprodukt $\left\langle \cdot,\cdot\right\rangle _{d}$
erzeugte Norm wird mit 
\[
\left\Vert \cdot\right\Vert _{d}:T\left(P_{0}\otimes Q_{0},\mathcal{P}\otimes\mathcal{Q}\right)\rightarrow\mathbb{R}_{\geq0},\,\,g\mapsto\sqrt{\left\langle g,g\right\rangle _{d}}
\]
 bezeichnet.\begin{korollar}Der Tangentialraum $T\left(P_{0}\otimes Q_{0},\mathcal{P}\otimes\mathcal{Q}\right)$
ist ein Hilbertraum bzgl. des Skalarproduktes $\left\langle \cdot,\cdot\right\rangle _{d}$.\end{korollar}\begin{proof}Sei
$g_{0}\in T\left(P_{0}\otimes Q_{0},\mathcal{P}\otimes\mathcal{Q}\right)$
beliebig. Dann existiert eine Folge $\left(g_{n}\right)_{n\in\mathbb{N}}$
in $T\left(P_{0}\otimes Q_{0},\mathcal{P}\otimes\mathcal{Q}\right)$
mit 
\[
\lim_{n\rightarrow\infty}\left\Vert g_{0}-g_{n}\right\Vert _{L_{2}\left(P_{0}\otimes Q_{0}\right)}=0.
\]
Für jedes $n\in\mathbb{N}_{0}$ besitzt die Tangente $g_{n}$ die
eindeutige Darstellung $g_{n}=g_{n1}\circ\pi_{1}+g_{n2}\circ\pi_{2}$
mit $g_{n1}\in T\left(P_{0},\mathcal{P}\right)$ und $g_{n2}\in T\left(Q_{0},\mathcal{Q}\right)$.
Man erhält somit 
\[
0=\lim_{n\rightarrow\infty}\left\Vert g_{0}-g_{n}\right\Vert _{L_{2}\left(P_{0}\otimes Q_{0}\right)}^{2}=\lim_{n\rightarrow\infty}\left\Vert g_{01}-g_{n1}\right\Vert _{L_{2}\left(P_{0}\right)}^{2}+\lim_{n\rightarrow\infty}\left\Vert g_{02}-g_{n2}\right\Vert _{L_{2}\left(Q_{0}\right)}^{^{2}}.
\]
Hieraus folgt $\lim\limits_{n\rightarrow\infty}\left\Vert g_{01}-g_{n1}\right\Vert _{L_{2}\left(P_{0}\right)}^{2}=0$
und $\lim\limits_{n\rightarrow\infty}\left\Vert g_{02}-g_{n2}\right\Vert _{L_{2}\left(P_{0}\right)}^{2}=0.$
Es ergibt sich nun
\begin{eqnarray*}
\textrm{} &  & \lim_{n\rightarrow\infty}\left\Vert g_{0}-g_{n}\right\Vert _{d}^{2}\\
 & = & \lim_{n\rightarrow\infty}\left\langle g_{0}-g_{n},\,g_{0}-g_{n}\right\rangle _{d}\\
 & = & \lim_{n\rightarrow\infty}\left(\left(1-d\right)\int\left(g_{01}-g_{n1}\right)^{2}dP_{0}+d\int\left(g_{02}-g_{n2}\right)^{2}dQ_{0}\right)\\
 & = & \left(1-d\right)\lim_{n\rightarrow\infty}\left\Vert g_{01}-g_{n1}\right\Vert _{L_{2}\left(P_{0}\right)}^{2}+d\lim_{n\rightarrow\infty}\left\Vert g_{02}-g_{n2}\right\Vert _{L_{2}\left(Q_{0}\right)}^{^{2}}\\
 & = & 0.
\end{eqnarray*}
 \end{proof}Der Hilbertraum $T\left(P_{0}\otimes Q_{0},\mathcal{P}\otimes\mathcal{Q}\right)$
bzgl. des Skalarproduktes $\left\langle \cdot,\cdot\right\rangle _{d}$
wird zur Abkürzung mit $H_{d}$ bezeichnet. 

Sei $g\in T\left(P_{0}\otimes Q_{0},\mathcal{P}\otimes\mathcal{Q}\right)$
beliebig aber fest gewählt. Für die Konstruktion einer $L_{2}\left(P_{0}\otimes Q_{0}\right)$-differenzierbaren
Kurve in $\mathcal{P}\otimes\mathcal{Q}$ mit Tangente $g$ lässt
sich das Master-Modell aus \cite{Janssen:2004b} verwenden, vgl. also
Satz \ref{konstruktion}. Die Tangente $g$ besitzt die eindeutige
Darstellung $g=g_{1}\circ\pi_{1}+g_{2}\circ\pi_{2}$ mit $g_{1}\in T\left(P_{0},\mathcal{P}\right)$
und $g_{2}\in T\left(Q_{0},\mathcal{Q}\right)$. Nach Satz \ref{konstruktion}
existieren dann eine $L_{2}\left(P_{0}\right)$-differenzierbare Kurve
$t\mapsto P_{t}$ in $\mathcal{M}_{1}\left(\Omega_{1},\mathcal{A}_{1}\right)$
mit Tangente $g_{1}$ und eine $L_{2}\left(Q_{0}\right)$-differenzierbare
Kurve $t\mapsto Q_{t}$ in $\mathcal{M}_{1}\left(\Omega_{2},\mathcal{A}_{2}\right)$
mit Tangente $g_{2}$. Die Kurve $t\mapsto P_{t}\otimes Q_{t}$ ist
$L_{2}\left(P_{0}\otimes Q_{0}\right)$-differenzierbar mit Tangente
$g$ nach Satz \ref{L2Produkt}. Für jedes $g\in T\left(P_{0}\otimes Q_{0},\mathcal{P}\otimes\mathcal{Q}\right)$
und alle $n\in\mathbb{N}$ definiert man ein Wahrscheinlichkeitsmaß
$P_{n,g}$ durch
\[
P_{n,g}:=P_{\frac{1}{\sqrt{n}}}^{n_{1}}\otimes Q_{\frac{1}{\sqrt{n}}}^{n_{2}},
\]
 wobei $t\mapsto P_{t}\otimes Q_{t}$ eine $L_{2}\left(P_{0}\otimes Q_{0}\right)$-differenzierbare
Kurve mit Tangente $g$ ist. Es wird nun die Folge 
\[
E_{d,n}=\left(\Omega_{1}^{n_{1}}\times\Omega_{2}^{n_{2}},\mathcal{A}_{1}^{n_{1}}\otimes\mathcal{A}_{2}^{n_{2}},\left\{ P_{n,g}:g\in T\left(P_{0}\otimes Q_{0},\mathcal{P}\otimes\mathcal{Q}\right)\right\} \right)
\]
der statistischen Experimente betrachtet, die der Modellbildung durch
implizite Alternativen und Hypothesen bei Zweistichprobenproblemen
entspricht. Außerdem sei 
\[
E_{d}=\left(\Omega,\mathcal{A},\left\{ P_{h}:h\in H_{d}\right\} \right)
\]
ein Gauß-Shift Experiment bzgl. des Hilbertraums $H_{d}$.\begin{satz}Die
Folge $\left(E_{d,n}\right)_{n\in\mathbb{N}}$ der Experimente konvergiert
schwach gegen ein Gauß-Shift Experiment $E_{d}$ und ist somit lokal
asymptotisch normal. \end{satz}\begin{proof}Es reicht die zwei Voraussetzungen
von Satz \ref{hin_not_lan} nachzuweisen. Sei $g\in T\left(P_{0}\otimes Q_{0},\mathcal{P}\otimes\mathcal{Q}\right)$
beliebig. Die Tangente $g$ besitzt die eindeutige Darstellung $g=g_{1}\circ\pi_{1}+g_{2}\circ\pi_{2}$.
Man berechnet zunächst
\begin{eqnarray*}
 &  & \lim_{n\rightarrow\infty}\mbox{Var}_{P_{n,0}}\left(\frac{1}{\sqrt{n}}\left(\sum_{i=1}^{n_{1}}g_{1}\circ\pi_{i}+\sum_{i=n_{1}+1}^{n}g_{2}\circ\pi_{i}\right)\right)\\
 & = & \left(1-d\right)\int g_{1}^{2}dP+d\int g_{2}^{2}dQ\\
 & = & \left\Vert g\right\Vert _{d}^{2}.
\end{eqnarray*}
 Nach Satz \ref{lan_produkt} und Anwendung \ref{ulan_anwendung}
ergibt sich
\[
\log\left(\frac{dP_{n,g}}{dP_{n,0}}\right)=\frac{1}{\sqrt{n}}\left(\sum_{i=1}^{n_{1}}g_{1}\circ\pi_{i}+\sum_{i=n_{1}+1}^{n}g_{2}\circ\pi_{i}\right)-\frac{1}{2}\left\Vert g\right\Vert _{d}^{2}+R_{n,g},
\]
 wobei $\lim\limits_{n\rightarrow\infty}P_{n,0}\left(\left|R_{n,g}\right|>\varepsilon\right)=0$
für jedes $\varepsilon>0$ gilt. Hieraus folgt 
\[
L_{n}\left(g\right)=\log\left(\frac{dP_{n,g}}{dP_{n,0}}\right)+\frac{1}{2}\left\Vert g\right\Vert _{d}^{2}=\frac{1}{\sqrt{n}}\left(\sum_{i=1}^{n_{1}}g_{1}\circ\pi_{i}+\sum_{i=n_{1}+1}^{n}g_{2}\circ\pi_{i}\right)+R_{n,h}.
\]
 Man erhält somit die schwache Konvergenz 
\begin{eqnarray*}
\mathcal{L}\left(\left.L_{n}\left(g\right)\right|P_{n,0}\right) & = & \mathcal{L}\left(\left.\frac{1}{\sqrt{n}}\left(\sum_{i=1}^{n_{1}}g_{1}\circ\pi_{i}+\sum_{i=n_{1}+1}^{n}g_{2}\circ\pi_{i}\right)+R_{n,h}\right|P_{n,0}\right)\\
 & \rightarrow & N\left(0,\left\Vert g\right\Vert _{d}^{2}\right)
\end{eqnarray*}
für $n\rightarrow\infty$ nach dem Zentralen Grenzwertsatz und nach
dem Lemma von Slutsky. Die $1.$ Bedingung von Satz \ref{hin_not_lan}
ist somit erfüllt. Die $2.$ Bedingung von Satz \ref{hin_not_lan}
lässt sich genauso nachweisen wie im Beweis von Satz \ref{tang_raum_lan}.
\end{proof}Es wird nun die bzgl. der $L_{2}\left(P_{0}\otimes Q_{0}\right)$-Norm
stetige lineare Funktion
\begin{eqnarray*}
f:T\left(P_{0}\otimes Q_{0},\mathcal{P}\otimes\mathcal{Q}\right) & \rightarrow & \mathbb{R}\\
g & \mapsto & \int g\widetilde{k}\,dP_{0}\otimes Q_{0}
\end{eqnarray*}
betrachtet. Sei $g=g_{1}\circ\pi_{1}+g_{2}\circ\pi_{2}$ mit $g_{1}\in L_{2}^{(0)}\left(P_{0}\right)$
und $g_{2}\in L_{2}^{(0)}\left(Q_{0}\right)$. Für die Funktion $f$
findet man leicht die Darstellung
\begin{eqnarray}
f\left(g\right) & = & \int\left(g_{1}\circ\pi_{1}+g_{2}\circ\pi_{2}\right)\left(\widetilde{k}_{1}\circ\pi_{1}+\widetilde{k}_{2}\circ\pi_{2}\right)dP_{0}\otimes Q_{0}\nonumber \\
 & = & \int g_{1}\widetilde{k}_{1}\,dP_{0}+\int g_{2}\widetilde{k}_{2}\,dQ_{0}\nonumber \\
 & = & \left\langle g_{1}\circ\pi_{1}+g_{2}\circ\pi_{2},\frac{1}{1-d}\widetilde{k}_{1}\circ\pi_{1}+\frac{1}{d}\widetilde{k}_{2}\circ\pi_{2}\right\rangle _{d}\nonumber \\
 & = & \left\langle g,\widehat{k}\right\rangle _{d},\label{d-darstellung}
\end{eqnarray}
wobei $\widehat{k}=\frac{1}{1-d}\widetilde{k}_{1}\circ\pi_{1}+\frac{1}{d}\widetilde{k}_{2}\circ\pi_{2}$
ist. Die Funktion $f:T\left(P_{0}\otimes Q_{0},\mathcal{P}\otimes\mathcal{Q}\right)\rightarrow\mathbb{R}$
lässt sich also als Skalarprodukt $\left\langle \cdot,\cdot\right\rangle _{d}$
mit $\widehat{k}$ darstellen. Man erhält außerdem
\[
\left\Vert f\right\Vert _{d}^{2}=\left\Vert \widehat{k}\right\Vert _{d}^{2}=\frac{1}{1-d}\int\widetilde{k}_{1}^{2}\,dP_{0}+\frac{1}{d}\int\widetilde{k}_{2}^{2}\,dQ_{0}.
\]
Die schwache Konvergenz der Folge $\left(E_{d,n}\right)_{n\in\mathbb{N}}$
der Experimente gegen das Gauß-Shift Experiment $E_{d}$ und die Darstellung
(\ref{d-darstellung}) der linearen Funktion $f$ ermöglichen eine
elegante Behandlung der Zweistichprobenprobleme beim Testen statistischer
Funktionale. Analog wie in Abschnitt \ref{sec:asymp_opt_ein} lässt
sich die asymptotische Optimalität der in Kapitel \ref{cha:Testen-implizit-definierter}
entwickelten Testfolgen für die einseitigen und zweiseitigen Zweistichprobenprobleme
nachweisen.

\subsection{Asymptotisch optimale einseitige Tests}

Die Testfolge 
\begin{eqnarray*}
\psi_{n} & = & \left\{ \begin{array}{cccc}
1 &  & >\\
 & T_{n} &  & u_{1-\alpha}\left\Vert \widehat{k}\right\Vert _{d}\\
0 &  & \leq
\end{array}\right.
\end{eqnarray*}
 mit 
\[
T_{n}=\frac{\sqrt{n}}{n_{1}}\sum_{i=1}^{n_{1}}\widetilde{k}_{1}\circ\pi_{i}+\frac{\sqrt{n}}{n_{2}}\sum_{i=n_{1}+1}^{n}\widetilde{k}_{2}\circ\pi_{i}
\]
eignet sich für das Testproblem $\mathcal{H}_{1}$ gegen $\mathcal{K}_{1}$,
vgl. Abschnitt \ref{sec:Asymptotische-Eigenschaften-der}. Unter den
schwachen Voraussetzungen an den Tangentenkegel $K\left(P_{0}\otimes Q_{0},\mathcal{P}\otimes\mathcal{Q}\right)$
lässt sich die asymptotische Optimalität der Testfolge $\left(\psi_{n}\right)_{n\in\mathbb{N}}$
nachweisen. \begin{defi} Eine Testfolge $\left(\varphi_{n}\right)_{n\in\mathbb{N}}$
für $\mathcal{H}_{1}$ gegen $\mathcal{K}_{1}$ heißt asymptotisch
Niveau $\alpha$-ähnlich, falls
\[
\lim_{n\rightarrow\infty}\int\psi_{n}\,dP_{t_{n}}^{n_{1}}\otimes Q_{t_{n}}^{n_{2}}=\alpha
\]
 für alle impliziten Hypothesen $\left(P_{t_{n}}^{n_{1}}\otimes Q_{t_{n}}^{n_{2}}\right)_{n\in\mathbb{N}}\in\mathcal{H}_{1}$
mit 
\[
\lim_{n\rightarrow\infty}\sqrt{n}\left(k\left(P_{t_{n}}\otimes Q_{t_{n}}\right)-k\left(P_{0}\otimes Q_{0}\right)\right)=0
\]
 gilt. \end{defi}\begin{satz}Es gelte $K\left(P_{0}\otimes Q_{0},\mathcal{P}\otimes\mathcal{Q}\right)=T\left(P_{0}\otimes Q_{0},\mathcal{P}\otimes\mathcal{Q}\right)$.
Die Testfolge $\left(\psi_{n}\right)_{n\in\mathbb{N}}$ ist dann asymptotisch
optimal innerhalb der Menge aller asymptotisch Niveau $\alpha$-ähnlichen
Testfolgen für $\mathcal{H}_{1}$ gegen $\mathcal{K}_{1}$. \end{satz}\begin{proof}Sei
$\left(\varphi_{n}\right)_{n\in\mathbb{N}}$ eine asymptotisch Niveau
$\alpha$-ähnliche Testfolge für $\mathcal{H}_{1}$ gegen $\mathcal{K}_{1}$.
Analog wie im Beweis von Satz \ref{asymp_opt_einseit_eistichprob_tang}
zeigt man, dass $\left(\varphi_{n}\right)_{n\in\mathbb{N}}$ bereits
eine asymptotisch Niveau $\alpha$-ähnliche Testfolge für das Testproblem
\[
H_{1}^{l}=\left\{ h\in T\left(P_{0}\otimes Q_{0},\mathcal{P}\otimes\mathcal{Q}\right):f(h)\leq0\right\} 
\]
 gegen 
\[
K_{1}^{l}=\left\{ h\in T\left(P_{0}\otimes Q_{0},\mathcal{P}\otimes\mathcal{Q}\right):f(h)>0\right\} 
\]
 ist. Nach Satz \ref{asymp_einseitig} ergibt sich 
\[
\limsup_{n\rightarrow\infty}\int\varphi_{n}\,dP_{n,h}\leq\Phi\left(\frac{f(h)}{\left\Vert f\right\Vert _{d}}-u_{1-\alpha}\right)
\]
für alle $h\in K_{1}^{l}$ und 
\[
\liminf_{n\rightarrow\infty}\int\varphi_{n}\,dP_{n,h}\geq\Phi\left(\frac{f(h)}{\left\Vert f\right\Vert _{d}}-u_{1-\alpha}\right)
\]
 für alle $h\in H_{1}^{l}$. Sei $\left(P_{t_{n}}^{n_{1}}\otimes Q_{t_{n}}^{n_{2}}\right)_{n\in\mathbb{N}}\in\mathcal{K}_{1}$
eine implizite Alternative mit der zugehörigen Tangente $h\in T\left(P_{0}\otimes Q_{0},\mathcal{P}\otimes\mathcal{Q}\right).$
Außerdem sei $t=\lim\limits_{n\rightarrow\infty}n^{\frac{1}{2}}t_{n}.$
Man erhält dann 
\begin{eqnarray*}
\limsup_{n\rightarrow\infty}\int\varphi_{n}\,dP_{t_{n}}^{n_{1}}\otimes Q_{t_{n}}^{n_{2}} & = & \limsup_{n\rightarrow\infty}\int\varphi_{n}\,dP_{n,th}\\
 & \leq & \Phi\left(\frac{f(th)}{\left\Vert f\right\Vert _{d}}-u_{1-\alpha}\right)\\
 & = & \Phi\left(\frac{t\int h\widetilde{k}\,dP_{0}\otimes Q_{0}}{\left(\frac{1}{1-d}\int\widetilde{k}_{1}^{2}\,dP_{0}+\frac{1}{d}\int\widetilde{k}_{2}^{2}\,dQ_{0}\right)^{\frac{1}{2}}}-u_{1-\alpha}\right)\\
 & = & \lim_{n\rightarrow\infty}\int\psi_{n}\,dP_{t_{n}}^{n_{1}}\otimes Q_{t_{n}}^{n_{2}}.
\end{eqnarray*}
 Für eine implizite Hypothese $\left(P_{t_{n}}^{n_{1}}\otimes Q_{t_{n}}^{n_{2}}\right)_{n\in\mathbb{N}}\in\mathcal{H}_{1}$
mit der zugehörigen Tangente $h\in T\left(P_{0}\otimes Q_{0},\mathcal{P}\otimes\mathcal{Q}\right)$
ergibt sich 
\begin{eqnarray*}
\liminf_{n\rightarrow\infty}\int\varphi_{n}\,dP_{t_{n}}^{n_{1}}\otimes Q_{t_{n}}^{n_{2}} & = & \liminf_{n\rightarrow\infty}\int\varphi_{n}\,dP_{n,th}\\
 & \geq & \Phi\left(\frac{f(th)}{\left\Vert f\right\Vert _{d}}-u_{1-\alpha}\right)\\
 & = & \Phi\left(\frac{t\int h\widetilde{k}\,dP_{0}\otimes Q_{0}}{\left(\frac{1}{1-d}\int\widetilde{k}_{1}^{2}\,dP_{0}+\frac{1}{d}\int\widetilde{k}_{2}^{2}\,dQ_{0}\right)^{\frac{1}{2}}}-u_{1-\alpha}\right)\\
 & = & \lim_{n\rightarrow\infty}\int\psi_{n}\,dP_{t_{n}}^{n_{1}}\otimes Q_{t_{n}}^{n_{2}}.
\end{eqnarray*}
 Somit ist alles bewiesen.\end{proof}\begin{satz} Es gelte $N\left(\widetilde{k},P_{0}\otimes Q_{0},\mathcal{P}\otimes\mathcal{Q}\right)\subset K\left(P_{0}\otimes Q_{0},\mathcal{P}\otimes\mathcal{Q}\right)$.
Die Folge $\left(\psi_{n}\right)_{n\in\mathbb{N}}$ der Tests ist
dann asymptotisch optimal in der Menge aller asymptotisch Niveau $\alpha$-ähnlichen
Testfolgen für $\mathcal{H}_{1}$ gegen $\mathcal{K}_{1}$.\end{satz}\begin{proof}vgl.
Beweis von Satz \ref{asymp_opt_einseit_eistichprob_null}.\end{proof}\begin{satz}\label{hsatz_zw_einseitig}Liegt
die Menge $N\left(\widetilde{k},P_{0}\otimes Q_{0},\mathcal{P}\otimes\mathcal{Q}\right)\cap K\left(P_{0}\otimes Q_{0},\mathcal{P}\otimes\mathcal{Q}\right)$
dicht in dem Vektorraum $N\left(\widetilde{k},P_{0}\otimes Q_{0},\mathcal{P}\otimes\mathcal{Q}\right)$
bzgl. der $L_{2}\left(P_{0}\otimes Q_{0}\right)$-Norm, so ist die
Testfolge $\left(\psi_{n}\right)_{n\in\mathbb{N}}$ asymptotisch optimal
innerhalb der Menge aller asymptotisch Niveau $\alpha$-ähnlichen
Testfolgen für $\mathcal{H}_{1}$ gegen $\mathcal{K}_{1}$.\end{satz}\begin{proof}vgl.
Beweis von Satz \ref{asymp_opt_einseit_eistichprob_null_app}.\end{proof}\begin{bem}\label{bem_4_1_1}Ist
der Tangentenkegel $K\left(P_{0}\otimes Q_{0},\mathcal{P}\otimes\mathcal{Q}\right)$
ein Vektorraum, so ist die Vo\-raussetzung, dass die Menge $N\left(\widetilde{k},P_{0}\otimes Q_{0},\mathcal{P}\otimes\mathcal{Q}\right)\cap K\left(P_{0}\otimes Q_{0},\mathcal{P}\otimes\mathcal{Q}\right)$
dicht in dem Vektorraum $N\left(\widetilde{k},P_{0}\otimes Q_{0},\mathcal{P}\otimes\mathcal{Q}\right)$
liegt, bereits erfüllt (vgl. Bemerkung \ref{dichte_bemerkung}). Der
Tangentenkegel $K\left(P_{0}\otimes Q_{0},\mathcal{P}\otimes\mathcal{Q}\right)$
ist seinerseits genau dann ein Vektorraum, wenn die Tangentenkegel
$K\left(P_{0},\mathcal{P}\right)$ und $K\left(Q_{0},\mathcal{Q}\right)$
beide Vektorräume sind.\end{bem}\begin{bem}\label{bem_4_1_2}Im
Allgemeinen liegt die Menge $N\left(\widetilde{k},P_{0}\otimes Q_{0},\mathcal{P}\otimes\mathcal{Q}\right)\cap K\left(P_{0}\otimes Q_{0},\mathcal{P}\otimes\mathcal{Q}\right)$
bzgl. der $L_{2}\left(P_{0}\otimes Q_{0}\right)$-Norm dicht in $N\left(\widetilde{k},P_{0}\otimes Q_{0},\mathcal{P}\otimes\mathcal{Q}\right)$
genau dann, wenn die Menge $N\left(\widetilde{k}_{1},P_{0},\mathcal{P}\right)\cap K\left(P_{0},\mathcal{P}\right)$
bzgl. der $L_{2}\left(P_{0}\right)$-Norm dicht in $N\left(\widetilde{k}_{1},P_{0},\mathcal{P}\right)$
und die Menge $N\left(\widetilde{k}_{2},Q_{0},\mathcal{Q}\right)\cap K\left(Q_{0},\mathcal{Q}\right)$
bzgl. der $L_{2}\left(Q_{0}\right)$-Norm dicht in $N\left(\widetilde{k}_{2},Q_{0},\mathcal{Q}\right)$
liegen.\end{bem}

\subsection{Asymptotisch optimale zweiseitige Tests}

Für das zweiseitige Testproblem $\mathcal{H}_{2}$ gegen $\mathcal{K}_{2}$
eignet sich die Testfolge 
\begin{eqnarray*}
\psi_{n} & = & \left\{ \begin{array}{cccc}
1 &  & >\\
 & \left|T_{n}\right| &  & u_{1-\frac{\alpha}{2}}\left\Vert \widehat{k}\right\Vert _{d}\\
0 &  & \leq
\end{array}\right.,
\end{eqnarray*}
die in Abschnitt \ref{sec:zw:twsts} entwickelt wurde. Analog zum
Abschnitt \ref{opt:zw:test:einstichprobe} wird nun gezeigt, dass
die Testfolge $\left(\psi_{n}\right)_{n\in\mathbb{N}}$ unter den
schwachen Voraussetzungen an den Tangentenkegel $K\left(P_{0}\otimes Q_{0},\mathcal{P}\otimes\mathcal{Q}\right)$
innerhalb der Menge aller zum Niveau $\alpha$ asymptotisch unverfälschten
Testfolgen für das Testproblem $\mathcal{H}_{2}$ gegen $\mathcal{K}_{2}$
optimal ist. \begin{satz}Es gelte $K\left(P_{0}\otimes Q_{0},\mathcal{P}\otimes\mathcal{Q}\right)=T\left(P_{0}\otimes Q_{0},\mathcal{P}\otimes\mathcal{Q}\right)$.
Die Testfolge $\left(\psi_{n}\right)_{n\in\mathbb{N}}$ ist dann asymptotisch
optimal innerhalb der Menge aller zum Niveau $\alpha$ asymptotisch
unverfälschten Testfolgen für das Testproblem $\mathcal{H}_{2}$ gegen
$\mathcal{K}_{2}$.\end{satz}\begin{proof}Eine Testfolge $\left(\varphi_{n}\right)_{n\in\mathbb{N}}$
für $\mathcal{H}_{2}$ gegen $\mathcal{K}_{2}$ sei asymptotisch unverfälscht
zum Niveau $\alpha$. Analog wie im Beweis von Satz \ref{asymp_opt_einseit_eistichprob_tang}
zeigt man, dass $\left(\varphi_{n}\right)_{n\in\mathbb{N}}$ dann
bereits eine zum Niveau $\alpha$ asymptotisch unverfälschte Testfolge
für das Testproblem
\[
H_{2}^{l}=\left\{ h\in T\left(P_{0}\otimes Q_{0},\mathcal{P}\otimes\mathcal{Q}\right):f(h)=0\right\} 
\]
 gegen 
\[
K_{2}^{l}=\left\{ h\in T\left(P_{0}\otimes Q_{0},\mathcal{P}\otimes\mathcal{Q}\right):f(h)\neq0\right\} 
\]
ist. Nach Satz \ref{asymp_zweiseitig} erhält man zunächst 
\[
\limsup_{n\rightarrow\infty}\int\varphi_{n}\,dP_{n,h}\leq\Phi\left(\frac{f\left(h\right)}{\left\Vert f\right\Vert _{d}}-u_{1-\frac{\alpha}{2}}\right)+\Phi\left(-\frac{f\left(h\right)}{\left\Vert f\right\Vert _{d}}-u_{1-\frac{\alpha}{2}}\right)
\]
für alle $h\in K_{2}^{l}$ und 
\[
\liminf_{n\rightarrow\infty}\int\varphi_{n}\,dP_{n,h}\geq\Phi\left(\frac{f\left(h\right)}{\left\Vert f\right\Vert _{d}}-u_{1-\frac{\alpha}{2}}\right)+\Phi\left(-\frac{f\left(h\right)}{\left\Vert f\right\Vert _{d}}-u_{1-\frac{\alpha}{2}}\right)
\]
 für alle $h\in H_{2}^{l}$. Sei $\left(P_{t_{n}}^{n_{1}}\otimes Q_{t_{n}}^{n_{2}}\right)_{n\in\mathbb{N}}\in\mathcal{K}_{2}$
eine implizite Alternative mit der zugehörigen Tangente $h\in T\left(P_{0}\otimes Q_{0},\mathcal{P}\otimes\mathcal{Q}\right).$
Außerdem sei $t=\lim\limits_{n\rightarrow\infty}n^{\frac{1}{2}}t_{n}.$
Nach Satz \ref{sgl_zw_test_asymp_gute} ergibt sich 
\begin{eqnarray*}
\limsup_{n\rightarrow\infty}\int\varphi_{n}\,dP_{t_{n}}^{n_{1}}\otimes Q_{t_{n}}^{n_{2}} & = & \limsup_{n\rightarrow\infty}\int\varphi_{n}\,dP_{n,th}\\
 & \leq & \Phi\left(\frac{f\left(th\right)}{\left\Vert f\right\Vert _{d}}-u_{1-\frac{\alpha}{2}}\right)+\Phi\left(-\frac{f\left(th\right)}{\left\Vert f\right\Vert _{d}}-u_{1-\frac{\alpha}{2}}\right)\\
 & = & \lim_{n\rightarrow\infty}\int\psi_{n}\,dP_{t_{n}}^{n_{1}}\otimes Q_{t_{n}}^{n_{2}}.
\end{eqnarray*}
Für eine implizite Hypothese $\left(P_{t_{n}}^{n_{1}}\otimes Q_{t_{n}}^{n_{2}}\right)_{n\in\mathbb{N}}\in\mathcal{H}_{2}$
mit der zugehörigen Tangente $h\in T\left(P_{0}\otimes Q_{0},\mathcal{P}\otimes\mathcal{Q}\right)$
erhält man 
\begin{eqnarray*}
\liminf_{n\rightarrow\infty}\int\varphi_{n}\,dP_{t_{n}}^{n_{1}}\otimes Q_{t_{n}}^{n_{2}} & = & \liminf_{n\rightarrow\infty}\int\varphi_{n}\,dP_{n,th}\\
 & \geq & \Phi\left(\frac{f\left(th\right)}{\left\Vert f\right\Vert _{d}}-u_{1-\frac{\alpha}{2}}\right)+\Phi\left(-\frac{f\left(th\right)}{\left\Vert f\right\Vert _{d}}-u_{1-\frac{\alpha}{2}}\right)\\
 & = & \lim_{n\rightarrow\infty}\int\psi_{n}\,dP_{t_{n}}^{n_{1}}\otimes Q_{t_{n}}^{n_{2}}.
\end{eqnarray*}
Somit ist alles bewiesen.\end{proof}\begin{satz}\label{zw_ze_opt}
Der Tangentenkegel $K\left(P_{0}\otimes Q_{0},\mathcal{P}\otimes\mathcal{Q}\right)$
liege bzgl. der $L_{2}\left(P_{0}\otimes Q_{0}\right)$-Norm dicht
in dem Tangentialraum $T\left(P_{0}\otimes Q_{0},\mathcal{P}\otimes\mathcal{Q}\right).$
Die Menge 
\[
K\left(P_{0}\otimes Q_{0},\mathcal{P}\otimes\mathcal{Q}\right)\cap N\left(\widetilde{k},P_{0},\mathcal{P}\right)
\]
 liege bzgl. der $L_{2}\left(P_{0}\otimes Q_{0}\right)$-Norm dicht
in $N\left(\widetilde{k},P_{0},\mathcal{P}\right)$. Die Testfolge
$\left(\varphi_{n}\right)_{n\in\mathbb{N}}$ ist dann asymptotisch
optimal innerhalb der Menge aller zum Niveau $\alpha$ asymptotisch
unverfälschten Testfolgen für das Testproblem $\mathcal{H}_{2}$ gegen
$\mathcal{K}_{2}$.\end{satz}\begin{proof}vgl. Beweis von Satz \ref{zw_zw_adv_asymp_opt}.\end{proof}\begin{bem}\label{bem_4_2_1}Die
Voraussetzungen von Satz \ref{zw_ze_opt} sind erfüllt, falls der
Tangentenkegel $K\left(P_{0}\otimes Q_{0},\mathcal{P}\otimes\mathcal{Q}\right)$
ein Vektorraum ist (vgl. Bemerkung \ref{dichte_bemerkung}).\end{bem}\begin{bem}\label{bem_4_2_2}Im
Allgemeinen erfüllt der Tangentenkegel $K\left(P_{0}\otimes Q_{0},\mathcal{P}\otimes\mathcal{Q}\right)$
die Voraussetzungen von Satz \ref{zw_ze_opt} genau dann, wenn die
Tangentenkegel $K\left(P_{0},\mathcal{P}\right)$ und $K\left(Q_{0},\mathcal{Q}\right)$
die Voraussetzungen von Satz \ref{zw_zw_adv_asymp_opt} erfüllen.
\end{bem}

\chapter{Anwendungen und Beispiele}

In diesem Kapitel werden die Anwendungen der Theorie des Testens statistischer
Funktionale für die Zweistichprobenprobleme dargestellt. Es werden
einige praxisrelevante Beispiele vorgestellt und ausführlich erklärt.
Auf die Vorteile und die Schwierigkeiten der Anwendung im nichtparametrischen
Kontext wird hingewiesen.

Seien $\mathcal{P}\subset\mathcal{M}_{1}\left(\Omega_{1},\mathcal{A}_{1}\right)$
und $\mathcal{Q}\subset\mathcal{M}_{1}\left(\Omega_{2},\mathcal{A}_{2}\right)$
zwei nichtparametrische Familien von Wahrscheinlichkeitsmaßen. Es
seien $\left(n_{1}\right)_{n\in\mathbb{N}}\subset\mathbb{N}$ und
$\left(n_{2}\right)_{n\in\mathbb{N}}\subset\mathbb{N}$ zwei Folgen
mit $n_{1}+n_{2}=n$ für alle $n\in\mathbb{N}$ und $\lim_{n\rightarrow\infty}\frac{n_{2}}{n}=d$
für ein $d\in\left(0,1\right).$ Mit $n_{i}$ wird der Umfang der
$i$-ten Stichprobe für $i=1,2$ bezeichnet. Der gesamte Stichprobenumfang
ist dann $n=n_{1}+n_{2}$. Sei $\pi_{i}:\Omega_{1}^{n_{1}}\otimes\Omega_{2}^{n_{2}}\rightarrow\Omega_{1}\cup\Omega_{2},\,\left(\omega_{1},\ldots,\omega_{n}\right)\mapsto\omega_{i}$
die $i$-te kanonische Projektion. 

Außerdem wird die Bezeichnung $X_{i}:=\pi_{i}$ für $i\in\left\{ 1,\ldots,n_{1}\right\} $
und $Y_{i}:=\pi_{i+n_{1}}$ für $i\in\left\{ 1,\ldots,n_{2}\right\} $
verwendet. Die Zufallsvariablen $X_{1},\ldots,X_{n_{1}}$ und $Y_{1},\ldots,Y_{n_{2}}$
sind stochastisch unabhängig bzgl. $P^{n_{1}}\otimes Q^{n_{2}}$ auf
$\Omega_{1}^{n_{1}}\otimes\Omega_{2}^{n_{2}}.$ Die Zufallsvariablen
$X_{1},\ldots,X_{n_{1}}$ sind identisch verteilt gemäß $P$ für ein
$P\in\mathcal{P}$, denn es gilt $\mathcal{L}\left(\left.X_{i}\right|P^{n_{1}}\otimes Q^{n_{2}}\right)=P$
für alle $i\in\left\{ 1,\ldots,n_{1}\right\} $. Die Zufallsvariablen
$Y_{1},\ldots,Y_{n_{2}}$ sind identisch verteilt gemäß $Q\in\mathcal{Q}$,
weil $\mathcal{L}\left(\left.Y_{i}\right|P^{n_{1}}\otimes Q^{n_{2}}\right)=Q$
für alle $j\in\left\{ 1,\ldots,n_{2}\right\} $ gilt. Die erste Stichprobe
wird durch die Zufallsvariablen $X_{1},\ldots,X_{n_{1}}$ modelliert
und die Zufallsvariablen $Y_{1},\ldots,Y_{n_{2}}$ entsprechen der
zweiten Stichprobe.

Seien $k:\mathcal{P}\otimes\mathcal{Q}\rightarrow\mathbb{R}$ ein
statistisches Funktional und $a\in\mathbb{R}$ eine Zahl. Wir beschäftigen
uns weiter mit dem einseitigen Testproblem 
\[
H_{1}=\left\{ P\otimes Q\in\mathcal{P}\otimes\mathcal{Q}:k\left(P\otimes Q\right)\leq a\right\} 
\]
\begin{equation}
\mbox{gegen}\quad K_{1}=\left\{ P\otimes Q\in\mathcal{P}\otimes\mathcal{Q}:k\left(P\otimes Q\right)>a\right\} \label{eins_a_b_tp}
\end{equation}
und dem zweiseitigen Testproblem
\[
H_{2}=\left\{ P\otimes Q\in\mathcal{P}\otimes\mathcal{Q}:k\left(P\otimes Q\right)=a\right\} 
\]
\begin{equation}
\mbox{gegen}\quad K_{2}=\left\{ P\otimes Q\in\mathcal{P}\otimes\mathcal{Q}:k\left(P\otimes Q\right)\neq a\right\} .\label{zweis_a_b_tp}
\end{equation}
Es wird stets vorausgesetzt, dass das statistische Funktional $k:\mathcal{P}\otimes\mathcal{Q}\rightarrow\mathbb{R}$
an jeder Stelle $P_{0}\otimes Q_{0}\in H_{2}$ differenzierbar ist
und der kanonische Gradient $\widetilde{k}=\widetilde{k}\left(P_{0}\otimes Q_{0}\right)$
an jeder Stelle $P_{0}\otimes Q_{0}\in H_{2}$ die Voraussetzung 
\begin{equation}
\left\Vert \widetilde{k}\right\Vert _{L_{2}\left(P_{0}\otimes Q_{0}\right)}\neq0\label{a_b_no_null_vor}
\end{equation}
erfüllt. Nach Satz \ref{umrechnung} und Bemerkung \ref{eind_tangente}
besitzt der kanonische Gradient $\widetilde{k}$ die eindeutige Darstellung
$\widetilde{k}=\widetilde{k}_{1}\circ\pi_{1}+\widetilde{k}_{2}\circ\pi_{2}$
mit $\widetilde{k}_{1}\in T\left(P_{0},\mathcal{P}\right)$ und $\widetilde{k}_{2}\in T\left(Q_{0},\mathcal{Q}\right)$.

Die Testfolgen (\ref{einseit_asym_opt_testfolge}) und (\ref{zweiset_asymp_opt_testfolge}),
die in Kapitel \ref{cha:Testen-implizit-definierter} entwickelt worden
sind, lösen die Testprobleme (\ref{eins_a_b_tp}) und (\ref{zweis_a_b_tp})
lokal für ein fest gewähltes Wahrscheinlichkeitsmaß $P_{0}\otimes Q_{0}\in H_{2}$.
Die Testfolge (\ref{einseit_asym_opt_testfolge}) ist zur Erinnerung
durch 
\begin{equation}
\varphi_{1n}=\left\{ \begin{array}{cccc}
1 &  & >\\
 & T_{n} &  & c_{1}\\
0 &  & \leq
\end{array}\right.\label{erinnerung_asymp_einseit_testfolge}
\end{equation}
mit dem kritischen Wert 
\begin{equation}
c_{1}=u_{1-\alpha}\left(\frac{1}{1-d}\int\widetilde{k}_{1}^{2}\,dP_{0}+\frac{1}{d}\int\widetilde{k}_{2}^{2}\,dQ_{0}\right)^{\frac{1}{2}}\label{erinnerung_c1}
\end{equation}
gegeben, wobei 
\[
T_{n}=\frac{\sqrt{n}}{n_{1}}\sum_{i=1}^{n_{1}}\widetilde{k}_{1}\circ X_{i}+\frac{\sqrt{n}}{n_{2}}\sum_{i=1}^{n_{2}}\widetilde{k}_{2}\circ Y_{i}
\]
 die Teststatistik (\ref{teststatistik T}) ist. Die Testfolge (\ref{zweiset_asymp_opt_testfolge})
ist durch 
\begin{equation}
\varphi_{2n}=\left\{ \begin{array}{cccc}
1 &  & >\\
 & \left|T_{n}\right| &  & c_{2}\\
0 &  & \leq
\end{array}\right.\label{erinnerung_asymp_zweiseit_testfolge}
\end{equation}
mit dem kritischen Wert 
\begin{equation}
c_{2}=u_{1-\frac{\alpha}{2}}\left(\frac{1}{1-d}\int\widetilde{k}_{1}^{2}\,dP_{0}+\frac{1}{d}\int\widetilde{k}_{2}^{2}\,dQ_{0}\right)^{\frac{1}{2}}\label{erinnerung_c2}
\end{equation}
definiert. In Kapitel \ref{cha:Nichtparametrisches-asymptotisch-optimale}
wurde gezeigt, dass die Testfolgen (\ref{einseit_asym_opt_testfolge})
und (\ref{zweiset_asymp_opt_testfolge}) unter der schwachen Voraussetzungen
an den Tangentenkegel $K\left(P_{0}\otimes Q_{0},\mathcal{P}\otimes\mathcal{Q}\right)$
asymptotisch optimal an der Stelle $P_{0}\otimes Q_{0}$ sind, vgl.
Satz \ref{hsatz_zw_einseitig} und Satz \ref{zw_ze_opt}. Wir können
stets annehmen, dass die Voraussetzungen von Satz \ref{hsatz_zw_einseitig}
und Satz \ref{zw_ze_opt} erfüllt sind. Aus der Sicht der nichtparametrischen
Statistik ist diese Annahme berechtigt, vgl. die Bemerkungen \ref{bem_4_1_1},
\ref{bem_4_1_2}, \ref{bem_4_2_1} und \ref{bem_4_2_2}. Die Testfolgen
(\ref{einseit_asym_opt_testfolge}) und (\ref{zweiset_asymp_opt_testfolge})
sind dann asymptotisch optimal an der Stelle $P_{0}\otimes Q_{0}$.
Das Wahrscheinlichkeitsmaß $P_{0}\otimes Q_{0}$ wird in diesem Fall
als Fußpunkt bezeichnet und kann als unbekannter Parameter angesehen
werden.

Zur Ausführung der Tests aus (\ref{einseit_asym_opt_testfolge}) und
(\ref{zweiset_asymp_opt_testfolge}) ist es allerdings nicht erforderlich,
das Wahrscheinlichkeitsmaß $P_{0}\otimes Q_{0}$ zu kennen, sondern
man braucht nur die Abbildungen $\widetilde{k}_{1},\widetilde{k}_{2}$
zu bestimmen und die Werte $\left\Vert \widetilde{k}_{1}\right\Vert _{L_{2}\left(P_{0}\right)}$
und $\left\Vert \widetilde{k}_{2}\right\Vert _{L_{2}\left(Q_{0}\right)}$
zu schätzen. Die Abbildungen $\widetilde{k}_{1}\in T\left(P_{0},\mathcal{P}\right)$
und $\widetilde{k}_{2}\in T\left(Q_{0},\mathcal{Q}\right)$ hängen
im Allgemeinen von dem Wahrscheinlichkeitsmaß $P_{0}\otimes Q_{0}$
ab. Da die Tangentialräume $T\left(P_{0},\mathcal{P}\right)$ und
$T\left(Q_{0},\mathcal{Q}\right)$ in der nichtparametrischen Statistik
in der Regel unendlichdimensional sind, können die Abbildungen $\widetilde{k}_{1}$
und $\widetilde{k}_{2}$ als unbekannte unendlichdimensionale Parameter
angesehen werden. Falls die Abbildungen $\widetilde{k}_{1}$ und $\widetilde{k}_{2}$
eine besonders günstige Gestalt besitzen, so kann die Teststatistik
$T_{n}$ anhand der Realisierungen von Zufallsvariablen $X_{1},\ldots,X_{n_{1}}$
und $Y_{1},\ldots,Y_{n_{2}}$ ausgewertet werden. In manchen Fällen
findet man eine Teststatistik $\widehat{T}_{n}$, die anhand der Realisierungen
von Zufallsvariablen $X_{1},\ldots,X_{n_{1}}$ und $Y_{1},\ldots,Y_{n_{2}}$
berechenbar ist und die Bedingung 
\[
\lim_{n\rightarrow\infty}\mathcal{L}\left(\left.T_{n}\right|P_{0}^{n_{1}}\otimes Q_{0}^{n_{2}}\right)=\lim_{n\rightarrow\infty}\mathcal{L}\left(\left.\widehat{T}_{n}\right|P_{0}^{n_{1}}\otimes Q_{0}^{n_{2}}\right)
\]
 erfüllt, d.h. die Teststatistiken $T_{n}$ und $\widehat{T}_{n}$
besitzen die gleiche asymptotische Verteilung unter $P_{0}^{n_{1}}\otimes Q_{0}^{n_{2}}$.
Außerdem braucht man ausreichend gute Schätzer für die unbekannte
Werte $\left\Vert \widetilde{k}_{1}\right\Vert _{L_{2}\left(P_{0}\right)}$
und $\left\Vert \widetilde{k}_{2}\right\Vert _{L_{2}\left(Q_{0}\right)}$.
Sind diese Voraussetzungen erfüllt, so kann man die Testfolgen (\ref{einseit_asym_opt_testfolge})
und (\ref{zweiset_asymp_opt_testfolge}) mit der Teststatistik $T_{n}$
oder $\widehat{T}_{n}$ und geschätzten Werten $\left\Vert \widetilde{k}_{1}\right\Vert _{L_{2}\left(P_{0}\right)}$,
$\left\Vert \widetilde{k}_{2}\right\Vert _{L_{2}\left(Q_{0}\right)}$
ausführen. Dabei erhält man eine Testfolge, die asymptotisch äquivalent
zu der Testfolge (\ref{einseit_asym_opt_testfolge}) bzw. (\ref{zweiset_asymp_opt_testfolge})
für das Testproblem (\ref{eins_a_b_tp}) bzw. (\ref{zweis_a_b_tp})
ist. \begin{defi}[lokale asymptotische Äquivalenz]Sei $P_{0}\otimes Q_{0}\in H_{2}$
beliebig aber fest gewählt. Zwei Testfolgen $\left(\widehat{\varphi}_{n}\right)_{n\in\mathbb{N}}$
und $\left(\varphi_{n}\right)_{n\in\mathbb{N}}$ heißen lokal asymptotisch
äquivalent für das Testproblem (\ref{eins_a_b_tp}) bzw. (\ref{zweis_a_b_tp})
an der Stelle $P_{0}\otimes Q_{0}$, falls 
\[
\lim_{n\rightarrow\infty}\int\varphi_{n}dP_{t_{n}}^{n_{1}}\otimes Q_{t_{n}}^{n_{2}}=\lim_{n\rightarrow\infty}\int\widehat{\varphi}_{n}dP_{t_{n}}^{n_{1}}\otimes Q_{t_{n}}^{n_{2}}
\]
für jede Folge $\left(P_{t_{n}}^{n_{1}}\otimes Q_{t_{n}}^{n_{2}}\right)_{n\in\mathbb{N}}\in\mathcal{F}_{2}$
von Wahrscheinlichkeitsmaßen gilt. \end{defi}\begin{bem}Sind zwei
Testfolgen $\left(\widehat{\varphi}_{n}\right)_{n\in\mathbb{N}}$
und $\left(\varphi_{n}\right)_{n\in\mathbb{N}}$ lokal asymptotisch
äquivalent an einer Stelle $P_{0}\otimes Q_{0}\in H_{2}$, so besitzen
sie die gleiche asymptotische Gütefunktion entlang der impliziten
Alternativen und Hypothesen an der Stelle $P_{0}\otimes Q_{0}\in H_{2}$.
Insbesondere ist die Testfolge $\left(\widehat{\varphi}_{n}\right)_{n\in\mathbb{N}}$
lokal asymptotisch optimal an der Stelle $P_{0}\otimes Q_{0}\in H_{2}$
in einer Menge $\Psi$ der Testfolgen, falls die Testfolge $\left(\varphi_{n}\right)_{n\in\mathbb{N}}$
in der Menge $\Psi$ der Testfolgen lokal asymptotisch optimal an
der Stelle $P_{0}\otimes Q_{0}\in H_{2}$ ist. \end{bem}\begin{defi}[asymptotische Äquivalenz]Zwei
Testfolgen $\left(\widehat{\varphi}_{n}\right)_{n\in\mathbb{N}}$
und $\left(\varphi_{n}\right)_{n\in\mathbb{N}}$ heißen asymptotisch
äquivalent für das Testproblem (\ref{eins_a_b_tp}) bzw. (\ref{zweis_a_b_tp}),
falls sie an jeder Stelle $P_{0}\otimes Q_{0}\in H_{2}$ lokal asymptotisch
äquivalent für das Testproblem (\ref{eins_a_b_tp}) bzw. (\ref{zweis_a_b_tp})
sind.\end{defi}

\begin{beisp}[Summe zweier von Mises Funktionale]\label{summ_mises_funk}

Seien $k_{1}:\mathcal{P}\rightarrow\mathbb{R},\,P\mapsto\int f_{1}\,dP$
und $k_{2}:\mathcal{Q}\rightarrow\mathbb{R},\,Q\mapsto\int f_{2}\,dQ$
zwei von Mises Funktionale. Für jedes $P_{0}\in\mathcal{P}$ existiere
ein $\varepsilon>0$ und ein $K>0$, so dass $0<\int f_{1}^{2}\,dP<K$
für alle $P\in\mathcal{P}$ mit $d\left(P_{0},P\right)<\varepsilon$
gilt. Nach Lemma \ref{LemmaIundH} und Satz \ref{topologie}  ist
das statistische Funktional $k_{1}$ differenzierbar an jeder Stelle
$P_{0}\in\mathcal{P}$ mit der Abbildung $f_{1}-E_{P_{0}}\left(f_{1}\right)$
als Gradient. Außerdem sei $f_{1}-E_{P_{0}}\left(f_{1}\right)\in T\left(P_{0},\mathcal{P}\right)$
für alle $P_{0}\in\mathcal{P}$. Die Abbildung $f_{1}-E_{P_{0}}\left(f_{1}\right)$
ist dann der kanonische Gradient von $k_{1}$ an der Stelle $P_{0}\in\mathcal{P}$.
Das statistische Funktional $k_{2}$ erfülle die analogen Voraussetzungen,
so dass $k_{2}$ an jeder Stelle $Q_{0}\in\mathcal{Q}$ differenzierbar
mit dem kanonischen Gradienten $f_{2}-E_{Q_{0}}\left(f_{2}\right)$
ist. Nach Beispiel \ref{bsp_zusamm_gesetzte_funk} ist das statistische
Funktional 
\[
k:\mathcal{P}\otimes\mathcal{Q}\rightarrow\mathbb{R},\,P\otimes Q\mapsto k_{1}\left(P\right)+k_{2}\left(Q\right)
\]
 differenzierbar mit dem kanonischen Gradienten 
\[
\widetilde{k}:\Omega_{1}\times\Omega_{2}\rightarrow\mathbb{R},\left(\omega_{1},\omega_{2}\right)\mapsto f_{1}\left(\omega_{1}\right)-E_{P_{0}}\left(f_{1}\right)+f_{2}\left(\omega_{2}\right)-E_{Q_{0}}\left(f_{2}\right)
\]
 an jeder Stelle $P_{0}\otimes Q_{0}\in\mathcal{P}$. Man erhält insbesondere
$\widetilde{k}_{1}=f_{1}-E_{P_{0}}\left(f_{1}\right)$ und $\widetilde{k}_{2}=f_{2}-E_{Q_{0}}\left(f_{2}\right)$.
Sei nun $P_{0}\otimes Q_{0}\in H_{2}$, das bedeutet 
\[
k\left(P_{0}\otimes Q_{0}\right)=E_{P_{0}}\left(f_{1}\right)+E_{Q_{0}}\left(f_{2}\right)=a.
\]
 Für die Teststatistik $T_{n}$ ergibt sich dann 
\begin{eqnarray*}
T_{n} & = & \frac{\sqrt{n}}{n_{1}}\sum_{i=1}^{n_{1}}\widetilde{k}_{1}\circ X_{i}+\frac{\sqrt{n}}{n_{2}}\sum_{i=1}^{n_{2}}\widetilde{k}_{2}\circ Y_{i}\\
 & = & \frac{\sqrt{n}}{n_{1}}\sum_{i=1}^{n_{1}}f_{1}\circ X_{i}+\frac{\sqrt{n}}{n_{2}}\sum_{i=1}^{n_{2}}f_{2}\circ Y_{i}-\sqrt{n}\left(E_{P_{0}}\left(f_{1}\right)+E_{Q_{0}}\left(f_{2}\right)\right)\\
 & = & \frac{\sqrt{n}}{n_{1}}\sum_{i=1}^{n_{1}}f_{1}\circ X_{i}+\frac{\sqrt{n}}{n_{2}}\sum_{i=1}^{n_{2}}f_{2}\circ Y_{i}-\sqrt{n}a.
\end{eqnarray*}
Es bezeichne $\overline{f}_{1,n}=\frac{1}{n_{1}}\sum_{i=1}^{n_{1}}f_{1}\circ X_{i}$
und $\overline{f}_{2,n}=\frac{1}{n_{2}}\sum_{i=1}^{n_{2}}f_{2}\circ Y_{i}$.
Außerdem seien 
\[
\widehat{\sigma}_{n}\left(f_{1}\right)^{2}=\frac{1}{n_{1}-1}\sum_{i=1}^{n_{1}}\left(f_{1}\circ X_{i}-\overline{f}_{1,n}\right)^{2}
\]
 und
\[
\widehat{\sigma}_{n}\left(f_{2}\right)^{2}=\frac{1}{n_{2}-1}\sum_{i=1}^{n_{2}}\left(f_{2}\circ Y_{i}-\overline{f}_{2,n}\right)^{2}
\]
 die kanonischen Schätzer für $\mbox{Var}_{P_{0}}\left(f_{1}\right)$
und $\mbox{Var}_{Q_{0}}\left(f_{2}\right)$. Die Testfolge $\left(\widehat{\varphi}_{1n}\right)_{n\in\mathbb{N}}$
wird definiert durch
\begin{equation}
\widehat{\varphi}_{1n}=\left\{ \begin{array}{cccc}
1 &  & >\\
 & T_{n} &  & \widehat{c}_{1,n}\\
0 &  & \leq
\end{array}\right.\label{dach_einseit_asymp_testfolge}
\end{equation}
mit dem kritischen Wert 
\begin{equation}
\widehat{c}_{1,n}=u_{1-\alpha}\left(\frac{n}{n_{1}}\widehat{\sigma}_{n}\left(f_{1}\right)^{2}+\frac{n}{n_{2}}\widehat{\sigma}_{n}\left(f_{2}\right)^{2}\right)^{\frac{1}{2}}.\label{krit_wert_c_1_dach}
\end{equation}
 Die Testfolgen (\ref{erinnerung_asymp_einseit_testfolge}) und (\ref{dach_einseit_asymp_testfolge})
sind asymptotisch äquivalent für das einseitige Testproblem (\ref{eins_a_b_tp}).
Zum Beweis betrachte ein beliebiges aber fest gewähltes Wahrscheinlichkeitsmaß
$P_{0}\otimes Q_{0}\in H_{2}$. Die Testfolgen (\ref{erinnerung_asymp_einseit_testfolge})
und (\ref{dach_einseit_asymp_testfolge}) werden zunächst folgendermaßen
äquivalent umgeformt:
\[
\widehat{\varphi}_{1n}=\left\{ \begin{array}{cccc}
1 &  & >\\
 & T_{n}-\widehat{c}_{1,n} &  & 0\\
0 &  & \leq
\end{array}\right.\qquad\mbox{und}\qquad\varphi_{1n}=\left\{ \begin{array}{cccc}
1 &  & >\\
 & T_{n}-c_{1} &  & 0\\
0 &  & \leq
\end{array}\right..
\]
Die beiden Testfolgen (\ref{erinnerung_asymp_einseit_testfolge})
und (\ref{dach_einseit_asymp_testfolge}) bestehen aus den nicht randomisierten
Tests. Nach Lemma \ref{benach_imp_asymp_eq} und Bemerkung \ref{asympt_banachb_equiv_tests}
reicht es 
\[
\lim_{n\rightarrow\infty}\int\left|\widehat{\varphi}_{1n}-\varphi_{1n}\right|dP_{0}^{n_{1}}\otimes Q_{0}^{n_{2}}=0
\]
 nachzuweisen. Nach Lemma \ref{asympt_aquivalent_testfolge} genügt
es zu zeigen, dass
\[
\lim_{n\rightarrow\infty}P_{0}^{n_{1}}\otimes Q_{0}^{n_{2}}\left(\left\{ \left|\widehat{c}_{1,n}-c_{1}\right|>\varepsilon\right\} \right)=0
\]
 für alle $\varepsilon>0$ gilt. Nach dem starken Gesetz der großen
Zahlen ergibt sich die fast sichere Konvergenz $\widehat{\sigma}_{n}\left(f_{1}\right)^{2}\rightarrow\mbox{Var}_{P_{0}}\left(f_{1}\right)$
und $\widehat{\sigma}_{n}\left(f_{2}\right)^{2}\rightarrow\mbox{Var}_{Q_{0}}\left(f_{2}\right)$
für $n\rightarrow\infty$, weil $f_{1}\in L_{2}\left(P_{0}\right)$
und $f_{2}\in L_{2}\left(Q_{0}\right)$ nach Voraussetzung sind. Es
gilt außerdem 
\[
\int\widetilde{k}_{1}^{2}\,dP_{0}=\mbox{Var}_{P_{0}}\left(\widetilde{k}_{1}\right)=\mbox{Var}_{P_{0}}\left(f_{1}-E_{P_{0}}\left(f_{1}\right)\right)=\mbox{Var}_{P_{0}}\left(f_{1}\right),
\]
 
\[
\int\widetilde{k}_{2}^{2}\,dQ_{0}=\mbox{Var}_{Q_{0}}\left(\widetilde{k}_{2}\right)=\mbox{Var}_{Q_{0}}\left(f_{2}-E_{Q_{0}}\left(f_{2}\right)\right)=\mbox{Var}_{Q_{0}}\left(f_{2}\right).
\]
Hieraus folgt die fast sichere Konvergenz
\begin{equation}
\left(\frac{n}{n_{1}}\widehat{\sigma}_{n}\left(f_{1}\right)^{2}+\frac{n}{n_{2}}\widehat{\sigma}_{n}\left(f_{2}\right)^{2}\right)^{\frac{1}{2}}\rightarrow\left(\frac{1}{1-d}\int\widetilde{k}_{1}^{2}\,dP_{0}+\frac{1}{d}\int\widetilde{k}_{2}^{2}\,dQ_{0}\right)^{\frac{1}{2}}\label{st_c_1_kogz}
\end{equation}
 für $n\rightarrow\infty$. Für alle $\varepsilon>0$ erhält man insbesondere
\[
\lim_{n\rightarrow\infty}P_{0}^{n_{1}}\otimes Q_{0}^{n_{2}}\left(\left\{ \left|\widehat{c}_{1,n}-c_{1}\right|>\varepsilon\right\} \right)=0.
\]

Die Testfolge $\left(\widehat{\varphi}_{2n}\right)_{n\in\mathbb{N}}$
für das zweiseitige Problem (\ref{zweis_a_b_tp}) wird definiert durch
\begin{equation}
\widehat{\varphi}_{2n}=\left\{ \begin{array}{cccc}
1 &  & >\\
 & \left|T_{n}\right| &  & \widehat{c}_{2,n}\\
0 &  & \leq
\end{array}\right.\label{dach_zweiseit_asymp_testfolge}
\end{equation}
mit dem kritischen Wert 
\begin{equation}
\widehat{c}_{2,n}=u_{1-\frac{\alpha}{2}}\left(\frac{n}{n_{1}}\widehat{\sigma}_{n}\left(f_{1}\right)^{2}+\frac{n}{n_{2}}\widehat{\sigma}_{n}\left(f_{2}\right)^{2}\right)^{\frac{1}{2}}.\label{krit_wert_c_2_dach}
\end{equation}
 Mit (\ref{st_c_1_kogz}) ergibt sich $\lim_{n\rightarrow\infty}P_{0}^{n_{1}}\otimes Q_{0}^{n_{2}}\left(\left\{ \left|\widehat{c}_{2,n}-c_{2}\right|>\varepsilon\right\} \right)=0$
für alle \\
$\varepsilon>0$. Nach Lemma \ref{asympt_aquivalent_testfolge}, Lemma
\ref{benach_imp_asymp_eq} und Bemerkung \ref{asympt_banachb_equiv_tests}
folgt nun analog, dass die Testfolgen (\ref{erinnerung_asymp_zweiseit_testfolge})
und (\ref{dach_zweiseit_asymp_testfolge}) für das zweiseitige Testproblem
(\ref{zweis_a_b_tp}) asymptotisch äquivalent sind.\end{beisp}\begin{anwendung}[Testen der Erwartungswerte]Es
gelten die Voraussetzungen aus Beispiel \ref{summ_mises_funk}. Außerdem
seien $\left(\Omega_{1},\mathcal{A}_{1}\right)=\left(\Omega_{2},\mathcal{A}_{2}\right)=\left(\mathbb{R},\mathcal{B}\right)$
und $f_{1}=-f_{2}=\mbox{Id}_{\mathbb{R}}$. Für die Testprobleme (\ref{eins_a_b_tp})
und (\ref{zweis_a_b_tp}) erhält man dann
\begin{equation}
H_{1}=\left\{ E_{P}\left(\mbox{Id}_{\mathbb{R}}\right)-E_{Q}\left(\mbox{Id}_{\mathbb{R}}\right)\leq a\right\} \;\mbox{gegen}\;K_{1}=\left\{ E_{P}\left(\mbox{Id}_{\mathbb{R}}\right)-E_{Q}\left(\mbox{Id}_{\mathbb{R}}\right)>a\right\} \label{eq:h1k1}
\end{equation}
 und 
\begin{equation}
H_{2}=\left\{ E_{P}\left(\mbox{Id}_{\mathbb{R}}\right)-E_{Q}\left(\mbox{Id}_{\mathbb{R}}\right)=a\right\} \;\mbox{gegen}\;K_{2}=\left\{ E_{P}\left(\mbox{Id}_{\mathbb{R}}\right)-E_{Q}\left(\mbox{Id}_{\mathbb{R}}\right)\neq a\right\} ,\label{eq:h2k2}
\end{equation}
 wobei $\left\{ E_{P}\left(\mbox{Id}_{\mathbb{R}}\right)-E_{Q}\left(\mbox{Id}_{\mathbb{R}}\right)\leq a\right\} $
die kurze Schreibweise für 
\[
\left\{ P\otimes Q\in\mathcal{P}\otimes\mathcal{Q}:E_{P}\left(\mbox{Id}_{\mathbb{R}}\right)-E_{Q}\left(\mbox{Id}_{\mathbb{R}}\right)\leq a\right\} 
\]
 ist. Es bezeichne $\overline{X}_{n_{1}}=\frac{1}{n_{1}}\sum_{i=1}^{n_{1}}X_{i}$
und $\overline{Y}_{n_{2}}=\frac{1}{n_{2}}\sum_{i=1}^{n_{2}}Y_{i}$.
Die Testfolge $\left(\widehat{\varphi}_{1n}\right)_{n\in\mathbb{N}}$
für das einseitige Testproblem (\ref{eq:h1k1}) ist durch die Teststatistik
\[
T_{n}=\frac{\sqrt{n}}{n_{1}}\sum_{i=1}^{n_{1}}X_{i}-\frac{\sqrt{n}}{n_{2}}\sum_{i=1}^{n_{2}}Y_{i}-\sqrt{n}a=\sqrt{n}\left(\overline{X}_{n_{1}}-\overline{Y}_{n_{2}}-a\right)
\]
 und den kritischen Wert
\[
\widehat{c}_{1,n}=u_{1-\alpha}\left(\frac{n}{n_{1}}\widehat{\sigma}_{1,n}^{2}+\frac{n}{n_{2}}\widehat{\sigma}_{2,n}^{2}\right)^{\frac{1}{2}}
\]
 gegeben, wobei 
\[
\widehat{\sigma}_{1,n}^{2}=\frac{1}{n_{1}-1}\sum_{i=1}^{n_{1}}\left(X_{i}-\overline{X}_{n_{1}}\right)^{2}\;\mbox{und}\quad\widehat{\sigma}_{2,n}^{2}=\frac{1}{n_{2}-1}\sum_{i=1}^{n_{2}}\left(Y_{i}-\overline{Y}_{n_{2}}\right)^{2}
\]
 sind. Die Testfolge $\left(\widehat{\varphi}_{2n}\right)_{n\in\mathbb{N}}$
für das zweiseitige Testproblem (\ref{eq:h2k2}) ist ebenfalls durch
die Teststatistik $T_{n}$ und den kritischen Wert 
\[
\widehat{c}_{2,n}=u_{1-\frac{\alpha}{2}}\left(\frac{n}{n_{1}}\widehat{\sigma}_{1,n}^{2}+\frac{n}{n_{2}}\widehat{\sigma}_{2,n}^{2}\right)^{\frac{1}{2}}
\]
 festgelegt. Es ist zu beachten, dass diese Testfolgen das Niveau
$\alpha$ nur asymptotisch einhalten. Für den endlichen Stichprobenumfang
gibt es bereits für zwei parametrische Familien $\mathcal{P}:=\left\{ N\left(\mu,\sigma^{2}\right):\mu\in\mathbb{R},\,\sigma^{2}\in\left(0,+\infty\right)\right\} $
und $\mathcal{Q}:=\mathcal{P}$ der Normalverteilungen nur triviale
exakte Niveau $\alpha$-Tests für das Testproblem (\ref{eins_a_b_tp})
bzw. (\ref{zweis_a_b_tp}), vgl. Behrens-Fisher Testproblem, \cite{Lehmann:1986},
Seite 304. \end{anwendung}\begin{anwendung}[Testen von Verteilungsfunktionen]Es
gelten die Voraussetzungen aus Beispiel \ref{summ_mises_funk}. Es
seien $\left(\Omega_{1},\mathcal{A}_{1}\right)=\left(\Omega_{2},\mathcal{A}_{2}\right)=\left(\mathbb{R},\mathcal{B}\right)$
und $f_{1}=-f_{2}=\mathbf{1}_{\left(-\infty,q\right]}$ für ein $q\in\mathbb{R}$.
Die Verteilungsfunktion eines Wahrscheinlichkeitsmaßes $P\in\mathcal{M}_{1}\left(\mathbb{R},\mathcal{B}\right)$
wird mit $F_{P}$ bezeichnet. Für die Testprobleme (\ref{eins_a_b_tp})
und (\ref{zweis_a_b_tp}) ergibt sich 
\begin{equation}
H_{1}=\left\{ F_{P}\left(q\right)-F_{Q}\left(q\right)\leq a\right\} \;\mbox{gegen}\;K_{1}=\left\{ F_{P}\left(q\right)-F_{Q}\left(q\right)>a\right\} \label{eq:q_quant_eins}
\end{equation}
 und 
\begin{equation}
H_{2}=\left\{ F_{P}\left(q\right)-F_{Q}\left(q\right)=a\right\} \;\mbox{gegen}\;K_{2}=\left\{ F_{P}\left(q\right)-F_{Q}\left(q\right)\neq a\right\} .\label{eq:q_quant_zwei}
\end{equation}
Als Teststatistik erhält man dann 
\[
T_{n}=\frac{\sqrt{n}}{n_{1}}\sum_{i=1}^{n_{1}}\mathbf{1}_{\left(-\infty,q\right]}\left(X_{i}\right)-\frac{\sqrt{n}}{n_{2}}\sum_{i=1}^{n_{2}}\mathbf{1}_{\left(-\infty,q\right]}\left(Y_{i}\right)-\sqrt{n}a.
\]
 Der Test $\widehat{\varphi}_{1n}$ für das einseitige Testproblem
(\ref{eq:q_quant_eins}) ist durch den kritischen Wert $\widehat{c}_{1,n}=u_{1-\alpha}\left(\frac{n}{n_{1}}\widehat{\sigma}_{1,n}^{2}+\frac{n}{n_{2}}\widehat{\sigma}_{2,n}^{2}\right)^{\frac{1}{2}}$
gegeben, wobei 
\[
\widehat{\sigma}_{1,n}^{2}=\frac{1}{n_{1}-1}\sum_{i=1}^{n_{1}}\left(\mathbf{1}_{\left(-\infty,q\right]}\left(X_{i}\right)-\frac{1}{n_{1}}\sum_{j=1}^{n_{1}}\mathbf{1}_{\left(-\infty,q\right]}\left(X_{j}\right)\right)^{2}
\]
 und
\[
\widehat{\sigma}_{2,n}^{2}=\frac{1}{n_{2}-1}\sum_{i=1}^{n_{2}}\left(\mathbf{1}_{\left(-\infty,q\right]}\left(Y_{i}\right)-\frac{1}{n_{2}}\sum_{j=1}^{n_{2}}\mathbf{1}_{\left(-\infty,q\right]}\left(Y_{j}\right)\right)^{2}
\]
 sind. Der Test $\widehat{\varphi}_{2n}$ für das zweiseitige Testproblem
(\ref{eq:q_quant_zwei}) wird durch den kritischen Wert $\widehat{c}_{2,n}=u_{1-\alpha}\left(\frac{n}{n_{1}}\widehat{\sigma}_{1,n}^{2}+\frac{n}{n_{2}}\widehat{\sigma}_{2,n}^{2}\right)^{\frac{1}{2}}$
festgelegt. \end{anwendung}\begin{beisp}[Wilcoxon Test]\label{WT_first}Sei
$\left(\Omega_{1},\mathcal{A}_{1}\right)=\left(\Omega_{2},\mathcal{A}_{2}\right)=\left(\mathbb{R},\mathcal{B}\right)$.
Es wird vorausgesetzt, dass die nichtparametrischen Familien $\mathcal{P}\subset\mathcal{M}_{1}\left(\mathbb{R},\mathcal{B}\right)$
und $\mathcal{Q}\subset\mathcal{M}_{1}\left(\mathbb{R},\mathcal{B}\right)$
nur die stetigen Wahrscheinlichkeitsmaße enthalten. Es bezeichne $k:\mathcal{P}\otimes\mathcal{Q}\rightarrow\mathbb{R},\,P\otimes Q\mapsto\int1_{\left\{ \left(x,y\right)\in\mathbb{R}^{2}:x\geq y\right\} }dP\otimes Q$
das Wilcoxon Funktional. Für jedes stetiges Wahrscheinlichkeitsmaß
$P\in\mathcal{M}_{1}\left(\mathbb{R},\mathcal{B}\right)$ gilt $k\left(P\otimes P\right)=\frac{1}{2}$.
Anstelle der Testprobleme (\ref{eins_a_b_tp}) und (\ref{zweis_a_b_tp})
betrachtet man häufig das einseitige Testproblem mit der verkleinerten
Hypothese 
\[
\widetilde{H}_{2}=\left\{ P\otimes Q\in\mathcal{P}\otimes\mathcal{Q}:\,k\left(P\otimes Q\right)=\frac{1}{2}\:\mbox{und}\:P=Q\right\} 
\]
gegen $K_{1}$ und das zweiseitige Testproblem $\widetilde{H}_{2}$
gegen $K_{2}$. Wir beschränken uns auf das einseitige Testproblem
$\widetilde{H}_{2}$ gegen $K_{1}$. Das zweiseitige Testproblem $\widetilde{H}_{2}$
gegen $K_{2}$ kann analog behandelt werden. Das Wilcoxon Funktional
ist differenzierbar an jeder Stelle $P_{0}\otimes Q_{0}\in\mathcal{P}\otimes\mathcal{Q}$
und die Abbildung 
\[
\widetilde{k}\left(P_{0}\otimes Q_{0}\right):\mathbb{R}^{2}\rightarrow\mathbb{R},\,\left(x,y\right)\mapsto Q_{0}\left(\left[x,+\infty\right)\right)+P_{0}\left(\left(-\infty,y\right]\right)-2k\left(P_{0}\otimes Q_{0}\right)
\]
 ist ein Gradient von $k$ an der Stelle $P_{0}\otimes Q_{0}$, vgl.
Anwendung \ref{wilcoxon}. Es wird außerdem vorausgesetzt, dass $\widetilde{k}\left(P_{0}\otimes Q_{0}\right)$
bereits der kanonische Gradient von $k$ an jeder Stelle $P_{0}\otimes Q_{0}\in\widetilde{H}_{2}$
ist, d.h. es gilt $\widetilde{k}\left(P_{0}\otimes Q_{0}\right)\in T\left(P_{0}\otimes Q_{0},\mathcal{P}\otimes\mathcal{Q}\right)$
für alle $P_{0}\otimes Q_{0}\in\widetilde{H}_{2}$. Sei nun $P_{0}\otimes Q_{0}\in\widetilde{H}_{2}$
beliebig aber fest gewählt. Für die Teststatistik $T_{n}$ erhält
man nach Beispiel \ref{Ts_WF} 
\begin{eqnarray*}
T_{n} & = & \frac{\sqrt{n}}{n_{1}}\sum_{i=1}^{n_{1}}Q_{0}\left(\left(-\infty,X_{i}\right]\right)+\frac{\sqrt{n}}{n_{2}}\sum_{j=1}^{n_{2}}P_{0}\left(\left[Y_{j},+\infty\right)\right)-2\sqrt{n}k\left(P_{0}\otimes Q_{0}\right)\\
 & = & \frac{\sqrt{n}}{n_{1}}\sum_{i=1}^{n_{1}}F_{Q_{0}}\left(X_{i}\right)+\frac{\sqrt{n}}{n_{2}}\sum_{j=1}^{n_{2}}\left(1-F_{P_{0}}\left(Y_{j}\right)\right)-\sqrt{n}\\
 & = & \frac{\sqrt{n}}{n_{1}}\sum_{i=1}^{n_{1}}F_{Q_{0}}\left(X_{i}\right)-\frac{\sqrt{n}}{n_{2}}\sum_{j=1}^{n_{2}}F_{P_{0}}\left(Y_{j}\right).
\end{eqnarray*}
Unter der Hypothese $\widetilde{H}_{2}$ gilt die Gleichheit $F_{P_{0}}=F_{Q_{0}}$
nach Voraussetzung. Die Verteilungsfunktion $F_{P_{0}}$ stellt einen
unbekannten unendlichdimensionalen Parameter dar. Durch den Übergang
zu den Rängen wird dieser Parameter eliminiert, vgl. \cite{Hajek:1999},
\cite{Janssen:1999a}, \cite{Janssen:1997} und \cite{Janssen:2000}.
Sei $\left(Z_{1},\ldots,Z_{n}\right)=\left(X_{1},\ldots,X_{n_{1}},Y_{1},\ldots,Y_{n_{2}}\right)$
die gesamte Stichprobe. Mit $R_{n}=\left(R_{n,1},\ldots,R_{n,n}\right)$
werden die gemeinsamen Ränge von Zufallsvariablen $\left(Z_{1},\ldots,Z_{n}\right)$
bezeichnet. Die unbekannte Verteilungsfunktion $F_{P_{0}}=F_{Q_{0}}$
wird durch die empirische Verteilungsfunktion 
\[
\widehat{F}\left(Z_{i}\right)=\frac{1}{n}R_{n,i}
\]
 ersetzt. Die lineare Rangstatistik $\widehat{T}_{n}$ wird durch
\[
\widehat{T}_{n}=\frac{1}{n_{1}\sqrt{n}}\sum_{i=1}^{n_{1}}R_{n,i}-\frac{1}{n_{2}\sqrt{n}}\sum_{j=1+n_{1}}^{n}R_{n,j}
\]
 definiert. Es gilt 
\[
\lim_{n\rightarrow\infty}\int\left(T_{n}-\widehat{T}_{n}\right)^{2}dP_{0}^{n_{1}}\otimes Q_{0}^{n_{2}}=\lim_{n\rightarrow\infty}\int\left(T_{n}-\widehat{T}_{n}\right)^{2}dP_{0}^{n}=0
\]
 nach \cite{Janssen:1998}, Satz 9.3 oder \cite{Hajek:1999}, Section
6.1. Hieraus folgt 
\[
\lim_{n\rightarrow\infty}P_{0}^{n_{1}}\otimes Q_{0}^{n_{2}}\left(\left|T_{n}-\widehat{T}_{n}\right|>\varepsilon\right)=0
\]
 für alle $\varepsilon>0$. Sei $S_{n}$ die symmetrische Gruppe,
d.h. die Gruppe aller Permutationen auf der Menge $\left\{ 1,\ldots,n\right\} $.
Sei nun 
\begin{equation}
\widehat{\varphi}_{1n}=\left\{ \begin{array}{cccc}
1 &  & >\\
\widehat{\gamma}_{n} & \widehat{T}_{n} & = & \widehat{c}_{1,n}(\alpha)\\
0 &  & <
\end{array}\right.\label{phi_1n_dach_test}
\end{equation}
 der Rangtest zum Niveau $\alpha,$ wobei $\widehat{c}_{1,n}(\alpha)$
das $(1-\alpha)$-Quantil der Verteilung der Statistik 
\[
S_{n}\rightarrow\mathbb{R},\pi\mapsto\frac{1}{n_{1}\sqrt{n}}\sum_{i=1}^{n_{1}}\pi(i)-\frac{1}{n_{2}\sqrt{n}}\sum_{j=1+n_{1}}^{n}\pi(j)
\]
 unter der Gleichverteilung auf der symmetrischen Gruppe $S_{n}$
ist. Man erhält die stochastische Konvergenz 
\[
\lim_{n\rightarrow\infty}P_{0}^{n_{1}}\otimes Q_{0}^{n_{2}}\left(\left|\widehat{c}_{1,n}(\alpha)-c_{1}\right|>\varepsilon\right)=0
\]
 für alle $\varepsilon>0$ nach \cite{Janssen:1998}, Satz 11.1 und
Lemma 11.3, vgl. ebenfalls \cite{Janssen:1997} und \cite{Janssen:1999a}.
Nach Satz 11.1 aus \cite{Janssen:1998} folgt außerdem 
\[
\lim_{n\rightarrow\infty}P_{0}^{n_{1}}\otimes Q_{0}^{n_{2}}\left(\left\{ \widehat{\varphi}_{1n}\in\left(0,1\right)\right\} \right)=\lim_{n\rightarrow\infty}P_{0}^{n_{1}}\otimes Q_{0}^{n_{2}}\left(\left\{ \widehat{T}_{n}=\widehat{c}_{1,n}(\alpha)\right\} \right)=0.
\]
Nach Lemma \ref{benach_imp_asymp_eq} und Bemerkung \ref{asympt_banachb_equiv_tests}
ergibt sich nun, dass die Testfolgen (\ref{phi_1n_dach_test}) und
(\ref{erinnerung_asymp_einseit_testfolge}) an jeder Stelle $P_{0}\otimes Q_{0}\in\widetilde{H}_{2}$
lokal asymptotisch äquivalent sind. Die Teststatistik $\widehat{T}_{n}$
lässt sich folgendermaßen äquivalent umformen: 
\begin{eqnarray*}
\widehat{T}_{n} & = & \frac{1}{n_{1}\sqrt{n}}\sum_{i=1}^{n_{1}}R_{n,i}-\frac{1}{n_{2}\sqrt{n}}\sum_{j=1+n_{1}}^{n}R_{n,j}\\
 & = & \frac{1}{n_{1}\sqrt{n}}\sum_{i=1}^{n_{1}}R_{n,i}-\frac{1}{n_{2}\sqrt{n}}\left(\frac{n\left(n+1\right)}{2}-\sum_{i=1}^{n_{1}}R_{n,i}\right)\\
 & = & \left(\frac{1}{n_{1}\sqrt{n}}+\frac{1}{n_{2}\sqrt{n}}\right)\sum_{i=1}^{n_{1}}R_{n,i}-\frac{n\left(n+1\right)}{2n_{2}\sqrt{n}}.
\end{eqnarray*}
Durch die äquivalente Transformation erhält man
\[
\widehat{\varphi}_{1n}=\left\{ \begin{array}{cccc}
1 &  & >\\
\widetilde{\gamma}_{n} & \sum_{i=1}^{n_{1}}R_{n,i} & = & \widetilde{c}_{1,n}(\alpha)\\
0 &  & <
\end{array}\right.
\]
 für geeignete $\widetilde{\gamma}_{n}$ und $\widetilde{c}_{1,n}(\alpha)$.
Der kritische Wert $\widetilde{c}_{1,n}(\alpha)$ lässt sich außerdem
als der $\left(1-\alpha\right)$-Quantil der Verteilung der Statistik
\[
S_{n}\rightarrow\mathbb{R},\pi\mapsto\sum_{i=1}^{n_{1}}\pi(i)
\]
 unter der Gleichverteilung auf der symmetrischen Gruppe $S_{n}$
bestimmen. Der Test $\widehat{\varphi}_{1n}$ ist also der übliche
Zweistichproben-Wilcoxon-Test. Es gibt studentisierte Versionen des
Wilcoxon-Tests, die unter der allgemeinen Nullhypothese $H_{2}=\left\{ P\otimes Q\in\mathcal{P}\otimes\mathcal{Q}:\,k\left(P\otimes Q\right)=\frac{1}{2}\right\} $
als Permutationstests ausgeführt werden, vgl. \cite{Hajek:1999},
\cite{Janssen:1997} und \cite{Janssen:1999a}. \end{beisp}Es bezeichne
$\mathbb{P}(\mathbb{N})$ die Potenzmenge von $\mathbb{N}$. Seien
$\mathcal{P}\subset\mathcal{M}_{1}\left(\mathbb{R},\mathcal{B}\right)$
und $\mathcal{Q}\subset\mathcal{M}_{1}\left(\mathbb{N},\mathbb{P}(\mathbb{N})\right)$
zwei Familien von Wahrscheinlichkeitsmaßen. Man betrachte die randomisierte
Summe 
\[
S=\sum_{i=1}^{Y_{1}}X_{i}\,.
\]
 Es gilt dann $E\left(S\right)=E\left(X_{1}\right)E\left(Y_{1}\right)=E_{P}\left(\mbox{Id}_{\mathbb{R}}\right)E_{Q}\left(\mbox{Id}_{\mathbb{N}}\right)$.
Die einseitigen und zweiseitigen Testprobleme für das statistische
Funktional $k:\mathcal{P}\otimes\mathcal{Q}\rightarrow\mathbb{R},\,P\otimes Q\mapsto E_{P}\left(\mbox{Id}_{\mathbb{R}}\right)E_{Q}\left(\mbox{Id}_{\mathbb{N}}\right)$
motivieren das nachfolgende Beispiel.\begin{beisp}[Produkt zweier von Mises Funktionale]\label{prod_mises_funk}Seien
$k_{1}:\mathcal{P}\rightarrow\mathbb{R},\,P\mapsto\int f_{1}\,dP$
und $k_{2}:\mathcal{Q}\rightarrow\mathbb{R},\,Q\mapsto\int f_{2}\,dQ$
zwei von Mises Funktionale. Die Voraussetzungen aus Beispiel \ref{summ_mises_funk}
seien erfüllt. Das Funktional $k_{1}$ ist dann differenzierbar an
jeder Stelle $P_{0}\in\mathcal{P}$ mit dem kanonischen Gradienten
$f_{1}-E_{P_{0}}\left(f_{1}\right)$ und das Funktional $k_{2}$ ist
differenzierbar an jeder Stelle $Q_{0}\in\mathcal{Q}$ mit dem kanonischen
Gradienten $f_{2}-E_{Q_{0}}\left(f_{2}\right)$. Nach Beispiel \ref{bsp_zusamm_gesetzte_funk}
ist das statistische Funktional 
\[
k:\mathcal{P}\otimes\mathcal{Q}\rightarrow\mathbb{R},\,P\otimes Q\mapsto k_{1}\left(P\right)k_{2}\left(Q\right)
\]
differenzierbar an jeder Stelle $P_{0}\otimes Q_{0}\in\mathcal{P}\otimes\mathcal{Q}$
mit dem kanonischen Gradienten
\[
\widetilde{k}:\left(\omega_{1},\omega_{2}\right)\mapsto E_{Q_{0}}\left(f_{2}\right)\left(f_{1}\left(\omega_{1}\right)-E_{P_{0}}\left(f_{1}\right)\right)+E_{P_{0}}\left(f_{1}\right)\left(f_{2}\left(\omega_{2}\right)-E_{Q_{0}}\left(f_{2}\right)\right).
\]
 Hieraus folgt $\widetilde{k}_{1}=E_{Q_{0}}\left(f_{2}\right)\left(f_{1}-E_{P_{0}}\left(f_{1}\right)\right)$
und $\widetilde{k}_{2}=E_{P_{0}}\left(f_{1}\right)\left(f_{2}-E_{Q_{0}}\left(f_{2}\right)\right)$.
Die Testprobleme (\ref{eins_a_b_tp}) und (\ref{zweis_a_b_tp}) für
$a\neq0$ werden betrachtet. Es bezeichne wieder $\overline{f}_{1,n}=\frac{1}{n_{1}}\sum_{i=1}^{n_{1}}f_{1}\circ X_{i}$
und $\overline{f}_{2,n}=\frac{1}{n_{2}}\sum_{i=1}^{n_{2}}f_{2}\circ Y_{i}$.
Sei $P_{0}\otimes Q_{0}\in H_{2}$ beliebig aber fest gewählt. Dann
gilt $a=k\left(P_{0}\otimes Q_{0}\right)=E_{P_{0}}\left(f_{1}\right)E_{Q_{0}}\left(f_{2}\right).$
Für die Teststatistik $T_{n}$ erhält man nun 
\begin{eqnarray*}
T_{n} & = & \frac{\sqrt{n}}{n_{1}}\sum_{i=1}^{n_{1}}E_{Q_{0}}\left(f_{2}\right)\left(f_{1}\left(X_{i}\right)-E_{P_{0}}\left(f_{1}\right)\right)\\
 &  & +\frac{\sqrt{n}}{n_{2}}\sum_{i=1}^{n_{2}}E_{P_{0}}\left(f_{1}\right)\left(f_{2}\left(Y_{i}\right)-E_{Q_{0}}\left(f_{2}\right)\right)\\
 & = & E_{Q_{0}}\left(f_{2}\right)\frac{\sqrt{n}}{n_{1}}\sum_{i=1}^{n_{1}}f_{1}\left(X_{i}\right)-\sqrt{n}E_{P_{0}}\left(f_{1}\right)E_{Q_{0}}\left(f_{2}\right)\\
 &  & +E_{P_{0}}\left(f_{1}\right)\frac{\sqrt{n}}{n_{2}}\sum_{i=1}^{n_{2}}f_{2}\left(Y_{i}\right)-\sqrt{n}E_{P_{0}}\left(f_{1}\right)E_{Q_{0}}\left(f_{2}\right)\\
 & = & \sqrt{n}E_{Q_{0}}\left(f_{2}\right)\overline{f}_{1,n}+\sqrt{n}E_{P_{0}}\left(f_{1}\right)\overline{f}_{2,n}-2\sqrt{n}a.
\end{eqnarray*}
 Die Werte $E_{P_{0}}\left(f_{1}\right)$ und $E_{Q_{0}}\left(f_{2}\right)$
sind unbekannt und werden durch $\overline{f}_{1,n}$ und $\overline{f}_{2,n}$
geschätzt. Die Teststatistik $\widehat{T}_{n}$ mit geschätzten Erwartungswerten
ist durch 
\[
\widehat{T}_{n}=\sqrt{n}\,\overline{f}_{1,n}\overline{f}_{2,n}-\sqrt{n}a
\]
definiert. Die Teststatistik $\widehat{T}_{n}$ läßt sich auch mit
Hilfe der kanonischen Zwei\-stich\-proben-U-Statistik für das statistische
Funktional $k$ herleiten. Die Zufallsvariable $X_{1}Y_{1}$ ist ein
erwartungstreuer Schätzer für $k$. Für die Zwei\-stich\-proben-U-Statistik
mit der Abbildung $\psi:\left(x,y\right)\mapsto xy$ als Kern erhält
man dann 
\[
\frac{1}{n_{1}n_{2}}\sum_{i=1}^{n_{1}}\sum_{j=1}^{n_{2}}X_{i}Y_{j}=\left(\frac{1}{n_{1}}\sum_{i=1}^{n_{1}}X_{i}\right)\left(\frac{1}{n_{2}}\sum_{j=1}^{n_{2}}Y_{j}\right)=\overline{f}_{1,n}\overline{f}_{2,n}.
\]
 Die asymptotische Verteilung der Teststatistik $\widehat{T}_{n}$
lässt sich mit Delta Methode bestimmen, vgl. \cite{Vaart:1998}, Seite
26, Theorem 3.1 oder \cite{Witting:1995}, Seite 107, Satz 5.107 (Cram\'er).
Nach dem starken Gesetz der großen Zahlen ergibt sich die fast sichere
Konvergenz $\left(\overline{f}_{1,n},\overline{f}_{2,n}\right)\rightarrow\left(E_{P_{0}}\left(f_{1}\right),E_{Q_{0}}\left(f_{2}\right)\right)$
für $n\rightarrow\infty$. Mit Cram\'er-Wold-Device zeigt man die
schwache Konvergenz 
\begin{eqnarray*}
\textrm{} &  & \mathcal{L}\left(\left.\sqrt{n}\left(\left(\begin{array}{c}
\overline{f}_{1,n}\\
\overline{f}_{2,n}
\end{array}\right)-\left(\begin{array}{c}
E_{P_{0}}\left(f_{1}\right)\\
E_{Q_{0}}\left(f_{2}\right)
\end{array}\right)\right)\right|P_{0}^{n_{1}}\otimes Q_{0}^{n_{2}}\right)\\
 & \rightarrow & N\left(\left(\begin{array}{c}
0\\
0
\end{array}\right),\left(\begin{array}{cc}
\frac{1}{1-d}\mbox{Var}_{P_{0}}\left(f_{1}\right) & 0\\
0 & \frac{1}{d}\mbox{Var}_{Q_{0}}\left(f_{2}\right)
\end{array}\right)\right)
\end{eqnarray*}
für $n\rightarrow\infty$. Die Abbildung $H_{1}:\mathbb{R}^{2}\rightarrow\mathbb{R},\left(x,y\right)\mapsto xy$
ist differenzierbar in $\mathbb{R}^{2}$ mit der Jacobi-Matrix 
\[
JH_{1}\left(x,y\right)=\left(\begin{array}{c}
y\\
x
\end{array}\right).
\]
Mit Delta Methode erhält man für $n\rightarrow\infty$ die schwache
Konvergenz
\begin{eqnarray*}
 &  & \mathcal{L}\left(\left.\widehat{T}_{n}\right|P_{0}^{n_{1}}\otimes Q_{0}^{n_{2}}\right)\\
 & = & \mathcal{L}\left(\left.\sqrt{n}\,\overline{f}_{1,n}\overline{f}_{2,n}-\sqrt{n}E_{P_{0}}\left(f_{1}\right)E_{Q_{0}}\left(f_{2}\right)\right|P_{0}^{n_{1}}\otimes Q_{0}^{n_{2}}\right)\\
 & = & \mathcal{L}\left(\left.\sqrt{n}\left(H_{1}\left(\overline{f}_{1,n},\overline{f}_{2,n}\right)-H_{1}\left(E_{P_{0}}\left(f_{1}\right),E_{Q_{0}}\left(f_{2}\right)\right)\right)\right|P_{0}^{n_{1}}\otimes Q_{0}^{n_{2}}\right)\\
 & \rightarrow & N\left(0,\,E_{Q_{0}}\left(f_{2}\right)^{2}\frac{1}{1-d}\mbox{Var}_{P_{0}}\left(f_{1}\right)+E_{P_{0}}\left(f_{1}\right)^{2}\frac{1}{d}\mbox{Var}_{Q_{0}}\left(f_{2}\right)\right)\\
 & = & N\left(0,\,\frac{1}{1-d}\int\widetilde{k}_{1}^{2}\,dP_{0}+\frac{1}{d}\int\widetilde{k}_{2}^{2}\,dQ_{0}\right).
\end{eqnarray*}
Die Testfolgen 
\begin{equation}
\widetilde{\varphi}_{1n}=\left\{ \begin{array}{cccc}
1 &  & >\\
 & \widehat{T}_{n} &  & c_{1}\\
0 &  & \leq
\end{array}\right.\label{pmf_sch_1}
\end{equation}
 und 
\begin{equation}
\widetilde{\varphi}_{2n}=\left\{ \begin{array}{cccc}
1 &  & >\\
 & \left|\widehat{T}_{n}\right| &  & c_{2}\\
0 &  & \leq
\end{array}\right.\label{pmf_sch_2}
\end{equation}
erfüllen dann die Bedingung $\lim_{n\rightarrow\infty}\int\widetilde{\varphi}_{in}\,dP_{0}^{n_{1}}\otimes Q_{0}^{n_{2}}=\alpha$
für $i=1,2$. Die kritischen Werte $c_{1}$ und $c_{2}$ sind in (\ref{erinnerung_c1})
und (\ref{erinnerung_c2})  definiert. Es wird nun bewiesen, dass
die Testfolgen (\ref{pmf_sch_1}) und (\ref{erinnerung_asymp_einseit_testfolge})
bzw. (\ref{pmf_sch_2}) und (\ref{erinnerung_asymp_zweiseit_testfolge})
lokal asymptotisch äquivalent für das Testproblem (\ref{eins_a_b_tp})
bzw. (\ref{zweis_a_b_tp}) an der Stelle $P_{0}\otimes Q_{0}$ sind.
Sei $t\mapsto P_{t}\otimes Q_{t}$ eine $L_{2}\left(P_{0}\otimes Q_{0}\right)$-differenzierbare
Kurve in $\mathcal{P}\otimes\mathcal{Q}$ mit Tangente $g\in L_{2}^{(0)}\left(P_{0}\otimes Q_{0}\right).$
Sei $\int g^{2}\,dP_{0}\otimes Q_{0}\neq0$ ohne Einschränkung, sonst
folgt die Behauptung nach Satz \ref{niveau_alpha_trivial}. Die Tangente
$g$ besitzt nach Satz \ref{L2Produkt} die Darstellung $g=g_{1}\circ\pi_{1}+g_{2}\circ\pi_{2}$
mit $g_{1}\in L_{2}^{(0)}\left(P_{0}\right)$ und $g_{2}\in L_{2}^{(0)}\left(Q_{0}\right)$.
Sei 
\[
\Lambda_{n}=\frac{1}{\sqrt{n}}\left(\sum_{i=1}^{n_{1}}g_{1}\left(X_{i}\right)+\sum_{i=1}^{n_{2}}g_{2}\left(Y_{i}\right)\right).
\]
Die asymptotische Gütefunktion der Testfolgen (\ref{erinnerung_asymp_einseit_testfolge})
und (\ref{erinnerung_asymp_zweiseit_testfolge}) entlang der impliziten
Alternativen und Hypothesen $\left(P_{t_{n}}^{n_{1}}\otimes Q_{t_{n}}^{n_{2}}\right)_{n\in\mathbb{N}}\in\mathcal{F}_{2}$
ist eindeutig bestimmt durch die gemeinsame asymptotische Verteilung
\[
\lim_{n\rightarrow\infty}\mathcal{L}\left(\left.\left(T_{n},\Lambda_{n}\right)\right|P_{0}^{n_{1}}\otimes Q_{0}^{n_{2}}\right)
\]
von $T_{n}$ und $\Lambda_{n}$. Die asymptotische Gütefunktion der
Testfolgen (\ref{pmf_sch_1}) und (\ref{pmf_sch_2}) entlang der impliziten
Alternativen und Hypothesen $\left(P_{t_{n}}^{n_{1}}\otimes Q_{t_{n}}^{n_{2}}\right)_{n\in\mathbb{N}}\in\mathcal{F}_{2}$
ist ebenfalls eindeutig bestimmt durch die gemeinsame asymptotische
Verteilung 
\[
\lim_{n\rightarrow\infty}\mathcal{L}\left(\left.\left(\widehat{T}_{n},\Lambda_{n}\right)\right|P_{0}^{n_{1}}\otimes Q_{0}^{n_{2}}\right)
\]
von $\widehat{T}_{n}$ und $\Lambda_{n}$, vgl. Abschnitte \ref{sec:Asymptotische-Eigenschaften-der}
und \ref{sec:zw:twsts}, insbesondere Satz \ref{asymp-guete} und
Satz \ref{asymp-guete_zw}. Es reicht also 
\begin{equation}
\lim_{n\rightarrow\infty}\mathcal{L}\left(\left.\left(T_{n},\Lambda_{n}\right)\right|P_{0}^{n_{1}}\otimes Q_{0}^{n_{2}}\right)=\lim_{n\rightarrow\infty}\mathcal{L}\left(\left.\left(\widehat{T}_{n},\Lambda_{n}\right)\right|P_{0}^{n_{1}}\otimes Q_{0}^{n_{2}}\right)\label{eq:asymp_gl_der_vert}
\end{equation}
 nachzuweisen. Nach Satz \ref{gem-asymp-vert} erhält man
\begin{equation}
\lim_{n\rightarrow\infty}\mathcal{L}\left(\left.\left(\begin{array}{c}
T_{n}\\
\Lambda_{n}
\end{array}\right)\right|P_{0}^{n_{1}}\otimes Q_{0}^{n_{2}}\right)=N\left(\left(\begin{array}{c}
0\\
0
\end{array}\right),\left(\begin{array}{cc}
\sigma_{1}^{2} & \sigma_{12}\\
\sigma_{12} & \sigma_{2}^{2}
\end{array}\right)\right),\label{eq:asymp_vert_von_t_n_lambda_n}
\end{equation}
 wobei 
\begin{eqnarray*}
\sigma_{1}^{2} & = & \frac{1}{1-d}\int\widetilde{k}_{1}^{2}\,dP_{0}+\frac{1}{d}\int\widetilde{k}_{2}^{2}\,dQ_{0}\\
 & = & \frac{1}{1-d}E_{Q_{0}}\left(f_{2}\right)^{2}\mbox{Var}_{P_{0}}\left(f_{1}\right)+\frac{1}{d}E_{P_{0}}\left(f_{1}\right)^{2}\mbox{Var}_{Q_{0}}\left(f_{2}\right),
\end{eqnarray*}
\[
\sigma_{2}^{2}=\left(1-d\right)\int g_{1}^{2}\,dP_{0}+d\int g_{2}^{2}\,dQ_{0}=\left(1-d\right)\mbox{Var}_{P_{0}}\left(g_{1}\right)+d\mbox{Var}_{Q_{0}}\left(g_{2}\right),
\]
\[
\sigma_{12}=\int\widetilde{k}g\,dP_{0}\otimes Q_{0}=E_{Q_{0}}\left(f_{2}\right)\mbox{Cov}_{P_{0}}\left(f_{1},g_{1}\right)+E_{P_{0}}\left(f_{1}\right)\mbox{Cov}_{Q_{0}}\left(f_{2},g_{2}\right)
\]
 sind. Die asymptotische Verteilung von $\left(\widehat{T}_{n},\Lambda_{n}\right)$
unter $P_{0}^{n_{1}}\otimes Q_{0}^{n_{2}}$ wird mit Delta Methode
bestimmt. Es bezeichne 
\[
Y_{n}=\left(\overline{f}_{1,n},\overline{f}_{2,n},\frac{1}{n}\sum_{i=1}^{n_{1}}g_{1}\left(X_{i}\right),\frac{1}{n}\sum_{i=1}^{n_{2}}g_{2}\left(Y_{i}\right)\right).
\]
 Es gilt dann 
\[
E\left(Y_{n}\right)=\left(E_{P_{0}}\left(f_{1}\right),E_{Q_{0}}\left(f_{2}\right),0,0\right).
\]
 Nach dem starken Gesetz der großen Zahlen ergibt sich die fast sichere
Konvergenz $Y_{n}\rightarrow\left(E_{P_{0}}\left(f_{1}\right),E_{Q_{0}}\left(f_{2}\right),0,0\right)$
für $n\rightarrow\infty$. Mit Cram\'er-Wold-Device erhält man die
schwache Konvergenz 
\[
\lim_{n\rightarrow\infty}\mathcal{L}\left(\left.\sqrt{n}\left(Y_{n}-E\left(Y_{n}\right)\right)\right|P_{0}^{n_{1}}\otimes Q_{0}^{n_{2}}\right)=N\left(0,\Sigma\right).
\]
Die Kovarianzmatrix $\Sigma$ ergibt sich als
\begin{eqnarray*}
\Sigma & = & \left(\begin{array}{cccc}
\frac{1}{1-d}\mbox{Var}_{P_{0}}\left(f_{1}\right) & 0 & \mbox{Cov}_{P_{0}}\left(f_{1},g_{1}\right) & 0\\
0 & \frac{1}{d}\mbox{Var}_{Q_{0}}\left(f_{2}\right) & 0 & \mbox{Cov}_{Q_{0}}\left(f_{2},g_{2}\right)\\
\mbox{Cov}_{P_{0}}\left(f_{1},g_{1}\right) & 0 & \left(1-d\right)\mbox{Var}_{P_{0}}\left(g_{1}\right) & 0\\
0 & \mbox{Cov}_{Q_{0}}\left(f_{2},g_{2}\right) & 0 & d\mbox{Var}_{Q_{0}}\left(g_{2}\right)
\end{array}\right).
\end{eqnarray*}
Die Abbildung $H_{2}:\mathbb{R}^{4}\rightarrow\mathbb{R}^{2},\left(x_{1},x_{2},x_{3},x_{4}\right)\mapsto\left(x_{1}x_{2},x_{3}+x_{4}\right)$
ist differenzierbar in $\mathbb{R}^{4}$ mit der Jacobi-Matrix 
\[
JH_{2}\left(x_{1},x_{2},x_{3},x_{4}\right)=\left(\begin{array}{cccc}
x_{2} & x_{1} & 0 & 0\\
0 & 0 & 1 & 1
\end{array}\right).
\]
Sei $B:=JH_{2}\left(\left(E_{P_{0}}\left(f_{1}\right),E_{Q_{0}}\left(f_{2}\right),0,0\right)\right)$.
Nach Delta Methode ergibt sich die schwache Konvergenz 
\begin{eqnarray*}
 &  & \lim_{n\rightarrow\infty}\mathcal{L}\left(\left.\left(\widehat{T}_{n},\Lambda_{n}\right)\right|P_{0}^{n_{1}}\otimes Q_{0}^{n_{2}}\right)\\
 & = & \lim_{n\rightarrow\infty}\mathcal{L}\left(\left.\sqrt{n}\left(H_{2}\left(Y_{n}\right)-H_{2}\left(E\left(Y_{n}\right)\right)\right)\right|P_{0}^{n_{1}}\otimes Q_{0}^{n_{2}}\right)\\
 & = & N\left(0,\,B\Sigma B^{t}\right).
\end{eqnarray*}
Man erhält außerdem $B\Sigma B^{t}=\left(\begin{array}{cc}
\sigma_{1}^{2} & \sigma_{12}\\
\sigma_{12} & \sigma_{2}^{2}
\end{array}\right)$, vgl. (\ref{eq:asymp_vert_von_t_n_lambda_n}). Die Gleichheit (\ref{eq:asymp_gl_der_vert})
ist somit bewiesen. Die Werte $c_{1}$ und $c_{2}$ sind unbekannt
und müssen geschätzt werden. Die Schätzer $\overline{f}_{1,n}$, $\overline{f}_{2,n}$,
$\widehat{\sigma}_{n}\left(f_{1}\right)^{2}$ und $\widehat{\sigma}_{n}\left(f_{2}\right)^{2}$
seien wie in Beispiel \ref{summ_mises_funk} definiert. Die geschätzten
kritischen Werte $\widehat{c}_{1,n}$ und $\widehat{c}_{2,n}$ sind
definiert durch 
\[
\widehat{c}_{1,n}=u_{1-\alpha}\left(\frac{n}{n_{1}}\left(\overline{f}_{2,n}\right)^{2}\widehat{\sigma}_{n}\left(f_{1}\right)^{2}+\frac{n}{n_{2}}\left(\overline{f}_{1,n}\right)^{2}\widehat{\sigma}_{n}\left(f_{2}\right)^{2}\right)^{\frac{1}{2}}
\]
 und
\[
\widehat{c}_{1,n}=u_{1-\frac{\alpha}{2}}\left(\frac{n}{n_{1}}\left(\overline{f}_{2,n}\right)^{2}\widehat{\sigma}_{n}\left(f_{1}\right)^{2}+\frac{n}{n_{2}}\left(\overline{f}_{1,n}\right)^{2}\widehat{\sigma}_{n}\left(f_{2}\right)^{2}\right)^{\frac{1}{2}}.
\]
 Nach Lemma \ref{asympt_aquivalent_testfolge}, Lemma \ref{benach_imp_asymp_eq}
und Bemerkung \ref{asympt_banachb_equiv_tests} sind die Testfolgen
\[
\widehat{\varphi}_{1n}=\left\{ \begin{array}{cccc}
1 &  & >\\
 & \widehat{T}_{n} &  & \widehat{c}_{1,n}\\
0 &  & \leq
\end{array}\right.\qquad\mbox{und}\qquad\widehat{\varphi}_{2n}=\left\{ \begin{array}{cccc}
1 &  & >\\
 & \left|\widehat{T}_{n}\right| &  & \widehat{c}_{2,n}\\
0 &  & \leq
\end{array}\right.
\]
 asymptotisch äquivalent zu den Testfolgen $\widetilde{\varphi}_{1n}$
und $\widetilde{\varphi}_{2n}$ für das Testproblem (\ref{eins_a_b_tp})
bzw. (\ref{zweis_a_b_tp}), weil $\widehat{c}_{1,n}\rightarrow c_{1}$
und $\widehat{c}_{2,n}\rightarrow c_{2}$ für $n\rightarrow\infty$
fast sicher konvergieren, vgl. Beispiel \ref{summ_mises_funk}. Die
Testfolge $\widehat{\varphi}_{1n}$ ist also asymptotisch äquivalent
zu der Testfolge (\ref{erinnerung_asymp_einseit_testfolge}) für das
Testproblem (\ref{eins_a_b_tp}) und die Testfolge $\widehat{\varphi}_{2n}$
ist asymptotisch äquivalent zu der Testfolge (\ref{erinnerung_asymp_zweiseit_testfolge})
für das Testproblem (\ref{zweis_a_b_tp}).\end{beisp}\begin{anwendung}[Testen des Wilcoxon Funktionals]In
dieser Anwendung kehren wir zu dem Wilcoxon Funktional $k:\mathcal{P}\otimes\mathcal{Q}\rightarrow\mathbb{R},\,P\otimes Q\mapsto\int1_{\left\{ \left(x,y\right)\in\mathbb{R}^{2}:x\geq y\right\} }dP\otimes Q$
zurück. Es gilt $k\left(P\otimes Q\right)=\int F_{Q}dP$. Sei $\left(\Omega_{1},\mathcal{A}_{1}\right)=\left(\Omega_{2},\mathcal{A}_{2}\right)=\left(\mathbb{R},\mathcal{B}\right)$.
Im Gegensatz zu Beispiel \ref{WT_first} wird nicht vorausgesetzt,
dass die nichtparametrischen Familien $\mathcal{P}\subset\mathcal{M}_{1}\left(\mathbb{R},\mathcal{B}\right)$
und $\mathcal{Q}\subset\mathcal{M}_{1}\left(\mathbb{R},\mathcal{B}\right)$
nur stetige Wahrscheinlichkeitsmaße enthalten. Es werden das einseitige
Testproblem
\[
H_{1}=\left\{ P\otimes Q\in\mathcal{P}\otimes\mathcal{Q}:k\left(P\otimes Q\right)\leq\frac{1}{2}\right\} 
\]
 gegen

\begin{equation}
K_{1}=\left\{ P\otimes Q\in\mathcal{P}\otimes\mathcal{Q}:k\left(P\otimes Q\right)>\frac{1}{2}\right\} \label{wt1}
\end{equation}
und das zweiseitige Testproblem 
\[
H_{2}=\left\{ P\otimes Q\in\mathcal{P}\otimes\mathcal{Q}:k\left(P\otimes Q\right)=\frac{1}{2}\right\} 
\]
 gegen

\begin{equation}
K_{2}=\left\{ P\otimes Q\in\mathcal{P}\otimes\mathcal{Q}:k\left(P\otimes Q\right)\neq\frac{1}{2}\right\} \label{wt2}
\end{equation}
betrachtet. Aus der Anwendung \ref{wilcoxon} ist es bereits bekannt,
dass das Wilcoxon Funktional an jeder Stelle $P_{0}\otimes Q_{0}\in\mathcal{P}\otimes\mathcal{Q}$
differenzierbar ist und die Abbildung 
\[
\widetilde{k}\left(P_{0}\otimes Q_{0}\right):\mathbb{R}^{2}\rightarrow\mathbb{R},\,\left(x,y\right)\mapsto Q_{0}\left(\left(-\infty,x\right]\right)+P_{0}\left(\left[y,+\infty\right)\right)-2k\left(P_{0}\otimes Q_{0}\right)
\]
 ein Gradient von $k$ an der Stelle $P_{0}\otimes Q_{0}$ ist. Es
wird vorausgesetzt, dass $\widetilde{k}\left(P_{0}\otimes Q_{0}\right)$
der kanonische Gradient von $k$ an jeder Stelle $P_{0}\otimes Q_{0}\in H_{2}$
ist, d.h. es gilt $\widetilde{k}\left(P_{0}\otimes Q_{0}\right)\in T\left(P_{0}\otimes Q_{0},\mathcal{P}\otimes\mathcal{Q}\right)$
für alle $P_{0}\otimes Q_{0}\in H_{2}$. Diese Voraussetzung besagt,
dass das Modell hinreichend groß ist. Sei nun $P_{0}\otimes Q_{0}\in H_{2}$
beliebig aber fest gewählt. Nach Beispiel \ref{Ts_WF} liegt die Teststatistik
\[
T_{n}=\frac{\sqrt{n}}{n_{1}}\sum_{i=1}^{n_{1}}Q_{0}\left(\left(-\infty,X_{i}\right]\right)+\frac{\sqrt{n}}{n_{2}}\sum_{j=1}^{n_{2}}P_{0}\left(\left[Y_{j},+\infty\right)\right)-2\sqrt{n}k\left(P_{0}\otimes Q_{0}\right)
\]
 den asymptotisch optimalen Testfolgen (\ref{erinnerung_asymp_einseit_testfolge})
und (\ref{erinnerung_asymp_zweiseit_testfolge}) zugrunde. Wegen $k\left(P_{0}\otimes Q_{0}\right)=\frac{1}{2}$
erhält man unter $H_{2}$
\[
T_{n}=\frac{\sqrt{n}}{n_{1}}\sum_{i=1}^{n_{1}}Q_{0}\left(\left(-\infty,X_{i}\right]\right)+\frac{\sqrt{n}}{n_{2}}\sum_{j=1}^{n_{2}}P_{0}\left(\left[Y_{j},+\infty\right)\right)-\sqrt{n}.
\]
 Im Folgenden wird eine adaptive Methode vorgestellt, mit der die
unbekannten Fußpunkte $P_{0}$ und $Q_{0}$ eliminiert werden können.
Die Werte $Q_{0}\left(\left(-\infty,X_{i}\right]\right)$ und $P_{0}\left(\left[Y_{j},+\infty\right)\right)$,
die in die Teststatistik $T_{n}$ hineingehen, sind unbekannt und
müssen geschätzt werden. Für $t\in\mathbb{R}$ wird $Q_{0}\left(\left(-\infty,t\right]\right)$
durch 
\[
\frac{1}{n_{2}}\sum_{i=1}^{n_{2}}\mathbf{1}_{\left(-\infty,t\right]}\left(Y_{i}\right)
\]
 geschätzt. Der Wert $P_{0}\left(\left[s,+\infty\right)\right)$ wird
für $s\in\mathbb{R}$ durch 
\[
\frac{1}{n_{1}}\sum_{i=1}^{n_{1}}\mathbf{1}_{\left[s,+\infty\right)}\left(X_{i}\right)
\]
 geschätzt. Man erhält somit die neue Teststatistik
\begin{eqnarray*}
\widehat{T}_{n} & = & \frac{\sqrt{n}}{n_{1}}\sum_{i=1}^{n_{1}}\frac{1}{n_{2}}\sum_{j=1}^{n_{2}}\mathbf{1}_{\left(-\infty,X_{i}\right]}\left(Y_{j}\right)+\frac{\sqrt{n}}{n_{2}}\sum_{j=1}^{n_{2}}\frac{1}{n_{1}}\sum_{i=1}^{n_{1}}\mathbf{1}_{\left[Y_{j},+\infty\right)}\left(X_{i}\right)-\sqrt{n}\\
 & = & 2\frac{\sqrt{n}}{n_{1}n_{2}}\sum_{i=1}^{n_{1}}\sum_{j=1}^{n_{2}}\mathbf{1}_{\left(-\infty,X_{i}\right]}\left(Y_{j}\right)-\sqrt{n}.
\end{eqnarray*}
 Weiterhin wird die Teststatistik $\widetilde{T}_{n}=\frac{1}{2}\widehat{T}_{n}$
betrachtet. Es bezeichne 
\[
\mathbb{F}_{n_{1},P_{0}}(t)=\frac{1}{n_{1}}\sum_{i=1}^{n_{1}}\mathbf{1}_{\left(-\infty,t\right]}\left(X_{i}\right)
\]
 die empirische Verteilungsfunktion von $P_{0}$ und 
\[
\mathbb{F}_{n_{2},Q_{0}}(t)=\frac{1}{n_{2}}\sum_{i=1}^{n_{2}}\mathbf{1}_{\left(-\infty,t\right]}\left(Y_{i}\right)
\]
 die empirische Verteilungsfunktion von $Q_{0}$. Mit diesen Bezeichnungen
erhält man 
\begin{eqnarray*}
\widetilde{T}_{n} & = & \frac{\sqrt{n}}{n_{1}n_{2}}\sum_{i=1}^{n_{1}}\sum_{j=1}^{n_{2}}\mathbf{1}_{\left(-\infty,X_{i}\right]}\left(Y_{j}\right)-\frac{1}{2}\sqrt{n}\\
 & = & \sqrt{n}\int\mathbb{F}_{n_{2},Q_{0}}d\mathbb{F}_{n_{1},P_{0}}-\frac{1}{2}\sqrt{n}\\
 & = & \sqrt{n}\left(\int\mathbb{F}_{n_{2},Q_{0}}d\mathbb{F}_{n_{1},P_{0}}-\int F_{Q_{0}}dP_{0}\right).
\end{eqnarray*}
 Die asymptotische Verteilung von $\widetilde{T}_{n}$ unter $P_{0}^{n_{1}}\otimes Q_{0}^{n_{2}}$
wird nun untersucht.\begin{lemma} Für alle $\varepsilon>0$ gilt
\[
\lim\limits_{n\rightarrow\infty}P_{0}^{n_{1}}\otimes Q_{0}^{n_{2}}\left(\left|T_{n}-\widetilde{T}_{n}\right|>\varepsilon\right)=0.
\]
 \end{lemma} \begin{proof}Es gilt 
\begin{eqnarray*}
 &  & \sqrt{n}\left(\int\mathbb{F}_{n_{2},Q_{0}}d\mathbb{F}_{n_{1},P_{0}}-\int F_{Q_{0}}dP_{0}\right)\\
\textrm{} & = & \sqrt{n}\int\left(\mathbb{F}_{n_{2},Q_{0}}-F_{Q_{0}}\right)dP_{0}-\sqrt{n}\int\left(\mathbb{F}_{n_{1},P_{0}}-F_{P_{0}}\right)_{-}dQ_{0}+R_{n}
\end{eqnarray*}
mit 
\begin{equation}
\lim\limits_{n\rightarrow\infty}P_{0}^{n_{1}}\otimes Q_{0}^{n_{2}}\left(\left|R_{n}\right|>\varepsilon\right)=0\label{eq:wf1}
\end{equation}
 für alle $\varepsilon>0$ nach der starken Version der funktionellen
Delta Methode, vgl. \cite{Vaart:1998}, Seite 299, Example 20.11.
Mit $\left(\mathbb{F}_{n_{1},P_{0}}-F_{P_{0}}\right)_{-}$ wird die
linksseitig stetige Version der Abbildung $t\mapsto\mathbb{F}_{n_{1},P_{0}}(t)-F_{P_{0}}(t)$
bezeichnet. Mit den äquivalenten Umformungen erhält man 
\begin{eqnarray*}
 &  & \sqrt{n}\int\left(\mathbb{F}_{n_{1},P_{0}}-F_{P_{0}}\right)_{-}dQ_{0}\\
 & = & \sqrt{n}\int\left(\frac{1}{n_{1}}\sum_{i=1}^{n_{1}}\mathbf{1}_{\left(-\infty,t\right]}\left(X_{i}\right)-P_{0}\left(\left(-\infty,t\right]\right)\right)_{-}dQ_{0}(t)\\
 & = & \sqrt{n}\int\left(\frac{1}{n_{1}}\sum_{i=1}^{n_{1}}\mathbf{1}_{\left(-\infty,t\right)}\left(X_{i}\right)-P_{0}\left(\left(-\infty,t\right)\right)\right)dQ_{0}\left(t\right)\\
 & = & \frac{\sqrt{n}}{n_{1}}\sum_{i=1}^{n_{1}}\left(\int\mathbf{1}_{\left(-\infty,t\right)}\left(X_{i}\right)dQ_{0}\left(t\right)-\int P_{0}\left(\left(-\infty,t\right)\right)dQ_{0}\left(t\right)\right)\\
 & = & \frac{\sqrt{n}}{n_{1}}\sum_{i=1}^{n_{1}}\left(\int\mathbf{1}_{\left(X_{i},+\infty\right)}\left(t\right)dQ_{0}\left(t\right)-1+k\left(P_{0}\otimes Q_{0}\right)\right)\\
 & = & \frac{\sqrt{n}}{n_{1}}\sum_{i=1}^{n_{1}}\left(Q_{0}\left(\left(X_{i},+\infty\right)\right)-1+k\left(P_{0}\otimes Q_{0}\right)\right)\\
 & = & -\frac{\sqrt{n}}{n_{1}}\sum_{i=1}^{n_{1}}\left(Q_{0}\left(\left(-\infty,X_{i}\right]\right)-k\left(P_{0}\otimes Q_{0}\right)\right),
\end{eqnarray*}
wobei 
\begin{eqnarray*}
\int P_{0}\left(\left(-\infty,t\right)\right)dQ_{0}\left(t\right) & = & \int\left(\int\mathbf{1}_{\left(-\infty,t\right)}\left(x\right)dP_{0}\left(x\right)\right)dQ_{0}\left(t\right)\\
 & = & \int1_{\left\{ \left(x,t\right)\in\mathbb{R}^{2}:x<y\right\} }dP_{0}\otimes Q_{0}\\
 & = & 1-\int1_{\left\{ \left(x,t\right)\in\mathbb{R}^{2}:x\geq y\right\} }dP_{0}\otimes Q_{0}\\
 & = & 1-k\left(P_{0}\otimes Q_{0}\right)
\end{eqnarray*}
 gilt. Analog ergibt sich 
\begin{eqnarray*}
 &  & \sqrt{n}\int\left(\mathbb{F}_{n_{2},Q_{0}}-F_{Q_{0}}\right)\\
 & = & \frac{\sqrt{n}}{n_{2}}\sum_{j=1}^{n_{2}}\int\left(\mathbf{1}_{\left(-\infty,t\right]}\left(Y_{j}\right)-F_{Q_{0}}\left(t\right)\right)dP_{0}\left(t\right)\\
 & = & \frac{\sqrt{n}}{n_{2}}\sum_{j=1}^{n_{2}}\left(\int\mathbf{1}_{\left[Y_{j},+\infty\right)}\left(t\right)dP_{0}\left(t\right)-\int F_{Q_{0}}\left(t\right)dP_{0}\left(t\right)\right)\\
 & = & \frac{\sqrt{n}}{n_{2}}\sum_{j=1}^{n_{2}}\left(P_{0}\left(\left[Y_{j},+\infty\right)\right)-k\left(P_{0}\otimes Q_{0}\right)\right).
\end{eqnarray*}
Insgesamt erhält man 
\begin{eqnarray*}
\textrm{} &  & \sqrt{n}\int\left(\mathbb{F}_{n_{2},Q_{0}}-F_{Q_{0}}\right)dP_{0}-\sqrt{n}\int\left(\mathbb{F}_{n_{1},P_{0}}-F_{P_{0}}\right)_{-}dQ_{0}\\
 & = & \frac{\sqrt{n}}{n_{2}}\sum_{j=1}^{n_{2}}\left(P_{0}\left(\left[Y_{j},+\infty\right)\right)-k\left(P_{0}\otimes Q_{0}\right)\right)\\
 &  & +\frac{\sqrt{n}}{n_{1}}\sum_{i=1}^{n_{1}}\left(Q_{0}\left(\left(-\infty,X_{i}\right]\right)-k\left(P_{0}\otimes Q_{0}\right)\right)\\
 & = & \frac{\sqrt{n}}{n_{1}}\sum_{i=1}^{n_{1}}Q_{0}\left(\left(-\infty,X_{i}\right]\right)+\frac{\sqrt{n}}{n_{2}}\sum_{j=1}^{n_{2}}P_{0}\left(\left[Y_{j},+\infty\right)\right)-2\sqrt{n}k\left(P_{0}\otimes Q_{0}\right)\\
 & = & T_{n}.
\end{eqnarray*}
Hieraus folgt
\begin{eqnarray}
\widetilde{T}_{n} & = & \int\mathbb{F}_{n_{2},Q_{0}}d\mathbb{F}_{n_{1},P_{0}}-\int F_{Q_{0}}dF_{P_{0}}\nonumber \\
 & = & \sqrt{n}\int\left(\mathbb{F}_{n_{2},Q_{0}}-F_{Q_{0}}\right)dP_{0}-\sqrt{n}\int\left(\mathbb{F}_{n_{1},P_{0}}-F_{P_{0}}\right)_{-}dQ_{0}+R_{n}\nonumber \\
 & = & T_{n}+R_{n}.\label{ref:wf2}
\end{eqnarray}
 Aus (\ref{eq:wf1}) und (\ref{ref:wf2}) folgt die Behauptung 
\[
\lim\limits_{n\rightarrow\infty}P_{0}^{n_{1}}\otimes Q_{0}^{n_{2}}\left(\left|T_{n}-\widetilde{T}_{n}\right|>\varepsilon\right)=0
\]
 für alle $\varepsilon>0$.\end{proof} Man kann bei den asymptotisch
optimalen Testfolgen (\ref{erinnerung_asymp_einseit_testfolge}) und
(\ref{erinnerung_asymp_zweiseit_testfolge}) die Teststatistik $T_{n}$
durch $\widetilde{T}_{n}$ ersetzen. Dabei erhält man die Testfolgen,
die zu der asymptotisch optimalen Testfolgen (\ref{erinnerung_asymp_einseit_testfolge})
und (\ref{erinnerung_asymp_zweiseit_testfolge}) für das entsprechende
Testproblem asymptotisch äquivalent sind.

Zur Berechnung des kritischen Wertes für die Tests (\ref{erinnerung_asymp_einseit_testfolge})
und (\ref{erinnerung_asymp_zweiseit_testfolge}) braucht man die asymptotische
Varianz 
\[
\lim_{n\rightarrow\infty}\mbox{Var}_{P_{0}^{n_{1}}\otimes Q_{0}^{n_{2}}}\left(\widetilde{T}_{n}\right)=\lim_{n\rightarrow\infty}\mbox{Var}_{P_{0}^{n_{1}}\otimes Q_{0}^{n_{2}}}\left(T_{n}\right).
\]
 Mit Beispiel \ref{wilcoxon} und Satz \ref{asymp_verhalt_der_T_stat} 
ergibt sich
\begin{eqnarray*}
 &  & \lim_{n\rightarrow\infty}\mbox{Var}_{P_{0}^{n_{1}}\otimes Q_{0}^{n_{2}}}\left(\widetilde{T}_{n}\right)\\
 & = & \frac{1}{1-d}\int\left(\int1_{\left\{ \left(x,y\right)\in\mathbb{R}^{2}:x\geq y\right\} }dQ_{0}(y)-k\left(P_{0}\otimes Q_{0}\right)\right)^{2}dP_{0}\left(x\right)\\
 &  & +\frac{1}{d}\int\left(\int1_{\left\{ \left(x,y\right)\in\mathbb{R}^{2}:x\geq y\right\} }dP_{0}(x)-k\left(P_{0}\otimes Q_{0}\right)\right)^{2}dQ_{0}\left(y\right),
\end{eqnarray*}
wobei $d=\lim\limits_{n\rightarrow\infty}\frac{n_{2}}{n}$ ist. Die
Werte 
\[
w_{1}=\int\left(\int1_{\left\{ \left(x,y\right)\in\mathbb{R}^{2}:x\geq y\right\} }dQ_{0}(y)-k\left(P_{0}\otimes Q_{0}\right)\right)^{2}dP_{0}\left(x\right)
\]
 und 
\[
w_{2}=\int\left(\int1_{\left\{ \left(x,y\right)\in\mathbb{R}^{2}:x\geq y\right\} }dP_{0}(x)-k\left(P_{0}\otimes Q_{0}\right)\right)^{2}dQ_{0}\left(y\right)
\]
 sind unbekannt und müssen geschätzt werden. Es reicht die konsistenten
und erwartungstreuen Schätzer $\widehat{w}_{1,n}$ und $\widehat{w}_{2,n}$
für $w_{1}$ und $w_{2}$ zu finden. Die Testfolgen mit den geschätzten
Werten von $w_{1}$ und $w_{2}$ sind dann asymptotisch äquivalent
zu den Testfolgen mit den exakten Werten von $w_{1}$ und $w_{2}$,
vgl. Beispiel \ref{summ_mises_funk} zum Beweis. Dieses Schätzproblem
wird nun in einer verallgemeinerten Form gelöst, die eine Anwendung
beim Testen der von Mises Funktionale ermöglicht, vgl. Beispiel \ref{von_mises}.

\end{anwendung}\begin{lemma}Sei $P\otimes Q\in\mathcal{P}\otimes\mathcal{Q}$
beliebig aber fest gewählt. Sei $h:\left(\Omega_{1}\times\Omega_{2},\mathcal{A}_{1}\otimes\mathcal{A}_{2}\right)\rightarrow\left(\mathbb{R},\mathcal{B}\right)$
eine Zufallsvariable mit $\int h^{4}dP\otimes Q<\infty$ und $\int h\,dP\otimes Q=0$.
Die Zufallsvariable 
\[
S_{n}=\frac{1}{n_{1}n_{2}\left(n_{2}-1\right)}\sum_{i=1}^{n_{1}}\sum_{j=1}^{n_{2}}\sum_{{k=1\atop k\neq j}}^{n_{2}}h\left(X_{i},Y_{k}\right)h\left(X_{i},Y_{j}\right)
\]
 ist dann ein erwartungstreuer und konsistenter Schätzer für den Wert
\begin{equation}
\int\left(\int h\left(x,y\right)dQ\left(y\right)\right)^{2}dP\left(x\right).\label{Wert_h_two}
\end{equation}
 \end{lemma}\begin{proof}Die Zufallsvariable $S_{n}$ ist ein erwartungstreuer
Schätzer für den Wert (\ref{Wert_h_two}), denn es gilt 
\begin{eqnarray*}
 &  & E_{P^{n_{1}}\otimes Q^{n_{2}}}\left(S_{n}\right)\\
 & = & \frac{1}{n_{1}n_{2}\left(n_{2}-1\right)}\sum_{i=1}^{n_{1}}\sum_{j=1}^{n_{2}}\sum_{{k=1\atop k\neq j}}^{n_{2}}\int h\left(x_{i},y_{k}\right)h\left(x_{i},y_{j}\right)dP\otimes Q^{2}\left(x_{i},y_{k},y_{j}\right)\\
 & = & \frac{1}{n_{1}n_{2}\left(n_{2}-1\right)}\sum_{i=1}^{n_{1}}\sum_{j=1}^{n_{2}}\sum_{{k=1\atop k\neq j}}^{n_{2}}\int\left(\int h\left(x,y\right)dQ\left(y\right)\right)^{2}dP\left(x\right)\\
 & = & \int\left(\int h\left(x,y\right)dQ\left(y\right)\right)^{2}dP\left(x\right)
\end{eqnarray*}
nach dem Satz von Fubini. Es wird nun gezeigt, dass der Schätzer $S_{n}$
für den Wert (\ref{Wert_h_two}) konsistent ist. Es reicht zu zeigen,
dass 
\begin{equation}
\lim_{n\rightarrow\infty}\mbox{Var}_{P^{n_{1}}\otimes Q^{n_{2}}}\left(S_{n}\right)=0\label{eq:konv_var}
\end{equation}
 gilt. Es seien 
\[
I_{n}=\left\{ \left(i,j,k\right):i\in\left\{ 1,\ldots,n_{1}\right\} \mbox{ und }j,k\in\left\{ 1,\ldots,n_{2}\right\} \mbox{ und }k\neq j\right\} 
\]
 die Indexmenge und 
\[
z_{n}=n_{1}n_{2}\left(n_{2}-1\right)
\]
die Anzahl der Elemente von $I_{n}$. Es bezeichne $f_{s}=h\left(X_{i},Y_{k}\right)h\left(X_{i},Y_{j}\right)$
für ein $s=\left(i,j,k\right)\in I_{n}$. Man erhält zunächst
\begin{eqnarray*}
 &  & \mbox{Var}_{P^{n_{1}}\otimes Q^{n_{2}}}\left(S_{n}\right)\\
 & = & \mbox{Var}_{P^{n_{1}}\otimes Q^{n_{2}}}\left(\frac{1}{z_{n}}\sum_{s\in I_{n}}f_{s}\right)\\
 & = & \frac{1}{z_{n}^{2}}\sum_{s\in I_{n}}\mbox{Var}_{P^{n_{1}}\otimes Q^{n_{2}}}\left(f_{s}\right)+\frac{1}{z_{n}^{2}}\sum_{s\in I_{n}}\,\sum_{t\in I_{n}\setminus\left\{ s\right\} }\mbox{Cov}_{P^{n_{1}}\otimes Q^{n_{2}}}\left(f_{s},f_{t}\right).
\end{eqnarray*}
 Für die Summe der Varianzen ergibt sich somit
\[
\frac{1}{z_{n}^{2}}\sum_{s\in I_{n}}\mbox{Var}_{P^{n_{1}}\otimes Q^{n_{2}}}\left(f_{s}\right)=\frac{1}{z_{n}}\mbox{Var}_{P^{n_{1}}\otimes Q^{n_{2}}}\left(f_{\left(1,1,2\right)}\right)\rightarrow0
\]
für $n\rightarrow\infty$, weil 
\begin{eqnarray}
\mbox{Var}_{P^{n_{1}}\otimes Q^{n_{2}}}\left(f_{\left(1,1,2\right)}\right) & = & \int\left(h\left(x_{1},y_{1}\right)h\left(x_{1},y_{2}\right)\right)^{2}dP\otimes Q^{2}\left(x_{1},y_{1},y_{2}\right)\nonumber \\
 & = & \int\left(\int h\left(x_{1},y_{1}\right)^{2}dQ\left(y_{1}\right)\right)^{2}dP\left(x_{1}\right)\nonumber \\
 & \leq & \int\left(\int h\left(x_{1},y_{1}\right)^{4}dQ\left(y_{1}\right)\right)dP\left(x_{1}\right)\nonumber \\
 & = & \int h^{4}dP\otimes Q\nonumber \\
 & < & \infty\label{eq:var_cov_ab}
\end{eqnarray}
nach Voraussetzung gilt. Es bleibt die Summe 
\begin{equation}
\frac{1}{z_{n}^{2}}\sum_{s\in I_{n}}\,\sum_{t\in I_{n}\setminus\left\{ s\right\} }\mbox{Cov}_{P^{n_{1}}\otimes Q^{n_{2}}}\left(f_{s},f_{t}\right)\label{sum_cov}
\end{equation}
der Kovarianzen zu betrachten. Mit der Abschätzung (\ref{eq:var_cov_ab})
erhält man 
\begin{equation}
\left|\mbox{Cov}_{P^{n_{1}}\otimes Q^{n_{2}}}\left(f_{s},f_{t}\right)\right|\leq\left(\mbox{Var}_{P^{n_{1}}\otimes Q^{n_{2}}}\left(f_{s}\right)\mbox{Var}_{P^{n_{1}}\otimes Q^{n_{2}}}\left(f_{t}\right)\right)^{\frac{1}{2}}\leq\int h^{4}dP\otimes Q<\infty.\label{eq:cov_ab}
\end{equation}
Es bezeichne $s=\left(s_{1},s_{2},s_{3}\right)$ und $t=\left(t_{1},t_{2},t_{3}\right)$.
Es gilt 
\[
\mbox{Cov}_{P^{n_{1}}\otimes Q^{n_{2}}}\left(f_{s},f_{t}\right)=0,
\]
falls $s_{k}\neq t_{k}$ für alle $k\in\left\{ 1,2,3\right\} $ sind.
Insgesamt gibt es also mindestens 
\[
n_{2}\left(n_{2}-1\right)n_{1}\left(n_{2}-2\right)\left(n_{2}-3\right)\left(n_{1}-1\right)=z_{n}\left(n_{2}-2\right)\left(n_{2}-3\right)\left(n_{1}-1\right)
\]
 verschiedene Kovarianzen $\mbox{Cov}_{P^{n_{1}}\otimes Q^{n_{2}}}\left(f_{s},f_{t}\right)$
in der Summe (\ref{sum_cov}), die gleich Null sind. Man erhält somit
die Abschätzung
\begin{eqnarray*}
\textrm{} &  & \left|\frac{1}{z_{n}^{2}}\sum_{s\in I_{n}}\,\sum_{t\in I_{n}\setminus\left\{ s\right\} }\mbox{Cov}_{P^{n_{1}}\otimes Q^{n_{2}}}\left(f_{s},f_{t}\right)\right|\\
 & \leq & \frac{1}{z_{n}^{2}}\left(z_{n}\left(z_{n}-1\right)-z_{n}\left(n_{2}-2\right)\left(n_{2}-3\right)\left(n_{1}-1\right)\right)\int h^{4}dP\otimes Q\\
 & = & \frac{1}{z_{n}}\left(z_{n}-1-\left(n_{2}-2\right)\left(n_{2}-3\right)\left(n_{1}-1\right)\right)\int h^{4}dP\otimes Q\\
 & = & \frac{n_{2}\left(n_{2}-1\right)n_{1}-\left(n_{2}-2\right)\left(n_{2}-3\right)\left(n_{1}-1\right)-1}{n_{2}\left(n_{2}-1\right)n_{1}}\int h^{4}dP\otimes Q\\
 & \rightarrow & 0
\end{eqnarray*}
 für $n\rightarrow\infty$. Somit ist die Konvergenz (\ref{eq:konv_var})
gezeigt.\end{proof}

\begin{beisp}Speziell für das Wilcoxon Funktional mit $h=1_{\left\{ \left(x,y\right)\in\mathbb{R}^{2}:x\geq y\right\} }-\frac{1}{2}$
erhält man den erwartungstreuen und konsistenten Schätzer 
\[
\widehat{w}_{1,n}=\frac{1}{n_{1}n_{2}\left(n_{2}-1\right)}\sum_{i=1}^{n_{1}}\sum_{j=1}^{n_{2}}\sum_{{k=1\atop k\neq j}}^{n_{2}}\left(\mathbf{1}_{\left\{ X_{i}\leq Y_{k}\right\} }-\frac{1}{2}\right)\left(\mathbf{1}_{\left\{ X_{i}\leq Y_{j}\right\} }-\frac{1}{2}\right)
\]
für $w_{1}$. Analog ergibt sich der erwartungstreue und konsistente
Schätzer 
\[
\widehat{w}_{2,n}=\frac{1}{n_{2}n_{1}\left(n_{1}-1\right)}\sum_{i=1}^{n_{2}}\sum_{j=1}^{n_{1}}\sum_{{k=1\atop k\neq j}}^{n_{1}}\left(\mathbf{1}_{\left\{ X_{j}\leq Y_{i}\right\} }-\frac{1}{2}\right)\left(\mathbf{1}_{\left\{ X_{k}\leq Y_{i}\right\} }-\frac{1}{2}\right)
\]
 für $w_{2}$. Insgesamt erhält man die asymptotisch optimale Folge
\[
\widetilde{\varphi}_{1n}=\left\{ \begin{array}{cccc}
1 &  & >\\
 & \widetilde{T}_{n} &  & u_{1-\alpha}\left(\frac{n}{n_{1}}\widehat{w}_{1,n}+\frac{n}{n_{2}}\widehat{w}_{2,n}\right)^{\frac{1}{2}}\\
0 &  & \leq
\end{array}\right.
\]
 der einseitigen Tests für das Testproblem (\ref{wt1}) und die asymptotisch
optimale Folge
\[
\widetilde{\varphi}_{2n}=\left\{ \begin{array}{cccc}
1 &  & >\\
 & \left|\widetilde{T}_{n}\right| &  & u_{1-\frac{\alpha}{2}}\left(\frac{n}{n_{1}}\widehat{w}_{1,n}+\frac{n}{n_{2}}\widehat{w}_{2,n}\right)^{\frac{1}{2}}\\
0 &  & \leq
\end{array}\right.
\]
 der zweiseitigen Tests für das Testproblem (\ref{wt2}).\end{beisp} 

\chapter*{Anhang}

\addcontentsline{toc}{chapter}{Anhang}

\section*{Symbol $o$}

\addcontentsline{toc}{section}{Symbol $o$}

\subsection*{Symbol $o$ für Folgen}

Es seien $\left(a_{n}\right)_{n\in\mathbb{N}}$ und $\left(b_{n}\right)_{n\in\mathbb{N}}$
zwei Folgen von reellen Zahlen.\\
Das Symbol $a_{n}\in o\left(b_{n}\right)$ bedeutet, dass $\lim\limits_{n\rightarrow\infty}\frac{a_{n}}{b_{n}}=0$
ist. \\

\subsection*{Symbol $o$ für Funktionen}

Es seien $M$ ein metrischer Raum und $f,g:M\rightarrow\mathbb{R}$
zwei Abbildungen. Das Symbol $f\in o\left(g\right)$ für $x\rightarrow x_{0}$
in $M$ bedeutet dann, dass $\lim\limits_{x\rightarrow x_{0}}\frac{f(x)}{g(x)}=0$
gilt.

Die hier vorgelegte Dissertation habe ich eigenständig und ohne unerlaubte
Hilfe angefertigt. Die Dissertation wurde in der vorgelegten oder
in ähnlicher Form noch bei keiner anderen Institution eingereicht.
Ich habe bisher keine erfolglosen Promotionsversuche unternommen.\vspace*{1.5cm}\\
Düsseldorf, den 17.02.2006\\
Vladimir Ostrovski.
\end{document}